\documentclass[letterpaper]{amsart}

\usepackage[letterpaper]{geometry}
\usepackage{amsmath, amsthm, amsfonts, amsbsy, thmtools, amssymb}

\usepackage{thm-restate}
\usepackage{hyperref} 
\usepackage[mathscr]{euscript}
\usepackage[T1]{fontenc}
\usepackage{tikz}
\usetikzlibrary{arrows}
\usepackage{cleveref}

\numberwithin{equation}{section}

\DeclareMathOperator{\Var}{Var}

\DeclareMathOperator{\Adm}{Adm}

\newtheorem{thm}{Theorem}[section]
\newtheorem{prop}[thm]{Proposition}
\newtheorem{lem}[thm]{Lemma}
\newtheorem{cor}[thm]{Corollary}

\newtheorem{exa}[thm]{Example}
\theoremstyle{remark}
\newtheorem{rem}[thm]{Remark}
\theoremstyle{definition}
\newtheorem{definition}[thm]{Definition}
\newtheorem{assumption}[thm]{Assumption}

\title{Universality for Lozenge Tiling Local Statistics}

\author{Amol Aggarwal}

\begin{document}

	\begin{abstract} 
		
		In this paper we consider uniformly random lozenge tilings of arbitrary domains approximating (after suitable normalization) a closed, simply-connected subset of $\mathbb{R}^2$ with piecewise smooth, simple boundary. We show that the local statistics of this model around any point in the liquid region of its limit shape are given by the infinite-volume, translation-invariant, extremal Gibbs measure of the appropriate slope, thereby confirming a prediction of Cohn-Kenyon-Propp from 2001 in the case of lozenge tilings. Our proofs proceed by locally coupling a uniformly random lozenge tiling with a model of Bernoulli random walks conditioned to never intersect, whose convergence of local statistics has been recently understood by the work of Gorin-Petrov. Central to implementing this procedure is to establish a \emph{local law} for the random tiling, which states that the associated height function is approximately linear on any mesoscopic scale.

	\end{abstract}

	\maketitle 
	
	\tableofcontents

	\section{Introduction} 
	
	\label{Tilings}

	\subsection{Preface}
	
	\label{Tilings1}
	
	Over the past several decades, a significant amount of effort has been directed towards understanding the effect of boundary conditions on the behavior of two-dimensional statistical mechanical models. A well-studied class of systems in this context consists of dimer (or tiling) models, due to their exact solvability. Indeed, for such models, many quantities of interest (including partition and correlation functions) can be expressed as determinants of an inverse Kasteleyn matrix. One may then hope to asymptotically analyze these determinants in order to either predict or prove precise, universal physical phenomena in the large scale limit; see, for instance, the surveys of Kenyon \cite{OD,PDMB}. 
	
	One of the first such results concerned the \emph{limit shape phenomenon} exhibited by domino tilings, stating that the height function of a uniformly random tiling of a large domain is likely to concentrate (after suitable normalization) around a global limit. This was first established for the Aztec diamond domain by Cohn-Elkies-Propp \cite{LSRT}, who provided an exact form for the limit shape. Later, this result was substantially generalized by Cohn-Kenyon-Propp \cite{VPDT} to domino tilings of essentially arbitrary domains. The latter work expressed the limit shape through a variational principle, namely, as the maximizer of a certain (explicit) concave surface tension functional. In the case of lozenge tilings, this was later rewritten by Kenyon-Okounkov \cite{LSCE} as the solution to a non-linear first order partial differential equation that can in many cases be solved through the method of characteristics. 
	
	A salient feature of these limit shapes is that they are inhomogeneous for generic choices of the boundary, that is, the local densities of tiles can (non-negligibly) differ in different regions of the rescaled domain $\mathfrak{R}$. In fact, for certain boundary data, these local densities can exhibit sharp phase transitions. These are commonly known as \emph{arctic boundaries}, across which the behavior of the tiling changes from asymptotically deterministic (\emph{frozen}) to random (\emph{liquid} or \emph{bulk}). In the context of dimer models, such a phenomenon was first established by Jockusch-Propp-Shor \cite{RTAC} for uniformly random domino tilings of an Aztec diamond, although a similar notion had been studied earlier for Ising crystals (see the books of Cerf \cite{CPM} and Dobrushin-Koteck\'{y}-Shlosman \cite{GLI}). 
	
	Thus a question of interest, originally mentioned by Kasteleyn \cite{SDL} in 1961, is how the boundary data affect the local statistics (joint law of nearly neighboring tiles) of a random tiling. A precise predicted answer to this question was proposed as Conjecture 13.5 of \cite{VPDT}, which suggests that these local statistics should be determined by the gradient of the global limiting height function $\mathcal{H}: \mathfrak{R} \rightarrow \mathbb{R}$ for the tiling model. More specifically, around some point $\mathfrak{v}$ in the liquid region of $\mathfrak{R}$, they should be given by the infinite-volume, translation-invariant, extremal Gibbs measure with slope equal to $\nabla \mathcal{H} (\mathfrak{v})$; the uniqueness of such a measure was established by Sheffield in \cite{RS}. For dimer models, these measures are now well-understood, as the work of Kenyon-Okounkov-Sheffield \cite{DA} expresses them as determinantal point processes with explicit kernels; in the case of lozenge tilings, they are known as certain \emph{extended discrete sine processes}.
	
	For lozenge and domino tilings, this prediction has been shown to hold true for several special choices of the boundary data. These include ($q$-weighted) lozenge tilings of the hexagon \cite{DOA,DBPP,NPOE}, of domains covered by finitely many trapezoids \cite{URB, AR}, of sectors of the plane \cite{AGRSD,SGWM}, and of bounded perturbations of these domains \cite{LLSBHP}; skew plane partitions \cite{RSPPPBW,LRSPBW,CFP,RSPP} with periodic weights \cite{PPTW} and symmetry constraints \cite{FBPA}; and domino tilings of the Aztec diamond \cite{ASD,CP} with periodic weights \cite{ST,TPMVOP} and of certain families of polygonal domains with no frozen regions \cite{CI,DPD,DT}. 
	
	Most of these results (with the exceptions of \cite{URB,LLSBHP}) are based on an asymptotic analysis of the inverse Kasteleyn matrix (or some other, boundary-specific determinantal structure). However, in many situations, the entries of this matrix can be unstable under boundary perturbations. Partly for this reason, results on local statistics had until now remained unavailable for any tiling model under general boundary data. 
	
	In this paper we establish the local statistics prediction for lozenge tilings of arbitrary domains approximating (after suitable normalization) a closed, simply-connected subset of $\mathbb{R}^2$ with piecewise smooth, simple boundary; see \Cref{localconverge} below. This attains the same level of universality for convergence of lozenge tiling local statistics as was shown for the global height profile in \cite{VPDT}. 
	
	Unlike most of the previously mentioned works, our method does not make direct use of a Kasteleyn matrix. Instead, we proceed by locally comparing a uniformly random lozenge tiling of a given domain to an ensemble of Bernoulli random walks conditioned to never intersect. The latter model was introduced by K\"{o}nig-O'Connell-Roch in \cite{NRWTQPE} and can be viewed as a random lozenge tiling, whose boundary conditions are partly free  (but not uniform). In particular, when run under suitable initial data and for a sufficiently long time, its limiting local statistics are also governed by an extended discrete sine process.
		
	The benefit to this random walk model is that its algebraic structure appears to be more amenable to asymptotic analysis than does the Kasteleyn matrix. In particular, the recent work \cite{ULSRW} of Gorin-Petrov provides a general result on the convergence of its local statistics, a special case which can essentially be stated as follows. Suppose that\footnote{Given sequences $(a_1, a_2, \ldots )$ and $(b_1, b_2, \ldots )$ of positive real numbers, we write $a_N \ll b_N$ if $\lim_{N \rightarrow \infty} \frac{a_N}{b_N} = 0$.} $1 \ll U \ll T \ll V$ are integers and that $\textbf{a}$ is an initial sequence of particle locations that is of approximately constant density on any length $U$ subinterval of $[x_0 - V, x_0 + V]$. Then the local statistics of the non-intersecting random walk model, run for time $T$ with initial data $\textbf{a}$, converge around site $x_0$ to an extended discrete sine process. This result will be used extensively in our analysis of lozenge tiling local statistics. 
	
	More specifically, let $\mathscr{M}$ denote a uniformly random tiling of a large domain $R$, and fix some vertex $v_0 = (x_0, y_0) \in R$. Our framework will proceed by coupling $\mathscr{M}$ with a non-intersecting random walk ensemble, in such a way that the two models coincide in a large neighborhood of $v_0$ with high probability. Let us outline how we exhibit such a coupling. In what follows, we let $N$ denote the diameter of $R$; fix an integer $1 \ll T \ll N$; and define the vertex $u_0 = v - (0, T) = (x_0, y_0 - T) \in R$.

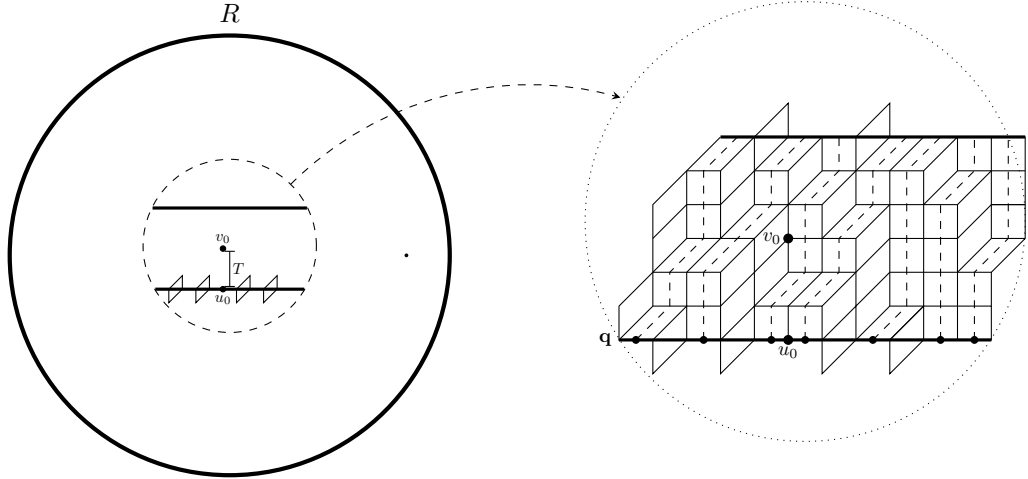
\begin{figure}

	\begin{center}

		\begin{tikzpicture}[
		>=stealth,
		auto,
		style={
			scale = .9
		}
		]

			\draw[black, ultra thick] (0, 0) circle[radius = 3.25];
			
			\draw[] (0, 3.5) circle[radius = 0] node[above = -5]{$R$};

			\path[draw, black] (-.9, -.5) -- (-.7, -.3) -- (-.7, -.5) -- (-.9, -.7) -- (-.9, -.5);
			\path[draw, black] (-.5, -.5) -- (-.3, -.3) -- (-.3, -.5) -- (-.5, -.7) -- (-.5, -.5);
			\path[draw, black] (.1, -.5) -- (.3, -.3) -- (.3, -.5) -- (.1, -.7) -- (.1, -.5);
			\path[draw, black] (.5, -.5) -- (.7, -.3) -- (.7, -.5) -- (.5, -.7) -- (.5, -.5);
			
			\filldraw[fill=black] (2.61, 0) circle [radius = .02];
			
			\filldraw[fill=black] (-.1, .1) circle [radius = .04] node[above, scale = .6]{$v_0$};
			\filldraw[fill=black] (-.1, -.5) circle [radius = .04] node[below, scale = .6]{$u_0$};
			
			\draw[-, black] (.075, .06) -- (-.075, .06);
			\draw[-, black] (0, .06) -- (0, -.46); 
			\draw[-, black] (.075, -.46) -- (-.075, -.46);
			\draw[] (0, -.2) circle[radius = 0] node[right = -1, scale = .6]{$T$};
			
			\draw[-, black, dashed] (0, .14) circle [radius = 1.28];
			\draw[-, black, very thick] (-1.1, -.5) -- (1.1, -.5);
			\draw[-, black, very thick] (-1.14, .7) -- (1.14, .7);
			
			\draw[->, black, dashed] (.9, 1.05) arc (135:74.5:5);

			\draw[-, dotted] (8.5, .5) circle [radius = 3.25];
			
			\filldraw[fill=black] (8.25, .25) circle [radius = .065] node[left, scale = .75]{$v_0$};
			\filldraw[fill=black] (8.25, -1.25) circle [radius = .065] node[below, scale = .75]{$u_0$};
			
			\draw[-, very thick] (5.75, -1.25) node[left, scale = .75]{$\textbf{q}$} -- (11.25, -1.25);
			\draw[-, very thick] (7.25, 1.75) -- (11.75, 1.75);
			
			\path[draw, black] (6.25, -1.25) -- (6.75, -.75) -- (6.75, -1.25) -- (6.25, -1.75) -- (6.25, -1.25);
			\path[draw, black] (7.25, -1.25) -- (7.75, -.75) -- (7.75, -1.25) -- (7.25, -1.75) -- (7.25, -1.25);
			\path[draw, black] (8.75, -1.25) -- (9.25, -.75) -- (9.25, -1.25) -- (8.75, -1.75) -- (8.75, -1.25);
			\path[draw, black] (9.75, -1.25) -- (10.25, -.75) -- (10.25, -1.25) -- (9.75, -1.75) -- (9.75, -1.25);
			\path[draw, black] (6.75, -.75) -- (6.25, -.75) -- (5.75, -1.25) -- (6.25, -1.25);
			\path[draw, black] (6.75, -1.25) -- (7.25, -1.25) -- (7.25, -.75) -- (6.75, -.75);
			\path[draw, black] (7.75, -.75) -- (8.25, -.75) -- (8.25, -1.25) -- (7.75, -1.25);
			\path[draw, black] (8.25, -1.25) -- (8.75, -1.25) -- (8.75, -.75) -- (8.25, -.75);
			\path[draw, black] (9.75, -1.25) -- (9.25, -1.25) -- (9.75, -.75) -- (10.25, -.75);
			\path[draw, black] (10.25, -1.25) -- (9.75, -1.25) -- (10.25, -.75) -- (10.75, -.75);
			\path[draw, black] (10.75, -1.25) -- (11.25, -1.25) -- (11.25, -.75) -- (10.75, -.75) -- (10.75, -1.25);
			
			\path[draw, black, dashed] (6, -1.25) -- (6.5, -.75) -- (6.5, -.25) -- (7, .25) -- (7, .75) -- (7, 1.25) -- (7.5, 1.75);
			\path[draw, black, dashed] (7, -1.25) -- (7, -.75) -- (7, -.25) -- (7.5, .25) -- (8, .75) -- (8, 1.25) -- (8, 1.25) -- (8.5, 1.75);
			\path[draw, black, dashed] (8, -1.25) -- (8, -.75) -- (8.5, -.25) -- (8.5, .25)  -- (8.5, .75) -- (9, 1.25) -- (9, 1.75);
			\path[draw, black, dashed] (8.5, -1.25) -- (8.5, -.75) -- (9, -.25) -- (9, .25) -- (9.5, .75) -- (9.5, 1.25) -- (10, 1.75);
			\path[draw, black, dashed] (9.5, -1.25) -- (10, -.75) -- (10, -.25) -- (10, .25) -- (10, .75) -- (10, 1.25) -- (10.5, 1.75);
			\path[draw, black, dashed] (10.5, -1.25) -- (10.5, -.75) -- (10.5, -.25) -- (10.5, .25) -- (10.5, .75) -- (11, 1.25) -- (11, 1.75);
			\path[draw, black, dashed] (11, -1.25) -- (11, -.75) -- (11, -.25) -- (11.5, .25) -- (11.5, .75) -- (11.5, 1.25) -- (11.5, 1.75);
			
			\path[draw, black] (6.25, -.75) -- (6.25, -.25) -- (6.75, .25) -- (6.75, .75) -- (6.75, 1.25) -- (7.25, 1.75);
			\path[draw, black] (6.75, -.75) -- (6.75, -.25) -- (7.25, .25) -- (7.25, .75) -- (7.25, 1.25) -- (7.75, 1.75);
			\path[draw, black] (7.25, -.75) -- (7.25, -.25) -- (7.75, .25) -- (8.25, .75) -- (8.25, 1.25) -- (8.75, 1.75);
			\path[draw, black] (7.25, .25) -- (7.75, .75) -- (7.75, 1.25) -- (8.25, 1.75);
			\path[draw, black] (7.75, -.75) -- (8.25, -.25) -- (8.25, .25) -- (8.25, .75) -- (8.75, 1.25) -- (8.75, 1.75);
			\path[draw, black] (8.25, -.75) -- (8.75, -.25) -- (8.75, .25) -- (8.75, .75) -- (9.25, 1.25) -- (9.25, 1.75);
			\path[draw, black] (8.75, -.75) -- (9.25, -.25) -- (9.25, .25) -- (9.75, .75) -- (9.75, 1.25) -- (10.25, 1.75);
			\path[draw, black] (8.75, .25) -- (9.25, .75) -- (9.25, 1.25) -- (9.75, 1.75);
			\path[draw, black] (9.75, -.75) -- (9.75, -.25) -- (9.75, .25) -- (9.75, .75);
			\path[draw, black] (10.25, -.75) -- (10.25, 1.25) -- (10.75, 1.75);
			\path[draw, black] (10.25, .75) -- (10.75, 1.25) -- (10.75, 1.75);
			\path[draw, black] (10.75, -.75) -- (10.75, .75) -- (11.25, 1.25) -- (11.25, 1.75);
			\path[draw, black] (10.75, -.25) -- (11.25, .25) -- (11.25, 1.25);
			\path[draw, black] (11.25, -.75) -- (11.25, -.25) -- (11.75, .25) -- (11.75, 1.75) -- (11.5, 1.75);
			\path[draw, black] (6.25, -.25) -- (7.25, -.25);
			\draw[draw, black] (7.25, -.75) -- (8.25, .25) -- (9.25, .25);
			\path[draw, black] (8.25, -.25) -- (9.25, -.25);
			\path[draw, black] (9.75, -.25) -- (11.25, -.25);
			\path[draw, black] (11.75, .75) -- (11.25, .75) -- (10.75, .25) -- (9.75, .25) -- (9.25, -.25) -- (9.25, -.75) -- (9.75, -.25);
			\path[draw, black] (6.75, .25) -- (7.75, .25) -- (7.75, -.75);
			\path[draw, black] (11.25, .25) -- (11.75, .25);
			\path[draw, black] (6.75, .75) -- (7.25, .75) -- (7.75, 1.25) -- (8.25, 1.25);
			\path[draw, black] (7.75, .75) -- (8.75, .75); 
			\path[draw, black] (9.25, .75) -- (10.75, .75); 
			\path[draw, black] (11.25, .75) -- (11.75, .75);
			\path[draw, black] (6.75, 1.25) -- (7.25, 1.25);
			\path[draw, black] (8.75, 1.25) -- (10.25, 1.25);
			\path[draw, black] (10.75, 1.25) -- (11.75, 1.25);  
			\path[draw, black] (8.25, 1.75) -- (8.25, 2.25) -- (7.75, 1.75) -- (7.75, 1.25);
			\path[draw, black] (9.75, 1.75) -- (9.75, 2.25) -- (9.25, 1.75);
			\path[draw, black] (6.75, 1.25) -- (6.25, .75) -- (6.25, -.25) -- (5.75, -.75) -- (5.75, -1.25);
			\path[draw, black] (6.75, .75) -- (6.25, .25); 
			
			\filldraw[fill=black] (6, -1.25) circle [radius = .05];
			\filldraw[fill=black] (7, -1.25) circle [radius = .05];
			\filldraw[fill=black] (8, -1.25) circle [radius = .05];
			\filldraw[fill=black] (8.5, -1.25) circle [radius = .05];
			\filldraw[fill=black] (9.5, -1.25) circle [radius = .05];
			\filldraw[fill=black] (10.5, -1.25) circle [radius = .05];
			\filldraw[fill=black] (11, -1.25) circle [radius = .05];

		\end{tikzpicture}
		
	\end{center}

	\caption{\label{pathslocal} Depicted to the left is the domain $R$ to be tiled; the vertices $v_0 = (x_0, y_0) \in R$ and $u_0 = (x_0, y_0 - T) \in R$; and a neighborhood of $v_0$. Depicted to the right is part of the tiling in this neighborhood and the associated non-intersecting path ensemble $\textbf{Q}$ (dashed).}

\end{figure} 

	First, we interpret $\mathscr{M}$ as an ensemble $\textbf{Q}$ of non-intersecting paths, and let $\textbf{q}$ denote the locations where these paths intersect the horizontal line $\{ y = y_0 - T \}$; see \Cref{pathslocal}. Second, we introduce two particle configurations $\textbf{p}$ and $\textbf{r}$ that coincide with $\textbf{q}$ inside a large neighborhood $u_0$, but are obtained by perturbing $\textbf{q}$ to the left and right, respectively, outside of this neighborhood. Third, we define two non-intersecting random walk ensembles $\textbf{P}$ and $\textbf{R}$ with initial data $\textbf{p}$ and $\textbf{r}$, respectively, and show that there exists a coupling between $(\textbf{P}, \textbf{Q}, \textbf{R})$ such that $\textbf{Q}$ is likely bounded between $\textbf{P}$ and $\textbf{R}$; see \Cref{paths2}. Fourth, we use identities from \cite{ULSRW} to prove that the expected difference between the height functions associated with $\textbf{P}$ and $\textbf{R}$ tends to $0$ in a large neighborhood of $u_0$ (containing $v_0$). Fifth, using the ordering between $(\textbf{P}, \textbf{Q}, \textbf{R})$ and a Markov bound, we conclude that one can couple these three ensembles to coincide near $v_0$ with high probability.

	\begin{figure}

		\begin{center}

			\begin{tikzpicture}[
			>=stealth,
			auto,
			style={
				scale = .75
			}
			]

			\fill[fill=white!75!gray] (6.75, 3.5) -- (9.75, 3.5) -- (9.75, 6.5) -- (6.75, 6.5) -- (6.75, 3.5);
			
			\draw[<->, black, thick] (-1, 3.5) node[left]{$\textbf{p}$} -- (17.5, 3.5);
			\draw[<->, black, thick] (-1, 6.5) -- (17.5, 6.5);
			\draw[] (-1, 5) circle [radius = 0] node[left]{$\textbf{P}$};
			
			\filldraw[fill=black] (8.25, 3.5) circle [radius = .085] node[below = 1, scale = .75]{$u_0$};
			\filldraw[fill=black] (8.25, 5) circle [radius = .085] node[below = 1, scale = .75]{$v_0$};

			\path[draw, black, dashed] (-.5, 3.5) -- (-.5, 4) -- (-.5, 4.5) -- (-.5, 5) -- (-.5, 5.5) -- (-.5, 6) -- (0, 6.5);
			\path[draw, black, dashed] (1.5, 3.5) -- (1.5, 4) -- (1.5, 4.5) -- (2, 5) -- (2.5, 5.5) -- (3, 6) -- (3, 6.5); 
			\path[draw, black, dashed] (2, 3.5) -- (2.5, 4) -- (2.5, 4.5) -- (3, 5) -- (3, 5.5) -- (3.5, 6) -- (4, 6.5);
			\path[draw, black, dashed] (3, 3.5) -- (3.5, 4) -- (4, 4.5) -- (4, 5) -- (4.5, 5.5) -- (4.5, 6) -- (4.5, 6.5);
			\path[draw, black, dashed] (4.5, 3.5) -- (4.5, 4) -- (4.5, 4.5) -- (4.5, 5) -- (5, 5.5) -- (5, 6) -- (5.5, 6.5);
			\path[draw, black, dashed] (5.5, 3.5) -- (5.5, 4) -- (5.5, 4.5) -- (6, 5) -- (6, 5.5) -- (6, 6) -- (6, 6.5);
			\path[draw, black, dashed] (6, 3.5) -- (6.5, 4) -- (6.5, 4.5) -- (7, 5) -- (7, 5.5) -- (7, 6) -- (7.5, 6.5);
			\path[draw, black, dashed] (7, 3.5) -- (7, 4) -- (7, 4.5) -- (7.5, 5) -- (8, 5.5) -- (8, 6) -- (8.5, 6.5);
			\path[draw, black, dashed] (8, 3.5) -- (8, 4) -- (8.5, 4.5) -- (8.5, 5) -- (8.5, 5.5) -- (9, 6) -- (9, 6.5);
			\path[draw, black, dashed] (8.5, 3.5) -- (8.5, 4) -- (9, 4.5) -- (9, 5) -- (9.5, 5.5) -- (9.5, 6) -- (10, 6.5);
			\path[draw, black, dashed] (9.5, 3.5) -- (10, 4) -- (10, 4.5) -- (10, 5) -- (10, 5.5) -- (10, 6) -- (10.5, 6.5);
			\path[draw, black, dashed] (10.5, 3.5) -- (10.5, 4) -- (10.5, 4.5) -- (10.5, 5) -- (10.5, 5.5) -- (11, 6) -- (11, 6.5);
			\path[draw, black, dashed] (11, 3.5) -- (11, 4) -- (11, 4.5) -- (11.5, 5) -- (11.5, 5.5) -- (11.5, 6) -- (12, 6.5);
			\path[draw, black, dashed] (11.5, 3.5) -- (12, 4) -- (12, 4.5) -- (12.5, 5) -- (13, 5.5) -- (13, 6) -- (13, 6.5);
			\path[draw, black, dashed] (12, 3.5) -- (12.5, 4) -- (13, 4.5) -- (13, 5) -- (13.5, 5.5) -- (14, 6) -- (14, 6.5);
			
			\filldraw[fill=black] (-.5, 3.5) circle [radius = .06];
			\filldraw[fill=black] (1.5, 3.5) circle [radius = .06];
			\filldraw[fill=black] (2, 3.5) circle [radius = .06];
			\filldraw[fill=black] (3, 3.5) circle [radius = .06];
			\filldraw[fill=white] (4.5, 3.5) circle [radius = .06];
			\filldraw[fill=white] (5.5, 3.5) circle [radius = .06];
			\filldraw[fill=white] (6, 3.5) circle [radius = .06];
			\filldraw[fill=white] (7, 3.5) circle [radius = .06];
			\filldraw[fill=white] (8, 3.5) circle [radius = .06];
			\filldraw[fill=white] (8.5, 3.5) circle [radius = .06];
			\filldraw[fill=white] (9.5, 3.5) circle [radius = .06];
			\filldraw[fill=white] (10.5, 3.5) circle [radius = .06];
			\filldraw[fill=white] (11, 3.5) circle [radius = .06];
			\filldraw[fill=black] (11.5, 3.5) circle [radius = .06];
			\filldraw[fill=black] (12, 3.5) circle [radius = .06];
			\filldraw[fill=black] (13, 3.5) circle [radius = .06];
			\filldraw[fill=black] (14, 3.5) circle [radius = .06];
			
			\path[draw, black, dashed] (13, 3.5) -- (13, 4) -- (13.5, 4.5) -- (14, 5) -- (14, 5.5) -- (14.5, 6) -- (15, 6.5);
			
			\path[draw, black, dashed] (14, 3.5) -- (14, 4) -- (14.5, 4.5) -- (14.5, 5) -- (14.5, 5.5) -- (15, 6) -- (15.5, 6.5);

			\fill[fill=white!75!gray] (6.75, 0) -- (9.75, 0) -- (9.75, 3) -- (6.75, 3) -- (6.75, 0);
			
			\draw[<->, black, thick] (-1, 0) node[left]{$\textbf{q}$} -- (17.5, 0);
			\draw[<->, black, thick] (-1, 3) -- (17.5, 3);
			\draw[] (-1, 1.5) circle [radius = 0] node[left]{$\textbf{Q}$};
			
			\path[draw, black, dashed] (.5, 0) -- (.5, .5) -- (.5, 1) -- (1, 1.5) -- (1, 2) -- (1, 2.5) -- (1, 3);
			\path[draw, black, dashed] (2, 0) -- (2, .5) -- (2, 1) -- (2.5, 1.5) -- (2.5, 2) -- (3, 2.5) -- (3.5, 3);
			\path[draw, black, dashed] (2.5, 0) -- (2.5, .5) -- (2.5, 1) -- (3, 1.5) -- (3.5, 2) -- (4, 2.5) -- (4.5, 3);
			\path[draw, black, dashed] (3.5, 0) -- (3.5, .5) -- (4, 1) -- (4, 1.5) -- (4.5, 2) -- (5, 2.5) -- (5, 3);
			\path[draw, black, dashed] (4.5, 0) -- (4.5, .5) -- (4.5, 1) -- (5, 1.5) -- (5, 2) -- (5.5, 2.5) -- (5.5, 3);
			\path[draw, black, dashed] (5.5, 0) -- (5.5, .5) -- (6, 1) -- (6, 1.5) -- (6, 2) -- (6.5, 2.5) -- (6.5, 3);
			\path[draw, black, dashed] (6, 0) -- (6.5, .5) -- (6.5, 1) -- (7, 1.5) -- (7, 2) -- (7, 2.5) -- (7.5, 3);
			\path[draw, black, dashed] (7, 0) -- (7, .5) -- (7, 1) -- (7.5, 1.5) -- (8, 2) -- (8, 2.5) -- (8.5, 3);
			\path[draw, black, dashed] (8, 0) -- (8, .5) -- (8.5, 1) -- (8.5, 1.5) -- (8.5, 2) -- (9, 2.5) -- (9, 3);
			\path[draw, black, dashed] (8.5, 0) -- (8.5, .5) -- (9, 1) -- (9, 1.5) -- (9.5, 2) -- (9.5, 2.5) -- (10, 3);
			\path[draw, black, dashed] (9.5, 0) -- (10, .5) -- (10, 1) -- (10, 1.5) -- (10, 2) -- (10, 2.5) -- (10.5, 3);
			\path[draw, black, dashed] (10.5, 0) -- (10.5, .5) -- (10.5, 1) -- (10.5, 1.5) -- (11, 2) -- (11, 2.5) -- (11, 3);
			\path[draw, black, dashed] (11, 0) -- (11, .5) -- (11, 1) -- (11.5, 1.5) -- (11.5, 2) -- (11.5, 2.5) -- (12, 3);
			\path[draw, black, dashed] (12, 0) -- (12.5, .5) -- (12.5, 1) -- (12.5, 1.5) -- (13, 2) -- (13.5, 2.5) -- (13.5, 3);
			\path[draw, black, dashed] (12.5, 0) -- (13, .5) -- (13.5, 1) -- (14, 1.5) -- (14, 2) -- (14.5, 2.5) -- (14.5, 3);
			\path[draw, black, dashed] (13.5, 0) -- (13.5, .5) -- (14, 1) -- (14.5, 1.5) -- (14.5, 2) -- (15, 2.5) -- (15.5, 3);
			\path[draw, black, dashed] (15, 0) -- (15, .5) -- (15, 1) -- (15.5, 1.5) -- (16, 2) -- (16, 2.5) -- (16, 3);

			\filldraw[fill=black] (.5, 0) circle [radius = .06];
			\filldraw[fill=black] (2, 0) circle [radius = .06];
			\filldraw[fill=black] (2.5, 0) circle [radius = .06];
			\filldraw[fill=black] (3.5, 0) circle [radius = .06];
			\filldraw[fill=white] (4.5, 0) circle [radius = .06];
			\filldraw[fill=white] (5.5, 0) circle [radius = .06];
			\filldraw[fill=white] (6, 0) circle [radius = .06];
			\filldraw[fill=white] (7, 0) circle [radius = .06];
			\filldraw[fill=white] (8, 0) circle [radius = .06];
			\filldraw[fill=white] (8.5, 0) circle [radius = .06];
			\filldraw[fill=white] (9.5, 0) circle [radius = .06];
			\filldraw[fill=white] (10.5, 0) circle [radius = .06];
			\filldraw[fill=white] (11, 0) circle [radius = .06];
			\filldraw[fill=black] (12, 0) circle [radius = .06];
			\filldraw[fill=black] (12.5, 0) circle [radius = .06];
			\filldraw[fill=black] (13.5, 0) circle [radius = .06];
			\filldraw[fill=black] (15, 0) circle [radius = .06];
			
			\filldraw[fill=black] (8.25, 0) circle [radius = .085] node[below = 1, scale = .75]{$u_0$};
			\filldraw[fill=black] (8.25, 1.5) circle [radius = .085] node[below = 1, scale = .75]{$v_0$};

			\fill[fill=white!75!gray] (6.75, -.5) -- (9.75, -.5) -- (9.75, -3.5) -- (6.75, -3.5) -- (6.75, -3.5);
			
			\draw[<->, black, thick] (-1, -.5)  -- (17.5, -.5);
			\draw[<->, black, thick] (-1, -3.5) node[left]{$\textbf{r}$} -- (17.5, -3.5);
			\draw[] (-1, -2) circle [radius = 0] node[left]{$\textbf{R}$};
			
			\filldraw[fill=black] (8.25, -3.5) circle [radius = .085] node[below = 1, scale = .75]{$u_0$};
			\filldraw[fill=black] (8.25, -2) circle [radius = .085] node[below = 1, scale = .75]{$v_0$};
			
			\path[draw, black, dashed] (1.5, -3.5) -- (1.5, -3) -- (1.5, -2.5) -- (2, -2) -- (2.5, -1.5) -- (2.5, -1) -- (2.5, -.5);
			\path[draw, black, dashed] (2.5, -3.5) -- (2.5, -3) -- (3, -2.5) -- (3, -2) -- (3.5, -1.5) -- (3.5, -1) -- (4, -.5);
			\path[draw, black, dashed] (3, -3.5) -- (3.5, -3) -- (3.5, -2.5) -- (3.5, -2) -- (4, -1.5) -- (4, -1) -- (4.5, -.5);
			\path[draw, black, dashed] (4, -3.5) -- (4, -3) -- (4, -2.5) -- (4.5, -2) -- (4.5, -1.5) -- (5, -1) -- (5, -.5);
			\path[draw, black, dashed] (4.5, -3.5) -- (5, -3) -- (5, -2.5) -- (5.5, -2) -- (5.5, -1.5) -- (5.5, -1) -- (6, -.5);
			\path[draw, black, dashed] (5.5, -3.5) -- (5.5, -3) -- (6, -2.5) -- (6, -2) -- (6.5, -1.5) -- (6.5, -1) -- (6.5, -.5);
			\path[draw, black, dashed] (6, -3.5) -- (6.5, -3) -- (6.5, -2.5) -- (7, -2) -- (7, -1.5) -- (7, -1) -- (7.5, -.5);
			\path[draw, black, dashed] (7, -3.5) -- (7, -3) -- (7, -2.5) -- (7.5, -2) -- (8, -1.5) -- (8, -1) -- (8.5, -.5);
			\path[draw, black, dashed] (8, -3.5) -- (8, -3) -- (8.5, -2.5) -- (8.5, -2) -- (8.5, -1.5) -- (9, -1) -- (9, -.5);
			\path[draw, black, dashed] (8.5, -3.5) -- (8.5, -3) -- (9, -2.5) -- (9, -2) -- (9.5, -1.5) -- (9.5, -1) -- (10, -.5);
			\path[draw, black, dashed] (9.5, -3.5) -- (10, -3) -- (10, -2.5) -- (10, -2) -- (10, -1.5) -- (10, -1) -- (10.5, -.5);
			\path[draw, black, dashed] (10.5, -3.5) -- (10.5, -3) -- (11, -2.5) -- (11, -2) -- (11.5, -1.5) -- (11.5, -1) -- (11.5, -.5);
			\path[draw, black, dashed] (11, -3.5) -- (11.5, -3) -- (11.5, -2.5) -- (12, -2) -- (12, -1.5) -- (12.5, -1) -- (12.5, -.5);
			\path[draw, black, dashed] (12.5, -3.5) -- (12.5, -3) -- (12.5, -2.5) -- (13, -2) -- (13, -1.5) -- (13.5, -1) -- (14, -.5);
			\path[draw, black, dashed] (13, -3.5) -- (13, -3) -- (13.5, -2.5) -- (14, -2) -- (14, -1.5) -- (14.5, -1) -- (15, -.5);
			\path[draw, black, dashed] (14, -3.5) -- (14, -3) -- (14, -2.5) -- (14.5, -2) -- (15, -1.5) -- (15.5, -1) -- (15.5, -.5);
			\path[draw, black, dashed] (16, -3.5) -- (16, -3) -- (16, -2.5) -- (16.5, -2) -- (16.5, -1.5) -- (17, -1) -- (17, -.5);

			\filldraw[fill=black] (1.5, -3.5) circle [radius = .06];
			\filldraw[fill=black] (2.5, -3.5) circle [radius = .06];
			\filldraw[fill=black] (3, -3.5) circle [radius = .06];
			\filldraw[fill=black] (4, -3.5) circle [radius = .06];
			\filldraw[fill=white] (4.5, -3.5) circle [radius = .06];
			\filldraw[fill=white] (5.5, -3.5) circle [radius = .06];
			\filldraw[fill=white] (6, -3.5) circle [radius = .06];
			\filldraw[fill=white] (7, -3.5) circle [radius = .06];
			\filldraw[fill=white] (8, -3.5) circle [radius = .06];
			\filldraw[fill=white] (8.5, -3.5) circle [radius = .06];
			\filldraw[fill=white] (9.5, -3.5) circle [radius = .06];
			\filldraw[fill=white] (10.5, -3.5) circle [radius = .06];
			\filldraw[fill=white] (11, -3.5) circle [radius = .06];
			\filldraw[fill=black] (12.5, -3.5) circle [radius = .06];
			\filldraw[fill=black] (13, -3.5) circle [radius = .06];
			\filldraw[fill=black] (14, -3.5) circle [radius = .06];
			\filldraw[fill=black] (16, -3.5) circle [radius = .06];

			\end{tikzpicture}
			
		\end{center}

		\caption{\label{paths2} Depicted in the middle is the non-intersecting path ensemble $\textbf{Q}$ with initial data $\textbf{q}$. Depicted above and below are the ensembles $\textbf{P}$ and $\textbf{R}$ whose initial data $\textbf{p}$ and $\textbf{r}$ are obtained by perturbing the particles in $\textbf{q}$ away from $u_0$ to the left and right, respectively (black vertices). Near $u_0$, these initial data coincide with $\textbf{q}$ (white vertices). Each path in $\textbf{Q}$ is between the corresponding ones in $\textbf{P}$ and $\textbf{R}$. In the shaded region, all three path ensembles coincide.  }

	\end{figure}
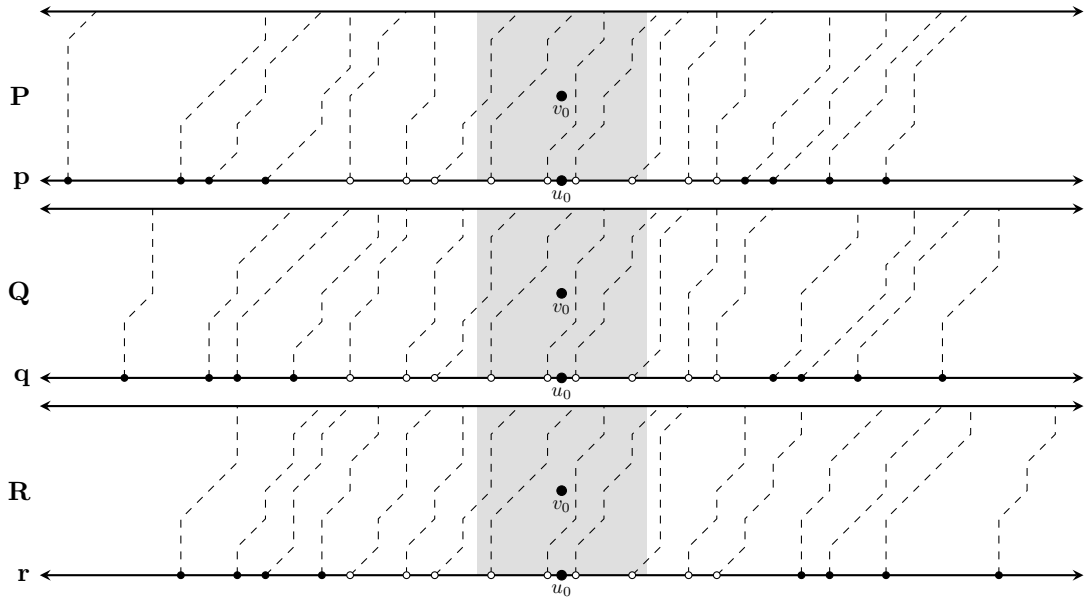

	Next, using the results of \cite{ULSRW}, we show that the local statistics of $\textbf{P}$ converge, after run for time $T \gg 1$, to an extended sine process around site $x_0$. Since $u_0 = (x_0, y_0 - T) = v_0 - (0, T)$, it follows that the local statistics of $\textbf{P}$ converge to this limit around $v_0$. Due to the coupling between $(\textbf{P}, \textbf{Q})$ around $v_0$, we then deduce that the same statement holds for $\mathscr{M}$. 
	
	Fundamental to implementing this procedure will be to first establish a \emph{local law} for uniformly random lozenge tilings; this result is given by \Cref{heightlocal1} below and essentially states the following. For any mesoscopic scale $1 \ll M \ll N = \text{diam} (R)$, the height function associated with $\mathscr{M}$ is approximately linear with slope $\nabla \mathcal{H} (N^{-1} v_0)$ in an $M$-neighborhood of $v_0$. Equivalently, the densities of tile types in $\mathscr{M}$ in any mesoscopic neighborhood of $v_0$ is governed by the (gradient of) the global profile $\mathcal{H}$.  
	
	In the context of the above coupling framework, this local law will serve two primary purposes. To explain the first, recall that the results of \cite{ULSRW} impose the existence of scales $1 \ll U \ll T \ll V$ such that the initial data $\textbf{a}$ for a random walk ensemble is regular on length $U$ subintervals of $[x_0 - V, x_0 + V]$. The local law will enable us to quickly verify this condition for $\textbf{p}$ and $\textbf{r}$, under essentially any choice of $1 \ll U \ll T \ll V$. In particular, we may take these scales to grow quite slowly with respect to $N$, which will be useful since the coupling between $(\textbf{P}, \textbf{Q}, \textbf{R})$ will only hold if $T$ is sufficiently small\footnote{Our choice for the dependence of $T$ on $N$ will not be explicit, but we at least require $T \ll \log \log N$.} compared to $N$. The second purpose of the local law will be to exhibit the previously mentioned coupling between $(\textbf{P}, \textbf{Q}, \textbf{R})$ such that $\textbf{Q}$ is bounded between $\textbf{P}$ and $\textbf{R}$. Indeed, although the results of \cite{ULSRW} can be used to approximate the drifts of the paths in $\textbf{P}$ and $\textbf{R}$, they do not seem to have direct implications about the paths in $\textbf{Q}$ (which are derived from a uniformly random tiling, and not from an ensemble of non-intersecting Bernoulli random walks). Therefore, the local law will be necessary to approximate the drifts of these paths.

	Thus, our task can be described as threefold. 
	
	\begin{enumerate}
		\item \emph{Local Law}: Establish a local law for $\mathscr{M}$. 
		
		\item \emph{Comparison}: Exhibit a coupling between $\mathscr{M}$ and an ensemble $\textbf{P}$ of non-intersecting Bernoulli random walks, such that the two models coincide around $v_0$ with high probability. 
		
		\item \emph{Universality}: Use results of \cite{ULSRW} to show that the local statistics of $\textbf{P}$ around $v_0$ are given by an extended discrete sine process, and conclude that the same holds for $\mathscr{M}$. 
	\end{enumerate}  

	The above outline closely resembles the \emph{three-step strategy} introduced by Erd\H{o}s-Yau for proving bulk universality for local eigenvalue statistics of Hermitian random matrices; see the survey \cite{URM}, the book \cite{DRM}, and references therein. There, they also first establish a local law, which states that the eigenvalue density of an $N \times N$ random matrix $\textbf{H}$ in intervals of length $N^{\delta - 1}$ is governed by the (global) semicircle law. Second, they consider the matrix $\textbf{H}_t = \textbf{H} + t^{1 / 2} \textbf{GUE}$, whose eigenvalues are obtained by running Dyson Brownian motion on the spectrum of $\textbf{H}$ for some short time $t \ll 1$, and show that its local statistics are universal. Third, they compare the local statistics of $\textbf{H}$ and $\textbf{H}_t$ to conclude that those of $\textbf{H}$ are universal.\footnote{The second and third steps of this three-step strategy in fact correspond to the third and second tasks listed above (in the context of lozenge tilings), respectively.} 
	
	From this perspective, our use of the non-intersecting Bernoulli random walk model is analogous to their use of Dyson Brownian motion as the ``source'' of universal local statistics. Moreover, the proofs of the short-time universality results for discrete random walk ensembles in \cite{ULSRW} can be compared to those for Dyson Brownian motion in the original work of Erd\H{o}s-P\'{e}ch\'{e}-Ram\'{i}rez-Schlein-Yau in \cite{UM}.\footnote{Since \cite{UM}, alternative routes to establishing universality for local statistics of Dyson Brownian motion, based on an analysis of the underlying stochastic differential equation, have been developed in \cite{URMLRF,URMFTD,CLSM,FEUM}.} Indeed, both were based on an asymptotic analysis of an explicit correlation kernel, which in the former work was due to Petrov \cite{AR} and in the latter work was due to Br\'{e}zin-Hikami \cite{LSE} and Johansson \cite{ULSDCM}. In fact, the discrete random walk model of \cite{NRWTQPE} can be viewed as a discretization of (the $\beta = 2$ case of) Dyson Brownian motion, as the former exhibits a limit degeneration to the latter.
	
	However, the implementation of the remaining tasks in our outline is considerably different from that in the random matrix three-step strategy. For instance, the comparison step in the latter is typically justified through either a reverse heat flow \cite{UM} or resolvent comparison theorems (originally due to Tao-Vu \cite{RME,ULES} and later developed by Erd\H{o}s-Yau-Yin \cite{BUGM}). Both are heavily based on the fact that the random point configuration of interest is governed by a matrix model. Since this feature is absent for lozenge tilings, we instead implement the comparison through the probabilistic coupling procedure described above. 
	
	Our proof of the lozenge tiling local law is also quite different from that of its random matrix counterpart. The latter was originally established by Erd\H{o}s-Schlein-Yau in \cite{LSLCDRM}, based on a combination of a multiscale analysis with linear algebraic identities satisfied by the random matrix. Local laws have also been proven for the two-dimensional Coulomb gas independently by Bauerschmidt-Bourgade-Nikula-Yau \cite{LTDP} and Lebl\'{e} \cite{LMBG}, and for a class of discrete $\beta$-ensembles by Guionnet-Huang \cite{REUDE}. The proofs of those results are based on the combination of a multiscale analysis with explicit identities for the probability density function of the underlying model. These methods do not seem to be applicable to random lozenge tilings of arbitrary domains, whose associated point process lacks both a matrix model and an explicit form for its probability density function.
	
	Still, a multiscale analysis will be central in our proof of the lozenge tiling local law, which will proceed as follows. Here, we suppose that $v_0 = (0, 0)$ and that the domain $R$ approximates the disk $B_N$ of radius $N$ centered at $v_0$, but impose no assumptions on the associated boundary height function.  
	
	1. We establish two \emph{effective global laws} for the height function $H$ associated with $\mathscr{M}$. These state that, after suitable normalization, $H$ approximates the global profile $\mathcal{H}$ to within an explicit error $\varpi_N$, with high probability. The first global law applies to essentially arbitrary boundary data, with an error of order $\varpi_N \sim (\log N)^{-c}$ for some small $c > 0$. The second has an improved error estimate of $\varpi_N \sim N^{-c}$, but only applies when $\mathcal{H}$ exhibits no frozen facets (is \emph{facetless}). Both laws will be necessary, since the $(\log N)^{-c}$ error appearing in the first would accumulate if applied too many (of order $\log N$) times. To prove the first, we largely proceed as in \cite{VPDT}, with the exception of a (seemingly new) effective variant of Rademacher's theorem quantifying the ``locally linearity'' of an arbitrary Lipschitz function. To prove the second, we closely follow the work \cite{LTGDMLS} of Laslier-Toninelli, by first using monotonicity results to compare the tiling $\mathscr{M}$ to one of a hexagonal domain and then using the integrability of the latter.

	2. We establish deterministic properties for maximizers appearing as global profiles in the variational principle. More specifically, we prove, (i) a stability estimate for how their gradients change upon perturbations of boundary data and, (ii) a condition for the boundary data under which the associated maximizer is facetless. The former will be used to compare two maximizers with approximately equal boundary data, and the latter will be used to verify when the second global law mentioned above is applicable. To prove these properties, we first use the results of De Silva-Savin \cite{MCFARS} to reduce them to estimates on solutions to uniformly elliptic partial differential equations, and then apply known bounds in the latter context.
	
	3. Using the two sets of results described above, we establish a \emph{scale reduction estimate} that essentially states the following. Let $M \gg 1$ denote an integer, and consider a random tiling of a domain approximating a disk $\mathcal{B}_M$ of radius $M$. Suppose that the associated boundary height function gives rise to a facetless global law $G$, within some error $\lambda \ge (\log M)^{-1 - c}$. Restrict the random tiling to a domain approximating the smaller disk $\mathcal{B}_{M / 8}$. Then, with high probability, the boundary height function on this smaller domain gives rise to a facetless global law $F$, within some improved error $\mu \ll \lambda$, and $F$ satisfies $|\nabla F - \nabla G| \lesssim \lambda$ around $N^{-1} v_0 = (0, 0)$; see the left side of \Cref{bm}.

	\begin{figure}

		\begin{center}

			\begin{tikzpicture}[
			>=stealth,
			auto,
			style={
				scale = .4
			}
			]

			\filldraw[white!50!gray] (0, 0) circle [radius = 2.5];
			\draw[] (0, 0) circle[radius = 8];
			\draw[] (0, 0) circle[radius = 4];
			
			\filldraw[fill = black] (0, 0) circle[radius = .08]	 node[below = 1, scale = .75]{$v_0$};
			
			\draw[] (0, 8) circle[radius = 0] node[above]{$\mathcal{B}_M$};
			\draw[] (0, 4) circle[radius = 0] node[above = -1, scale = .85]{$\mathcal{B}_{M / 8}$};
			
			\draw[->, black] (0, 0) -- (-6.93, -4);
			\draw[->, black] (0, 0) -- (4, 0);

			\filldraw[fill = black] (-5.2, -3) circle[radius = 0] node[below = 1, scale = .85]{$M$};
			\filldraw[fill = black] (3, 0) circle[radius = 0] node[below = 1, scale = .78]{$\frac{M}{8}$};
			\filldraw[fill = black] (0, -2.9) circle[radius = 0] node[scale = .7]{$M^{-1} H \approx F \approx G$};

			\filldraw[fill = black] (0, 1) circle[radius = 0] node[scale = .8]{$\nabla F \approx \nabla G$};
			
			\filldraw[white!50!gray] (20, 0) circle [radius = 2.7];
			\draw[] (20, 0) circle[radius = 8];
			\draw[] (20, 0) circle[radius = 5.5];
			\draw[] (20, 0) circle[radius = 3.6];

			\filldraw[fill = black] (20, 0) circle[radius = .08] node[below, scale = .75]{$v_0$};
			
			\draw[] (20, 8) circle[radius = 0] node[above]{$R \approx \mathcal{B}_N$};
			\draw[] (20, 5.5) circle[radius = 0] node[above = -1, scale = .85]{$\mathcal{B}_{N / 8}$};
			\draw[] (20, 3.6) circle[radius = 0] node[above = -1, scale = .75]{$\mathcal{B}_{N / 64}$};

			\draw[] (20, .65) circle[radius = 0] node[scale = .65]{$\nabla F \approx \nabla \mathcal{H} (N^{-1} v_0)$};
			
			\end{tikzpicture}
			
		\end{center}

		\caption{\label{bm} Depicted to the left is a diagram for the scale reduction estimate. In $\mathcal{B}_{M / 8}$, the normalized height function $M^{-1} H$ is approximately equal to $F$ and $G$. In the shaded region, we have that $\nabla F \approx \nabla G$. Depicted to the right is a diagram for the inductive proof of the local law.}

	\end{figure}
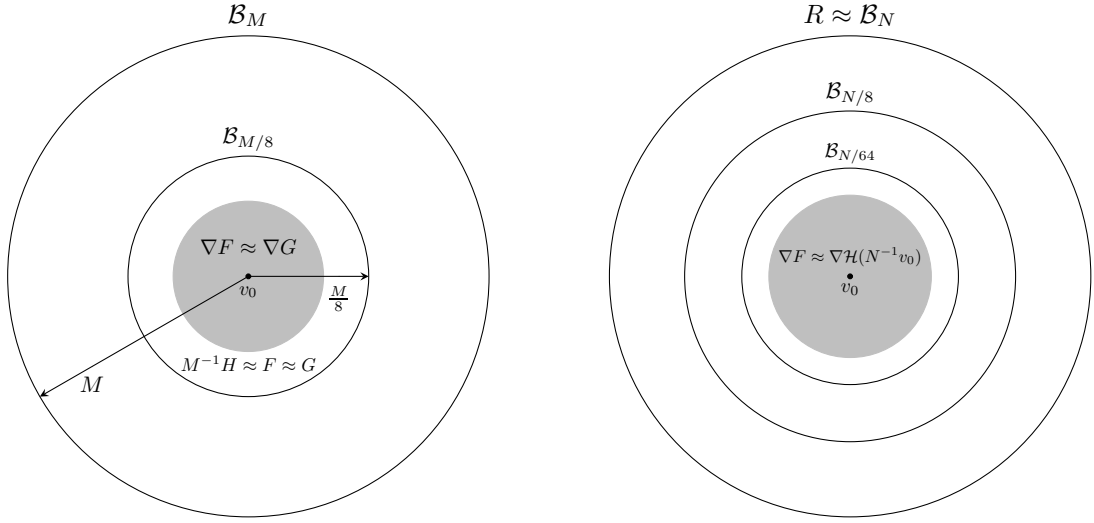

	4. We inductively apply the scale reduction estimate on a decreasing sequence of disks centered at $v_0 = (0, 0)$; see the right side of \Cref{bm}. In particular, after of order $\log N$ such applications, we pass from approximating the height function $H$ by $\mathcal{H}$ on the original domain $R$ of size $N$ to approximating it by some profile $F$ on an arbitrary mesoscopic neighborhood of $v_0$. Since the gradient around $(0, 0)$ of the profile approximating $H$ remains approximately constant after each application of the scale reduction estimate, $F$ will be approximately linear of slope $\nabla \mathcal{H} (0, 0)$. Hence, $H$ will be approximately linear of this slope on any mesoscopic disk centered at $v_0$, thereby establishing the lozenge tiling local law.

	The remainder of this section is organized as follows. In \Cref{TilingsHeight} we define the lozenge tilings and height functions on the triangular lattice; in \Cref{InfiniteMeasures} we recall the definition of a certain extended discrete sine process; in \Cref{ConvergenceStatistics} we state our results on universality for local lozenge tiling statistics; and in \Cref{Outline} we outline the organization for the remainder of the paper. 
	
	It would be of interest to extend the results of this article to dimer models on other lattices. Currently, the most accessible such model appears to be the uniformly random domino tiling on a subdomain of $\mathbb{Z}^2$. Although the framework described above largely applies in this context, it uses two exactly solvable features of lozenge tilings that have not yet been developed for domino tilings. The first is the integrability of the model on a sufficiently broad family of domains so as to encompass all possible second order Taylor expansions of any global height profile.\footnote{This is used in the proof of the second effective global law, with the improved power-law error; the proof of the first effective global law (with logarithmic error) should apply to other models with little modification.} In the lozenge tiling case, these constitute hexagonal domains \cite{LLSBHP,AR,ARTP}; for domino tilings, it seems likely that such a family of domains is provided in the work \cite{ARTR} of Bufetov-Knizel. The second is the convergence of local statistics under general initial data for an ensemble of discrete random walks, conditioned to never intersect, associated with the tiling model. For lozenge tilings, this was provided in \cite{ULSRW}. Given both of these statements for the domino tiling model, it might be possible to apply our methods to analyze its local statistics on arbitrary domains.
	
	Before proceeding, let us introduce some notation that will be used throughout the article. For any integer $d \in \mathbb{Z}_{\ge 1}$ and subset $\mathcal{S} \subseteq \mathbb{R}^d$, we let $\overline{\mathcal{S}}$ and $\partial \mathcal{S} = \partial \overline{\mathcal{S}}$ denote its closure and boundary, respectively. Additionally, for any function $f: \mathcal{S} \rightarrow \mathbb{R}$ and subset $\mathcal{S}' \subseteq \mathcal{S}$, let $f |_{\mathcal{S}'}$ denote the restriction of $f$ to $\mathcal{S}'$. Moreover, for any real number $c \in \mathbb{R}$ and vector $w \in \mathbb{R}^d$, we set $c \mathcal{S} + w = \big\{ cs + w: s \in \mathcal{S} \big\} \subseteq \mathbb{R}^d$. We further let $E^c$ denote the complement of any event $E$.

	\subsection{Lozenge Tilings and Height Functions} 
	
	\label{TilingsHeight} 
	
	In this section we describe the model of interest to us, namely, lozenge tilings on the triangular lattice and their associated height functions. 
	
	To that end, we first define the \emph{triangular lattice} $\mathbb{T}$ to be the graph whose vertex set is $\mathbb{Z}^2$ and whose edge set consists of edges connecting $(x, y), (x', y') \in \mathbb{Z}^2$ if and only if $(x' - x, y' - y) \in \big\{ (1, 0), (-1, 0), (0, 1), (0, -1), (1, 1), (-1, -1) \big\}$. The faces of $\mathbb{T}$ are therefore triangles with vertices of the form $\big\{ (x, y), (x + 1, y), (x + 1, y + 1) \big\}$ or $\big\{ (x, y), (x, y + 1), (x + 1, y + 1) \big\}$. We refer to the left side of \Cref{latticet} for a depiction.\footnote{It is typical to depict lozenge tilings on the hexagonal lattice, in which adjacent edges meet at angles of $\frac{\pi}{3}$ and not $\frac{\pi}{2}$. Our reason for instead depicting them on $\mathbb{Z}^2$ is to simplify notation for the height function when embedding the tiling in $\mathbb{R}^2$.} For any induced subgraph $R \subseteq \mathbb{T}$, we let $\mathbb{V}(R)$ denote its set of vertices and $\mathbb{F}(R)$ denote the union of its faces (viewed as a two-dimensional subset of $\mathbb{R}^2$). We say that $R$ is simply-connected if $\mathbb{F}(R)$ is. A \emph{domain} $R \subseteq \mathbb{T}$ is a simply-connected induced subgraph of $\mathbb{T}$. Its \emph{boundary} $\partial R \subseteq \mathbb{V}(R)$ is the set of vertices $v \in \mathbb{V}(R)$ adjacent to a vertex in $\mathbb{T} \setminus \mathbb{V}(R)$.

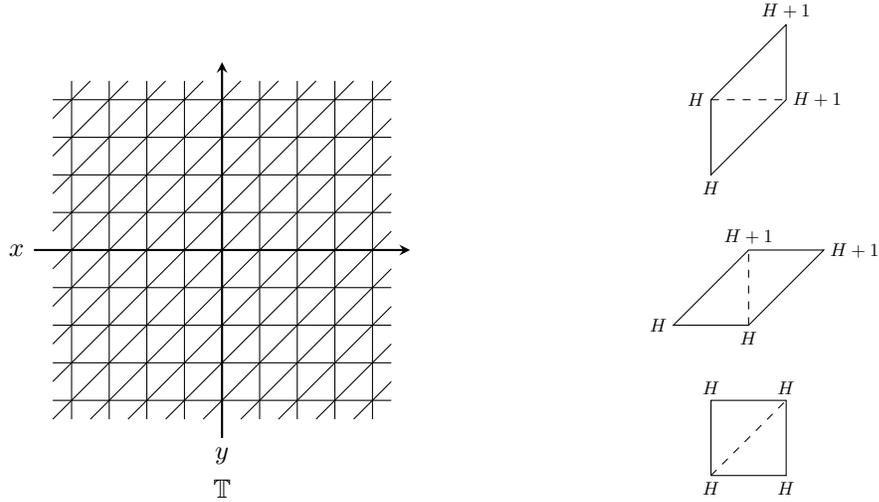
\begin{figure}

	\begin{center}

		\begin{tikzpicture}[
		>=stealth,
		auto,
		style={
			scale = .5
		}
		]
		
		\draw[-, black] (-4.5, -1) -- (4.5, -1); 
		\draw[-, black] (-4.5, -2) -- (4.5, -2); 
		\draw[-, black] (-4.5, -3) -- (4.5, -3); 
		\draw[-, black] (-4.5, -4) -- (4.5, -4); 
		\draw[->, thick, black] (-5, 0) node[left]{$x$} -- (5, 0); 
		\draw[-, black] (-4.5, 1) -- (4.5, 1); 
		\draw[-, black] (-4.5, 2) -- (4.5, 2); 
		\draw[-, black] (-4.5, 3) -- (4.5, 3); 
		\draw[-, black] (-4.5, 4) -- (4.5, 4);

		\draw[-, black] (-1, -4.5) -- (-1, 4.5);
		\draw[-, black] (-2, -4.5) -- (-2, 4.5);
		\draw[-, black] (-3, -4.5) -- (-3, 4.5);
		\draw[-, black] (-4, -4.5) -- (-4, 4.5);
		\draw[->, thick, black] (0, -5) node[below]{$y$} node[below = 12]{$\mathbb{T}$} -- (0, 5);
		\draw[-, black] (1, -4.5) -- (1, 4.5);
		\draw[-, black] (2, -4.5) -- (2, 4.5);
		\draw[-, black] (3, -4.5) -- (3, 4.5);
		\draw[-, black] (4, -4.5) -- (4, 4.5);

		\draw [-, black] (-4.5, -4.5) -- (4.5, 4.5);
		\draw [-, black] (-4.5, -3.5) -- (3.5, 4.5);
		\draw [-, black] (-4.5, -2.5) -- (2.5, 4.5);
		\draw [-, black] (-4.5, -1.5) -- (1.5, 4.5);
		\draw [-, black] (-4.5, -.5) -- (.5, 4.5);
		\draw [-, black] (-4.5, .5) -- (-.5, 4.5);
		\draw [-, black] (-4.5, 1.5) -- (-1.5, 4.5);
		\draw [-, black] (-4.5, 2.5) -- (-2.5, 4.5);
		\draw [-, black] (-4.5, 3.5) -- (-3.5, 4.5);

		\draw [-, black] (-3.5, -4.5) -- (4.5, 3.5);
		\draw [-, black] (-2.5, -4.5) -- (4.5, 2.5);
		\draw [-, black] (-1.5, -4.5) -- (4.5, 1.5);
		\draw [-, black] (-.5, -4.5) -- (4.5, .5);
		\draw [-, black] (.5, -4.5) -- (4.5, -.5);
		\draw [-, black] (1.5, -4.5) -- (4.5, -1.5);
		\draw [-, black] (2.5, -4.5) -- (4.5, -2.5);
		\draw [-, black] (3.5, -4.5) -- (4.5, -3.5);

		\draw[-, black] (15, 6) node[above, scale = .7]{$H + 1$}-- (15, 4) node[right, scale = .7]{$H + 1$} -- (13, 2) node[below, scale = .7]{$H$} -- (13, 4) node[left, scale = .7]{$H$} -- (15, 6); 
		\draw[-, dashed, black] (13, 4) -- (15, 4);

		\draw[-, black] (14, 0) node[above, scale = .7]{$H + 1$} -- (16, 0) node[right, scale = .7]{$H + 1$} -- (14, -2) node[below, scale = .7]{$H$} -- (12, -2) node[left, scale = .7]{$H$}-- (14, 0);
		\draw[-, dashed, black] (14, -2) -- (14, 0);
		
		\draw[-, black] (13, -6) node[below, scale = .7]{$H$} -- (13, -4) node[above, scale = .7]{$H$} -- (15, -4) node[above, scale = .7]{$H$} -- (15, -6) node[below, scale = .7]{$H$} -- (13, -6);
		\draw[-, dashed, black] (13, -6) -- (15, -4);
		
		\end{tikzpicture}
		
	\end{center}

	\caption{\label{latticet} The triangular lattice $\mathbb{T}$ is depicted to the left. The three different types of lozenges (tiles), along with their associated height functions, are depicted to the right.}

\end{figure}

A \emph{dimer covering} of a domain $R \subseteq \mathbb{T}$ is defined to be a perfect matching on the dual graph of $R$. A pair of adjacent triangular faces in any such matching forms a parallelogram, which we will also refer to as a \emph{lozenge} or \emph{tile}. Lozenges can be oriented in one of three ways; see the right side of \Cref{latticet} for a depiction. We refer to the topmost lozenge there (that is, one with vertices of the form $\big\{ (x, y), (x, y + 1), (x + 1, y + 2), (x + 1, y + 1) \big\}$) as a \emph{type $1$} lozenge. Similarly, we refer to the middle (with vertices of the form $\big\{ (x, y), (x + 1, y), (x + 2, y + 1), (x + 1, y + 1) \big\}$) and bottom (vertices of the form $\big\{ (x, y), (x, y + 1), (x + 1, y + 1), (x + 1, y) \big\}$) ones there as \emph{type $2$} and \emph{type $3$} lozenges, respectively. 

A dimer covering of $R$ can equivalently be interpreted as a tiling of $R$ by lozenges of types $1$, $2$, and $3$. Therefore, we will also refer to a dimer covering of $R$ as a \emph{(lozenge) tiling}. Denote by $\mathfrak{E} (R)$ the set of such tilings of $R$; if $\mathfrak{E} (R)$ is nonempty, then we call $R$ \emph{tileable}. 

\begin{exa}
	
	\label{domainabc}
	
	For any $A, B, C \in \mathbb{Z}_{\ge 1}$, the hexagon with vertices $\big\{ (0, 0), (A, 0), (0, B), (C, B + C), (A + C, B + C), (A + C, C) \big\}$ is an example of a tileable domain. We will refer to this domain as a \emph{$A \times B \times C$ hexagon} (although its sidelengths are in fact $A$, $B$, and $C \sqrt{2}$). See the right side of \Cref{tilinghexagon} for an example.  
	
\end{exa}

Following Section 3 of \cite{VPDT}, a \emph{free (lozenge) tiling} of a (not necessarily tileable) domain $R$ is defined to be a lozenge tiling of a possibly larger domain $R'$ containing $R$, each tile of which intersects $R$ in at least one face. Stated equivalently, it is a tiling of $R$, in which lozenges are permitted to extend past $\partial R$ and include at most one face of $\mathbb{T} \setminus R$.

Free lozenge tilings admit alternative, equivalent descriptions in terms of height functions. More specifically, a \emph{height function} $H: \mathbb{V}(R) \rightarrow \mathbb{Z}$ is a function on the vertices of $\mathbb{V} (R)$ that satisfies 
\begin{flalign*}
f(v) - f (u) \in \{ 0, 1 \}, \quad \text{whenever $u = (x, y)$ and $v \in \big\{ (x + 1, y), (x, y + 1), (x + 1, y + 1) \big\}$},
\end{flalign*}

\noindent for some $(x, y) \in \mathbb{Z}^2$. We refer to the restriction $h = H|_{\partial R}$ as a \emph{boundary height function}.

For a fixed vertex $v \in \mathbb{V}(R)$ and integer $m \in \mathbb{Z}$, one can associate with any free tiling of $R$ a height function $H: \mathbb{V}(R) \rightarrow \mathbb{Z}$ as follows. First, set $H (v) = m$, and then define $H$ at the remaining vertices of $R$ in such a way that the height functions along the four vertices of any lozenge in the tiling are of the form depicted on the right side of \Cref{latticet}. In particular, we require that $H (x + 1, y) - H (x, y) = 1$ if and only if $(x, y)$ and $(x + 1, y)$ are vertices of the same type $1$ lozenge, and that $H (x, y + 1) - H (x, y) = 1$ if and only if $(x, y)$ and $(x, y + 1)$ are vertices of the same type $2$ lozenge. Since $R$ is simply-connected, a height function on $R$ is uniquely determined by these conditions (and the value of $H(v) = m$). 

Furthermore, it can be quickly verified that any height function $H: \mathbb{V}(R) \rightarrow \mathbb{Z}$ gives rise to a (unique) free tiling of $R$. We refer to the right side of \Cref{tilinghexagon} for an example; as depicted there, we can also view a lozenge tiling of $R$ as a packing of $\mathbb{F}(R)$ by boxes of the type shown on the left side of \Cref{tilinghexagon}. In this case, the value $H (u)$ of the height function associated with this tiling at some vertex $u \in \mathbb{V}(R)$ denotes the height of the stack of boxes at $u$. Observe in particular that, if there exists a tiling $\mathscr{M} \in \mathfrak{E} (R)$ associated with some height function $H$, then the boundary height function $h = H |_{\partial R}$ is independent of $\mathscr{M}$ and is uniquely determined by $R$ (except for a global shift, which was above fixed by the value of $H(v) = m$).

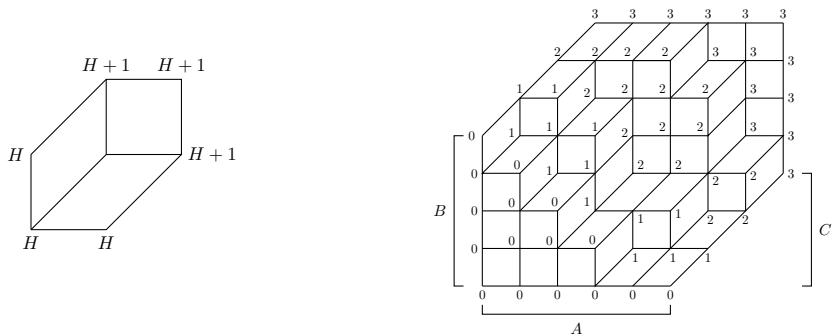
\begin{figure}

	\begin{center}

		\begin{tikzpicture}[
		>=stealth,
		auto,
		style={
			scale = .5
		}
		]

		\draw[-, black] (-12, 1.5) node[below, scale = .7]{$H$} -- (-10, 1.5) node[below, scale = .7]{$H$} -- (-8, 3.5) node[right, scale = .7]{$H + 1$} -- (-8, 5.5) node[above, scale = .7]{$H + 1$} -- (-10, 5.5) node[above, scale = .7]{$H + 1$} -- (-12, 3.5) node[left, scale = .7]{$H$} -- (-12, 1.5);
		\draw[-, black] (-12, 1.5) -- (-10, 3.5) -- (-10, 5.5);
		\draw[-, black] (-10, 3.5) -- (-8, 3.5);

		\draw[-, black] (0, 0) -- (5, 0);
		\draw[-, black] (0, 0) -- (0, 4);
		\draw[-, black] (5, 0) -- (8, 3);
		\draw[-, black] (0, 4) -- (3, 7);
		\draw[-, black] (8, 3) -- (8, 7); 
		\draw[-, black] (8, 7) -- (3, 7);

		\draw[-, black] (0, -.5) -- (0, -.75) -- (5, -.75) -- (5, -.5);
		\draw[-, black] (-.5, 0) -- (-.75, 0) -- (-.75, 4) -- (-.5, 4); 
		\draw[-, black] (8.5, 0) -- (8.75, 0) -- (8.75, 3) -- (8.5, 3); 
		
		\draw[] (2.5, -1.125) circle[radius = 0] node[scale = .6]{$A$};	
		\draw[] (-1.125, 2) circle[radius = 0] node[scale = .6]{$B$};
		\draw[] (9.125, 1.5) circle[radius = 0] node[scale = .6]{$C$};

		\draw[-, black] (1, 0) -- (1, 3) -- (0, 3) -- (1, 4) -- (1, 5) -- (2, 5) -- (4, 7);
		\draw[-, black] (0, 2) -- (2, 2) -- (2, 0);
		\draw[-, black] (0, 1) -- (3, 1) -- (3, 0) -- (4, 1) -- (6, 1);
		\draw[-, black] (2, 6) -- (5, 6) -- (6, 7) -- (6, 6) -- (8, 6);
		\draw[-, black] (4, 6) -- (5, 7);
		\draw[-, black] (1, 4) -- (2, 4) -- (2, 5); 
		\draw[-, black] (2, 4) -- (3, 5) -- (3, 6);
		\draw[-, black] (1, 3) -- (2, 4);
		\draw[-, black] (2, 2) -- (3, 3) -- (3, 4) -- (2, 4) -- (2, 3) -- (1, 2);
		\draw[-, black] (2, 3) -- (3, 3) -- (3, 2) -- (4, 3) -- (4, 6) -- (5, 6) -- (5, 3) -- (4, 3);
		\draw[-, black] (3, 5) -- (5, 5) -- (6, 6);
		\draw[-, black] (7, 7) -- (7, 4) -- (8, 4) -- (7, 3) -- (7, 2) -- (6, 2) -- (4, 0);
		\draw[-, black] (8, 5) -- (7, 5) -- (6, 4) -- (6, 5) -- (7, 6);
		\draw[-, black] (3, 3) -- (4, 4) -- (6, 4);
		\draw[-, black] (3, 4) -- (4, 5);
		\draw[-, black] (6, 5) -- (5, 5);
		\draw[-, black] (7, 4) -- (6, 3) -- (7, 3);
		\draw[-, black] (2, 1) -- (3, 2) -- (5, 2) -- (6, 3);
		\draw[-, black] (3, 1) -- (5, 3) -- (6, 3) -- (6, 4); 
		\draw[-, black] (4, 1) -- (4, 2); 
		\draw[-, black] (5, 1) -- (5, 2); 
		\draw[-, black] (6, 2) -- (6, 3);

		\draw[] (0, 0) circle [radius = 0] node[below, scale = .5]{$0$};
		\draw[] (1, 0) circle [radius = 0] node[below, scale = .5]{$0$};
		\draw[] (2, 0) circle [radius = 0] node[below, scale = .5]{$0$};
		\draw[] (3, 0) circle [radius = 0] node[below, scale = .5]{$0$};
		\draw[] (4, 0) circle [radius = 0] node[below, scale = .5]{$0$};
		\draw[] (5, 0) circle [radius = 0] node[below, scale = .5]{$0$};

		\draw[] (0, 1) circle [radius = 0] node[left, scale = .5]{$0$};
		\draw[] (1, 1) circle [radius = 0] node[above = 3, left = 0, scale = .5]{$0$};
		\draw[] (2, 1) circle [radius = 0] node[above = 3, left = 0, scale = .5]{$0$};
		\draw[] (3, 1) circle [radius = 0] node[left = 1, above = 0, scale = .5]{$0$};
		\draw[] (4, 1) circle [radius = 0] node[right = 1, below = 0, scale = .5]{$1$};
		\draw[] (5, 1) circle [radius = 0] node[right = 1, below = 0, scale = .5]{$1$};
		\draw[] (6, 1) circle [radius = 0] node[below, scale = .5]{$1$};

		\draw[] (0, 2) circle [radius = 0] node[left, scale = .5]{$0$};
		\draw[] (1, 2) circle [radius = 0] node[above = 3, left = 0, scale = .5]{$0$};
		\draw[] (2, 2) circle [radius = 0] node[left = 1, above = 0, scale = .5]{$0$};
		\draw[] (3, 2) circle [radius = 0] node[left = 3, above = 0, scale = .5]{$1$};
		\draw[] (4, 2) circle [radius = 0] node[below = 3, right = 0, scale = .5]{$1$};
		\draw[] (5, 2) circle [radius = 0] node[below = 1, right = 0, scale = .5]{$1$};
		\draw[] (6, 2) circle [radius = 0] node[right = 1, below = 0, scale = .5]{$2$};
		\draw[] (7, 2) circle [radius = 0] node[below, scale = .5]{$2$};

		\draw[] (0, 3) circle [radius = 0] node[left, scale = .5]{$0$};
		\draw[] (1, 3) circle [radius = 0] node[left = 1, above = 0, scale = .5]{$0$};
		\draw[] (2, 3) circle [radius = 0] node[above = 1, left = 0, scale = .5]{$1$};
		\draw[] (3, 3) circle [radius = 0] node[left = 3, above = 0, scale = .5]{$1$};
		\draw[] (4, 3) circle [radius = 0] node[right = 3, above = 0, scale = .5]{$2$};
		\draw[] (5, 3) circle [radius = 0] node[right = 3, above = 0, scale = .5]{$2$};
		\draw[] (6, 3) circle [radius = 0] node[below = 3, right = 0, scale = .5]{$2$};
		\draw[] (7, 3) circle [radius = 0] node[below = 1, right = 0, scale = .5]{$2$};
		\draw[] (8, 3) circle [radius = 0] node[right, scale = .5]{$3$};
		
		\draw[] (0, 4) circle [radius = 0] node[left = 1, scale = .5]{$0$};
		\draw[] (1, 4) circle [radius = 0] node[above = 1, left = 0, scale = .5]{$1$};
		\draw[] (2, 4) circle [radius = 0] node[left = 3, above = 0, scale = .5]{$1$};
		\draw[] (3, 4) circle [radius = 0] node[above, scale = .5]{$1$};
		\draw[] (4, 4) circle [radius = 0] node[above = 1, left = 0, scale = .5]{$2$};
		\draw[] (5, 4) circle [radius = 0] node[left = 3, above = 0, scale = .5]{$2$};
		\draw[] (6, 4) circle [radius = 0] node[left = 3, above = 0, scale = .5]{$2$};
		\draw[] (7, 4) circle [radius = 0] node[right = 3, above = 0, scale = .5]{$3$};
		\draw[] (8, 4) circle [radius = 0] node[right, scale = .5]{$3$};

		\draw[] (1, 5) circle [radius = 0] node[above, scale = .5]{$1$};
		\draw[] (2, 5) circle [radius = 0] node[left = 1, above = 0, scale = .5]{$1$};
		\draw[] (3, 5) circle [radius = 0] node[above = 2, left = 0, scale = .5]{$2$};
		\draw[] (4, 5) circle [radius = 0] node[left = 3, above = 0, scale = .5]{$2$};
		\draw[] (5, 5) circle [radius = 0] node[left = 3, above = 0, scale = .5]{$2$};
		\draw[] (6, 5) circle [radius = 0] node[left = 1, above = 0, scale = .5]{$2$};
		\draw[] (7, 5) circle [radius = 0] node[right = 3, above = 0, scale = .5]{$3$};
		\draw[] (8, 5) circle [radius = 0] node[right, scale = .5]{$3$};

		\draw[] (2, 6) circle [radius = 0] node[above, scale = .5]{$2$};
		\draw[] (3, 6) circle [radius = 0] node[above, scale = .5]{$2$};
		\draw[] (4, 6) circle [radius = 0] node[left = 2, above = 0, scale = .5]{$2$};
		\draw[] (5, 6) circle [radius = 0] node[left = 2, above = 0, scale = .5]{$2$};
		\draw[] (6, 6) circle [radius = 0] node[right = 3, above = 0, scale = .5]{$3$};
		\draw[] (7, 6) circle [radius = 0] node[right = 3, above = 0, scale = .5]{$3$};
		\draw[] (8, 6) circle [radius = 0] node[right, scale = .5]{$3$};

		\draw[] (3, 7) circle [radius = 0] node[above, scale = .5]{$3$};
		\draw[] (4, 7) circle [radius = 0] node[above, scale = .5]{$3$};
		\draw[] (5, 7) circle [radius = 0] node[above, scale = .5]{$3$};
		\draw[] (6, 7) circle [radius = 0] node[above, scale = .5]{$3$};
		\draw[] (7, 7) circle [radius = 0] node[above, scale = .5]{$3$};
		\draw[] (8, 7) circle [radius = 0] node[above, scale = .5]{$3$};

		\end{tikzpicture}
		
	\end{center}

	\caption{\label{tilinghexagon} Depicted to the right is a lozenge tiling of the $A \times B \times C$ hexagon, with $(A, B, C) = (5, 4, 3)$. One may view this tiling as a packing of boxes (of the type depicted on the left) into a large corner, which gives rise to a height function (shown on the right). }

\end{figure}

\subsection{Infinite-Volume Gibbs Measures}

\label{InfiniteMeasures}

Although in this paper we will be primarily interested in uniformly random lozenge tilings of finite domains, certain infinite-volume measures on $\mathbb{T}$ will arise as limit points of local statistics for our model. Therefore, in this section we recall a two-parameter family of such measures from \cite{OD,LSLD,CFP,RS}. 

To that end, let $\mathfrak{P} (R)$ denote the space of probability measures on the set $\mathfrak{E} (R)$ of all lozenge tilings of some domain $R \subseteq \mathbb{T}$. We say that $\mu \in \mathfrak{P} (\mathbb{T})$ is a \emph{Gibbs measure} if it satisfies the following property for any finite domain $R \subset \mathbb{T}$. The probability under $\mu$ of selecting any $\mathscr{M} \in \mathfrak{E} (\mathbb{T})$, conditional on the restriction of $\mathscr{M}$ to $\mathbb{T} \setminus R$, is uniform.

Moreover, for any $w \in \mathbb{Z}^2$, define the translation map $\mathfrak{S}_w: \mathbb{Z}^2 \rightarrow \mathbb{Z}^2$ by setting $\mathfrak{S}_w (v) = v - w$ for any $v \in \mathbb{Z}^2$. Then $\mathfrak{S}_w$ induces an operator on $\mathfrak{P} (\mathbb{T})$ that translates a tiling by $-w$; we also refer it to by $\mathfrak{S}_w$. A measure $\mu \in \mathfrak{P} (\mathbb{T})$ is called \emph{translation-invariant} if $\mathfrak{S}_w \mu = \mu$, for any $w \in \mathbb{Z}^2$.  We further call a translation-invariant measure $\mu \in \mathfrak{P} (\mathbb{T})$ \emph{extremal} if, for any $p \in (0, 1)$ and translation-invariant measures $\mu_1, \mu_2 \in \mathfrak{P} (\mathbb{T})$ such that $\mu = p \mu_1 + (1 - p) \mu_2$, we have $\mu_1 = \mu = \mu_2$. 
	
Let $\mu \in \mathfrak{P} (\mathbb{T})$ be a translation-invariant measure, and let $H$ denote the height function associated with a randomly sampled lozenge tiling under $\mu$, such that $H(0, 0) = 0$. Setting $s = \mathbb{E} \big[ H (1, 0) - H (0, 0) \big]$ and $t = \mathbb{E} \big[ H(0, 1) - H (0, 0) \big]$, the \emph{slope} of $\mu$ is defined to be the pair $(s, t)$. 

Observe that, if we denote 
\begin{flalign}
\label{t}
\mathcal{T} = \{ (s, t) \in \mathbb{R}_{> 0}^2: s + t < 1 \},
\end{flalign}

\noindent and its closure by $\overline{\mathcal{T}}$, then we must have $(s, t) \in \overline{\mathcal{T}}$. Indeed, $s$, $t$, and $1 - s - t$ denote the probabilities of a given lozenge in a random tiling (under $\mu$) being of types $1$, $2$, and $3$, respectively. 

It was shown as Theorem 9.1.1 of \cite{RS} that, for any $(s, t) \in \overline{\mathcal{T}}$, there exists a unique translation-invariant, extremal Gibbs measure $\mu = \mu_{s, t} \in \mathfrak{P} (\mathbb{T})$ of slope $(s, t)$; such a measure is known to arise as a limit point of lozenge tilings of a large torus, conditional on its height function having average slope $(s, t)$ (see Section 5.21 of \cite{OD} or Section 4 of \cite{LSLD}). An explicit form for the correlation functions for these measures was provided in Proposition 8.5 of \cite{VPDT} and Section 6.23 of \cite{OD}, based on methods introduced earlier in \cite{SDL}. To explain this further, we require the following definition. 

\begin{definition} 
	
\label{xm} 

Given a domain $R \subseteq \mathbb{T}$ and a tiling $\mathscr{M} \in \mathfrak{E} (R)$, let $\mathscr{X} = \mathscr{X} (\mathscr{M})$ denote the set of all $(x, y) \in \mathbb{Z}^2$ such that $\big( x + \frac{1}{2}, y \big)$ is the center of some type $1$ lozenge in $\mathscr{M}$.  

\end{definition}

\begin{figure}

	\begin{center}

		\begin{tikzpicture}[
		>=stealth,
		auto,
		style={
			scale = .5
		}
		]

		\draw[-, black] (0, 0) -- (5, 0);
		\draw[-, black] (0, 0) -- (0, 4);
		\draw[-, black] (5, 0) -- (8, 3);
		\draw[-, black] (0, 4) -- (3, 7);
		\draw[-, black] (8, 3) -- (8, 7); 
		\draw[-, black] (8, 7) -- (3, 7);

		\draw[-, black] (1, 0) -- (1, 3) -- (0, 3) -- (1, 4) -- (1, 5) -- (2, 5) -- (4, 7);
		\draw[-, black] (0, 2) -- (2, 2) -- (2, 0);
		\draw[-, black] (0, 1) -- (3, 1) -- (3, 0) -- (4, 1) -- (6, 1);
		\draw[-, black] (2, 6) -- (5, 6) -- (6, 7) -- (6, 6) -- (8, 6);
		\draw[-, black] (4, 6) -- (5, 7);
		\draw[-, black] (1, 4) -- (2, 4) -- (2, 5); 
		\draw[-, black] (2, 4) -- (3, 5) -- (3, 6);
		\draw[-, black] (1, 3) -- (2, 4);
		\draw[-, black] (2, 2) -- (3, 3) -- (3, 4) -- (2, 4) -- (2, 3) -- (1, 2);
		\draw[-, black] (2, 3) -- (3, 3) -- (3, 2) -- (4, 3) -- (4, 6) -- (5, 6) -- (5, 3) -- (4, 3);
		\draw[-, black] (3, 5) -- (5, 5) -- (6, 6);
		\draw[-, black] (7, 7) -- (7, 4) -- (8, 4) -- (7, 3) -- (7, 2) -- (6, 2) -- (4, 0);
		\draw[-, black] (8, 5) -- (7, 5) -- (6, 4) -- (6, 5) -- (7, 6);
		\draw[-, black] (3, 3) -- (4, 4) -- (6, 4);
		\draw[-, black] (3, 4) -- (4, 5);
		\draw[-, black] (6, 5) -- (5, 5);
		\draw[-, black] (7, 4) -- (6, 3) -- (7, 3);
		\draw[-, black] (2, 1) -- (3, 2) -- (5, 2) -- (6, 3);
		\draw[-, black] (3, 1) -- (5, 3) -- (6, 3) -- (6, 4); 
		\draw[-, black] (4, 1) -- (4, 2); 
		\draw[-, black] (5, 1) -- (5, 2); 
		\draw[-, black] (6, 2) -- (6, 3);

		\filldraw[fill=black] (0, 4) circle [radius = .1];
		\filldraw[fill=black] (1, 3) circle [radius = .1];
		\filldraw[fill=black] (2, 2) circle [radius = .1];
		\filldraw[fill=black] (2, 5) circle [radius = .1];
		\filldraw[fill=black] (3, 1) circle [radius = .1];
		\filldraw[fill=black] (3, 3) circle [radius = .1];	
		\filldraw[fill=black] (3, 4) circle [radius = .1];
		\filldraw[fill=black] (5, 2) circle [radius = .1];
		\filldraw[fill=black] (5, 6) circle [radius = .1];
		\filldraw[fill=black] (6, 4) circle [radius = .1];
		\filldraw[fill=black] (6, 5) circle [radius = .1];
		\filldraw[fill=black] (7, 3) circle [radius = .1];

		\filldraw[fill=white] (.5, 4) circle [radius = .1];
		\filldraw[fill=white] (1.5, 3) circle [radius = .1];
		\filldraw[fill=white] (2.5, 2) circle [radius = .1];
		\filldraw[fill=white] (2.5, 5) circle [radius = .1];
		\filldraw[fill=white] (3.5, 1) circle [radius = .1];
		\filldraw[fill=white] (3.5, 3) circle [radius = .1];	
		\filldraw[fill=white] (3.5, 4) circle [radius = .1];
		\filldraw[fill=white] (5.5, 2) circle [radius = .1];
		\filldraw[fill=white] (5.5, 6) circle [radius = .1];
		\filldraw[fill=white] (6.5, 4) circle [radius = .1];
		\filldraw[fill=white] (6.5, 5) circle [radius = .1];
		\filldraw[fill=white] (7.5, 3) circle [radius = .1];

		\draw[<->, thick] (14, 4) -- (22, 4);
		\draw[<->, thick] (16, 0) -- (16, 8);
		\filldraw[fill = black] (17.5, 6.598) circle [radius = .1] node[above]{$\xi$};
		\filldraw[fill = black] (17.5, 1.402) circle [radius = .1] node[below]{$\overline{\xi}$};
		\filldraw[fill = black] (16, 4) circle [radius = .1] node[below = 5, right = 0]{$0$};
		\filldraw[fill = black] (18, 4) circle [radius = .1] node[below]{$1$};

		\draw[->,black,dashed] (17.5, 1.402) arc (259.1:102:2.646);
		\draw[->,black] (17.5, 1.402) arc (210:151:5.196);

		\end{tikzpicture}
		
	\end{center}

	\caption{\label{xmcontours} Depicted to the left is a tiling $\mathscr{M}$; the black and white vertices shown there are the elements of $\mathscr{X} (\mathscr{M})$ and $\mathscr{X} (\mathscr{M}) + \big( \frac{1}{2}, 0 \big)$, respectively. Depicted to the right are two possibilities for the contour for $z$ in \Cref{kernellimitdefinition}; the dashed curve corresponds to the case $y_1 < y_2$ and the solid one to the case $y_1 \ge y_2$.}

\end{figure}
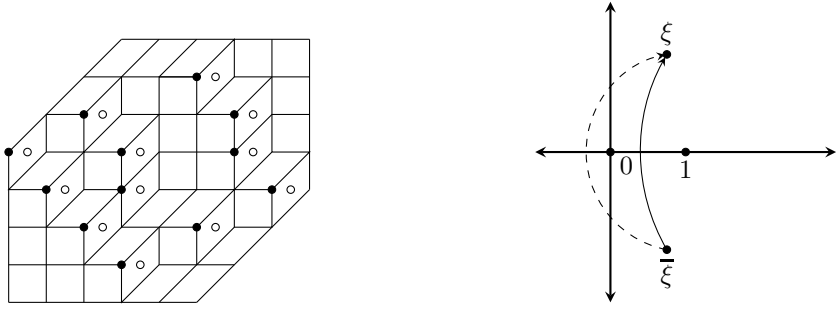

We refer to the left side of \Cref{xmcontours} for a depiction. It is quickly verified that the tiling $\mathscr{M} \in \mathfrak{E}(R)$ is determined by the location set $\mathscr{X} (\mathscr{M})$ of its type $1$ lozenges (unless both $R = \mathbb{T}$ and $\mathscr{X} (\mathscr{M})$ is empty, which will not be relevant for this work). As explained in Sections 3.1.7 and 3.1.9 of \cite{CFP}, if a random tiling $\mathscr{M} \in \mathfrak{E} (\mathbb{T})$ is sampled under the translation-invariant, extremal Gibbs measure $\mu_{s, t}$ of some slope $(s, t) \in \mathcal{T}$, then the law of $\mathscr{X} = \mathscr{X} (\mathscr{M})$ can be expressed as a determinantal point process. Its correlation kernel is an extended discrete sine kernel, which was introduced as Definition 4 of \cite{CFP} and is given as follows. In the below, $\mathbb{H} = \{ z \in \mathbb{C}: \Im z > 0 \} \subset \mathbb{C}$ denotes the upper-half plane.

\begin{definition}[{\cite[Definition 4]{CFP}}]
	
	\label{kernellimitdefinition} 
	
	Fix $\xi \in \mathbb{H}$. For any $x_1, x_2, y_1, y_2 \in \mathbb{Z}$, the \emph{extended discrete sine kernel} (also called the \emph{incomplete beta kernel}) $\mathcal{K}_{\xi} (x_1, y_1; x_2, y_2)$ is defined by	
	\begin{flalign*}
	\mathcal{K}_{\xi} (x_1, y_1; x_2, y_2) = \displaystyle\frac{1}{2 \pi \textbf{i}} \displaystyle\int_{\overline{\xi}}^{\xi} (1 - z)^{y_1 - y_2} z^{x_2 - x_1 - 1} dz,
	\end{flalign*}
	
	\noindent where $z$ is integrated along an arc of a ellipse from $\overline{\xi}$ to $\xi$, which intersects $\mathbb{R}$ in the interval $(0, 1)$ if $y_1 \ge y_2$ and which intersects $\mathbb{R}$ in the interval $(-\infty, 0)$ if $y_1 < y_2$. 
	
\end{definition}	

We refer to the right side of \Cref{xmcontours} for a depiction of the possible contours for $z$ in \Cref{kernellimitdefinition}. For any $\xi \in \mathbb{H}$, Theorem 2 of \cite{CFP} indicates that there exists a measure $\mu = \mu_{\xi} \in \mathfrak{P} (\mathbb{T})$ such that 
\begin{flalign}
\label{muxi}
\mathbb{P} \Bigg[ \bigcap_{k = 1}^m \big\{ (x_k, y_k) \in \mathscr{X} (\mathscr{M}) \big\} \Bigg] = \det \big[  \mathcal{K}_{\xi} (x_i, y_i; x_j, y_j) \big]_{1 \le i, j \le m},
\end{flalign}

\noindent where $\mathscr{M} \in \mathfrak{E} (\mathbb{T})$ is sampled under $\mu_{\xi}$, and $m \in \mathbb{Z}_{\ge 1}$ and $(x_1, y_1), (x_2, y_2), \ldots , (x_m, y_m) \in \mathbb{Z}^2$ are arbitrary. The measure $\mu_{\xi}$ is called the \emph{extended discrete sine process} with \emph{complex slope} $\xi$. If $(s, t) \in \mathcal{T}$ and $\xi = e^{\pi \textbf{i} s} \frac{\sin (\pi t)}{\sin (\pi - \pi s - \pi t)}$, then it is further known (see equation (36) of \cite{CFP} and Section 6.23 of \cite{OD}) that $\mu_{\xi} = \mu_{s, t}$.

\subsection{Results} 

\label{ConvergenceStatistics} 

In this article we will be interested in understanding the convergence of local statistics of uniformly random lozenge tilings of large domains. More specifically, we will consider a uniformly random tiling of some large tileable domain $R_N$ that approximates (after normalization by $\frac{1}{N}$) some subset $\mathfrak{R} \subset \mathbb{R}^2$ and whose boundary height function approximates (again, after normalization by $\frac{1}{N}$) some function $\mathfrak{h}: \partial \mathfrak{R} \rightarrow \mathbb{R}$. Then, we will understand how the local statistics of this random tiling behave around some $v_N \in \mathbb{V}(R_N)$, as $N$ tends to $\infty$. As \Cref{localconverge} below indicates, they will converge to one of the translation-invariant, extremal Gibbs measures $\mu_{s, t}$ from \Cref{InfiniteMeasures}. 

The determination of this slope $(s, t)$ will require some additional notation. To that end, recall the set $\mathcal{T}$ from \eqref{t} and its closure $\overline{\mathcal{T}}$. Define the \emph{Lobachevsky function} $L: \mathbb{R}_{> 0} \rightarrow \mathbb{R}$ and \emph{surface tension} $\sigma: \overline{\mathcal{T}} \rightarrow \mathbb{R}$ by, for any $x \in \mathbb{R}_{> 0}$ and $(s, t) \in \overline{\mathcal{T}}$, setting
\begin{flalign}
\label{lsigma} 
L(x)  = - \displaystyle\int_0^x \log | 2 \sin z| dz; \qquad	\sigma (s, t) = \displaystyle\frac{1}{\pi} \Big( L(\pi s) + L (\pi t) + L \big( \pi (1 - s - t) \big) \Big).
\end{flalign}

Now, fix a bounded, open, nonempty set $\mathfrak{R} \subset \mathbb{R}$ with boundary $\partial \mathfrak{R}$. Let $\Adm (\mathfrak{R})$ denote the set of Lipschitz functions $F: \overline{\mathfrak{R}} \rightarrow \mathbb{R}$ such that $\nabla F (z) \in \overline{\mathcal{T}}$ for almost every $z \in \mathfrak{R}$; for any function $f: \partial \mathfrak{R} \rightarrow \mathbb{R}$, set $\Adm (\mathfrak{R}; f) = \big\{ F \in \Adm (\mathfrak{R}): F |_{\partial \mathfrak{R}} = f \big\}$. We say that $f$ \emph{admits an admissible extension to $\mathfrak{R}$} if $\Adm (\mathfrak{R}; f)$ is not empty. For any $F \in \Adm (\mathfrak{R})$, define the \emph{entropy functional} 
\begin{flalign*}
\mathcal{E} (F) = \displaystyle\int_{\mathfrak{R}} \sigma \big( \nabla F (z) \big) dz.
\end{flalign*}

Given some $\mathfrak{h}: \partial \mathfrak{R} \rightarrow \mathbb{R}$ admitting an admissible extension to $\mathfrak{R}$, we call a function $\mathcal{H} \in \Adm (\mathfrak{R}; \mathfrak{h})$ a \emph{maximizer of $\mathcal{E}$ on $\mathfrak{R}$ with boundary data $\mathfrak{h}$} if $\mathcal{E} (\mathcal{H}) \ge \mathcal{E} (\mathcal{G})$ for any $\mathcal{G} \in \Adm (\mathfrak{R}; \mathfrak{h})$. 

\begin{rem} 
	
	\label{hemaximum} 
	
	It is known (see Proposition 4.5 of \cite{MCFARS}) that there exists a unique maximizer $\mathcal{H}: \overline{\mathfrak{R}} \rightarrow \mathbb{R}$ of $\mathcal{E}$ on $\mathfrak{R}$ with given boundary data $\mathfrak{h}$.

\end{rem}

Under the notation above, the limit shape for the height function associated with a random lozenge tiling of $R_N$ with boundary height function $h_N$ (which, as mentioned previously, converge to $\mathfrak{R}$ and $\mathfrak{h}$, respectively, after suitable normalization) was shown as Theorem 1.1 of \cite{VPDT} to converge in probability to this maximizer $\mathcal{H}$, as $N$ tends to $\infty$; this result is restated as \Cref{hnh} below. Such a phenomenon is known as a \emph{variational principle} for tilings. The pair $(s, t)$ described previously indicating the slope for the limiting local statistics of this tiling around $v_N$ will then given by $\nabla \mathcal{H} (\mathfrak{v})$, if $\lim_{N \rightarrow \infty} N^{-1} v_N = \mathfrak{v} \in \mathfrak{R}$. 

Let us now explain the sense in which $N^{-1} R_N$ its normalized height function $N^{-1} h_N$ ``converge'' to $\mathfrak{R}$ and $\mathfrak{h}$, respectively. To that end, for any subset $\mathcal{S} \subseteq \mathbb{R}^2$ and points $x_1, x_2 \in \mathcal{S}$, let $d_{\mathcal{S}} (x_1, x_2) = \inf_{\gamma} |\gamma|$, where $\gamma \subseteq \mathcal{S}$ is taken over all paths in $\mathcal{S}$ connecting $x_1$ and $x_2$, and $|\gamma|$ denotes the length of $\gamma$. Next, we say that a sequence of subsets \emph{$\mathfrak{R}_1, \mathfrak{R}_2, \ldots \subset \mathbb{R}^2$ converges to $\mathfrak{R} \subset \mathbb{R}^2$}, and write $\lim_{N \rightarrow \infty} \mathfrak{R}_N = \mathfrak{R}$, if for any $\delta > 0$ there exists an integer $N_0 = N_0 (\delta) > 1$ such that the following two properties hold whenever $N > N_0$. First, for any $x \in \mathfrak{R}$, there exist $x_N \in \mathfrak{R}_N$ and $x' \in \mathfrak{R} \cap \mathfrak{R}_N$ such that 
\begin{flalign} 
\label{dxxnrrn}
\displaystyle\max \big\{ d_{\mathfrak{R}} (x, x'), d_{\mathfrak{R}_N} (x_N, x') \big\} < \delta.
\end{flalign} 

\noindent Second, for any $x_N \in \mathfrak{R}_N$, there exist $x \in \mathfrak{R}$ and $x' \in \mathfrak{R} \cap \mathfrak{R}_N$ such that \eqref{dxxnrrn} again holds.\footnote{In \cite{VPDT,DA,RS}, variational principles for tiling models were established when the domains $\mathfrak{R}_N$ ``approximate $\mathfrak{R}$ from the inside.'' This constitutes the specialization of the above definition when $\mathfrak{R}_N \subseteq \mathfrak{R}$, in which case one may take $x' = x_N$ in \eqref{dxxnrrn}. Although our results can be phrased under this more restrictive notion of domain convergence, one might also be interested in the case when $\mathfrak{R}$ does not contain the $\mathfrak{R}_N$. Since admissible and height functions are $1$-Lipschitz on the metric spaces $(\mathfrak{R}, d_{\mathfrak{R}})$ and $(\mathfrak{R}_N, d_{\mathfrak{R}_N})$, respectively, it can quickly be seen that the results from \cite{VPDT,DA,RS} also apply to the slightly more general notion of domain convergence described above.} 

In this case, we moreover say that a sequence of functions \emph{$\mathfrak{h}_N: \partial \mathfrak{R}_N \rightarrow \mathbb{R}$ converges to $\mathfrak{h}: \partial \mathfrak{R} \rightarrow \mathbb{R}$}, and write $\lim_{N \rightarrow \infty} \mathfrak{h}_N = \mathfrak{h}$, if for every real number $\delta > 0$ there exists an integer $N_0 = N_0 (\delta) > 1$ such that the following two properties hold whenever $N > N_0$. First, for each $x \in \partial \mathfrak{R}$, there exist $x_N \in \partial \mathfrak{R}_N$ and $x' \in  \mathfrak{R} \cap \mathfrak{R}_N$ such that \eqref{dxxnrrn} and $\big| \mathfrak{h}_N (x_N) - \mathfrak{h} (x) \big| < \delta$ both hold. Second, for each $x_N \in \partial \mathfrak{R}_N$, there exist $x \in \partial \mathfrak{R}$ and $x' \in \mathfrak{R} \cap \mathfrak{R}_N$ such that the same inequalities are satisfied.

	We can now state the following theorem indicating that the local statistics of a uniformly random lozenge tiling of $R_N$ around $v_N$ converge as $N$ tends to $\infty$ to the unique translation-invariant, extremal Gibbs measure of slope $\nabla \mathcal{H} (\mathfrak{v})$. In the below, we recall the shift operator $\mathfrak{S}_w$ and measure $\mu_{s, t}$ from \Cref{InfiniteMeasures}. 
	
	\begin{thm} 
		
		\label{localconverge} 
		
		Let $\mathfrak{R} \subset \mathbb{R}^2$ denote an open, simply-connected, bounded domain whose boundary is a piecewise smooth, simple curve. Let $\mathfrak{h}: \partial \mathfrak{R} \rightarrow \mathbb{R}$ denote a function admitting an admissible extension to $\mathfrak{R}$, and let $\mathcal{H} \in \Adm (\mathfrak{R}; \mathfrak{h})$ denote the maximizer of $\mathcal{E}$ on $\mathfrak{R}$ with boundary data $\mathfrak{h}$. Further let $\mathfrak{v} \in \mathfrak{R}$ denote a point in the interior of $\mathfrak{R}$ such that $\nabla \mathcal{H} (\mathfrak{v}) = (s, t) \in \mathcal{T}$. 
		
		Let $R_1, R_2, \ldots \subset \mathbb{T}$ denote simply-connected, tileable domains with boundary height functions $h_1, h_2, \ldots $, respectively. Define $\mathfrak{h}_N: \partial (N^{-1} R_N) \rightarrow \mathbb{R}$ by setting $\mathfrak{h}_N (N^{-1} u) = N^{-1} h_N (u)$ for each $u \in \partial R_N$, and suppose that $\lim_{N \rightarrow \infty} N^{-1} R_N = \mathfrak{R}$ and $\lim_{N \rightarrow \infty} \mathfrak{h}_N = \mathfrak{h}$. For each $N \in \mathbb{Z}_{\ge 1}$, let $\mu_N \in \mathfrak{P} (R_N)$ denote the uniform measure on $\mathfrak{E} (R_N)$ and $v_N \in \mathbb{V}(R_N)$ denote a vertex such that $\lim_{N \rightarrow \infty} N^{-1} v_N = \mathfrak{v}$. Then, $\lim_{N \rightarrow \infty} \mathfrak{S}_{v_N} \mu_N = \mu_{s, t}$ in distribution.
		
	\end{thm}

	\subsection{Outline} 
	
	\label{Outline} 
	
	The remainder of this paper is organized as follows. We begin in \Cref{TilingsFunction} with some preliminary definitions and results concerning uniformly random tilings, maximizers of $\mathcal{E}$, and non-intersecting Bernoulli random walk ensembles. We then implement the three tasks listed in \Cref{Tilings1}. However, in order to most directly explain how these statements together imply \Cref{localconverge}, the exposition will proceed in reverse order. 
	
	So, in \Cref{ConvergenceProof} we first state the local law for lozenge tilings (\Cref{heightlocal1}) and the local coupling between a uniformly random lozenge tiling and ensembles of Bernoulli random walks conditioned to never intersect (\Cref{plpprcouple}). Then, we establish \Cref{localconverge} assuming these two results. In \Cref{ProofHHlHr} we prove the coupling statement (still assuming the local law), following the outline provided in \Cref{Tilings1}. 
	
	Next, in \Cref{Local} we state two effective global laws for the height function of a random lozenge tiling. The first (\Cref{heightapproximate}) applies to arbitrary boundary data and provides an error estimate of $(\log N)^{-c}$; the second (\Cref{estimateboundaryheight}) only applies to boundary data giving rise to a global profile exhibiting no frozen facets but has an improved error estimate of order $N^{-c}$. We also state a condition (\Cref{euv1v2estimategradient}) for when the boundary data of a maximizer of $\mathcal{E}$ gives rise to a facetless global profile. Assuming these results, we then establish a scale reduction estimate (\Cref{estimaten2n}) indicating that estimates for height functions associated with lozenge tilings are retained (and sometimes improved) when one decreases the size of the domain. 
	
	Next, in \Cref{EstimateHLocalProof} we  establish \Cref{heightlocal1} (the local law) through an inductive application of \Cref{estimaten2n} (the scale reduction estimate) on a decreasing sequence of domains. In \Cref{GlobalLaw} we establish the first effective global law applicable to general boundary data, and in \Cref{GlobalEstimate2} we establish its more precise variant that only applies to facetless global profiles. 
	 
	 In \Cref{ProofEstimateu} and \Cref{ProofGradientEstimateu} we establish properties about maximizers of $\mathcal{E}$. More specifically, in the former we prove a stability result (\Cref{perturbationboundary}) for the gradients of these maximizers under boundary perturbations, and in the latter we establish \Cref{euv1v2estimategradient} (the condition on which boundary data give rise to facetless global profiles). We conclude in \Cref{ProofstPn} by establishing an effective estimate (\Cref{numbersigma}) for the number of lozenge tilings of a torus with a given approximate slope, which is used in the proof of \Cref{heightapproximate} in \Cref{GlobalLaw}.

	\subsection*{Acknowledgments}

	The author heartily thanks Alexei Borodin for many stimulating conversations and valuable encouragements. The author also would like to express his profound gratitude to Tristan Collins, Vadim Gorin, Jiaoyang Huang, Benjamin Landon, Nicolai Reshetikhin, Chenglong Yu, Boyu Zhang, and Jonathan Zhu for helpful discussions. The author further thanks the anonymous referees for their useful comments on an earlier draft of this paper. This work was partially supported by the NSF Graduate Research Fellowship under grant numbers DGE1144152 and DMS-1664619.

	\section{Miscellaneous Preliminaries} 	
	
	\label{TilingsFunction} 
	
	In this section we introduce some notation and collect several preliminary results that will be useful for the proof of \Cref{localconverge}. In particular, in \Cref{EstimateGlobal} we recall the global law, a monotone coupling, and a concentration estimate for random lozenge tilings. Then, in \Cref{PropertyH} we state several properties concerning maximizers of $\mathcal{E}$, and in \Cref{NonIntersectingCorrelation} we recall the definition of and a concentration estimate for a certain model of random non-intersecting path ensembles.

		\subsection{A Global Law and Monotonicity for Tilings} 
	
	\label{EstimateGlobal} 
	
	In this section we recall three results concerning random tilings from \cite{LSRT,VPDT}. The first is a variational principle (limit shape result) for the height function associated with a uniformly random tiling of a given domain. It was established as Theorem 1.1 of \cite{VPDT} (see also Theorem 9 of \cite{OD}) and can be stated as follows.

	\begin{prop}[{\cite[Theorem 1.1]{VPDT}}]
		
		\label{hnh}
		
		Adopt the notation of \Cref{localconverge}, and let $H_N$ denote the height function associated with a uniformly random lozenge tiling of $R_N$ (with boundary height function $h_N$). Then, for any $\varpi > 0$, we have that 
		\begin{flalign*} 
		\displaystyle\lim_{N \rightarrow \infty} \mathbb{P} \left[ \displaystyle\max_{v \in \mathbb{V}(R_N) \cap N \mathfrak{R}} \big| N^{-1} H_N (v) - \mathcal{H} (N^{-1} v) \big| > \varpi \right] = 0.
		\end{flalign*}
		
	\end{prop}
	
	As \Cref{hnh} applies to general boundary data, it will be useful for the proof of \Cref{localconverge} in \Cref{Convergence}. However, observe that it does not indicate how $\varpi$ can be chosen to depend on $N$. More quantitative results of this type will be given by \Cref{heightapproximate} and \Cref{estimateboundaryheight} below.
	
	The second result of this section provides a monotone (order-preserving) coupling between two different uniformly random (free) lozenge tilings. To describe it further, we first require the following definition for the set of height functions on a domain with a given boundary height function.
	
	\begin{definition} 
		
		\label{gh} 
		
		For any finite domain $R \subset \mathbb{T}$ and boundary height function $h: \partial R \rightarrow \mathbb{Z}$, let $\mathfrak{G} (h) = \mathfrak{G} (h; R)$ denote the set of all height functions $H: \mathbb{V} (R) \rightarrow \mathbb{Z}$ for which $H |_{\partial R} = h$. We will often implicitly extend $H$ by linearity to a function on the union of the faces of $R$, sometimes writing $H:  \mathbb{F} (R) \rightarrow \mathbb{R}$ or $H : R \rightarrow \mathbb{R}$.
		
	\end{definition} 
	
	As explained in \Cref{TilingsHeight}, if $R$ is tileable and $h$ denotes a boundary height function associated with some tiling of $R$, then there is a bijection between $\mathfrak{G} (h)$ and the set $\mathfrak{E} (R)$ of all tilings of $R$. If $R$ is not tileable, then there is a bijection between $\mathfrak{G} (h)$ and the set of all free tilings of $R$ with boundary height function given by $h$. 
	
	Now we can state the following result, which appears as Lemma 18 of \cite{LSRT} and provides a coupling between uniformly random height functions on a domain, in such a way that their ordering on the domain's boundary is preserved in the domain's interior.

	\begin{lem}[{\cite[Lemma 18]{LSRT}}]
		
		\label{monotoneheightcouple}
		
		Fix a finite domain $R \subset \mathbb{T}$ and two boundary height functions $h_1, h_2: \partial R \rightarrow \mathbb{Z}$ such that $h_1 (u) \le h_2 (u)$, for each $u \in \partial R$. Let $H_1, H_2: \mathbb{V}(R) \rightarrow \mathbb{Z}$ denote uniformly random elements of $\mathfrak{G} (h_1), \mathfrak{G} (h_2)$, respectively. Then it is possible to couple $H_1$ and $H_2$ on a common probability space such that $H_1 (v) \le H_2 (v)$ holds almost surely, for each $v \in \mathbb{V}(R)$. 
		
	\end{lem}
	
	The third result of this section is a concentration estimate for a uniformly random height function on a domain; it appears as Proposition 22 of \cite{LSRT} and is given as follows. 
	
	\begin{lem}[{\cite[Proposition 22]{LSRT}}]
		
		\label{hvuexpectation}
		
		Fix a finite domain $R \subset \mathbb{T}$ and a boundary height function $h: \partial R \rightarrow \mathbb{R}$; let $H: \mathbb{V}(R) \rightarrow \mathbb{Z}$ denote a uniformly random element of $\mathfrak{G} (h)$. Further fix a real number $M > 0$ and a vertex $v \in \mathbb{V}(R)$ such that there exists a path in $R$ from $v$ to $\partial R$ of length at most $M$. Then, for any $r \in \mathbb{R}_{> 0}$, we have that
		\begin{flalign}
		\label{hxyh00m}
		\mathbb{P} \bigg[ \Big| H(v) - \mathbb{E} \big[ H (v) \big] \Big| > r M^{1 / 2} \bigg] \le 2 \exp \bigg( -\frac{r^2}{32} \bigg).
		\end{flalign}
	\end{lem}

	\subsection{Maximizers of \texorpdfstring{$\mathcal{E}$}{}}
	
	\label{PropertyH}

	In this section we state several properties and bounds satisfied by maximizers of $\mathcal{E}$; throughout this section we recall the notation of \Cref{ConvergenceStatistics}. The first is a (likely known) comparison principle stating that, if two maximizers of $\mathcal{E}$ on some domain $\mathfrak{R}$ are ordered on the boundary $\partial \mathfrak{R}$, then they are also ordered on the interior of $\mathfrak{R}$. 
	
	\begin{lem}
		
		\label{h1h2compare} 
		
		Fix some bounded, open, connected subset $\mathfrak{R} \subset \mathbb{R}^2$ with Lipschitz boundary, and two functions $\mathfrak{h}_1, \mathfrak{h}_2: \partial \mathfrak{R} \rightarrow \mathbb{R}$ admitting admissible extensions to $\mathfrak{R}$. For each $i \in \{ 1, 2 \}$, let $\mathcal{H}_i \in \Adm (\mathfrak{R}; \mathfrak{h}_i)$ denote the maximizer of $\mathcal{E}$ on $\mathfrak{R}$ with boundary data $\mathfrak{h}_i$. If $\mathfrak{h}_1 (z) \le \mathfrak{h}_2 (z)$ for each $z \in \partial \mathfrak{R}$, then $\mathcal{H}_1 (z) \le \mathcal{H}_2 (z)$ for each $z \in \mathfrak{R}$. 
		
	\end{lem} 
	
	\begin{proof}
		
		Define the subset $\mathfrak{U} = \big\{ z \in \mathfrak{R}: \mathcal{H}_1 (z) > \mathcal{H}_2 (z) \big\} \subseteq \mathfrak{R}$, which is open since both $\mathcal{H}_i$ are continuous. Furthermore, since $\mathfrak{h}_1 \le \mathfrak{h}_2$, the continuity of the $\mathcal{H}_i$ implies that $\mathcal{H}_1 (z) = \mathcal{H}_2 (z)$, for each $z \in \partial \mathfrak{U}$. Next define the functions $\mathcal{H}_3 : \overline{\mathfrak{R}} \rightarrow \mathbb{R}$ and $\mathcal{H}_4 : \overline{\mathfrak{R}} \rightarrow \mathbb{R}$ by setting 
		\begin{flalign*} 
			\mathcal{H}_3 (z) = \mathcal{H}_1 (z), \quad \text{for $z \notin \mathfrak{U}$, and} \quad \mathcal{H}_3 (z) = \mathcal{H}_2 (z), \quad \text{for $z \in \mathfrak{U}$}; \\
			\mathcal{H}_4 (z) = \mathcal{H}_2 (z), \quad \text{for $z \notin \mathfrak{U}$, and} \quad \mathcal{H}_4 (z) = \mathcal{H}_1 (z), \quad \text{for $z \in \mathfrak{U}$}.
		\end{flalign*} 
	
		\noindent Then, $\mathcal{H}_3$ and $\mathcal{H}_4$ are $1$-Lipschitz, since both $\mathcal{H}_1 (z)$ and $\mathcal{H}_2 (z)$ are. In particular, $\nabla \mathcal{H}_3 (z)$ and $\nabla \mathcal{H}_4 (z)$ exist almost everywhere (with respect to Lebesgue measure), and so $\nabla \mathcal{H}_3 (z) \in \overline{\mathcal{T}}$ and $\nabla \mathcal{H}_4 (z) \in \overline{\mathcal{T}}$ hold for almost every $z \in \mathfrak{R}$, as the same holds for $\mathcal{H}_1$ and $\mathcal{H}_2$. 
		
		We claim that $\mathcal{H}_3$ and $\mathcal{H}_4$ are maximizers for $\mathcal{E}$ on $\mathfrak{R}$. Indeed,  we have  
		\begin{flalign*}
			\mathcal{E}(\mathcal{H}_3) + \mathcal{E}(\mathcal{H}_4) & = \displaystyle\int_{\mathfrak{R} \setminus \mathfrak{U}} \sigma \big( \nabla \mathcal{H}_3 (z) \big) dz + \displaystyle\int_{\mathfrak{U}} \sigma \big( \mathcal{H}_3 (z) \big) dz  + \displaystyle\int_{\mathfrak{R} \setminus \mathfrak{U}} \sigma \big( \nabla \mathcal{H}_4 (z) \big) dz + \displaystyle\int_{\mathfrak{U}} \sigma \big(\nabla \mathcal{H}_4 (z) \big) dz \\
			& = \displaystyle\int_{\mathfrak{R} \setminus \mathfrak{U}} \sigma \big( \nabla \mathcal{H}_1 (z) \big) dz + \displaystyle\int_{\mathfrak{U}} \sigma \big( \mathcal{H}_2 (z) \big) dz  + \displaystyle\int_{\mathfrak{R} \setminus \mathfrak{U}} \sigma \big( \nabla \mathcal{H}_2 (z) \big) dz + \displaystyle\int_{\mathfrak{U}} \sigma \big(\nabla \mathcal{H}_1 (z) \big) dz \\
			& = \mathcal{E} (\mathcal{H}_1) + \mathcal{E} (\mathcal{H}_2),
		\end{flalign*}

		\noindent which implies that $\mathcal{H}_3$ and $\mathcal{H}_4$ are maximizers for $\mathcal{E}$, since $\mathcal{H}_1$ and $\mathcal{H}_2$ are (and $\mathcal{H}_1 |_{\partial \mathfrak{R}} = \mathcal{H}_3 |_{\partial \mathfrak{R}}$ and $\mathcal{H}_2 |_{\partial \mathfrak{R}} = \mathcal{H}_4 |_{\partial \mathfrak{R}}$). Since $\mathcal{H}_1$ and $\mathcal{H}_3$ have the same boundary data on $\mathfrak{R}$, it follows from \Cref{hemaximum} that $\mathcal{H}_1 = \mathcal{H}_3$. This implies that $\mathfrak{U}$ must be empty, and hence $\mathcal{H}_1 (z) \le \mathcal{H}_2 (z)$ for all $z \in \mathfrak{R}$.
	\end{proof} 

	\begin{rem}
		
		\label{h1h2gamma} 
		
		If we adopt the notation of \Cref{h1h2compare} but omit the assumption $\mathfrak{h}_1 \le \mathfrak{h}_2$, then we have that $\sup_{z \in \mathfrak{R}} \big| \mathcal{H}_1 (z) - \mathcal{H}_2 (z) \big| \le \sup_{z \in \partial \mathfrak{R}} \big| \mathfrak{h}_1 (z) - \mathfrak{h}_2 (z) \big|$. Indeed, denoting the latter quantity by $\gamma$, this follows by applying \Cref{h1h2compare} to the pairs $(\mathfrak{h}_1, \mathfrak{h}_2 + \gamma)$ and $(\mathfrak{h}_1 + \gamma, \mathfrak{h}_2)$. 
	\end{rem}

	Next, recall $\mathcal{T}$ from \eqref{t}. For any $\varepsilon \in \big( 0, \frac{1}{4} \big)$, we define the set
	\begin{flalign}
	\label{tset2} 
	\mathcal{T}_{\varepsilon} = \Big\{ (s, t) \in \mathcal{T}: d \big( (s, t), \partial \mathcal{T} \big) > \varepsilon \Big\} \subset \mathcal{T},
	\end{flalign}
	
	\noindent where $d (z, \mathcal{S}) = \inf_{w \in \mathcal{S}} |z - w|$ denotes the distance from a point $z \in \mathbb{R}^2$ to a set $\mathcal{S} \subseteq \mathbb{R}^2$.

	\begin{rem} 
		
		\label{concavesigmat}
		
		For any fixed $\varepsilon \in \big( 0, \frac{1}{4} \big)$, it was shown as Theorem 10.1 of \cite{VPDT} that the function $\sigma$ from \eqref{lsigma} is uniformly concave and smooth on the set $\mathcal{T}_{\varepsilon}$. Moreover, on $\overline{\mathcal{T}}$, the function $\sigma$ is (weakly) concave and uniformly H\"{o}lder continuous with any fixed exponent $\alpha \in (0, 1)$. 
		
	\end{rem} 
	
	Before proceeding, it will be useful to recall the elliptic partial differential equation satisfied by maximizers of $\mathcal{E}$.

	\begin{rem} 
		
		\label{haijequations} 
		
		Suppose that $\mathfrak{R}$ is an open set with smooth boundary and there exists a real number $\varepsilon \in \big( 0, \frac{1}{4} \big)$ such that $\nabla \mathcal{H} (z) \in \mathcal{T}_{\varepsilon}$, for each $z \in \mathfrak{R}$. Then it is known (see Theorem 11.9 and Theorem 15.19 of \cite{EDSO}) that all second derivatives of $\mathcal{H}$ are continuous in the interior of $\mathfrak{R}$, and that $\mathcal{H}$ solves \emph{Euler-Lagrange equations} for $\sigma$. Abbreviating $\partial_z = \frac{\partial}{\partial z}$ for any variable $z$, these are given by 
		\begin{flalign*}
		\partial_x \Big( \partial_s \sigma \big( \nabla F (z) \big) \Big) + \partial_y \Big( \partial_t \sigma \big( \nabla F (z) \big) \Big) = 0,
		\end{flalign*}
		
		\noindent which is equivalent to
		\begin{flalign}
		\label{aijh} 
		\displaystyle\sum_{j, k \in \{ x, y \}} \mathfrak{a}_{jk} \big( \nabla F (z) \big) \partial_j \partial_k F (z) = 0,
		\end{flalign}
		
		\noindent where 
		\begin{flalign}
		\label{aijst}
		\begin{aligned}
		\mathfrak{a}_{xx} (s, t) = \displaystyle\frac{1}{\tan \big( \pi (1 - s - t) \big)} & + \displaystyle\frac{1}{\tan (\pi s)}; \qquad \mathfrak{a}_{yy} (s, t) = \displaystyle\frac{1}{\tan \big( \pi (1 - s - t) \big)} + \displaystyle\frac{1}{\tan (\pi t)}; \\
		&  \mathfrak{a}_{xy} (s, t)  = \displaystyle\frac{1}{\tan \big( \pi (1 - s - t) \big)} =  \mathfrak{a}_{yx} (s, t).
		\end{aligned} 
		\end{flalign} 
		
		 Additionally, any solution $F$ to \eqref{aijh} on $\mathfrak{R}$ with $F |_{\partial \mathfrak{R}} = \mathfrak{h}$, such that $\nabla F (z) \in \mathcal{T}$ for each $z \in \mathfrak{R}$, is a local maximizer of $\mathcal{E}$ on $\mathfrak{R}$ with boundary data $\mathfrak{h}$. Since $\sigma$ is concave, $F$ is in fact a global maximizer of $\mathcal{E}$ and so $F = \mathcal{H}$ by the uniqueness of such maximizers (recall \Cref{hemaximum}). 
		
	\end{rem}

	\begin{rem}
		
		\label{derivativehcontinuouss}
		
		Although \Cref{haijequations} indicates that $\mathcal{H}$ satisfies an elliptic partial differential equation when $\nabla \mathcal{H} (z)$ is bounded away from $\partial \mathcal{T}$, it does not directly provide information about the set $\mathcal{S} = \big\{ z \in \mathfrak{R}: \nabla \mathcal{H} (z) \in \mathcal{T} \big\}$ on which this occurs. Theorem 4.1 of \cite{MCFARS} implies that $\mathcal{S}$ is open and that $\nabla \mathcal{H}$ is continuous on $\mathcal{S}$, if $\mathfrak{R}$ has Lipschitz boundary.
		
	\end{rem}

	Let us next state the second result, which bounds the $\mathcal{C}^2$-norm (defined in \Cref{LawGlobal2}) of a maximizer of $\mathcal{E}$ whose gradient is everywhere uniformly in the interior of $\mathcal{T}$; its proof, which follows quickly from \Cref{haijequations} and known estimates on uniformly elliptic partial differential equations, will be given in \Cref{EstimateEquation} below. In what follows, for any $z \in \mathbb{R}^2$ and $r \in \mathbb{R}_{> 0}$, we define the open disks 
	\begin{flalign}
	\label{brzdefinition} 
	\mathcal{B}_r (z) = \big\{ w \in \mathbb{R}^2: |w - z| < r \big\} \subset \mathbb{R}^2; \qquad \mathcal{B}_r = \mathcal{B}_r (0, 0); \qquad \mathcal{B} = \mathcal{B}_1 = \mathcal{B}_1 (0, 0).	
	\end{flalign}

	\begin{lem}
		
		\label{derivativeshestimate}
		
		For any fixed $\varepsilon \in \big( 0, \frac{1}{4} \big)$, there exists a constant $C = C (\varepsilon) > 1$ such that the following holds. Let $\mathfrak{h}: \partial \mathcal{B} \rightarrow \mathbb{R}$ denote a function admitting an admissible extension to $\mathcal{B}$, and let $\mathcal{H} \in \Adm (\mathcal{B}; \mathfrak{h})$ denote the maximizer of $\mathcal{E}$ on $\mathcal{B}$ with boundary data $\mathfrak{h}$. If $\nabla \mathcal{H} (z) \in \mathcal{T}_{\varepsilon}$ for each $z \in \mathcal{B}$, then $\| \mathcal{H} - \mathcal{H} (0, 0) \|_{\mathcal{C}^2 (\overline{\mathcal{B}}_{1 / 2})} \le C$.
		
	\end{lem}
	
	\begin{rem} 
		
	\label{estimatehrho} 
	
	Although \Cref{derivativeshestimate} was only stated on $\mathcal{B}$, an analog of it also holds on $\mathcal{B}_{\rho}$ for any $\rho \in \mathbb{R}_{> 0}$ (where the associated constant $C$ will now depend on both $\varepsilon$ and $\rho$). Indeed this follows from the scale-invariance of the variational principle, which states that if we define $\mathcal{H}^{(\rho)}: \mathcal{B}_{\rho} \rightarrow \mathbb{R}$ by setting $\mathcal{H}^{(\rho)} (z) = \rho^{-1} \mathcal{H} (\rho^{-1} z)$ for each $z \in \mathcal{B}_{\rho}$, then $\mathcal{H}^{(\rho)}$ is a maximizer of $\mathcal{E}$ on $\mathcal{B}_{\rho}$. 
	
	\end{rem} 
	
	Observe that \Cref{derivativeshestimate} is an \emph{interior estimate} on the $\mathcal{C}^2$-norm of $\mathcal{H}$, since it does not provide bounds close to the boundary $\partial \mathcal{B}$ of the domain. A \emph{global estimate} of the latter type is false without further regularity assumptions on the boundary data $\mathfrak{h}$. 
	
	Before stating the next result, we require some notation on (nearly) linear functions. 
	
	\begin{definition} 
		
	\label{linearst} 
	
	Fix a subset $U \subseteq \mathbb{R}^2$ and pair $(s, t) \in \mathbb{R}^2$. We say that a function $\Lambda: U \rightarrow \mathbb{R}$ is \emph{linear of slope $(s, t)$} if $\Lambda (z) - \Lambda (z') = (s, t) \cdot (z - z')$, for any $z, z' \in U$. Furthermore, for any $\delta > 0$, we say that a function $f: U \rightarrow \mathbb{R}$ is \emph{$\delta$-nearly linear of slope $(s, t)$ (on $U$)} if there exists a linear function $\Lambda: U \rightarrow \mathbb{R}$ of slope $(s, t)$ such that $\sup_{z \in U} \big| f(z) - \Lambda (z) \big| < \delta$.

	\end{definition} 
	
	Now we have the following proposition that provides an interior estimate for the effect of a $\mathcal{C}^0$ boundary perturbation on the gradient of a nearly linear maximizer of $\mathcal{E}$ (and also provides a $\mathcal{C}^2$ estimate for such maximizers). Its proof will be provided in \Cref{BoundaryPerturbationVariational}, where we will first use results from \cite{MCFARS} to reduce it to an estimate on uniformly elliptic partial differential equations, and then apply known bounds (Schauder and interpolation estimates) in the latter context.

	\begin{prop}
		
		\label{perturbationboundary}
		
		For any fixed real number $\varepsilon \in \big( 0, \frac{1}{4} \big)$, there exist constants $\delta = \delta (\varepsilon) > 0$ and $C = C (\varepsilon) > 1$ such that the following holds. Let $\mathfrak{h}_1, \mathfrak{h}_2: \partial \mathcal{B} \rightarrow \mathbb{R}$ denote two functions admitting admissible extensions to $\mathcal{B}$ and, for each $i \in \{ 1, 2 \}$, let $\mathcal{H}_i \in \Adm (\mathcal{B}; \mathfrak{h}_i)$ denote the maximizer of $\mathcal{E}$ on $\mathcal{B}$ with boundary data $\mathfrak{h}_i$. If there exists a pair $(s, t) \in \mathcal{T}_{\varepsilon}$ such that $\mathfrak{h}_i$ is $\delta$-nearly linear of slope $(s, t)$ on $\partial \mathcal{B}$, for each $i \in \{ 1, 2 \}$, then 
		\begin{flalign}
		\label{perturbationboundaryestimate} 
		\displaystyle\sup_{z \in \mathcal{B}_{1 / 4}} \big| \nabla \mathcal{H}_1 (z) - \nabla \mathcal{H}_2 (z) \big| \le C \displaystyle\max_{z \in \partial \mathcal{B}} \big| \mathfrak{h}_1 (z) - \mathfrak{h}_2 (z) \big|; \qquad \displaystyle\max_{i \in \{ 1, 2 \}} \big\| \mathcal{H}_i - \mathcal{H}_i (0, 0) \big\|_{\mathcal{C}^2 (\overline{\mathcal{B}}_{1 / 4})} \le C.
		\end{flalign}

	\end{prop}

	\subsection{Non-Intersecting Path Ensembles}
	
	\label{NonIntersectingCorrelation}
	
	In this section we recall the definition and some properties of a certain model of random non-intersecting paths. To that end, we begin with the following definition. In the below, for any $k \in \mathbb{Z}_{\ge 1}$ and two $k$-tuples $\textbf{a} = \big( a (1), a (2), \ldots , a (k) \big) \in \mathbb{R}^k$ and $\textbf{b} = \big( b (1), b (2), \ldots , b (k) \big) \in \mathbb{R}^k$, we say that $\textbf{a} < \textbf{b}$ if $a (j) < b (j)$ for each $j \in [1, k]$. We additionally define the \emph{Weyl chamber} 
	\begin{flalign*} 
	\mathbb{W}_k = \big\{ (i_1, i_2, \ldots , i_k) \in \mathbb{Z}^k: i_1 < i_2 < \cdots < i_k \big\}; \qquad \mathbb{W} = \bigcup_{j = 1}^{\infty} \mathbb{W}_j. 
	\end{flalign*}

	\begin{definition} 
		
		\label{pmpn} 
		
		Fix an integer $t \ge 0$. A \emph{path} of \emph{length} $t$ is a sequence $\textbf{q} = \big( q (0), q (1), \ldots , q (t) \big) \subset \mathbb{Z}$ such that $q (s) - q (s - 1) \in \{ 0, 1 \}$, for each $s \in [1, t]$. A \emph{path ensemble} is a family $\textbf{Q} = \big( \textbf{q}_{-m}, \textbf{q}_{1 - m}, \ldots , \textbf{q}_n \big) \subset \mathbb{Z}^{t + 1}$ of sequences, each $\textbf{q}_k = \big( q_k (0), q_k (1), \ldots , q_k (t) \big) \subset \mathbb{Z}$ of which is a path of length $t$. The ensemble $\textbf{Q}$ is called \emph{non-intersecting} if $\textbf{q}_k < \textbf{q}_{k + 1}$, for each $k \in [-m, n - 1]$. Setting $\textbf{q} (s) = \big( q_{-m} (s), q_{1 - m} (s), \ldots , q_n (s) \big)$ for every $s \in [0, t]$, the path ensemble $\textbf{Q}$ is non-intersecting if and only if each $\textbf{q} (s) \in \mathbb{W}$. 
		
	\end{definition} 
	
	By viewing $s$ as a time index, $\textbf{q}_k = \big\{ q_k (s) \big\}_s$ describes the space-time trajectory of path $k$ in the ensemble $\textbf{Q}$, which may either not move or jump to the right after each time step. Moreover, $\textbf{q} (s) = \big\{ q_k (s) \big\}_k$ denotes the set of positions of the paths in $\textbf{Q}$ at time $s$. 
	
	It will be useful to embed $\textbf{Q}$ in the discrete upper-half plane $\mathbb{Z} \times \mathbb{Z}_{\ge 0}$ by first placing a vertex at $\big( q_k (s), s \big)$ for each $(k, s) \in [-m, n] \times [0, t]$, and then by connecting $\big( q_k (s), s \big)$ and $\big( q_k (s + 1), s + 1 \big)$ by an edge for each $(k, s) \in [-m, n] \times [0, t - 1]$. This produces a family of non-intersecting paths on $\mathbb{Z} \times \mathbb{Z}_{\ge 0}$, as shown on the left side of \Cref{pathsfigure}.

	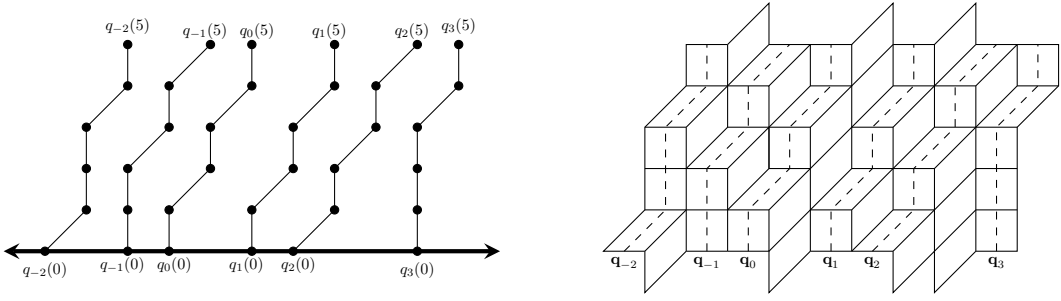
\begin{figure}

		\begin{center}

			\begin{tikzpicture}[
			>=stealth,
			auto,
			style={
				scale = .55
			}
			]
			
			\draw[<->, black, ultra thick] (-.5, 0) -- (11.5, 0);	
			
			\draw[-, black] (.5, 0) -- (1.5, 1) -- (1.5, 2) -- (1.5, 3) -- (2.5, 4) -- (2.5, 5); 
			\draw[-, black] (2.5, 0) -- (2.5, 1) -- (2.5, 2) -- (3.5, 3) -- (3.5, 4) -- (4.5, 5); 
			\draw[-, black] (3.5, 0) -- (3.5, 1) -- (4.5, 2) -- (4.5, 3) -- (5.5, 4) -- (5.5, 5);
			\draw[-, black] (5.5, 0) -- (5.5, 1) -- (6.5, 2) -- (6.5, 3) -- (7.5, 4) -- (7.5, 5); 
			\draw[-, black] (6.5, 0) -- (7.5, 1) -- (7.5, 2) -- (8.5, 3) -- (8.5, 4) -- (9.5, 5);
			\draw[-, black] (9.5, 0) -- (9.5, 1) -- (9.5, 2) -- (9.5, 3) -- (10.5, 4) -- (10.5, 5);
			
			\filldraw[fill=black] (.5, 0) circle [radius = .1] node[below = 2, scale = .6]{$q_{-2} (0)$};	
			\filldraw[fill=black] (2.5, 0) circle [radius = .1] node[left = 2, below, scale = .6]{$q_{-1} (0)$};	
			\filldraw[fill=black] (3.5, 0) circle [radius = .1] node[right = 2, below, scale = .6]{$q_0 (0)$};	
			\filldraw[fill=black] (5.5, 0) circle [radius = .1] node[left = 2, below, scale = .6]{$q_1 (0)$};	
			\filldraw[fill=black] (6.5, 0) circle [radius = .1] node[right = 2, below, scale = .6]{$q_2 (0)$};	
			\filldraw[fill=black] (9.5, 0) circle [radius = .1] node[below = 2, scale = .6]{$q_3 (0)$};
			
			\filldraw[fill=black] (1.5, 1) circle [radius = .1];
			\filldraw[fill=black] (2.5, 1) circle [radius = .1];
			\filldraw[fill=black] (3.5, 1) circle [radius = .1];
			\filldraw[fill=black] (5.5, 1) circle [radius = .1];
			\filldraw[fill=black] (7.5, 1) circle [radius = .1];
			\filldraw[fill=black] (9.5, 1) circle [radius = .1];	
			
			\filldraw[fill=black] (1.5, 2) circle [radius = .1];
			\filldraw[fill=black] (2.5, 2) circle [radius = .1];
			\filldraw[fill=black] (4.5, 2) circle [radius = .1];
			\filldraw[fill=black] (6.5, 2) circle [radius = .1];
			\filldraw[fill=black] (7.5, 2) circle [radius = .1];
			\filldraw[fill=black] (9.5, 2) circle [radius = .1];
			
			\filldraw[fill=black] (1.5, 3) circle [radius = .1];
			\filldraw[fill=black] (3.5, 3) circle [radius = .1];
			\filldraw[fill=black] (4.5, 3) circle [radius = .1];
			\filldraw[fill=black] (6.5, 3) circle [radius = .1];
			\filldraw[fill=black] (8.5, 3) circle [radius = .1];
			\filldraw[fill=black] (9.5, 3) circle [radius = .1];
			
			\filldraw[fill=black] (2.5, 4) circle [radius = .1];
			\filldraw[fill=black] (3.5, 4) circle [radius = .1];
			\filldraw[fill=black] (5.5, 4) circle [radius = .1];
			\filldraw[fill=black] (7.5, 4) circle [radius = .1];
			\filldraw[fill=black] (8.5, 4) circle [radius = .1];
			\filldraw[fill=black] (10.5, 4) circle [radius = .1];
			
			\filldraw[fill=black] (2.5, 5) circle [radius = .1] node[above = 2, scale = .6]{$q_{-2} (5)$};
			\filldraw[fill=black] (4.5, 5) circle [radius = .1] node[left = 2, above, scale = .6]{$q_{-1} (5)$};
			\filldraw[fill=black] (5.5, 5) circle [radius = .1] node[right = 2, above, scale = .6]{$q_0 (5)$};
			\filldraw[fill=black] (7.5, 5) circle [radius = .1] node[left = 2, above, scale = .6]{$q_1 (5)$};
			\filldraw[fill=black] (9.5, 5) circle [radius = .1] node[left = 2, above, scale = .6]{$q_2 (5)$};
			\filldraw[fill=black] (10.5, 5) circle [radius = .1] node[above = 2, scale = .6]{$q_3 (5)$};

			\draw[-, black, dashed] (16.5, 0) node[below = 0, scale = .6]{$\textbf{q}_{-1}$} -- (16.5, 1) -- (16.5, 2) -- (17.5, 3) -- (17.5, 4) -- (18.5, 5); 
			\draw[-, black, dashed] (17.5, 0) node[below = 0, scale = .6]{$\textbf{q}_0$} -- (17.5, 1) -- (18.5, 2) -- (18.5, 3) -- (19.5, 4) -- (19.5, 5);
			\draw[-, black, dashed] (19.5, 0) node[below, scale = .6]{$\textbf{q}_1$} -- (19.5, 1) -- (20.5, 2) -- (20.5, 3) -- (21.5, 4) -- (21.5, 5); 
			\draw[-, black, dashed] (20.5, 0) node[below, scale = .6]{$\textbf{q}_2$} -- (21.5, 1) -- (21.5, 2) -- (22.5, 3) -- (22.5, 4) -- (23.5, 5);

			\draw[-, black, dashed] (14.5, 0) node[below, scale = .6]{$\textbf{q}_{-2}$} -- (15.5, 1) -- (15.5, 2) -- (15.5, 3) -- (16.5, 4) -- (16.5, 5); 
			\draw[-, black] (16, 0) -- (16, 1) -- (16, 2) -- (17, 3) -- (17, 4) -- (18, 5); 
			\draw[-, black] (17, 0) -- (17, 1) -- (18, 2) -- (18, 3) -- (19, 4) -- (19, 5);
			\draw[-, black] (19, 0) -- (19, 1) -- (20, 2) -- (20, 3) -- (21, 4) -- (21, 5); 
			\draw[-, black] (20, 0) -- (21, 1) -- (21, 2) -- (22, 3) -- (22, 4) -- (23, 5);
			\draw[-, black] (23, 0) -- (23, 1) -- (23, 2) -- (23, 3) -- (24, 4) -- (24, 5);
			\draw[-, black, dashed] (23.5, 0) node[below, scale = .6]{$\textbf{q}_3$} -- (23.5, 1) -- (23.5, 2) -- (23.5, 3) -- (24.5, 4) -- (24.5, 5);
			
			\draw[-, black] (15, 0) -- (14, 0) -- (15, 1) -- (15, 2) -- (15, 3) -- (16, 4) -- (16, 5) -- (17, 5);
			\draw[-, black] (15, 1) -- (16, 1); 
			\draw[-, black] (15, 2) -- (16, 2); 
			\draw[-, black] (15, 3) -- (16, 3); 
			\draw[-, black] (16, 4) -- (17, 4); 
			\draw[-, black] (16, 3) -- (17, 4);
			\draw[-, black] (17, 5) -- (17, 4);
			\draw[-, black] (16, 0) -- (17, 0);
			\draw[-, black] (16, 1) -- (17, 1) -- (17, 2) -- (16, 2);
			\draw[-, black] (17, 2) -- (18, 3) -- (18, 2);
			\draw[-, black] (18, 3) -- (18, 4) -- (17, 4);
			\draw[-, black] (18, 5) -- (19, 5) -- (18, 4);
			\draw[-, black] (19, 5) -- (20, 5) -- (20, 4) -- (19, 4);
			\draw[-, black] (20, 4) -- (19, 3) -- (19, 2) -- (18, 1) -- (18, 0) -- (17, 0); 
			\draw[-, black] (17, 1) -- (18, 1);
			\draw[-, black] (18, 2) -- (19, 2);
			\draw[-, black] (18, 3) -- (19, 3);
			\draw[-, black] (18, 0) -- (19, 1) -- (19, 2) -- (20, 3) -- (20, 4) -- (21, 5);
			\draw[-, black] (20, 0) -- (20, 1) -- (21, 2) -- (21, 3) -- (22, 4) -- (22, 5) -- (21, 5);
			\draw[-, black] (21, 0) -- (22, 1) -- (22, 2) -- (23, 3) -- (23, 4) -- (24, 5) -- (23, 5);
			\draw[-, black] (19, 0) -- (21, 0);
			\draw[-, black] (19, 1) -- (20, 1);
			\draw[-, black] (20, 2) -- (22, 2);
			\draw[-, black] (20, 3) -- (21, 3);
			\draw[-, black] (21, 4) -- (23, 4);
			\draw[-, black] (22, 3) -- (23, 3);
			\draw[-, black] (23, 2) -- (22, 1) -- (22, 0) -- (23, 1);
			\draw[-, black] (15, 0) -- (15, -1) -- (16, 0); 
			\draw[-, black] (18, 0) -- (18, -1) -- (19, 0); 
			\draw[-, black] (21, 0) -- (21, -1) -- (22, 0) -- (22, -1) -- (23, 0);
			\draw[-, black] (17, 5) -- (18, 6) -- (18, 5);  
			\draw[-, black] (20, 5) -- (21, 6) -- (21, 5);  
			\draw[-, black] (22, 5) -- (23, 6) -- (23, 5);  
			\draw[-, black] (21, 1) -- (22, 1);
			\draw[-, black] (17, 3) -- (18, 3);
			\draw[-, black] (23, 0) -- (24, 0) -- (24, 1) -- (24, 2) -- (24, 3) -- (25, 4) -- (25, 5) -- (24, 5);
			\draw[-, black] (23, 1) -- (24, 1); 
			\draw[-, black] (23, 2) -- (24, 2); 
			\draw[-, black] (23, 3) -- (24, 3); 
			\draw[-, black] (24, 4) -- (25, 4);
			\draw[-, black] (15, 0) -- (16, 1); 
			\draw[-, black] (16, 2) -- (16, 3);
			
			\end{tikzpicture}
			
		\end{center}

		\caption{\label{pathsfigure} Depicted to the left is an ensemble $\textbf{Q} = \big( \textbf{q}_{-2}, \textbf{q}_{-1}, \textbf{q}_0, \textbf{q}_1, \textbf{q}_2, \textbf{q}_3 \big)$ consisting of six non-intersecting paths. Depicted to the right is (part of) the associated free lozenge tiling of $\Delta = \Delta^5$. }

	\end{figure} 
	
	It will also be useful to introduce notation for the sets of non-intersecting path ensembles that start and end at given locations; this is given by the following definition. 
	
	\begin{definition} 
		
		\label{fgwp0pt}
		
		For any $t \in \mathbb{Z}_{\ge 0}$, let $\mathfrak{W}^t$ denote the set of all non-intersecting path ensembles $\textbf{Q}$ whose paths are all of length $t$. Next, fix $m, n \in \mathbb{Z}$ with $-m \le n$ and two sequences $\textbf{a} = \big( a_{-m}, a_{1 - m}, \ldots , a_n \big) \in \mathbb{W}$ and $\textbf{b} = \big( b_{-m}, b_{1 - m}, \ldots , b_n \big) \in \mathbb{W}$. We say that a non-intersecting path ensemble $\textbf{Q} = \big( \textbf{q}_{-m}, \textbf{q}_{1 - m}, \ldots , \textbf{q}_n \big) \in \mathfrak{W}^t$ has \emph{initial data} $\textbf{a}$ if $q_k (0) = a_k$, for each $k \in [-m, n]$. Similarly, it has \emph{ending data} $\textbf{b}$ if $q_k (t) = b_k$, for each $k \in [-m, n]$. 
		
		Let $\mathfrak{W}_{\textbf{a}}^t \subseteq \mathfrak{W}^t$ denote the set of all non-intersecting path ensembles $\textbf{Q} \in \mathfrak{W}^t$ with initial data $\textbf{a}$. Additionally, let $\mathfrak{W}_{\textbf{a}; \textbf{b}}^t \subseteq \mathfrak{W}_{\textbf{a}}^t$ denote the set of all non-intersecting path ensembles $\textbf{Q} \in \mathfrak{W}^t$ that have initial data $\textbf{a}$ and ending data $\textbf{b}$. 
	\end{definition}

	Fix $t \in \mathbb{Z}_{\ge 1}$, and define the domain $\Delta = \Delta^t = \mathbb{Z} \times \{ 0, 1, \ldots , t \} \subset \mathbb{Z}^2$. Recalling the notation from \Cref{TilingsHeight} on free tilings, let us associate with any non-intersecting path ensemble $\textbf{Q} = \big( \textbf{q}_{-m}, \textbf{q}_{1 - m}, \ldots , \textbf{q}_n \big) \in \mathfrak{W}^t$ a free lozenge tiling $\mathscr{M} = \mathscr{M} (\textbf{Q})$ of $\Delta$. To that end, define the set 
	\begin{flalign}
	\label{xset}
	\mathscr{X} (\textbf{Q}) = \Big\{ (x, s) \in \Delta^t: x \notin \big\{ q_{-m} (s), q_{1 - m} (s), \ldots , q_n (s) \big\} \Big\} \subset \Delta^t,
	\end{flalign}  
	
	\noindent and associate with $\textbf{Q}$ the unique free tiling $\mathscr{M}$ of $\Delta$ whose type $1$ lozenges are centered at the vertices of $\mathscr{X} (\textbf{Q}) + \big( \frac{1}{2}, 0 \big)$. Stated alternatively, $\mathscr{M}$ determined by setting $\mathscr{X} (\mathscr{M}) = \mathscr{X} (\textbf{Q})$, where we recall the set $\mathscr{X} (\mathscr{M})$ from \Cref{xm}. See the right side of \Cref{pathsfigure} for a depiction (where there we have shifted the non-intersecting paths to the right by $\frac{1}{2}$, to make them more visible). 
	
	Under this correspondence, lozenges of types $2$ and type $3$ in $\mathscr{M}$ correspond to space-time locations where a path jumps to the right or does not move, respectively; these constitute the tiles on the right side of \Cref{pathsfigure} through which a path passes. In this way, any free tiling $\mathscr{M}$ of $\Delta$ is associated with a unique non-intersecting path ensemble $\textbf{Q} = \textbf{Q} (\mathscr{M}) \in \mathfrak{W}^t$ (possibly with infinitely many paths). Thus, any height function $H: \Delta \rightarrow \mathbb{Z}$ gives rise to a unique such ensemble. 
	
	\begin{rem} 
		
		\label{heightpaths}
		
		It is quickly verified that, for any $(x_1, s), (x_2, s) \in \Delta$ such that $x_1 \le x_2$, the number of paths in $\textbf{Q}$ that intersect the interval $\big[ \big( x_1 - \frac{1}{2}, s \big), \big( x_2 - \frac{1}{2}, s \big) \big]$ is equal to $x_2 - x_1 - \big( H(x_2, s) - H(x_1, s) \big)$. Moreover, for any $(x, s_1), (x, s_2) \in \Delta$ such that $s_1 \le s_2$, the number of paths in $\textbf{Q}$ that intersect the interval $\big[ \big( x - \frac{1}{2}, s_1 \big), \big( x - \frac{1}{2}, s_2 \big) \big]$ is equal to $H (x, s_2) - H(x, s_1)$. 
		
	\end{rem} 
	
	Next, let us recall a model for random non-intersecting path ensembles that was introduced in Section 3.1 of \cite{NRWTQPE}. Here, we set $|\textbf{p}| = \sum_{j = -m}^n p_j$ for any $\textbf{p} = \big( p_{-m}, p_{1 - m}, \ldots , p_n \big) \in \mathbb{R}^{m + n + 1}$. 
	
	\begin{definition}[{\cite[Section 3.1]{NRWTQPE}}] 
		
		\label{betaprobability}
		
		Fix a real number $\beta \in (0, 1)$; integers $t, m, n$ with $t \ge 0$ and $-m \le n$; and a sequence $\textbf{a} = \big( a_{-m}, a_{1 - m}, \ldots , a_n \big) \in \mathbb{W}$. Define the probability measure $\mathbb{P} = \mathbb{P}_{\beta; \textbf{a}} = \mathbb{P}_{\beta; \textbf{a}; t}$ on $\mathfrak{W}_{\textbf{a}}^t$ as follows. For any non-intersecting path ensemble $\textbf{Q} = \big( \textbf{q}_{-m}, \textbf{q}_{1 - m}, \ldots , \textbf{q}_n \big) \in \mathfrak{W}_{\textbf{a}}^t$ with initial data $\textbf{q} (0) = \textbf{a}$, set
		\begin{flalign}
		\label{pprobability} 
		\mathbb{P} [\textbf{Q}] = \beta^{|\textbf{q} (t)| - |\textbf{a}|} (1 - \beta)^{(m + n + 1) t - |\textbf{q} (t)| + |\textbf{a}|} \displaystyle\prod_{-m \le j < k \le n} \displaystyle\frac{q_k (t) - q_j (t)}{a_k - a_j}.
		\end{flalign} 
		
		\noindent As in \Cref{pmpn}, here we have denoted $\textbf{q}_k = \big( q_k (0), q_k (1), \ldots , q_k (t) \big)$ for each $k \in [-m, n]$, and $\textbf{q} (s) = \big( q_{-m} (s), q_{1 - m} (s), \ldots , q_n (s) \big)$ for each $s \in [0, t]$. 
		
	\end{definition}

		As indicated in Proposition 3.1 of \cite{NRWTQPE}, the measure $\mathbb{P}_{\beta; \textbf{a}}$ is the probability distribution for a collection of $m + n + 1$ mutually independent random walks starting at positions $(a_{-m}, a_{1 - m}, \ldots , a_n)$, which at each time either do not move or jump to the right with probabilities $1 - \beta$ and $\beta$, respectively, conditioned to never intersect.

	\begin{rem}
		
		\label{prprobability} 
		
		Fix sequences $\textbf{a} = \big( a_{-m}, a_{1 - m}, \ldots , a_n \big) \in \mathbb{W}$ and $\textbf{b} = \big( b_{-m}, b_{1 - m}, \ldots , b_n \big) \in \mathbb{W}$, and let $\textbf{Q} \in \mathfrak{W}_{\textbf{a}}^t$ denote a random non-intersecting path ensemble sampled from the measure $\mathbb{P}_{\beta; \textbf{a}}$, conditioned to have ending data $\textbf{b}$. Then it quickly follows from the explicit form \eqref{pprobability} for $\mathbb{P}$ that the conditional law of $\textbf{Q}$ is uniform on $\mathfrak{W}_{\textbf{a}; \textbf{b}}^t$. 
		
	\end{rem}
	
	\begin{rem}
		
		\label{qtprocess} 
		
		Using the explicit form \eqref{pprobability} for $\mathbb{P}$, it can be shown that the family $\big\{ \textbf{q} (s) \big\}_s$ of sequences forms a Markov process with time parameter $s$. 
		
	\end{rem}

	Next we have the following lemma, which is a concentration estimate for the height function associated with a non-intersecting random path ensemble sampled from the measure $\mathbb{P}_{\beta; \textbf{a}}$. Its proof is similar to that of Proposition 22 of \cite{LSRT} (restated as \Cref{hvuexpectation} above), but we include it below. 
	
	\begin{lem}
		
		\label{walksexpectationnear} 
		
		Fix an integer $t \ge 0$; a real number $\beta \in (0, 1)$; and an integer sequence $\textbf{\emph{a}} \in \mathbb{W}$. Let $\textbf{\emph{Q}} \in \mathfrak{W}_{\textbf{\emph{a}}}^t$ denote a random non-intersecting path ensemble sampled under the measure $\mathbb{P}_{\beta; \textbf{\emph{a}}}$, and let $H: \mathbb{Z} \times \{ 0, 1, \ldots , t \} \rightarrow \mathbb{Z}$ denote the unique height function associated with $\textbf{\emph{Q}}$ such that $H(0, 0) = 0$. Then, for any real number $A > 1$ and vertex $v \in \mathbb{Z} \times \{ 0, 1, \ldots , t \}$, we have that
		\begin{flalign*} 
		\mathbb{P} \bigg[ \Big| H (v) - \mathbb{E} \big[ H(v) \big] \Big| > A \bigg] \le 2 e^{-A^2 / 2 t}.
		\end{flalign*}

	\end{lem}

	\begin{proof} 
		For each integer $s \in [0, t]$, let $\mathcal{F}_s$ denote the $\sigma$-algebra generated by the random variables $\big\{ H (u): u \in \mathbb{Z} \times \{ 0, 1, \ldots , s \} \big\}$. Furthermore, for each vertex $v \in \mathbb{Z} \times \{ 0, 1, \ldots , t \}$ and integer $s \in [0, t]$, define the random variable $H^{(s)} (v) = \mathbb{E} \big[ H (v) | \mathcal{F}_s \big]$. Then, as $s$ ranges over $[0, t]$, the sequence $\big\{ H^{(s)} (v) \big\}$ forms a martingale with the property that $\big| H^{(s + 1)} (v) - H^{(s)} (v) \big| \le 1$ for each $v \in \mathbb{Z} \times \{ 0, 1, \ldots , t \}$ and $s \in [0, t - 1] \cap \mathbb{Z}$. Therefore, the Azuma-Hoeffding inequality yields 
		\begin{flalign*} 
		\mathbb{P} \bigg[ \Big| H (v) - \mathbb{E} \big[ H (v) \big] \Big| > A \bigg] = \mathbb{P} \bigg[ \Big| H^{(t)} (v) - H^{(0)} (v) \Big| > A \bigg] \le 2 e^{-A^2 / 2t},
		\end{flalign*}
		
		\noindent for any vertex $v \in \mathbb{Z} \times \{ 0, 1, \ldots , t \}$. 	
	\end{proof}

	\section{Convergence of Local Statistics} 
	
	\label{ConvergenceProof}

	In this section we establish \Cref{localconverge} assuming two results (namely, a local law and a local coupling statement, given by \Cref{heightlocal1} and \Cref{plpprcouple} below, respectively) that will be proven later in this paper. This will essentially proceed by coupling a random tiling $\mathscr{M}_N$ of the domain $R_N$ locally around the vertex $v_N$ with a random non-intersecting path ensemble sampled according to the measure $\mathbb{P}_{\beta; \textbf{a}}$ from \Cref{betaprobability} (for a suitable choice of $\beta \in (0, 1)$ and $\textbf{a} \in \mathbb{W}$). The benefit of the latter model is that it is more amenable to methods of exact solvability, as the results of \cite{ULSRW} show that it is a determinantal point process with an explicit kernel whose asymptotics can be analyzed for general $\beta$ and $\textbf{a}$. 
	
	Thus, we begin in \Cref{ProbabilityP} by recalling these results of \cite{ULSRW}. Then, in \Cref{RegularLocal}, we state a local law for uniformly random lozenge tilings with nearly arbitrary boundary data, which implies that the limit shape result \Cref{hnh} holds on scales much smaller than $N$. Next, in \Cref{EnsemblesCompare} we define two sequences $\textbf{p}, \textbf{r} \in \mathbb{W}$ and state a result indicating that $\mathscr{M}_N$ can be locally coupled with random non-intersecting path ensembles with initial data $\textbf{p}$ and $\textbf{r}$. In \Cref{DEstimate} we use the local law \Cref{heightlocal1} to establish certain estimates on $\textbf{p}$ and $\textbf{r}$ that will be necessary to apply the results from \cite{ULSRW}. We then conclude with the proof of \Cref{localconverge} in \Cref{Convergence}.

	\subsection{The Correlation Kernel of \texorpdfstring{$\mathbb{P}_{\beta; \textbf{a}}$}{}}
	
	\label{ProbabilityP}
	
	In this section we recall results from \cite{ULSRW} realizing the random walk model $\mathbb{P}_{\beta; \textbf{a}}$ from \Cref{betaprobability} as a determinantal point process. To that end, for $T \in \mathbb{Z}_{\ge 0}$ and a non-intersecting path ensemble $\textbf{Q} \in \mathfrak{W}^T$, we recall the set $\mathscr{X} (\textbf{Q})$ from \eqref{xset}. The following lemma, which follows from Theorem 2.1 of \cite{ULSRW} and Kerov's complementation principle for determinantal point processes (see Proposition A.8 of \cite{AMSG} or Remark 2.4 of \cite{ULSRW}), indicates that $\mathscr{X} (\textbf{Q})$ is a determinantal point process and evaluates its kernel, if $\textbf{Q}$ is distributed according to $\mathbb{P}_{\beta; \textbf{a}}$. It can be viewed as a discrete analog of the Br\'{e}zin-Hikami identity (see Section 4 of \cite{LSE} or Section 2 of \cite{ULSDCM}) that realizes the $\beta = 2$ case of Dyson Brownian motion as a determinantal point process with an explicit kernel. 
	
	In the below, we recall the \emph{Pochhammer symbol} $(a)_k = \prod_{j = 0}^{k - 1} (a + j)$, for any $a \in \mathbb{C}$ and $k \in \mathbb{Z}_{\ge 0}$.

	\begin{lem}[{\cite[Theorem 2.1]{ULSRW}}]
		
		\label{determinantnonintersecting} 
		
		Fix integers $T \ge 0$ and $-m \le n$; a real number $\beta \in (0, 1)$; and an integer sequence $\textbf{\emph{a}} = \big( a_{-m}, a_{1 - m}, \ldots , a_n \big) \in \mathbb{W}$. Let $\textbf{\emph{Q}} \in \mathfrak{W}_{\textbf{\emph{a}}}^T$ denote a random non-intersecting path ensemble sampled from the measure $\mathbb{P}_{\beta; \textbf{\emph{a}}}$, and recall the set $\mathscr{X} (\textbf{\emph{Q}})$ from \eqref{xset}. Then for any integers $k \ge 0$; $t_1, t_2, \ldots , t_k \in \{ 0, 1, \ldots , T \}$; and $x_1, x_2, \ldots , x_k \in \mathbb{Z}$, we have that 
		\begin{flalign*}
		\mathbb{P} \Bigg[ \bigcap_{j = 1}^k \big\{ (x_j, t_j) \in \mathscr{X} (\textbf{\emph{Q}}) \big\} \Bigg] = \det \big[ K (x_i, t_i; x_j, t_j) \big]_{1 \le i, j \le k},
		\end{flalign*}
		
		\noindent where for any $x, y \in \mathbb{Z}$ and $t, s \in \mathbb{Z}_{\ge 0}$ the kernel $K (x, t; y, s) = K_{\beta; \textbf{\emph{a}}} (x, t; y, s)$ is given by
		\begin{flalign}
		\label{kxtys}
		\begin{aligned}
		K (x, t; y, s) & =  \textbf{\emph{1}}_{x = y} \textbf{\emph{1}}_{t = s} - \textbf{\emph{1}}_{x \ge y} \textbf{\emph{1}}_{t > s} (-1)^{x - y + 1} \binom{t - s}{x - y}  \\
		& \qquad - \displaystyle\frac{t!}{(s - 1)!} \displaystyle\frac{1}{(2 \pi \textbf{\emph{i}})^2} \displaystyle\oint \displaystyle\oint \displaystyle\frac{(z - y + 1)_{s - 1}}{(w - x)_{t + 1}} \left( \displaystyle\frac{1 - \beta}{\beta} \right)^{w - s} \displaystyle\prod_{j = -m}^n \displaystyle\frac{z - a_j}{w - a_j} \\
		& \qquad \qquad \qquad \qquad \qquad \qquad \times \displaystyle\frac{\sin (\pi w)}{\sin (\pi z)} \displaystyle\frac{1}{w - z}  dw dz.
		\end{aligned}
		\end{flalign}
		
		\noindent Here, the contour for $z$ is a line from $y - s + \frac{1}{2} - \textbf{\emph{i}} \infty$ to $y - s + \frac{1}{2} + \textbf{\emph{i}} \infty$, and the contour for $w$ is positively oriented and encloses the elements in the intersection $\{ x - t, x - t + 1, \ldots , x \} \cap \{ a_{-m}, a_{1 - m}, \ldots , a_n \}$, but does not intersect the contour for $z$. 
		
	\end{lem}

	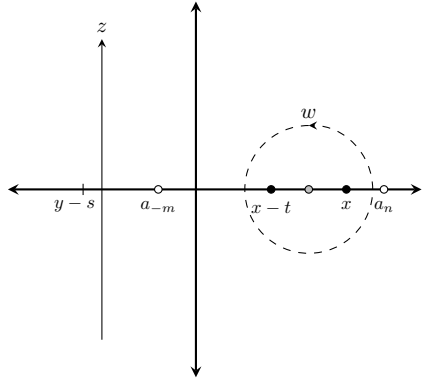
\begin{figure}

		\begin{center}

			\begin{tikzpicture}[
			>=stealth,
			auto,
			style={
				scale = .5
			}
			]

			\draw[<->, thick] (-5, 0) -- (6, 0);
			\draw[<->, thick] (0, -5) -- (0, 5);
			
			\filldraw[fill=white] (5, 0) circle [radius = .1] node[below = 2, scale = .7]{$a_n$};
			\filldraw[fill = black] (4, 0) circle [radius = .1] node[below = 2, scale = .7]{$x$};
			\filldraw[fill = gray!50!white] (3, 0) circle [radius = .1];
			\filldraw[fill = black] (2, 0) circle [radius = .1] node[below = 2, scale = .7]{$x - t$};
			\filldraw[fill=white] (-1, 0) circle [radius = .1] node[below = 2, scale = .7]{$a_{-m}$};

			\draw[->, black] (-2.5, -4) -- (-2.5, 4) node[above, scale = .8]{$z$};
			
			\draw[->, dashed] (3, 1.7) arc(90:450:1.7) node[above, scale = .8]{$w$};

			\draw[-] (-3, -.15) -- (-3, .15) node[below = 8, left = -7, scale = .7]{$y - s$};

			\end{tikzpicture}
			
		\end{center}

		\caption{\label{kcontours} Depicted above are the contours for $w$ (dashed) and $z$ (solid) from \eqref{kxtys}. The black vertices are the elements of $[x - t, x] \cap \textbf{a}$; the gray ones are those of $[x - t, x] \setminus \textbf{a}$, and the white ones are those of $\textbf{a} \setminus [x - t, x]$.}

	\end{figure}

	We refer to \Cref{kcontours} for a depiction of the contours in \Cref{determinantnonintersecting}; if $x - t \le y - s < x$, then the contour for $w$ might instead be a union of two disjoint circles. The independence of the right side of \eqref{kxtys} on $T$ is a consequence of the fact that $\textbf{Q}$ is a Markov process (recall \Cref{qtprocess}). 
	
	The next result we require from \cite{ULSRW} provides large-time asymptotics for the kernel $K (x, t; y, s)$ from \eqref{kxtys}. To that end, we first recall the following definition from equation (2.5) of \cite{ULSRW}. 
	
	\begin{definition}[{\cite[Equation (2.5)]{ULSRW}}]
		
		\label{dqut1t2}
		
		Fix an integer sequence $\textbf{a} = (a_{-m}, a_{1 - m}, \ldots , a_n) \in \mathbb{W}$. For any real numbers $0 < X \le Y$ (we allow $Y = \infty$), set 
		\begin{flalign*}
		\mathfrak{D} (\textbf{a}; X, Y) = \left| \displaystyle\sum_{|a_k| \in [X, Y]} \displaystyle\frac{1}{a_k} \right|. 
		\end{flalign*}

	\end{definition}

	Let us next describe a condition, due to Theorem 2.12 of \cite{ULSRW}, under which the kernel $K_{\beta; \textbf{a}} (x, t + T; y, s + T)$ from \eqref{kxtys} converges to the extended discrete sine kernel $\mathcal{K}_{\xi} (x, t; y, s)$ from \Cref{kernellimitdefinition}, as $T$ tends to $\infty$. Before providing its precise statement, let us take a moment to outline it.
	
	Let $\textbf{a} \in \mathbb{W}_N$, with $T \ll N$, and assume that $\textbf{a}$ satisfies the following two conditions. First, $\mathfrak{D} (\textbf{a}; X, \infty) \ll 1$ whenever $T \lesssim X \lesssim N$. Second, there exist scales $1 \ll U \ll T \ll V \ll N$ (which can all be $N$-dependent) and some $\rho \in (0, 1)$ such that $\textbf{a}$ approximates the Lebesgue measure with density $\rho$ on subintervals of $[-V, V]$ of length $U$. Then, for any $\beta \in (0, 1)$ and bounded $x, y, t, s \in \mathbb{Z}$, $K_{\beta; \textbf{a}} (x, t + T; y, s + T)$ converges to $\mathcal{K}_{\xi} (x, t; y, s)$ for some explicit $\xi = \xi (\beta, \rho) \in \mathbb{H}$, as $T$ tends to $\infty$. 
	
	It will be useful for this rate of convergence, given by $\vartheta_N$ below, to be uniform in various parameters. So, in what follows, $B$ will bound $|x|, |y|, |s|, |t|$ from above; $c$ will bound the distance from $\beta$ and $\rho$ to $\{ 0, 1 \}$ from below; and $\kappa_N$ will dictate the separation of scales $1 \ll U \ll T \ll V \ll N$, the rate of convergence of $\textbf{a}$ to (a multiple of) Lebesgue measure on $[-V, V]$, and the bound on $\mathfrak{D} (\textbf{a}; X, \infty)$. 
	
	\begin{lem}[{\cite[Theorem 2.12]{ULSRW}}]
		
		\label{limitdeterminantnonintersecting}
		
		For any fixed real numbers $0 < c < \frac{1}{2} < B$ and sequence $\kappa = (\kappa_1, \kappa_2, \ldots ) \subset \mathbb{R}_{> 0}$ of real numbers tending to $0$, there exists a sequence $\vartheta = \vartheta (c, B, \kappa)= (\vartheta_1, \vartheta_2, \ldots ) \subset \mathbb{R}_{> 0}$ of real numbers tending to $0$ such that the following holds. Let $U, T, V, N \in \mathbb{Z}$ denote integers such that $\kappa_N^{-1} < U < \kappa_N T < \kappa_N^2 V < \kappa_N^3 N$. Further fix $\rho \in (c, 1 - c)$ and $\textbf{\emph{a}} \in \mathbb{W}_N$ satisfying the following two assumptions. 
		
		\begin{enumerate}
			
			\item For each real number $X \in [\kappa_N T, \kappa_N^{-1} N]$, we have that $\mathfrak{D} (\textbf{\emph{a}}; X; \infty) < \kappa_N$. 
			
			\item For any interval $I \subset [-V, V]$ of length $U$, we have that $(\rho - \kappa_N) U \le |\textbf{\emph{a}} \cap I| \le (\rho + \kappa_N) U$.
			
		\end{enumerate}
		
		\noindent Additionally let $\beta \in (c, 1 - c)$ be a real number, and set $\xi = \beta (1 - \beta)^{-1} e^{\pi \textbf{\emph{i}} (1 - \rho)}$. Then for any integers $t, s \in [0, B]$ and $x, y \in [-B, B]$, we have that $\big| K_{\beta; \textbf{\emph{a}}} (x, t + T; y, s + T) - \mathcal{K}_{\xi} (x, t; y, s) \big| < \vartheta_N$.
		
	\end{lem}

	\begin{rem} 
		
		The uniformity of the error $|K_{\beta; \textbf{a}} - \mathcal{K}_{\xi}| < \vartheta_N$ in $(c, B, \kappa)$ was not explicitly stated in Theorem 2.12 of \cite{ULSRW} but can be quickly derived from Theorem 2.7 of \cite{ULSRW} and the proofs of Theorem 2.12 and Lemma 6.9 there. 
		
	\end{rem}

	One might view \Cref{limitdeterminantnonintersecting} as a discrete (and slightly more precise) analog of the $\beta = 2$ special case of Theorem 2.4 of \cite{CLSM} and Theorem 2.2 of \cite{FEUM} on local statistics of Dyson Brownian motion.

	\subsection{The Local Law} 
	
	\label{RegularLocal}

	In this section we state a local law for uniformly random height functions (equivalently, free tilings), given by \Cref{heightlocal1}, which will be established in \Cref{EstimateHLocalProof} below. In what follows, we will assume that the domain of the tiling approximates a disk (in the sense of \Cref{regularestimate} below), but make no restrictions on the boundary height function. This will suffice for the purposes of establishing \Cref{localconverge}, through a suitable restriction of the tiling there.
	
	In what follows, we recall the notation on disks from \eqref{brzdefinition}. 
	
	\begin{assumption}
	
	\label{regularestimate} 
	
	Suppose a real number $\varepsilon \in \big( 0, \frac{1}{4} \big)$ is given. Let $N \in \mathbb{Z}_{\ge 1}$ denote an integer, and define the simply-connected induced subgraph $R = \mathcal{B}_N \cap \mathbb{T} \subset \mathbb{T}$. Further let $h: \partial R \rightarrow \mathbb{Z}$ denote a boundary height function on $\partial R$. Let $H: \mathbb{V}(R) \rightarrow \mathbb{Z}$ denote a uniformly random element of $\mathfrak{G} (h)$ (recall \Cref{gh}); also let $\mathscr{M} = \mathscr{M} (H)$ denote the associated free tiling of $R$. Extend $H$ to $\overline{\mathcal{B}}_N$ by linearity on $\mathbb{F}(R)$ and arbitrarily on $\overline{\mathcal{B}}_N \setminus \mathbb{F}(R)$, in such a way that it is $1$-Lipschitz on $\overline{\mathcal{B}}_N$.

	Set $\mathfrak{R} = \mathcal{B} = \mathcal{B}_1$, and define $\mathfrak{h}: \partial \mathfrak{R} \rightarrow \mathbb{R}$ by setting $\mathfrak{h} (z) = N^{-1} H (Nz)$ for each $z \in \partial \mathfrak{R}$. Let $\mathcal{H} \in \Adm (\mathfrak{R}; \mathfrak{h})$ denote the maximizer of $\mathcal{E}$ on $\mathfrak{R}$ with boundary data $\mathfrak{h}$; also let $v_0 \in \mathbb{V}(R)$ denote some vertex such that $\mathcal{B}_{\varepsilon N} (v_0) \cap \mathbb{T} \subset \mathbb{F}(R)$ and $\nabla \mathcal{H} (z) \in \mathcal{T}_{\varepsilon}$ (recall \eqref{tset2}) for each $z \in \mathcal{B}_{\varepsilon} (N^{-1} v_0)$. 
	\end{assumption} 
	
	Now let us proceed to describe the local law for lozenge tilings. To that end, under the notation of \Cref{regularestimate}, first observe that \Cref{hnh} and a Taylor expansion together suggest that the random height function $H$ is with high probability approximately linear in an $\omega N$-neighborhood of $v_0$, with slope $\nabla \mathcal{H} (N^{-1} v_0)$, if $\omega > 0$ is small but independent of $N$. A \emph{local law} for tilings would indicate that $H$ is with high probability approximately linear in an $M$-neighborhood of $v_0$, with the same slope $\nabla \mathcal{H} (N^{-1} v_0)$, for (essentially) any scale $1 \ll M \ll N$. The following theorem states that this holds; its proof will be provided in \Cref{LocalProof} below.

	\begin{thm}
		
		\label{heightlocal1}

		For any fixed real numbers $\varepsilon \in \big( 0, \frac{1}{4} \big)$ and $D > 0$, there exist constants $C_1 = C_1 (\varepsilon) > 1$ and $C_2 = C_2 (\varepsilon, D) > 1$ such that the following holds. Adopt the notation of \Cref{regularestimate}, set $c = \frac{1}{20000}$, and let $M \in \mathbb{Z}$ satisfy 
		\begin{flalign*} 
		C_1 \le M \le N \exp \big( -5 (\log \log N)^2 \big).
		\end{flalign*} 
		
		\noindent Then, there exists a random (dependent on $H$) pair $(s, t) = \big( s (v_0; M), t (v_0; M) \big) \in \mathcal{T}_{\varepsilon / 2}$ and an event $\Gamma = \Gamma (v_0; M)$	 such that
		\begin{flalign}
		\label{gammahnv0estimate} 
		\mathbb{P} [\Gamma] \le C_2 M^{-D}; \qquad \big| (s, t) - \nabla \mathcal{H} (N^{-1} v_0) \big| \textbf{\emph{1}}_{\Gamma^c} < (\log M)^{-c},
		\end{flalign}
		
		\noindent and 
		\begin{flalign}
		\label{mhuhv0estimate} 
		\displaystyle\sup_{u \in \mathcal{B}_M (v_0)} \Big|  M^{-1} \big( H (u) - H (v_0) \big) - M^{-1} (u - v_0) \cdot (s, t)   \Big| \textbf{\emph{1}}_{\Gamma^c} < (\log M)^{-1 - c}. 
		\end{flalign}
		
	\end{thm}
	
	If $M \gg 1$, then the first bound in \eqref{gammahnv0estimate} shows that $\Gamma^c$ is an event of high probability.\footnote{However, this is no longer true if $M$ is of constant order, which is consistent with the non-degeneracy of the local statistics around $v_0$ in the large $N$ limit.} Moreover, \eqref{mhuhv0estimate} and the second estimate in \eqref{gammahnv0estimate} together yield 
	\begin{flalign}
	\label{mhuhv0estimate2} 
	\displaystyle\sup_{u \in \mathcal{B}_M (v_0)} \Big|  M^{-1} \big( H (u) - H (v_0) \big) - M^{-1} (u - v_0) \cdot \nabla \mathcal{H} (N^{-1} v_0)   \Big| \textbf{1}_{\Gamma^c} < 2 (\log M)^{- c},
	\end{flalign}
	
	\noindent which indeed indicates that $H$ is approximately linear on $\mathcal{B}_M (v_0)$, with slope $\nabla \mathcal{H} (N^{-1} v_0)$. However, the error in this linear approximation is of order $(\log M)^{-c}$, while the bound \eqref{mhuhv0estimate} shows that there exists a linear approximation (of possibly slightly different slope) with the improved error of $(\log M)^{-1- c}$. This improved bound will be useful for estimating the quantity $\mathfrak{D}$ in \Cref{DEstimate}.

	\subsection{Coupling Tilings With Non-Intersecting Random Path Ensembles}
	
	\label{EnsemblesCompare}

	In this section we state a result, given by \Cref{plpprcouple} below, that locally couples the restriction (denoted by $\textbf{Q}$) of a uniformly random free lozenge tiling $\mathscr{M}$ for a domain $R$ to a neighborhood of some vertex $u_0 \in \mathbb{V}(R)$, with two random non-intersecting path ensembles (denoted by $\textbf{P}$ and $\textbf{R}$) under suitably chosen initial data. To that end, we must first explain how to select these initial data, which will be perturbations of a sequence $\textbf{q} = \textbf{q} (0) \in \mathbb{W}$ given by the following definition. In the below, we will assume for notational convenience that the vertex $u_0$ around which we will eventually couple is equal to $(0, 0)$; this restriction can be removed by translating $R$. 
	
	\begin{definition} 
		
		\label{qunlambda}
		
		Suppose an integer $\ell \in \mathbb{Z}_{> 1}$; a finite domain $R \subseteq \mathbb{T}$ with $[- 8 \ell^4, 8 \ell^4] \times [0, \ell] \subset \mathbb{F}(R)$; and a free tiling $\mathscr{M}$ of $R$ are given. Recalling the set $\mathscr{X} (\mathscr{M})$ from \Cref{xm}, define the integer sequence $\textbf{q} = \textbf{q} (0) = \textbf{q}^{(\mathscr{M}; \ell)} (0)$ by 
		\begin{flalign*}
		\textbf{q} = \big\{ x \in [ - 2 \ell^4, 2 \ell^4] \cap \mathbb{Z}: (x, 0) \notin \mathscr{X} (\mathscr{M}) \big\}. 
		\end{flalign*}
		
		\noindent Stated alternatively, $\textbf{q} = \textbf{q} (0)$ denotes the set of positions $(x, 0) \in [-2 \ell^4, 2 \ell^4] \times \{ 0 \}$ along the $x$-axis such that $\big( x + \frac{1}{2}, 0 \big)$ is on an edge of a type $2$ or type $3$ lozenge in $\mathscr{M}$. After interpreting $\mathscr{M}$ as a non-intersecting path ensemble (through the discussion in \Cref{NonIntersectingCorrelation}), these correspond to the locations where paths of $\mathscr{M}$ intersect the interval $[-2 \ell^4, 2 \ell^4] \times \{ 0 \}$ of the $x$-axis; see \Cref{pathslocal}.
		
		Label the elements of $\textbf{q} (0) = \big( q_{-m} (0), q_{1 - m} (0), \ldots , q_n (0) \big) \in \mathbb{W}$ so that $q_0 (0) \le 0 < q_1 (0)$, and abbreviate $q_j = q_j (0)$ for each $j \in [-m, n]$. Viewing $\mathscr{M}$ as a path ensemble as above, define the non-intersecting path ensemble $\textbf{Q} = \textbf{Q}^{(\mathscr{M}; \ell)} = \big( \textbf{q}_{-m}, \textbf{q}_{1 - m}, \ldots , \textbf{q}_n \big) \in \mathfrak{W}_{\textbf{q}}^{\ell}$ to be the one obtained by restricting the $m + n + 1$ paths in $\mathscr{M}$ passing through $\textbf{q} (0)$ to the strip $\Delta^{\ell} = \mathbb{Z} \times \{ 0, 1, \ldots , \ell \}$. In particular, the initial data for $\textbf{Q}$ is $\textbf{q} = \textbf{q} (0)$. Set $\textbf{q}_j = \big( q_j (0), q_j (1), \ldots , q_j (\ell) \big)$, for each $j \in [-m, n]$, and $\textbf{q} (s) = \big( q_{-m} (s), q_{1 - m} (s) \ldots , q_n (s) \big)$, for each $s \in [0, \ell]$. 
		
	\end{definition}

	\begin{rem} 
		
		\label{hqh}
		
		Let $H_{\textbf{Q}}$ denote the height function associated with $\textbf{Q}$ (as explained in and above \Cref{heightpaths}), such that $H_{\textbf{Q}} (0, 0) = 0$. Then, for any height function $H: \mathbb{V}(R) \rightarrow \mathbb{Z}$ associated with $\mathscr{M}$, we have that $H_{\textbf{Q}} (u) = H(u) - H (0, 0)$ for each vertex $u \in [-2 \ell^4, 2 \ell^4] \times [0, \ell]$.
		
	\end{rem} 

	To continue, it will be useful to define a subset of the Weyl chamber consisting of sequences without long consecutive subsequences or large gaps. 
	
	\begin{definition} 
	 
	\label{zl} 
	
	For any $k \in \mathbb{Z}_{> 1}$, let $\mathcal{Z} (k) \subseteq \mathbb{W}$ denote the set of $\textbf{a} = (a_{-m}, a_{1 - m}, \ldots , a_n) \in \mathbb{W}$ such that $a_{j + k - 1} \ge a_j + k$ for each $j \in [-m, n - k + 1]$ and such that $a_{j + 1} - a_j < k$ for each $j \in [-m, n - 1]$. Equivalently, $\mathcal{Z} (k)$ is the set of $\textbf{a} \in \mathbb{W}$ that do not contain any subsequence of $k$ consecutive integers or any gap of size at least $k$. 
	
	\end{definition}

	Next, we introduce two perturbations $\textbf{p}$ and $\textbf{r}$ of the sequence $\textbf{q}$ from \Cref{qunlambda} that will serve as initial data for the random non-intersecting path ensembles to be coupled with $\mathscr{M}$. Assuming that $\textbf{q} \in \mathcal{Z} (\ell)$, these sequences $\textbf{p}$ and $\textbf{r}$ will coincide with $\textbf{q}$ on the subinterval $[-2 \ell, 2 \ell] \subset [-2 \ell^4, 2 \ell^4]$ and will be determined outside of this subinterval by (mildly) shifting $\textbf{q}$ to the left and right, respectively.

	\begin{definition} 
		
		\label{lrqdefinition}
		
		Adopt the notation and assumptions of \Cref{qunlambda}. If $\textbf{q} \in \mathcal{Z} (\ell)$, then define the sequence $\textbf{p} = \textbf{p} (0) = \textbf{p}^{(\textbf{q}; \ell)} (0) = \big( p_{-m} (0), p_{1 - m} (0), \ldots , p_n (0) \big) \in \mathbb{W}$ as follows. In the below, we abbreviate $p_j = p_j (0)$ for each $j \in [-m, n]$. 
		
		\begin{enumerate} 
			
			\item For each $j \in [-m, n]$ such that $-2 \ell \le q_j \le 2 \ell$, set $p_j = q_j$. 
			
			\item For each $i \in [1, \ell^3)$ and $j \in [-m, n]$ such that $-2 (i + 1) \ell \le q_j < - 2 i \ell$, set $p_j = q_j - i$. 
			
			\item For each $i \in [1, \ell^3)$, let $j (i) \in [1, n]$ denote the smallest integer such that $2 i \ell < q_{j (i)} \le 2 (i + 1) \ell$ and $q_{j (i)} - q_{j (i) - 1} > 1$; such an index $j(i)$ is guaranteed to exist, since $\textbf{q} \in \mathcal{Z} (\ell)$. Also denote $j (\ell^3) = \infty$, and set $p_j = q_j - i$ whenever $j(i) \le j < j (i + 1)$ (and $p_j = q_j$ when $2 \ell < q_j < q_{j(1)}$).

		\end{enumerate} 
		
		\noindent If $\textbf{q} \in \mathcal{Z} (\ell)$, then similarly define $\textbf{r} = \textbf{r} (0) = \textbf{r}^{(\textbf{q}; \ell)} (0) = \big(r_{-m} (0), r_{1 - m} (0), \ldots , r_n (0) \big) \in \mathbb{W}$ as follows. In the below we again abbreviate $r_j = r_j (0)$ for each $j \in [-m, n]$. 
		
		\begin{enumerate} 
			
			\item For each $j \in [-m, n]$ such that $-2 \ell \le q_j \le 2 \ell$, set $r_j = q_j$. 
			
			\item For each $i \in [1, \ell^3)$, let $j (i) \in [-m, 0]$ denote the largest integer such that $- 2 (i + 1) \ell \le q_{j(i)} < - 2 i \ell$ and $q_{j(i) + 1} - q_{j(i)} > 1$; also let $j (0) = -\infty$. Then, set $r_j = q_j + i$ whenever $j(i - 1) < j \le j (i)$ (and $r_j = q_j$ when $q_{j(1)} < q_j < -2 \ell$).
			
			\item For each $i \in [1, \ell^3)$ and $j \in [-m, n]$ such that $2 i \ell < q_j \le 2 (i + 1) \ell$, set $r_j = q_j + i$. 
		\end{enumerate}
	
		\noindent If instead $\textbf{q} \notin \mathcal{Z} (\ell)$, then define $\textbf{p}, \textbf{r} \in \mathbb{W}$ arbitrarily. 
		
		\end{definition} 
			
		Before proceeding, let us briefly take a moment to explain our choices in \Cref{lrqdefinition}. We will eventually wish to couple $\textbf{Q}$ with (suitably chosen) random non-intersecting path ensembles with initial data $\textbf{p}$ and $\textbf{r}$ around $(0, 0)$; see \Cref{plpprcouple} below. This is our reason for imposing that $\textbf{p}$, $\textbf{q}$, and $\textbf{r}$ coincide on $[-2 \ell, 2 \ell]$. 
		
		It will also be useful to further impose that $p_j + \ell \le q_j - \ell < q_j + \ell < r_j - \ell$ when $|j|$ is sufficiently large, as this will guarantee that the $j$-th path in $\textbf{Q}$ is between the $j$-th path in any ensembles in $\mathfrak{W}_{\textbf{p}}^{\ell}$ and $\mathfrak{W}_{\textbf{r}}^{\ell}$; the latter will in turn be useful to produce couplings using the monotonicity result \Cref{monotoneheightcouple}. In order to ensure this while altering the local density of $\textbf{q}$ as little as possible, we stipulated that $p_j = q_j - i$ whenever $- 2 (i + 1) \ell \le q_j < - 2i \ell$ (and similarly for $r_j$). 
			
		However, if we had replaced the third part in the definition of $\textbf{p}$ with the analogous $p_j = q_j - i$ whenever $2 (i + 1) \ell \le q_j < 2 i \ell$, then we might have $p_k = p_{k + 1}$ for some $k \in [-m, n]$ (for example, if $q_k = 2 i \ell$ and $q_{k + 1} = 2 i \ell + 1$). In this case, $\textbf{p} \notin \mathbb{W}$ and would therefore not serve as suitable initial data for a non-intersecting path ensemble. Thus, we should only allow $q_{k + 1} - p_{k + 1} > q_k - p_k$ if $q_{k + 1} > q_k + 1$. This is the reason for introducing the indices $j(i)$ in \Cref{lrqdefinition}, whose existence is guaranteed by imposing that $\textbf{q} \in \mathcal{Z} (\ell)$ (which occurs with high probability, as we will show in \Cref{ait1t2b2}).

	Given \Cref{lrqdefinition}, we next require notation on random non-intersecting path ensembles with initial data $\textbf{p}$ and $\textbf{r}$. 
	
	\begin{definition}
		
	\label{pensemblerdefinition} 
	
	Suppose real numbers $\beta_1, \beta_2 \in (0, 1)$ are given, and adopt the notation and assumptions of \Cref{lrqdefinition}. Let $\textbf{P} = (\textbf{p}_{-m}, \textbf{p}_{1 - m}, \ldots , \textbf{p}_n) \in \mathfrak{W}_{\textbf{p}}^{\ell}$ denote a random non-intersecting path ensemble sampled from the measure $\mathbb{P}_{\beta_1; \textbf{p}}$ from \Cref{betaprobability}. Similarly, let $\textbf{R} = (\textbf{r}_{-m}, \textbf{r}_{1 - m}, \ldots , \textbf{r}_n) \in \mathfrak{W}_{\textbf{r}}^{\ell}$ denote a random ensemble sampled from $\mathbb{P}_{\beta_2; \textbf{r}}$.
	
	For each $j \in [-m, n]$, set $\textbf{p}_j = \big( p_j (0), p_j (1), \ldots , p_j (\ell) \big) \in \mathbb{Z}^{\ell + 1}$ and, for each $s \in [0, \ell]$, set $\textbf{p} (s) = \big( p_{-m} (s), p_{1 - m} (s), \ldots , p_n (s) \big) \in \mathbb{W}$. Define $\textbf{r}_j$ and $\textbf{r} (s)$ similarly. 
	
	\end{definition}

	Now we state the following proposition, which will be established in \Cref{ProofPCouple} below. It exhibits nearly equal $\beta_1, \beta_2 \in (0, 1)$ such that, under the notation of \Cref{regularestimate}, \Cref{qunlambda}, and \Cref{pensemblerdefinition},  it is possible to couple $\textbf{P}$, $\textbf{Q}$, and $\textbf{R}$ so that they coincide on a large region around $(0, 0)$, off of a small probability event. In the below, the ``near equality'' of $\beta_1$ and $\beta_2$ will be quantified by $\delta_N$; the ``small probability'' by $\varsigma_N$; and the ``large region'' by $A_N$. We also recall the notation on disks from \eqref{brzdefinition} and say that $\textbf{P} |_S = \textbf{Q} |_S = \textbf{R} |_S$ for any subset $S \subset \mathbb{Z}^2$ if $p_j = q_j = r_j$ whenever either $(p_j, j) \in S$, $(q_j, j) \in S$, or $(r_j, j) \in S$.

	\begin{prop}
		
		\label{plpprcouple}
		
		For any fixed real number $\varepsilon \in \big( 0, \frac{1}{4} \big)$, there exist two sequences $\delta = \delta (\varepsilon) = (\delta_1, \delta_2, \ldots ) \subset \mathbb{R}_{> 0}$ and $\varsigma = \varsigma (\varepsilon) = (\varsigma_1, \varsigma_2, \ldots ) \subset \mathbb{R}_{> 0}$ of real numbers tending to $0$, and a sequence $A = A(\varepsilon) = (A_1, A_2, \ldots ) \subset \mathbb{Z}_{\ge 1}$ of integers tending to $\infty$, such that the following holds. Adopt the notation of \Cref{regularestimate}, suppose that $v_0 \in \mathcal{B}_{\varepsilon N / 2}$, let $\ell = \lfloor N^{1 / 9} \rfloor$, and set 
		\begin{flalign}
		\label{stbeta12}
		(s, t) = \nabla \mathcal{H} (0, 0); \qquad \beta = \displaystyle\frac{\sin (\pi t)}{\sin (\pi t) + \sin \big( \pi (1 - s - t) \big)}; \qquad \beta_1 = \beta - \delta_N; \qquad \beta_2 = \beta + \delta_N. 
		\end{flalign} 
		
		\noindent Further adopt the notation of \Cref{qunlambda} and \Cref{pensemblerdefinition}, and set $\mathscr{A} = [-A_N, A_N] \times [0, A_N] \subset \mathbb{Z}^2$. Then, there exists a mutual coupling of $\textbf{\emph{P}}$, $\textbf{\emph{Q}}$, and $\textbf{\emph{R}}$ on a common probability space such that $\mathbb{P} \big[ \textbf{\emph{P}} |_{\mathscr{A}} = \textbf{\emph{Q}} |_{\mathscr{A}} = \textbf{\emph{R}} |_{\mathscr{A}} \big] \ge 1 - \varsigma_N$.
		
	\end{prop}

	We refer to \Cref{paths2} for a depiction.

	\subsection{Bounding \texorpdfstring{$\mathfrak{D}$}{}}
	
	\label{DEstimate} 
	
	To establish \Cref{localconverge}, we wish to show that the local statistics of $\textbf{Q}$ (from \Cref{qunlambda}) are governed by the discrete sine kernel, as $N$ tends to $\infty$. Due to the coupling provided by \Cref{plpprcouple}, it suffices to instead establish that those of $\textbf{P}$ (or $\textbf{R}$) are. To do this, we will apply \Cref{limitdeterminantnonintersecting} to the pairs $(\textbf{a}, \beta) \in \big\{ (\textbf{p}, \beta_1), (\textbf{r}, \beta_2) \big\}$, to which end we must confirm that $\textbf{p}$ and $\textbf{r}$ satisfy the two assumptions listed in that lemma. In this section we implement this task, by establishing \Cref{ait1t2b2} below (which will also be used to prove \Cref{plpprcouple} in \Cref{ProofHHlHr}). 
	
	To verify the first assumption stated in \Cref{limitdeterminantnonintersecting}, we begin with the following lemma that provides a general condition on an integer sequence $\textbf{w} \in \mathbb{W}$ under which $\mathfrak{D} (\textbf{w}; X, Y)$ is small; this will later be applied when $\textbf{w}$ is (possibly a shift of) $\textbf{p}$ or $\textbf{r}$. In the below, we recall from \Cref{NonIntersectingCorrelation} that associated with any non-intersecting path ensemble is a tiling and therefore a height function, which is unique except for a global shift. Thus, we may associate with any integer sequence $\textbf{w} = (w_{-m}, w_{1 - m}, \ldots , w_n) \in \mathbb{W}$ a height function $h: \mathbb{Z} \rightarrow \mathbb{Z}$, which is determined by imposing that $h(0) = 0$ and
	\begin{flalign}
	\label{hxhyw} 
	h(y) - h(x) = y - x - \big| \textbf{w} \cap [x, y) \big|, \qquad \text{for any $x, y \in \mathbb{Z}$ such that $x \le y$}.
	\end{flalign}
	
	\noindent By \Cref{heightpaths}, this is the restriction to the $x$-axis of a height function for the tiling associated with any path ensemble with initial data $\textbf{w}$. 
	
	\begin{lem} 
		
		\label{at1t2b} 
		
		Fix real numbers $\varepsilon, c \in \big( 0, \frac{1}{2} \big)$, $C > 1$, and $2 \le X < Y$, as well as an integer sequence $\textbf{\emph{w}} = (w_{-m}, w_{1 - m}, \ldots , w_n) \in \mathbb{W}$. Let $h_{\textbf{\emph{w}}}: \mathbb{Z} \rightarrow \mathbb{Z}$ denote the height function associated with $\textbf{\emph{w}}$ satisfying \eqref{hxhyw} and $h(0) = 0$. Suppose that, for each $k \in \mathbb{Z}_{\ge 1}$ with $X \le 2^k \le Y$, there exists some $\rho_k \in (\varepsilon, 1 - \varepsilon)$ such that 
		\begin{flalign}
		\label{hwyhwx} 
		\big| h_{\textbf{\emph{w}}} (y) - h_{\textbf{\emph{w}}} (x) - \rho_k (y - x) \big| < C k^{-1 - c} 2^k,
		\end{flalign} 
		
		\noindent for any $x, y \in [-2^{k + 1}, 2^{k + 1}] \cap [-Y, Y] \cap \mathbb{Z}$. Then, $\mathfrak{D} (\textbf{\emph{w}}; X, Y) < 16 C c^{-1} \varepsilon^{-1} (\log X)^{-c}$.

	\end{lem} 
	
	\begin{proof} 
		
		Fix $k \in \mathbb{Z}_{\ge 1}$ with $X \le 2^k \le Y$, and set $U, V \in \mathbb{Z}$ by $U = \max \{ 2^{k-1}, X \}$ and $V = \min \{ 2^{k + 1}, Y \}$; let us estimate
		\begin{flalign*}
		\mathfrak{D} (\textbf{w}; U, V) = \left| \displaystyle\sum_{-V \le w_j \le -U} \displaystyle\frac{1}{w_j} + \displaystyle\sum_{U \le w_j \le V} \displaystyle\frac{1}{w_j} \right|.
		\end{flalign*}
		
		To that end, set $\textbf{w} \cap [-V, -U] = (-x_a, -x_{a - 1}, \ldots , -x_1) \in \mathbb{W}_a$ and $\textbf{w} \cap [U, V] = (y_1, y_2, \ldots , y_b) \in \mathbb{W}_b$, for some $a, b \in \mathbb{Z}_{\ge 0}$. For notational convenience, also set $x_i = \infty$ if $i > a$ and $y_i = \infty$ if $i > b$. Then observe from \eqref{hxhyw} that $x_j = U - h_{\textbf{w}} (-x_j) + h_{\textbf{w}} (-U) + j$ for any $j \in [1, a]$, and $y_j = U + h_{\textbf{w}} (y_j) - h_{\textbf{w}} (U) + j$ for any $j \in [1, b]$. 
		
		Denote $\mathscr{K} = \lceil C \varepsilon^{-1} k^{-1 -c} 2^{k + 1} \rceil$. Then since $\rho_k \in (\varepsilon, 1 - \varepsilon)$ and $\big| h_{\textbf{w}} (y) - h_{\textbf{w}} (x) - \rho_k (y - x) \big| < \frac{\varepsilon \mathscr{K}}{2}$ for any integers $x, y \in [-V, V]$ (by \eqref{hwyhwx}), the above two identities for $x_j$ and $y_j$ imply that $x_{i - \mathscr{K}} \le y_i \le x_{i + \mathscr{K}}$ for each $i \in \mathbb{Z}_{> \mathscr{K}}$. Thus, 
		\begin{flalign}
		\label{vqju1}
		\displaystyle\sum_{-V \le w_j \le -U} \displaystyle\frac{1}{w_j} + \displaystyle\sum_{U \le w_j \le V} \displaystyle\frac{1}{w_j} = \displaystyle\sum_{i = \mathscr{K} + 1}^{\infty} \left( \displaystyle\frac{1}{y_i} - \displaystyle\frac{1}{x_{i - \mathscr{K}}}  \right) + \displaystyle\sum_{i = 1}^{\mathscr{K}} \displaystyle\frac{1}{y_i} \le \mathscr{K} 2^{1-k} \le \displaystyle\frac{8C}{\varepsilon k^{1 + c}},  
		\end{flalign}
		
		\noindent where to deduce the second statement in \eqref{vqju1} we used the bound $y_i \ge 2^k$. By similar reasoning,
		\begin{flalign}
		\label{vqju2}
		\displaystyle\sum_{-V \le w_j \le -U} \displaystyle\frac{1}{w_j} + \displaystyle\sum_{U \le w_j \le V} \displaystyle\frac{1}{w_j} \ge - \displaystyle\frac{8C}{\varepsilon k^{1 + c}}.
		\end{flalign}
		
		\noindent Now let $k_0 \in \mathbb{Z}_{\ge 1}$ denote the minimal integer such that $2^{k_0} > \frac{X}{2}$. Then, summing \eqref{vqju1} and \eqref{vqju2} over $k$ yields 	
		\begin{flalign*}
		\mathfrak{D} (\textbf{w}; X, Y) = \left| \displaystyle\sum_{X \le |w_j| \le Y} \displaystyle\frac{1}{w_j} \right| \le 4 \varepsilon^{-1} C \displaystyle\sum_{k = k_0}^{\infty} k^{- 1 - c} \le \displaystyle\frac{16 C}{\varepsilon c k_0^c} \le \displaystyle\frac{32 C}{c \varepsilon (\log X)^c}, 
		\end{flalign*}
		
		\noindent from which we deduce the lemma. 
	\end{proof}

	Using \Cref{at1t2b} and the local law \Cref{heightlocal1}, we can now establish the following corollary that verifies with high probability that $\textbf{w} \in \mathcal{Z} (\ell)$ (recall \Cref{zl}) and that it satisfies the two assumptions listed in \Cref{limitdeterminantnonintersecting}, if $\textbf{w}$ is a shift of one of the sequences $\textbf{p}$ or $\textbf{r}$ from \Cref{plpprcouple} (it will be useful to allow such shifts for the proof of \Cref{x1x2x} in \Cref{ProofCoupling}). The third condition stated in this corollary does not appear as an assumption in \Cref{limitdeterminantnonintersecting} but will be useful for \Cref{continuityk} below. In what follows we recall for any $k \in \mathbb{Z}$ and $\textbf{a} = (a_{-m}, a_{1 - m}, \ldots , a_n) \in \mathbb{Z}^{m + n + 1}$ that $\textbf{a} + k = (a_{-m} + k, a_{1 - m} + k, \ldots , a_n + k)$.

		\begin{cor}
			
			\label{ait1t2b2}

			Fix real numbers $\varepsilon \in \big( 0, \frac{1}{4} \big)$ and $D > 0$; adopt the notation of \Cref{regularestimate}; suppose that $v_0 \in \mathcal{B}_{\varepsilon N / 2}$; and set $\ell = \lfloor N^{1 / 9} \rfloor$. Further adopt the notation of \Cref{qunlambda} and \Cref{lrqdefinition}; set $(s, t) = \nabla \mathcal{H} (0, 0) \in \mathcal{T}_{\varepsilon}$; and fix a real number $A \in [2, \ell]$. 
			
			Then, there exist a constant $C = C(\varepsilon, D) > 1$ and an event $\Gamma = \Gamma_A$ with $\mathbb{P} [\Gamma] \le C A^{-D}$ such that the following holds. Set $c = \frac{1}{20000}$; let $k \in [-2 A^3, 2 A^3]\cap \mathbb{Z}$; and let $\textbf{\emph{w}} = (w_{-m}, w_{1 - m}, \ldots , w_n) \in \mathbb{W}_{m + n + 1}$ denote either the sequence $\textbf{\emph{p}} + k$ or $\textbf{\emph{r}} + k$. Then, on the event $\Gamma^c$, we have $\textbf{\emph{q}} \in \mathcal{Z} (\ell)$ and the following three bounds. 
			
			\begin{enumerate}  
				
				\item For any real number $X \ge A^{1 / 2}$, we have that $\mathfrak{D} (\textbf{\emph{w}}; X, \infty) < 384 c^{-1} \varepsilon^{-1} (\log X)^{-c}$.
				
				\item For each interval $I \subset [-2 A^3, 2 A^3]$ of length $\lfloor A^{1 / 2} \rfloor$, we have that 
				\begin{flalign*}
				\big( 1 - s - 5 (\log A)^{-c} \big) A^{1 / 2} \le \big| \textbf{\emph{w}} \cap I \big| \le (1 - s + 5 (\log A)^{-c} \big) A^{1 / 2}.
				\end{flalign*} 
				
				\item Assume $k = 0$. Then for each integer $j \in [A^{1 / 2}, n]$, we have that $\big( 1 + \frac{\varepsilon}{6} \big) j \le w_j \le \frac{6j}{\varepsilon}$ and, for each integer $j \in [-m, -A^{1 / 2}]$, we have that $\frac{6j}{\varepsilon} \le w_j \le \big( 1 + \frac{\varepsilon}{6} \big) j$. 
				
			\end{enumerate} 
			
		\end{cor}

	\begin{proof}
		
		Let us begin by defining the event $\Gamma$. To that end, fix a vertex $v \in \mathcal{B}_{\varepsilon N / 2} \cap \mathbb{T} \subset \mathbb{F}(R) \cap \mathcal{B}_{\varepsilon N / 2} (v_0)$ and an integer $M \in \big[ A^{1 / 2}, 6 \ell^4]$. Then \Cref{heightlocal1} (applied with the $v_0$ and $\varepsilon$ there equal to $v$ and $\frac{\varepsilon}{2}$ here, respectively) yields a constant $C = C (\varepsilon, D) > 1$; an event $\Gamma (v; M)$; and a pair $\big( s(v; M), t(v; M) \big) \in \mathcal{T}_{\varepsilon / 4}$ such that we have the three estimates
		\begin{flalign}
		\label{gammau0mstst}
		\begin{aligned}
		& \quad \mathbb{P} \big[ \Gamma (v; M) \big] \le C M^{-15 D}; \qquad \Big| \big( s (v; M), t (v; M) \big) - \nabla \mathcal{H} (N^{-1} v) \Big| \textbf{1}_{\Gamma (v; M)^c} < (\log M)^{-c}; \\
		& \displaystyle\sup_{u \in \mathcal{B}_M (v)} \Big| M^{-1} \big( H(u) - H(v) \big) - M^{-1} (u - v) \cdot \big( s(v; M), t (v; M) \big) \Big| \textbf{1}_{\Gamma (v; M)^c} < (\log M)^{-1 - c}.
		\end{aligned}
		\end{flalign}
		
		\noindent Then, set 
		\begin{flalign*} 
		\Gamma = \Gamma_A = \bigcup_{M = \lfloor A^{1 / 2} \rfloor}^{6 \ell^4} \bigcup_{v \in \mathcal{B}_{2 A^3} \cap \mathbb{T}} \Gamma (v; M), 
		\end{flalign*} 
	
		\noindent so that the first estimate in \eqref{gammau0mstst} and a union bound (using the fact that $|\mathcal{B}_{2A^3} \cap \mathbb{T} | \le 10 (2A^3)^2 = 40A^6$) together imply that 
		\begin{flalign*} 
		 \mathbb{P} [\Gamma] \le 40 A^6 \cdot C \displaystyle\sum_{M = A^{1/2}}^{\infty} M^{-15D} \le 40C A^6 \cdot A^{(1-15D)/2} \le 40 C A^{-D}.
		\end{flalign*}

		Next let us verify that $\textbf{q} \in \mathcal{Z} (\ell)$ holds on $\Gamma^c$. To that end, the third bound in \eqref{gammau0mstst} yields 
		\begin{flalign}
		\label{hvhu0}
		\displaystyle\sup_{u \in \mathcal{B}_M (v)} \Big| M^{-1} \big( H(u) - H(v) \big) - M^{-1} (u - v) \cdot \big( s(v; M), t (v; M) \big) \Big| \textbf{1}_{\Gamma^c} < (\log M)^{-1 - c},
		\end{flalign}
		
		\noindent for each $M \in \big[ A^{1 / 2}, 6 \ell^4 \big]$ and $v \in \mathcal{B}_{2 A^3} \cap \mathbb{T}$. In particular, since $\big( s(v; M), t (v; M) \big) \in \mathcal{T}_{\varepsilon / 4}$, we have that $s (v; M) \in \big( \frac{\varepsilon}{4}, 1 - \frac{\varepsilon}{4} \big)$. So, \eqref{hvhu0} and the fact that $A \le \ell$ together imply that, on $\Gamma^c$, 
		\begin{flalign*} 
		\displaystyle\frac{\varepsilon}{8} < \frac{\varepsilon}{4} - (\log \ell)^{-1 - c} < \ell^{-1} \big( H ( j + \ell, 0) - H(j, 0) \big) < 1 - \frac{\varepsilon}{4} + (\log \ell)^{-1 - c} < 1 - \displaystyle\frac{\varepsilon}{8},
		\end{flalign*}
		
		\noindent for each $j \in [-2 \ell^4, 2\ell^4]$, if $\ell$ is sufficiently large. By \eqref{hxhyw}, this implies that $q_{j + \ell - 1} - q_j \ge \ell$ and $q_{j + 1} - q_j < \ell$ for each $j$, and so $\textbf{q} \in \mathcal{Z} (\ell)$ on $\Gamma^c$ for sufficiently large $\ell$. 
		
		Thus it remains to verify that, on $\Gamma^c$, the sequence $\textbf{w}$ satisfies the three bounds listed in the corollary. Let us fix some $k \in [- 2 A^3, 2 A^3]$ and suppose that $\textbf{w} = \textbf{p} + k$, as the alternative case $\textbf{w} = \textbf{r} + k$ is entirely analogous. We begin by establishing the first bound in the corollary, to which end it suffices to show that $\mathfrak{D} (\textbf{w}; X, Y) < 384 c^{-1} \varepsilon^{-1} (\log X)^{-c}$ whenever $A^{1 / 2} \le X \le Y \le 2 \ell^4 + 3 \ell^3$, since $|w_{-m}|, |w_n| \le 2 \ell^4 + 3 \ell^3$. To do this, we apply \Cref{at1t2b} and must therefore verify the approximate linearity bound \eqref{hwyhwx}. 
		
		Now, fix some $x \in [-2 A^3 - k, 2 A^3 - k]$ and let $h_{\textbf{p}}$ denote the height function associated with $\textbf{p}$ as defined in \eqref{hxhyw}. Since $\textbf{q} \in \mathcal{Z} (\ell)$, it follows from \Cref{lrqdefinition} that for any interval $I \subset [- 2\ell^4 - 3 \ell^3, 2\ell^4 + 3 \ell^3]$ there exists a translation $\widetilde{I}$ of $I$ by at most $\ell^3$ such that $\big| |\textbf{q} \cap \widetilde{I}| - |\textbf{p} \cap I| \big| \le \ell^{-1} |I| + 1$. Therefore, by \eqref{hxhyw}; \eqref{hvhu0}; \Cref{hqh}; and the fact that $\ell^{-1} |I| \le 5 I^{3 / 4}$ whenever $|I| \le 2 \ell^4 + 3 \ell^3 \le 5 \ell^4$, there exists for any $M \in [A^{1 / 2}, 2 \ell^4 + 3 \ell^3]$ some integer $\widetilde{x} = \widetilde{x} (x; M) \in [-6 \ell^4, 6 \ell^4]$ such that the following holds. Denoting $\widetilde{u} = (\widetilde{x}, 0) \in \mathbb{V}(R)$, we have for sufficiently large $A$ that
		\begin{flalign}
		\label{hlvu0estimate} 
		\begin{aligned}
		\displaystyle\sup_{y \in [-M, M]} \Big| M^{-1} \big( h_{\textbf{p}} (x + y + k) - & h_{\textbf{p}} (x + k) \big) - M^{-1} y s(\widetilde{u}; M) \Big| \textbf{1}_{\Gamma^c} \\
		& \le (\log M)^{-1 - c} + 6 M^{-1 / 4} < 2 (\log M)^{-1 - c}.
		\end{aligned} 
		\end{flalign}
		
		Thus, \Cref{at1t2b}, \eqref{hlvu0estimate}, and the fact that $s (u_1; M) \in \big( \frac{\varepsilon}{4}, 1 - \frac{\varepsilon}{4} \big)$ together yield $\mathfrak{D} (\textbf{w}; X, Y) < 384 c^{-1} \varepsilon^{-1} (\log X)^{-c}$ whenever $A^{1 / 2} \le X \le Y \le 2 \ell^4 + 3 \ell^3$. As mentioned previously, it follows that $\textbf{w}$ satisfies the first bound listed in the corollary. The second bound there follows from \eqref{hxhyw}; \eqref{hlvu0estimate}; the second statement of \eqref{gammau0mstst}; the second bound in \eqref{perturbationboundaryestimate} (with \Cref{estimatehrho}) on $\mathcal{B}_{\varepsilon} (N^{-1} v_0)$; and the fact that $|N^{-1} v - N^{-1} v_0| \le 2 N^{-1} A^3 < 2 N^{-2 / 3}$ (which holds since $A \le \ell \le N^{1 / 9}$). 
		
		To verify the third, observe that \eqref{hlvu0estimate} and $\big( s (\widetilde{u}; M), t (\widetilde{u}; M) \big) \in \mathcal{T}_{\varepsilon / 4}$ together imply on $\Gamma^c$ that
		\begin{flalign*}
		\displaystyle\frac{\varepsilon M}{5} \le h_{\textbf{p}} (M) \le \left( 1 - \displaystyle\frac{\varepsilon}{5} \right)M; \qquad \left( \displaystyle\frac{\varepsilon}{5} - 1 \right) M \le h_{\textbf{p}} (-M) \le - \displaystyle\frac{\varepsilon M}{5},
		\end{flalign*}
		
		\noindent for each integer $M \in [A^{1 / 2}, 2 \ell^4 + 3 \ell^3]$ and sufficiently large $A$. Thus, for any positive integers $M_1, M_2 \ge A^{1 / 2}$ with $M_1 + M_2 \le 2 \ell^4 + 3 \ell^3$, we have on $\Gamma^c$ that  
		\begin{flalign*}
		\displaystyle\frac{\varepsilon (M_1 + M_2 + 1)}{6} \le \big| \textbf{p} \cap [-M_1, M_2] \big| \le \left(1 - \displaystyle\frac{\varepsilon}{6} \right) (M_1 + M_2 + 1),
		\end{flalign*}
		
		\noindent for sufficiently large $A$, which by \eqref{hxhyw} quickly implies that $\textbf{p}$ satisfies the third estimate listed in the corollary. 
	\end{proof}

	\subsection{Proof of \Cref{localconverge}}

	\label{Convergence} 
	
	In this section we establish \Cref{localconverge}. To that end, we begin with the following theorem that establishes the convergence of local statistics of uniformly random tilings of large domains subject to \Cref{regularestimate}. In what follows, we recall the extended discrete sine kernel $\mathcal{K}_{\xi}$ from \Cref{kernellimitdefinition} and the set $\mathscr{X} (\mathscr{M})$ associated with a tiling $\mathscr{M}$ from \Cref{xm}.

	\begin{thm} 	
		
		\label{localconverge2}
		
		For any fixed real number $\varepsilon \in \big( 0, \frac{1}{4} \big)$ and integers $k, B \ge 1$, there exists a sequence $\varpi = \varpi (\varepsilon, k, B) = (\varpi_1, \varpi_2, \ldots ) \subset \mathbb{R}_{> 0}$ of real numbers tending to $0$ such that the following holds. Adopt the notation of \Cref{regularestimate}; set $v_0 = (x, y) \in \mathbb{V}(R)$; and define $\nabla \mathcal{H} (N^{-1} v_0) = (s, t) \in \mathcal{T}_{\varepsilon}$ and $\xi = e^{\pi \textbf{\emph{i}} s} \frac{\sin (\pi t)}{\sin (\pi - \pi s - \pi t)} \in \mathbb{H}$. Further fix integer sequences $\textbf{\emph{x}} = (x_1, x_2, \ldots , x_k) \in \mathbb{Z}^k$ and $\textbf{\emph{y}} = (y_1, y_2, \ldots , y_k) \in \mathbb{Z}^k$ such that $|x_j|, |y_j| \le B$ for each $j \in [1, k]$. Then, letting $\textbf{\emph{K}} = \textbf{\emph{K}}_{\textbf{\emph{x}}, \textbf{\emph{y}}} = \textbf{\emph{K}}_{\textbf{\emph{x}}, \textbf{\emph{y}}; \xi}$ denote the $k \times k$ matrix whose $(i, j)$ entry is equal to $\mathcal{K}_{\xi} (x_i, y_i; x_j, y_j)$, we have that
		\begin{flalign}
		\label{probabilityxmdeltak}
		\Bigg| \mathbb{P} \bigg[ \bigcap_{j = 1}^k \big( (x + x_j, y + y_j) \in \mathscr{X} (\mathscr{M}) \big)\bigg] - \det \textbf{\emph{K}} \Bigg| < \varpi_N. 
		\end{flalign}

	\end{thm} 
	
	\begin{proof}
		
		We will proceed by first locally (around $v_0$) coupling the non-intersecting path ensemble associated with $\mathscr{M}$ with suitably chosen random path ensembles sampled according to the measure $\mathbb{P}$ from \Cref{betaprobability}, and then by using the results from \Cref{ProbabilityP} to analyze the latter ensembles. To that end, let $\delta = \delta (\varepsilon) = (\delta_1, \delta_2, \ldots ) \subset \mathbb{R}_{> 0}$ and $\varsigma = \varsigma (\varepsilon) = (\varsigma_1, \varsigma_2, \ldots ) \subset \mathbb{R}_{> 0}$ denote the sequences of real numbers tending to $0$ and $A = A (\varepsilon) = (A_1, A_2, \ldots ) \subset \mathbb{Z}_{\ge 1}$ denote the sequence of integers tending to $\infty$ provided by \Cref{plpprcouple}; after decreasing the elements of $A$ if necessary (in such a way that they still tend to $\infty$), we may assume that $A_N < N^{1 / 20}$. Then set $T = T_N = \big\lfloor \frac{A_N}{2} \big\rfloor$ and define $u_0 = \big( x, y - T \big) \in \mathbb{Z}^2$, so that $\mathcal{B}_{\varepsilon N / 2} (u_0) \subset \mathcal{B}_{\varepsilon N} (v_0)$. We may further assume after shifting $R$ if necessary that $u_0 = (0, 0)$, and let us denote
		\begin{flalign*}
		 \nabla \mathcal{H} (0, 0) = (s_0, t_0) \in \mathcal{T}_{\varepsilon}; \qquad  \beta_0 = \frac{\sin (\pi t_0)}{\sin (\pi t_0) + \sin \big( \pi (1 - s_0 - t_0) \big)}; \qquad \ell = \ell_N = \lfloor N^{1 / 9} \rfloor. 
		\end{flalign*} 
		
		Now, recall the sequence $\textbf{q} \in \mathbb{W}$ and non-intersecting path ensemble $\textbf{Q} \in \mathfrak{W}_{\textbf{q}}^{\ell}$ from \Cref{qunlambda}, as well as the initial data $\textbf{p}, \textbf{r} \in \mathbb{W}$ from \Cref{lrqdefinition}. Further set $\beta_1 = \beta_0 - \delta_N$ and $\beta_2 = \beta_0 + \delta_N$, and recall the random non-intersecting path ensembles $\textbf{P}, \textbf{R} \in \mathfrak{W}^{\ell}$ from \Cref{pensemblerdefinition}. Additionally defining the domain $\mathscr{A} = [-A_N, A_N] \times [0, A_N] \subset \mathbb{Z}^2$, it follows from \Cref{plpprcouple} (where the $(s, t)$ and $\beta$ there are equal to $(s_0, t_0)$ and $\beta_0$ here, respectively) that there exists a coupling between $\textbf{P}$, $\textbf{Q}$, and $\textbf{R}$ for which $\mathbb{P} \big[ \textbf{P} |_{\mathscr{A}} = \textbf{Q} |_{\mathscr{A}} = \textbf{R} |_{\mathscr{A}} \big] \ge 1 - \varsigma_N$ (in what follows, $\textbf{R}$ will not be used, but $\textbf{P}$ will be).
		
		Letting $\mathscr{M}_{\textbf{P}} = \mathscr{M} (\textbf{P})$ denote the tiling associated with $\textbf{P}$ (as in \Cref{NonIntersectingCorrelation}), we therefore obtain 
		\begin{flalign}
		\label{xmxmpk} 
		\begin{aligned} 
		\Bigg| \mathbb{P} \bigg[ \bigcap_{j = 1}^k \big\{ (x + x_j, y + y_j) \in \mathscr{X} (\mathscr{M}) \big\} \bigg] - \det \textbf{K} \Bigg| & \le \bigg| \mathbb{P} \Big[ \bigcap_{j = 1}^k \big\{ ( x_j, y_j + T)   \in \mathscr{X} (\mathscr{M}_{\textbf{P}}) \big\} \Big] - \det \textbf{K} \bigg| + \varsigma_N \\
		& = \Big| \det \big[ K_{\beta_1; \textbf{p}} ( x_i, y_i + T; x_j, y_j + T) \big]  - \det \textbf{K} \Big| + \varsigma_N,
		\end{aligned} 
		\end{flalign} 
		
		\noindent where to obtain the last equality in \eqref{xmxmpk} we recalled the kernel $K$ from \eqref{kxtys} and used \Cref{determinantnonintersecting}. 
		
		To estimate the right side of \eqref{xmxmpk}, let $\xi_1 = \beta_1 (1 - \beta_1)^{-1} e^{\pi \textbf{i} s_0}$ and $\textbf{K}_1$ denote the $k \times k$ matrix whose $(i, j)$ entry is $\mathcal{K}_{\xi_1} (x_i, y_i; x_j, y_j)$. We will first use \Cref{limitdeterminantnonintersecting} to establish the convergence of $\det \big[ K_{\beta_1; \textbf{p}} ( x_i, y_i + T; x_j, y_j + T) \big]$ to $\det \textbf{K}_1$, and then show that $\det \textbf{K}_1$ converges to $\det \textbf{K}$. 
		
		To implement the former task, we must verify that $\textbf{p}$ satisfies the two assumptions listed in \Cref{limitdeterminantnonintersecting}. To that end we apply \Cref{ait1t2b2}, where the $D$, $A$, $(s, t)$, and $\textbf{w}$ there are equal to $1$, $A_N$, $(s_0, t_0)$, and $\textbf{p}$ here, respectively. This yields a constant $C_1 = C_1 (\varepsilon) > 1$ and an event $\Gamma$ with $\mathbb{P} [\Gamma] \le C_1 A_N^{-1}$, such that on $\Gamma^c$ we have $\textbf{q} \in \mathcal{Z} (\ell)$ (recall \Cref{zl}) and that $\textbf{p}$ satisfies the following two properties. First, denoting $c = \frac{1}{20000}$, we have that $\mathfrak{D} (\textbf{p}; X, \infty) < 384 c^{-1} \varepsilon^{-1} (\log X)^{-c}$ for any real number $X \ge A_N^{1 / 2}$. Second, for any interval $I \subset [-A_N^3, A_N^3]$ of length $\lfloor A_N^{1 / 2} \rfloor$, we have that $\big( 1 - s_0 - 5 (\log A_N)^{-c} \big) A_N^{1 / 2} \le |\textbf{p} \cap I| \le \big( 1 - s_0 + 5 (\log A_N)^{-c} \big) A_N^{1 / 2}$. 
		
		We now apply \Cref{limitdeterminantnonintersecting}, with the $c$ there equal to $\frac{\varepsilon}{3}$ here; the $\beta$ there equal to $\beta_1$ here; the $\rho$ there equal to $1 - s_0$ here; the $\kappa_N$ there equal to $(\log \log A_N)^{-1}$ here; the $U$ there equal to $\lfloor A_N^{1 / 2} \rfloor$; the $V$ there equal to $A_N^3$ here; the $\textbf{a}$ there equal to $\textbf{p}$ here; and the $N$ there equal to $m + n + 1 \in [4 \ell^3, 4 \ell^4 + 1]$ here. On $\Gamma^c$, it follows that $\textbf{p}$ satisfies the two assumptions listed in that lemma, which therefore yields a sequence $\vartheta = \vartheta (\varepsilon, B) = (\vartheta_1, \vartheta_2, \ldots ) \subset \mathbb{R}_{> 0}$ of real numbers tending to $0$ such that 
		\begin{flalign*} 
		\big| K_{\beta_1; \textbf{p}} (x_i, y_i + T; x_j, y_j + T) - \mathcal{K}_{\xi_1} (x_i, y_i; x_j, y_j) \big| \textbf{1}_{\Gamma^c} < \vartheta_N, 
		\end{flalign*} 
		
		\noindent for any $i, j \in [1, k]$. Due to the uniform boundedness of the $\big| \mathcal{K}_{\xi_1} (x_i, y_i; x_j, y_j) \big|$ over the parameters $|x_i|, |y_i|, |x_j|, |y_j| \le B$, it follows after increasing the elements of $\vartheta$ if necessary (in such a way that they still tend to $0$) that
		\begin{flalign}
		\label{klbetalpsi}
		\Big| \det \big[ K_{\beta_1; \textbf{p}} (x_i, y_i + T; x_j, y_j + T) \big] - \det \textbf{K}_1 \Big| \textbf{1}_{\Gamma^c} < \vartheta_N. 
		\end{flalign} 
		
		\noindent Since $\det \big[ K_{\beta_1; \textbf{p}} (x_i, y_i + T; x_j, y_j + T) \big], \det \textbf{K}_1 \in [0, 1]$ (as both are probabilities associated with a determinantal point process), \eqref{xmxmpk}, \eqref{klbetalpsi}, and the fact that $\mathbb{P} [\Gamma] \le C_1 A_N^{-1}$ together yield  
		\begin{flalign}
		\label{xmxmpk1} 
		\Bigg| \mathbb{P} \bigg[ \bigcap_{j = 1}^k \big\{ (x + x_j, y + y_j) \in \mathscr{X} (\mathscr{M}) \big\} \bigg] - \det \textbf{K}\Bigg| & \le | \det \textbf{K} - \det \textbf{K}_1 | + \vartheta_N + \varsigma_N + C_1 A_N^{-1}.
		\end{flalign} 
		
		It therefore remains to bound $| \det \textbf{K} - \det \textbf{K}_1|$, to which end one must bound $|\xi - \xi_1|$. To do this, recall from \Cref{regularestimate} that $\nabla \mathcal{H} (z) \in \mathcal{T}_{\varepsilon}$ for each $z \in \mathcal{B}_{\varepsilon / 2} \subset \mathcal{B}_{\varepsilon} (N^{-1} v_0)$. Since $\mathcal{H}$ is a maximizer of $\mathcal{E}$ on $\mathcal{B}_{\varepsilon / 2}$, it follows from \Cref{derivativeshestimate} (and \Cref{estimatehrho}) that there exists a constant $C_2 = C_2 (\varepsilon) > 1$ such that $\big| \nabla \mathcal{H} (0, 0) - \nabla \mathcal{H} (N^{-1} v_0) \big| < C_2 N^{-1} |v_0| \le C_2 A_N N^{-1}$, and so $|s - s_0| + |t - t_0| \le 2 C_2 A_N N^{-1} < C_2 N^{-1 / 2}$. 
		
		Due to the uniform continuity of $\xi$ in $(s, t)$ over $(s, t) \in \mathcal{T}_{\varepsilon}$; the uniform boundedness and continuity of $\mathcal{K}_{\xi} (x_i, y_i; x_j, y_j)$ over $|x_i|, |y_i|, |x_j|, |y_j| \le B$ and $\xi$ in compact subsets of $\mathbb{H}$; and the uniform continuity of $\det \textbf{K}$ in the $\mathcal{K}_{\xi} (x_i, y_i; x_j; y_j)$, we deduce the existence of a sequence $\varpi = \varpi (\varepsilon, B) = (\varpi_1, \varpi_2, \ldots ) \subset \mathbb{R}_{> 0}$ of real numbers tending to $0$ such that $| \det \textbf{K} - \det \textbf{K}_1| < \frac{\varpi_N}{2}$. Combining this bound with \eqref{xmxmpk1} and decreasing the elements of $\varpi$ if necessary (in such a way that they still tend to $0$) yields \eqref{probabilityxmdeltak}.
	\end{proof}	
	
	Now we can establish \Cref{localconverge}.

	\begin{proof}[Proof of \Cref{localconverge}]
		
		Through a suitable shift, we may assume that $\mathfrak{v} = (0, 0)$. Moreover, throughout this proof, let $\mathscr{M}_N \in \mathfrak{E} (R_N)$ denote a uniformly random tiling of $R_N$, and abbreviate the set $\mathscr{X}_N = \mathscr{X} (\mathscr{M}_N)$ from \Cref{xm}. Further fix integers $1 \le k \le B$ and two $k$-tuples $\textbf{x} = (x_1, x_2, \ldots , x_k) \in \mathbb{Z}^k$ and $\textbf{y} = (y_1, y_2, \ldots , y_k) \in \mathbb{Z}^k$, such that $|x_j|, |y_j| < B$ for each $j \in [1, k]$. Additionally denote $\xi = e^{\pi \textbf{i} s} \frac{\sin (\pi t)}{\sin (\pi - \pi s - \pi t)}$, and let $\textbf{K}$ denote the $k \times k$ matrix whose $(i, j)$ entry is equal to the extended discrete sine kernel $\mathcal{K}_{\xi} (x_i, y_i; x_j, y_j)$ from \Cref{kernellimitdefinition}. For arbitrary such $k, B, \textbf{x}, \textbf{y}$, it suffices to show that 
		\begin{flalign}
		\label{mnxjyjvkappak}
		\displaystyle\lim_{N \rightarrow \infty} \Bigg| \mathbb{P} \bigg[ \bigcap_{j = 1}^k \big\{ v_N + (x_j, y_j) \in \mathscr{X}_N \big\} \bigg] - \det \textbf{K} \Bigg| = 0. 
		\end{flalign}
		
		The difference between \eqref{probabilityxmdeltak} and \eqref{mnxjyjvkappak} is twofold. First, the former only applies to domains approximating disks in the sense of \Cref{regularestimate}, while the $R_N$ are not restricted to satisfy this assumption. Second, the latter states convergence of local statistics to a Gibbs measure of slope $\nabla \mathcal{H} (\mathfrak{v}) = (s, t)$, while the former would to one of slope $\nabla \mathcal{H}_N (N^{-1} v_N)$, where $\mathcal{H}_N$ is the maximizer of $\mathcal{E}$ with boundary data $\mathfrak{h}_N$ (instead of $\mathfrak{h}$). To remedy the first point, we restrict $\mathscr{M}_N$ to a free tiling of some sufficiently large disk. To remedy the second, we show that $\lim_{N \rightarrow \infty} \nabla \mathcal{H}_N (N^{-1} v_N) = \mathcal{H} (\mathfrak{v})$. 
		
		We begin with the former. To that end, first recall from \Cref{derivativehcontinuouss} that $\nabla \mathcal{H}$ is continuous on the open set $\mathcal{S} = \big\{ z \in \mathfrak{R}: \nabla \mathcal{H} (z) \in \mathcal{T} \big\} \subseteq \mathfrak{R}$. Therefore, since $\nabla \mathcal{H} (\mathfrak{v}) \in \mathcal{T}$, there exists some $\varepsilon = \varepsilon (\mathcal{H}) > 0$ such that $\mathcal{B}_{\varepsilon} (\mathfrak{v}) \subset \mathfrak{R} \cap \bigcap_{N = 1}^{\infty} N^{-1} R_N$ and $\nabla \mathcal{H} (z) \in \mathcal{T}_{\varepsilon}$ for each $z \in \mathcal{B}_{\varepsilon} (\mathfrak{v})$.
			
		Now, let $H_N$ denote the height function on $R_N$ associated with $\mathscr{M}_N$. Then \Cref{hnh} yields the existence of a sequence $\varpi = \varpi \big( \mathfrak{R}, \mathfrak{h}, \{ R_N \}, \{ h_N \} \big) = (\varpi_1, \varpi_2, \ldots ) \subset \mathbb{R}_{> 0}$ of real numbers tending to $0$ such that, if we define the event 
		\begin{flalign*}
		\Omega_N = \Big\{ \displaystyle\sup_{u \in \mathbb{V}(R_N) \cap N \mathfrak{R}} \big| N^{-1} H_N (u) - \mathcal{H} (N^{-1} u) \big| > \varpi_N \Big\},
		\end{flalign*} 
		
		\noindent for each $N \in \mathbb{Z}_{\ge 1}$, then $\lim_{N \rightarrow \infty} \mathbb{P} [\Omega_N] = 0$. 
		
		Thus, let us restrict to the event $\Omega_N^c$ and further restrict $\mathscr{M}_N$ to a free tiling $\widetilde{\mathscr{M}}_N$ of $\mathcal{B}_{\lfloor \varepsilon N \rfloor} (N \mathfrak{v}) \cap \mathbb{T} = \mathcal{B}_{\lfloor \varepsilon N \rfloor} \cap \mathbb{T}$. Let $\widetilde{R}_N \subset R_N$ denote the domain tiled by this free tiling; let $\widetilde{h}_N: \partial \widetilde{R}_N \rightarrow \mathbb{Z}$ denote the associated boundary height function; and define $\widetilde{\mathfrak{h}}_N:  \partial (N^{-1} \widetilde{R}_N) \rightarrow \mathbb{R}$ by setting $\widetilde{\mathfrak{h}}_N (N^{-1} u) = N^{-1} \widetilde{h}_N (u)$ for each $u \in \partial \widetilde{R}_N$. Further setting $\widetilde{\mathfrak{R}} = \mathcal{B}_{\varepsilon}$; $\widetilde{\mathfrak{h}} = \mathcal{H} |_{\partial \widetilde{\mathfrak{R}}}$; and $\widetilde{\mathcal{H}} = \mathcal{H} |_{\widetilde{\mathfrak{R}}}$, it follows that $(\widetilde{\mathfrak{R}}, \widetilde{\mathfrak{h}}, \widetilde{\mathcal{H}}, \mathfrak{v}, \widetilde{R}_N, \widetilde{h}_N,\widetilde{\mathfrak{h}}_N, v_N)$ satisfy the conditions of \Cref{localconverge} on $\Omega_N^c$. Moreover, since a uniformly random tiling of any domain satisfies the Gibbs property (recall \Cref{InfiniteMeasures}), the law of the tiling $\widetilde{\mathscr{M}}_N$ is uniform on $\mathfrak{E} (\widetilde{R}_N)$. In particular, it suffices to establish this theorem assuming that $\mathfrak{R} = \widetilde{\mathfrak{R}}$ and $R_N = \widetilde{R}_N$, and thus we replace $(\widetilde{\mathfrak{R}}, \widetilde{\mathfrak{h}}, \widetilde{\mathcal{H}}, \widetilde{R}_N, \widetilde{h}_N, \widetilde{\mathfrak{h}}_N)$ with $(\mathfrak{R}, \mathfrak{h}, \mathcal{H}, R_N, h_N, \mathfrak{h}_N)$ to simplify notation. So, after a suitable shift, we may assume that we are in the setting of \Cref{regularestimate} (whose $N$ is equal to $\lfloor \varepsilon N \rfloor$ here). 
		
		In this case, $\mathfrak{R} = \mathcal{B}_{\varepsilon}$. As in \Cref{regularestimate}, extend $H_N$ to $\overline{\mathcal{B}}_{\varepsilon N}$ by linearity on $\mathbb{F} (R_N)$ and arbitrarily on $\overline{\mathcal{B}}_{\varepsilon N} \setminus \mathbb{F}(R_N)$, in such a way that it is $1$-Lipschitz on $\overline{\mathcal{B}}_{\varepsilon N}$. Also define $\mathfrak{g}_N: \partial \mathfrak{R} \rightarrow \mathbb{R}$ by setting $\mathfrak{g}_N (z) = N^{-1} H_N (Nz)$ for any $z \in \partial \mathfrak{R}$, and let $\mathcal{H}_N \in \Adm (\mathfrak{R}; \mathfrak{g}_N)$ denote the maximizer of $\mathcal{E}$ on $\mathfrak{R}$ with boundary data $\mathfrak{g}_N$. Denoting $(s_N, t_N) = \nabla \mathcal{H}_N (N^{-1} v_N)$, let us show that $(s_N, t_N) \approx (s, t) = \mathcal{H} (\mathfrak{v}) = \nabla \mathcal{H} (0, 0)$, to which end we apply \Cref{perturbationboundary}. 
		
		Recalling the constants $\delta = \delta (\varepsilon) \in (0, 1)$ and $C = C (\varepsilon) > 1$ from that proposition, the continuity of $\nabla \mathcal{H}$ on $\mathcal{S}$ yields a constant $\omega = \omega (\mathcal{H}) \in \big( 0, \frac{\varepsilon \delta}{4} \big)$ such that 
		\begin{flalign}
		\label{hzst} 
		\displaystyle\sup_{z \in \mathcal{B}_{\omega}} \big| \nabla \mathcal{H} (z) - (s, t) \big| < \frac{\delta}{4}.
		\end{flalign}
		
		\noindent Furthermore, since $\mathcal{H}$ and the $\mathcal{H}_N$ are admissible; since $\mathfrak{R} = \mathcal{B}_{\varepsilon}$; and since we have $\lim_{N \rightarrow \infty} \mathfrak{g}_N = \lim_{N\rightarrow \infty} \mathfrak{h}_N = \mathfrak{h}$, it follows from \Cref{h1h2gamma} that, for any fixed $\vartheta \in (0, 1)$,  
		\begin{flalign} 
		\label{hnzhzdelta} 
		\sup_{z \in \mathcal{B}_{\omega}} \big| \mathcal{H}_N (z) - \mathcal{H} (z) \big| \le \sup_{z \in \partial \mathcal{B}_{\varepsilon}} \big| \mathcal{H}_N (z) - \mathcal{H} (z) \big| \le \frac{\delta \omega \vartheta}{4}, 
		\end{flalign} 
		
		\noindent for sufficiently large $N$. Thus, \eqref{hzst} and \eqref{hnzhzdelta} together imply (by the mean value theorem) the existence of a linear function $\Lambda: \overline{\mathcal{B}}_{\omega} (\mathfrak{v}) \rightarrow \mathbb{R}$ of slope $(s, t)$ such that $\sup_{z \in \mathcal{B}_{\omega}} \big| \mathcal{H} (z) - \Lambda (z)\big| < \frac{\delta \omega}{4}$ and $\sup_{z \in \mathcal{B}_{\omega}} \big| \mathcal{H}_N (z) - \Lambda (z)\big| < \frac{\delta\omega}{2}$, for sufficiently large $N$. Hence, $\mathcal{H}_N$ and $\mathcal{H}$ are $\frac{\delta}{2}$-nearly linear of slope $(s, t)$ on $\mathcal{B}_{\omega}$, and so (by \Cref{estimatehrho}) the conditions of \Cref{perturbationboundary} apply. Then the first bound in \eqref{perturbationboundaryestimate}, \eqref{hnzhzdelta}, the continuity of $\nabla \mathcal{H}$ on $\mathcal{S}$, and the fact that $\lim_{N \rightarrow \infty} N^{-1} v_N = \mathfrak{v} = (0, 0)$ together imply for sufficiently large $N$ that 
		\begin{flalign*}
		\big| (s_N, t_N) - (s, t) \big| & \le \big| \nabla \mathcal{H}_N (N^{-1} v_N) - \nabla \mathcal{H} (N^{-1} v_N) \big| + \big| \nabla \mathcal{H} (N^{-1} v_N) - \nabla \mathcal{H} (\mathfrak{v}) \big| \\
		& \le C \displaystyle\sup_{z \in \mathcal{B}_{\omega}} \big| \mathcal{H}_N (z) - \mathcal{H} (z) \big| + \vartheta < 2 C \vartheta.
		\end{flalign*}
		
		\noindent Since $\vartheta$ was arbitrary, this implies that $\lim_{N \rightarrow \infty} (s_N, t_N) = (s, t)$. 
		
		Now, let $\xi_N = e^{\pi \textbf{i} s_N} \frac{\sin (\pi t_N)}{\sin (\pi - \pi s_N - \pi t_N)}$ for each $N \in \mathbb{Z}_{> 1}$, and define the $k \times k$ matrix $\textbf{K}_N$ whose $(i, j)$ entry is equal to $\mathcal{K}_{\xi_N} (x_i, y_i; x_j, y_j)$. Since we have assumed that we are in the setting of \Cref{regularestimate}, it then follows from \Cref{localconverge2} that 
		\begin{flalign*} 
		\displaystyle\lim_{N \rightarrow \infty} \Bigg| \mathbb{P} \bigg[ \bigcap_{j = 1}^k \big\{ v_N + (x_j, y_j) \in \mathscr{M}_N \big\} \bigg] - \det \textbf{K}_N \Bigg| = 0,	
		\end{flalign*} 
		
		\noindent which yields \eqref{mnxjyjvkappak}, in view of the facts that $\lim_{N \rightarrow \infty} \xi_N = \xi$; that $\mathcal{K}_{\xi} (x_i, y_i; x_j, y_j)$ is continuous in $\xi \in \mathbb{H}$ when $|x_i|, |y_i|, |x_j|, |y_j|$ is uniformly bounded; and that $\det \textbf{K}$ is continuous in $\mathcal{K}_{\xi}$.  
	\end{proof}

	\section{Coupling Random Non-Intersecting Path Ensembles}
	
	\label{ProofHHlHr}
	
	In this section we establish \Cref{plpprcouple}. To do this, we first exhibit a mutual coupling between $\textbf{P}$, $\textbf{Q}$, and $\textbf{R}$ such that their height functions are ordered on $[-\ell^3, \ell^3] \times [0, \ell]$ (stated alternatively, each path of $\textbf{Q}$ in $[-\ell^3, \ell^3] \times [0, \ell]$ lies between the corresponding paths in $\textbf{P}$ and $\textbf{R}$) with high probability. Then, we show that the expected difference between the height functions associated with $\textbf{P}$ and $\textbf{R}$ tends to $0$ on $\mathscr{A}$, as $N$ tends to $\infty$. Together with a Markov estimate, these two facts will yield a coupling between $\textbf{P}$, $\textbf{Q}$, and $\textbf{R}$, such that they likely coincide on $\mathscr{A}$. 
	
	We implement the former task in \Cref{ProofCoupling}. To establish the latter fact, we first require an estimate bounding the effect on the kernel $K_{\beta; \textbf{a}}$ from \eqref{kxtys} resulting from a perturbation of $\beta$ and $\textbf{a}$; this will be provided in \Cref{EstimateKernel}. We then bound the expected difference between the height functions associated with $\textbf{P}$ and $\textbf{R}$ and conclude the proof of \Cref{plpprcouple} in \Cref{ProofPCouple}.

	\subsection{An Ordered Coupling} 
	
	\label{ProofCoupling}
	
	Before establishing \Cref{plpprcouple}, which exhibits an exact coupling between $(\textbf{P}, \textbf{Q}, \textbf{R})$ around $(0, 0)$, we will in this section establish as \Cref{x1x2x} the existence of a coupling between $(\textbf{P}, \textbf{Q}, \textbf{R})$ that ensures with high probability that they are ordered, in the following sense. For any integer $t \ge 1$ and non-intersecting path ensemble $\textbf{E} \in \mathfrak{W}^t$, let $H_{\textbf{E}}: \mathbb{Z} \times \{ 0, 1, \ldots , t \}$ denote the height function associated with $\textbf{E}$ (as explained in and above \Cref{heightpaths}), such that $H_{\textbf{E}} (0, 0) = 0$. 
	
	The following proposition then indicates that it is possible to select $\delta = (\delta_1, \delta_2, \ldots )$ from \Cref{plpprcouple} in such a way that there exists a coupling between $(\textbf{P}, \textbf{Q}, \textbf{R})$ so that $H_{\textbf{P}} \le H_{\textbf{Q}} \le H_{\textbf{R}}$ occurs on $[-\ell^3, \ell^3] \times [0, \ell]$ with high probability. In the below, for any set $\mathcal{S}$, subset $\mathcal{S}' \subseteq \mathcal{S}$, and functions $F, G: \mathcal{S} \rightarrow \mathbb{Z}$, we say $F |_{\mathcal{S}'} \le G |_{\mathcal{S}'}$ if $F(z) \le G(z)$ for each $z \in \mathcal{S}'$.

	\begin{prop}
		
		\label{x1x2x}

		For any fixed real numbers $\varepsilon \in \big( 0, \frac{1}{4} \big)$ and $D > 0$, there exist a constant $C = C (\varepsilon, D) > 1$ and a sequence $\delta = \delta (\varepsilon) = (\delta_1, \delta_2, \ldots ) \subset \mathbb{R}_{> 0}$ of real numbers tending to $0$ such that the following holds. Adopt the notation of \Cref{regularestimate}; suppose that $v_0 \in \mathcal{B}_{\varepsilon N / 2}$; let $\ell = \lfloor N^{1 / 9} \rfloor$; and set $\big( (s, t), \beta, \beta_1, \beta_2 \big)$ as in \eqref{stbeta12}. Further adopt the notation of \Cref{qunlambda} and \Cref{pensemblerdefinition}, and set $\mathscr{S} = [-\ell^3, \ell^3] \times [0, \ell] \subset \mathbb{Z}^2$. Then, there exists a mutual coupling of $\textbf{\emph{P}}$, $\textbf{\emph{Q}}$, and $\textbf{\emph{R}}$ on a common probability space under which $\mathbb{P} \big[ H_{\textbf{\emph{P}}} |_{\mathscr{S}} \le H_{\textbf{\emph{Q}}} |_{\mathscr{S}} \le H_{\textbf{\emph{R}}} |_{\mathscr{S}} \big] \ge 1 - C N^{-D}$.
		
	\end{prop}

	In order to establish \Cref{x1x2x}, we will apply \Cref{monotoneheightcouple}, to which end we must exhibit a coupling under which $H_{\textbf{P}} |_{\partial \mathscr{S}} \le H_{\textbf{Q}} |_{\partial \mathscr{S}} \le H_{\textbf{R}} |_{\partial \mathscr{S}}$ occurs with high probability. On the event $\big\{ q \in \mathcal{Z} (\ell) \big\}$ (recall \Cref{zl}), \Cref{lrqdefinition} will be seen to imply that $H_{\textbf{P}} \le H_{\textbf{Q}} \le H_{\textbf{R}}$ holds on the bottom, left, and right sides of $\partial \mathscr{S}$. 
	
	To establish this bound on the top side of $\partial \mathscr{S}$, we will approximate the drifts of the paths in $\textbf{P}$ and $\textbf{R}$. To do this, recall from above \Cref{heightpaths} in \Cref{NonIntersectingCorrelation} that the space-time locations of the right jumps performed by paths in a non-intersecting path ensemble $\textbf{E}$ coincide with the locations of type $2$ lozenges in the associated tiling $\mathscr{M} = \mathscr{M} (\textbf{E})$. The following definition, which is analogous to \Cref{xm} for $\mathscr{X} (\mathscr{M})$, provides notation for the set of such locations.  
	
	\begin{definition}
		
	\label{ym} 
	
	For any domain $R \subseteq \mathbb{T}$ and tiling $\mathscr{M} \in \mathfrak{E} (R)$ of $R$, let $\mathscr{Y} (\mathscr{M})$ denote the set of vertices $(x, y) \in \mathbb{V}(R)$ such that $\big( x, y + \frac{1}{2} \big)$ is the center of a type $2$ lozenge in $\mathscr{M}$. If $\mathscr{M} = \mathscr{M} (\textbf{E})$ is the tiling associated with a non-intersecting path ensemble $\textbf{E}$, then we abbreviate $\mathscr{Y} (\textbf{E}) = \mathscr{Y} \big( \mathscr{M} (\textbf{E}) \big)$.

	\end{definition}
	
	After suitable averaging over $(x, y)$, one might view $\mathbb{P} \big[ (x, y) \in \mathscr{Y} (\textbf{E}) \big]$ as prescribing the typical number of rightwards jumps in $\textbf{E}$ around $(x, y)$. When $\textbf{E} = \textbf{Q}$, the local law \Cref{heightlocal1} implies that this quantity (after averaging around $(0, 0)$) is approximately equal to $t$. Now we can establish the following proposition, which indicates the existence of a sequence $\delta$ in \Cref{x1x2x} such that $\mathbb{P} \big[ (x, y) \in \mathscr{Y} (\textbf{P}) \big]$ is slightly smaller than $t$ and $\mathbb{P} \big[ (x, y) \in \mathscr{Y} (\textbf{R}) \big]$ is slightly larger than $t$.

	\begin{prop}
		
		\label{betalbetar} 
		
		For any fixed real number $\varepsilon \in \big( 0, \frac{1}{4} \big)$, there exists a sequence $\delta = \delta (\varepsilon) = (\delta_1, \delta_2, \ldots ) \subset \mathbb{R}_{> 0}$ of real numbers tending to $0$ such that the following holds. Let $\ell = \lfloor N^{1 / 9} \rfloor$; adopt the notation of \Cref{regularestimate}; and suppose that $v_0 \in \mathcal{B}_{\varepsilon N / 2}$. Further set $\big( (s, t), \beta, \beta_1, \beta_2 \big)$ as in \eqref{stbeta12}; adopt the notation of \Cref{qunlambda} and \Cref{pensemblerdefinition}; and recall the event $\Gamma = \Gamma_{\ell}$ from \Cref{ait1t2b2}. Then, after restricting to the event $\Gamma^c$, we have  
		\begin{flalign}
		\label{sdeltanx1x2}
		\begin{aligned}
		\mathbb{P} \big[ (k, T) \in \mathscr{Y} (\textbf{\emph{P}}) \big] \le t - 2 (\log \log N)^{-1}; \qquad \mathbb{P} \big[ (k, T) \in \mathscr{Y} (\textbf{\emph{R}}) \big] \ge t + 2 (\log \log N)^{-1},
		\end{aligned}
		\end{flalign}
		
		\noindent for any integers $k \in [-\ell^3, \ell^3]$ and $T \in [\ell^{3 / 4}, \ell]$. 
	\end{prop} 
	
	\begin{proof}
		
		The bounds in \eqref{sdeltanx1x2} will follow from a suitable application of \Cref{determinantnonintersecting} and \Cref{limitdeterminantnonintersecting}. To that end, first recall the kernel $K_{\textbf{a}; \beta}$ from \eqref{kxtys}; fix integers $k \in [-\ell^3, \ell^3]$; and $T \in [\ell^{3 / 4}, \ell]$; and let $\textbf{w}$ denote either the sequence $\textbf{p} - k$ or $\textbf{r} - k$. 
		
		Let us verify that $\textbf{w}$ satisfies the two assumptions listed in \Cref{limitdeterminantnonintersecting} on $\Gamma^c$. To do this, apply \Cref{ait1t2b2} (where the $A$ there is equal to $\ell$ here), which implies upon restriction to $\Gamma^c$ that $\textbf{w}$ satisfies the following two properties. First, denoting $c = \frac{1}{20000}$, we have that $\mathfrak{D} (\textbf{w}; X, \infty) < 384 c^{-1} \varepsilon^{-1} (\log X)^{-c}$ for any real number $X \ge \ell^{1 / 2}$. Second, for any interval $I \subset [-2 \ell^3, 2 \ell^3]$ of length $\lfloor \ell^{1 / 2} \rfloor$, we have that $\big( 1 - s - 5 (\log \ell)^{-c} \big) \ell^{1 / 2} \le |\textbf{w} \cap I| \le \big( 1 - s + 5 (\log \ell)^{-c} \big) \ell^{1 / 2}$. 
		
		Next, fix an arbitrary real number $\mathfrak{b} \in \big( \frac{\varepsilon}{4}, 1 - \frac{\varepsilon}{4} \big)$, and set $\zeta = \zeta_{\mathfrak{b}} = \mathfrak{b} (1 - \mathfrak{b})^{-1} e^{\pi \textbf{i} s}$. Then apply \Cref{limitdeterminantnonintersecting}, with the $c$ there equal to $\frac{\varepsilon}{4}$ here; the $\beta$ there equal to $\mathfrak{b}$ here; the $\rho$ there equal to $1 - s$ here; the $\kappa_N$ there equal to $(\log \log \ell)^{-1}$ here; the $U$ there equal to $\lfloor \ell^{1 / 2} \rfloor$ here; the $V$ there equal to $\ell^3$ here; the $\textbf{a}$ there equal to $\textbf{w}$ here; and the $N$ there equal to $m + n + 1$ here. As shown above, $\textbf{w}$ satisfies the two assumptions listed in that lemma, so there exist sequences $\vartheta = \vartheta (\varepsilon) = (\vartheta_1, \vartheta_2, \ldots ) \subset \mathbb{R}_{> 0}$ of real numbers tending to $0$ and $B = B (\varepsilon) = (B_1, B_2, \ldots ) \subset \mathbb{Z}_{\ge 1}$ of integers tending to $\infty$ such that 
		\begin{flalign}
		\label{kx1tt1x2t2} 
		\big| K_{\mathfrak{b}; \textbf{w}} (x_1, T + y_1; x_2, T + y_2) - \mathcal{K}_{\zeta} (x_1, y_1; x_2, y_2) \big| < \vartheta_N. 
		\end{flalign} 
		
		\noindent for any $x_1, x_2, y_1, y_2 \in \mathbb{Z}$ satisfying $|x_1|, |y_1|, |x_2|, |y_2| \le B_N$. 
		
		Now let $t_{\mathfrak{b}} \in (0, 1)$ be such that $\mathfrak{b} = \sin (\pi t_{\mathfrak{b}}) \big(\sin (\pi t_{\mathfrak{b}}) + \sin (\pi - \pi s - \pi t_{\mathfrak{b}}) \big)^{-1}$, and let $\textbf{W} \in \mathfrak{W}_{\textbf{w}}^{\ell}$ denote a random non-intersecting path ensemble sampled from the measure $\mathbb{P}_{\mathfrak{b}; \textbf{w}}$ from \Cref{betaprobability}. To establish \eqref{sdeltanx1x2}, we use the fact that $\mathcal{K}_{\zeta}$ is the kernel for the translation-invariant Gibbs measure $\mu_{\zeta} \in \mathfrak{P} (\mathbb{T})$, defined by \eqref{muxi}, with slope $(s, t_{\mathfrak{b}})$ (recall \Cref{InfiniteMeasures}). In particular, if $\mathscr{M}_{\zeta} \in \mathfrak{E} (\mathbb{T})$ is randomly sampled under $\mu_{\zeta}$, then $\mathbb{P} \big[ (x, y) \in \mathscr{Y} (\mathscr{M}_{\zeta}) \big] = t_{\mathfrak{b}}$ for any $(x, y) \in \mathbb{Z}^2$.  
		
		This fact, \Cref{determinantnonintersecting}, \eqref{kx1tt1x2t2}, and the limits $\lim_{N \rightarrow \infty} \vartheta_N = 0$ and $\lim_{N \rightarrow \infty} B_N = \infty$ together yield a sequence $\gamma = \gamma (\varepsilon) = (\gamma_1, \gamma_2, \ldots ) \subset \mathbb{R}_{> 0}$ of real numbers tending to $0$ such that 
		\begin{flalign}
		\label{xtx2estimate} 
		\Big| \mathbb{P} \big[ (0, T) \in \mathscr{Y} (\textbf{W}) \big] - t_{\mathfrak{b}} \Big| = \Big| \mathbb{P} \big[ (0, T) \in \mathscr{Y} (\textbf{W}) \big] - \mathbb{P} \big[ (0, T) \in \mathscr{Y} (\mathscr{M}_{\zeta}) \big] \Big| <  \gamma_N.	
		\end{flalign}
		
		Let $\delta = \delta (\varepsilon) = (\delta_1, \delta_2, \ldots ) \subset \mathbb{R}_{> 0}$ denote a sequence of real numbers tending to $0$ with the following property. If we denote $\beta_1 = \beta - \delta_N$ and $\beta_2 = \beta + \delta_N$, then 
		\begin{flalign*} 
		t_{\beta_1} < t - 2 (\log \log N)^{-1} - \gamma_N; \qquad t_{\beta_2} > t + 2 (\log \log N)^{-1} + \gamma_N.
		\end{flalign*}  
		
		\noindent The existence of this sequence $\delta$ and the fact that it only depends on $\varepsilon$ follow from the definition \eqref{stbeta12} of $\beta$ and the fact that $(s, t) \in \mathcal{T}_{\varepsilon}$, respectively. 
		 
		 Now set $\mathfrak{b} \in \big\{ \beta_1, \beta_2 \big\}$ in \eqref{xtx2estimate} depending on whether $\textbf{w} \in \{ \textbf{p} - k, \textbf{r} - k \}$, respectively. Since $\mathbb{P} \big[ (0, T) \in \mathscr{Y} (\textbf{W}) \big]$ is equal to $\mathbb{P} \big[ (k, T) \in \mathscr{Y} (\textbf{P}) \big]$ or $\mathbb{P} \big[ (k, T) \in \mathscr{Y} (\textbf{R}) \big]$ depending on whether $\textbf{w} = \textbf{p} - k$ or $\textbf{w} = \textbf{r} - k$, respectively, \eqref{xtx2estimate} yields \eqref{sdeltanx1x2}. 
	\end{proof} 
	
	Now we can establish \Cref{x1x2x}. 	
	
	\begin{proof}[Proof of \Cref{x1x2x}]

		Throughout this proof we adopt the notation of \Cref{betalbetar}. As mentioned above, we will establish this proposition using \Cref{monotoneheightcouple}, to which end we must exhibit a coupling between $(\textbf{P}, \textbf{Q}, \textbf{R})$ such that $H_{\textbf{P}} |_{\partial \mathscr{S}} \le H_{\textbf{Q}} |_{\partial \mathscr{S}} \le H_{\textbf{R}} |_{\partial \mathscr{S}}$ with high probability. To do this, recall that the event $\Gamma = \Gamma_{\ell}$ satisfies $\mathbb{P} [\Gamma] \le C \ell^{-10 D} \le C N^{-D}$ for some constant $C = C(\varepsilon, D) > 1$, since $\ell = \lfloor N^{1 / 9} \rfloor$. Let us condition on $\textbf{q}$ and further restrict to the event $\Gamma^c$, so that $\textbf{q} \in \mathcal{Z} (\ell)$ (recall \Cref{zl} and \Cref{ait1t2b2}). 
		
		Then \Cref{lrqdefinition} implies that $p_j \le q_j \le r_j$ for each $j \in [-m, n]$. Thus, since $\textbf{p} = \textbf{q} = \textbf{r}$ on $[- 2 \ell, 2 \ell]$, it follows from \Cref{heightpaths} (or \eqref{hxhyw}) that $H_{\textbf{P}} (x, 0) \le H_{\textbf{Q}} (x, 0) \le H_{\textbf{R}} (x, 0)$, for each $x \in [-\ell^3, \ell^3]$. This verifies $H_{\textbf{P}} \le H_{\textbf{Q}} \le H_{\textbf{R}}$ on the bottom side of $\partial \mathscr{S}$. 
		
		Let us next address the left side of $\partial \mathscr{S}$. To that end, \Cref{lrqdefinition} implies that $p_j \le q_j - \frac{\ell^2}{4}$ and $r_j \ge q_j + \frac{\ell^2}{4}$ when $q_j \le - \frac{\ell^3}{2}$. Moreover, the second part of \Cref{ait1t2b2} and the fact that $(s, t) \in \mathcal{T}_{\varepsilon}$ together imply for $N$ sufficiently large that $\frac{\varepsilon |I|}{2} \le |\textbf{p} \cap I|, |\textbf{q} \cap I|, |\textbf{r} \cap I| \le \big( 1 - \frac{\varepsilon}{2} \big) |I|$, for any interval $I \subseteq [-\ell^3, \ell^3]$ of length at least $\ell^{1 / 2}$. Together with \eqref{hxhyw}, these two facts imply that $H_{\textbf{P}} (-\ell^3, 0) \le H_{\textbf{Q}} (-\ell^3, 0) - \frac{\varepsilon \ell^2}{8} \le H_{\textbf{Q}} (-\ell^3, 0) - 2 \ell$ and $H_{\textbf{Q}} (-\ell^3, 0) + 2 \ell \le H_{\textbf{Q}} (-\ell^3, 0) + \frac{\varepsilon \ell^2}{8} \le H_{\textbf{R}} (-\ell^3, 0)$. Since $H_{\textbf{P}}$, $H_{\textbf{Q}}$, and $H_{\textbf{R}}$ are all $1$-Lipschitz, it follows that $H_{\textbf{P}} \le H_{\textbf{Q}} \le H_{\textbf{R}}$ on $\{ - \ell^3 \} \times [0, \ell] \subset \partial \mathscr{S}$; this addresses the left side of $\partial \mathscr{S}$. By similar reasoning, we have the same bound on $\{ \ell^3 \} \times [0, \ell]$, which addresses the right side of $\partial \mathscr{S}$.  
		
		To establish this bound with high probability on the top side of the boundary $[-\ell^3, \ell^3] \times \{ \ell \} \subset \partial \mathscr{S}$, we first use \Cref{betalbetar} to bound the expectation $\mathbb{E} \big[ H_{\textbf{W}} (k, \ell) - H_{\textbf{W}} (k, 0) \big]$ for each $\textbf{W} \in \{ \textbf{P}, \textbf{R} \}$ and $k \in [-\ell^3, \ell^3]$, and then we apply the concentration estimate \Cref{walksexpectationnear}. When $\textbf{W} = \textbf{P}$, \Cref{heightpaths}, \Cref{betalbetar}, and the fact that $H_{\textbf{P}}$ is $1$-Lipschitz together imply that 
		\begin{flalign}
		\label{hkphpk0}
		\mathbb{E} \big[ H_{\textbf{P}} (k, \ell) - H_{\textbf{P}} (k, 0) \big] \le 2 \ell^{3 / 4} + \ell \displaystyle\max_{T \in [\ell^{3 / 4}, \ell]} \mathbb{P} \big[ (k, T) \in \mathscr{Y} (\textbf{P}) \big] \le 2 \ell^{3 / 4} + \big( t - 2 (\log \log N)^{-1} \big) \ell, 
		\end{flalign}
		
		\noindent on $\Gamma^c$, for each $k \in [-\ell^3, \ell^3]$. Therefore, denoting the event 
		\begin{flalign*}
		\Omega_{\textbf{P}} = \Bigg\{ \displaystyle\max_{x \in [-\ell^3, \ell^3]} \big( H_{\textbf{P}} (x, \ell) - H_{\textbf{P}} (x, 0) \big) > \ell \bigg( t - \displaystyle\frac{1}{\log \log N} \bigg) \Bigg\} \cup \Gamma,
		\end{flalign*}
		
		\noindent we deduce from \Cref{walksexpectationnear}, \Cref{ait1t2b2}, \eqref{hkphpk0}, and the fact that $\ell = \lfloor N^{1 / 9} \rfloor$ that $\mathbb{P} [\Omega_{\textbf{P}}] \le C N^{-D}$, after increasing $C$ if necessary. Similarly, set $c_0 = \frac{1}{25000}$, and define the events
		\begin{flalign*}
		& \Omega_{\textbf{Q}} = \Bigg\{ \displaystyle\max_{x \in [-\ell^3, \ell^3]} \big| H_{\textbf{Q}} (x, \ell) - H_{\textbf{Q}} (x, 0) - t \ell \big| > \displaystyle\frac{\ell}{(\log N)^{c_0}	} \Bigg\} \cup \Gamma; \\
		& \Omega_{\textbf{R}} = \Bigg\{ \displaystyle\max_{x \in [-\ell^3, \ell^3]} \big( H_{\textbf{R}} (x, \ell) - H_{\textbf{R}} (x, 0) \big) < \ell \bigg( t + \displaystyle\frac{1}{\log \log N} \bigg) \Bigg\} \cup \Gamma.
		\end{flalign*}
		
		Implementing analogous reasoning as above, we obtain $\mathbb{P} [\Omega_{\textbf{R}}] \le C N^{-D}$. Furthermore, the local law \Cref{heightlocal1}, the fact that $\ell = \lfloor N^{1 / 9} \rfloor$, and a union bound together yield $\mathbb{P} [\Omega_{\textbf{Q}}] \le C N^{-D}$, after again increasing $C$ if necessary. Setting $\Omega = \Omega_{\textbf{P}} \cup \Omega_{\textbf{Q}} \cup \Omega_{\textbf{R}}$, it then follows that $\mathbb{P} [\Omega] \le 3 C N^{-D}$. 
		
		Now, recall from \Cref{lrqdefinition} that on $\Gamma^c$ we have $H_{\textbf{P}} (x, 0) \le H_{\textbf{Q}} (x, 0) \le H_{\textbf{R}} (x, 0)$, for each $x \in [-\ell^3, \ell^3]$. Hence, it follows that $H_{\textbf{P}} (x, \ell) \le H_{\textbf{Q}} (x, \ell) \le H_{\textbf{R}} (x, \ell)$ for each $x \in [-\ell^3, \ell^3]$ on $\Omega^c$, which in view of the above implies that $H_{\textbf{P}} |_{\partial \mathscr{S}} \le H_{\textbf{Q}} |_{\partial \mathscr{S}} \le H_{\textbf{R}} |_{\partial \mathscr{S}}$ on $\Omega^c$. Thus, \Cref{monotoneheightcouple} and \Cref{prprobability} (the latter of which quickly implies that $H_{\textbf{W}} |_{\mathscr{S}}$, conditional on $H_{\textbf{W}} |_{\partial \mathscr{S}}$, is uniformly distributed on $\mathfrak{G} \big( H_{\textbf{W}} |_{\partial \mathscr{S}} \big)$, for each $\textbf{W} \in \{ \textbf{P}, \textbf{R} \}$) yields a mutual coupling between $(\textbf{P}, \textbf{Q}, \textbf{R})$ with $\mathbb{P} \big[ H_{\textbf{P}} |_{\mathscr{S}} \le H_{\textbf{Q}} |_{\mathscr{S}} \le H_{\textbf{R}} |_{\mathscr{S}} \big] \ge 1 - \mathbb{P} [\Omega] \ge 1 - 3 C N^{-D}$. 
	\end{proof}

	\subsection{A Perturbative Kernel Estimate}
	
	\label{EstimateKernel}
	
	We would next like to refine the coupling from \Cref{x1x2x} to ensure with high probability that $H_{\textbf{P}} = H_{\textbf{Q}} = H_{\textbf{R}}$ on some large neighborhood $\mathscr{A}$ of $(0, 0)$. To do this, we will bound the expected difference $\mathbb{E} \big[ H_{\textbf{R}} (x, y) - H_{\textbf{P}} (x, y) \big]$ for each $(x, y)$ near $(0, 0)$. Since $\textbf{P}$ and $\textbf{R}$ are sampled according to $\mathbb{P}_{\beta_1; \textbf{p}}$ and $\mathbb{P}_{\beta_2; \textbf{r}}$, respectively (recall \Cref{betaprobability}), \Cref{determinantnonintersecting} indicates that $\mathbb{E} \big[ H_{\textbf{R}} (x, y) - H_{\textbf{P}} (x, y) \big]$ can be expressed through the correlation kernel $K_{\beta; \textbf{a}} (x, t; y, s)$ from \eqref{kxtys}. Thus, to bound this expected difference, we will require a bound on $|K_{\beta_2; \textbf{r}} - K_{\beta_1; \textbf{p}}|$. 
	
	In this section we establish the following proposition, which provides such an estimate on $|K_{\beta; \textbf{a}} - K_{\beta'; \textbf{a}'}|$ in the more general setting when $|\beta - \beta'|$ is small and $\textbf{a} = \textbf{a}'$ in a neighborhood of $0$. In what follows, the parameter $\ell$ will dictate the size of this neighborhood, and we will also impose further assumptions (the second and third one below) on $\textbf{a}$ and $\textbf{a}'$ that essentially bounds the associated quantities $\mathfrak{D}$ (recall \Cref{dqut1t2}) and the global densities of $\textbf{a}$ and $\textbf{a}'$.
	
	\begin{prop}
		
		\label{continuityk} 
		
		For any fixed real number $\varepsilon \in \big( 0, \frac{1}{4} \big)$ there exists a constant $C = C(\varepsilon) > 1$ such that the following holds. Fix real numbers $c, \delta, B, \beta, \beta' \in \mathbb{R}_{> 0}$ and an integer $\ell > C$ satisfying 
		\begin{flalign*} 
		c, \delta \in (0, 1); \qquad B \in (1, \ell^{1 / 2}); \qquad \beta, \beta' \in (\varepsilon, 1 - \varepsilon); \qquad |\beta - \beta'| \le \delta.
		\end{flalign*} 
		
		\noindent Further let $\textbf{\emph{a}} = (a_{-m}, a_{1 - m}, \ldots , a_n) \in \mathbb{W}$ and $\textbf{\emph{a}}' = (a_{-m}', a_{1 - m}', \ldots , a_n') \in \mathbb{W}$ denote two integer sequences with $a_0 \le 0 < a_1$ and $a_0' \le 0 \le a_1'$, satisfying the following three assumptions. 
		
		\begin{enumerate}
			
			\item We have that $\textbf{\emph{a}} \cap [-\ell, \ell] = \textbf{\emph{a}}' \cap [-\ell, \ell]$. 
			
			\item For any $v \in \mathbb{Z}$ and $X, Y \in \mathbb{R}$ such that $|v| \le 3 B \le X \le Y \le \max \{ - a_{-m}, a_n \}$, we have that 
			\begin{flalign*} 
			\displaystyle\max \big\{ \mathfrak{D} (\textbf{\emph{a}} - v; X, Y), \mathfrak{D} (\textbf{\emph{a}}' - v; X, Y) \big\} < (\log X)^{-c}.
			\end{flalign*} 
			
			\item For each integer $j \in [B, n]$, we have that $a_j, a_j' \in \big( (1 + \varepsilon) j, \varepsilon^{-1} j \big)$ and, for each integer $j \in [-m, -B]$, we have that $a_j, a_j' \in \big( \varepsilon^{-1} j, (1 + \varepsilon) j \big)$.
		\end{enumerate} 
		
		\noindent If $x, y \in \mathbb{Z}$ and $t, s \in \mathbb{Z}_{\ge 0}$ satisfy $|x|, |y|, t, s \le B$, then we have that 
		\begin{flalign}
		\label{kxtys12} 
		\begin{aligned} 
		& \displaystyle\max \Big\{ \big| K_{\beta; \textbf{\emph{a}}} (x, t; y, s) \big|, \big| K_{\beta'; \textbf{\emph{a}}'} (x, t; y, s) \big| \Big\} < (CB)^{CB}; \\
		& \big| K_{\beta; \textbf{\emph{a}}} (x, t; y, s) - K_{\beta'; \textbf{\emph{a}}'} (x, t; y, s) \big| < (CB)^{CB} \big( (\log \ell)^{-c / CB} + \delta^{1 / CB} \big).
		\end{aligned} 
		\end{flalign}

	\end{prop}
	
	\begin{proof}
		
		Throughout this proof, we assume that $\ell > C > 400 \varepsilon^{-4}$. For any $w, z \in \mathbb{C}$, denote the integrand appearing in the definition \eqref{kxtys} of $K(x, t; y, s)$ by 
		\begin{flalign*}
		\mathcal{I} (\textbf{a}; \beta) = \mathcal{I}(\textbf{a}; \beta; w; z) = \displaystyle\frac{(z - y + 1)_{s - 1}}{(w - x)_{t + 1}} \displaystyle\frac{1}{w - z} \displaystyle\frac{\sin (\pi w)}{\sin (\pi z)} \left( \displaystyle\frac{1 - \beta}{\beta} \right)^{w - s} \displaystyle\prod_{j = -m}^n \displaystyle\frac{z - a_j}{w - a_j},
		\end{flalign*}
		
		\noindent and define $\mathcal{I} (\textbf{a}'; \beta') = \mathcal{I}(\textbf{a}'; \beta'; w; z)$ similarly. It is quickly verified that we may select contours for $w$ and $z$ in \eqref{kxtys} such that the bounds  
		\begin{flalign}
		\label{wz} 
		|w| \le [-5B, 5B]; \qquad d (w, \mathbb{Z}) \ge \displaystyle\frac{1}{4}; \qquad |\Re z| \le 3B; \qquad |w - z| \ge \frac{1}{16}; \qquad \Re z - \frac{1}{2} \in \mathbb{Z},
		\end{flalign}
		
		\noindent all hold, and such that the length of the $w$-contour is at most $40B$; here, we recall that $d (w, \mathbb{Z}) = \min_{k \in \mathbb{Z}} |w - k|$. Thus, we begin by estimating $\big| \mathcal{I} (\textbf{a}; \beta) \big|$ and $\big| \mathcal{I} (\textbf{a}'; \beta') \big|$ assuming \eqref{wz}. 
		
		Let $U, V \in \mathbb{Z}$ satisfy $-20 \varepsilon^{-1} B \le -U \le V \le 20 \varepsilon^{-1} B$. Then, under \eqref{wz}, we have the six bounds 
		\begin{flalign}
		\label{wzxystestimates} 
		\begin{aligned} 
		& \left| \displaystyle\frac{1}{w - z} \right| \le 16; \qquad \left| \displaystyle\prod_{j = -U}^{V} \displaystyle\frac{z - a_j}{w - a_j}\right| \le (24 \varepsilon^{-2} B)^{50 B / \varepsilon} \big( |\Im z| + 1 \big)^{50B / \varepsilon}; \qquad \left| \displaystyle\frac{\beta}{1 - \beta} \right|^{w - s} \le \varepsilon^{-6B}; \\
		& \left| \displaystyle\frac{(z - y + 1)_{s - 1} }{(w - x)_{t + 1}} \right| \le 16 \big( 5B + |\Im z| \big)^B;  \qquad \left| \displaystyle\frac{\sin (\pi w) }{\sin (\pi z)} \right| \le 2 e^{5 \pi B -\pi |\Im z|}; \qquad  \left| \displaystyle\prod_{j = -m}^n \displaystyle\frac{1}{w - a_j}\right| \le 16.
		\end{aligned} 
		\end{flalign}
		
		\noindent  Here, $[-U, V] \subseteq [- 20 \varepsilon^{-1} B, 20 \varepsilon^{-1} B] \subseteq [- \varepsilon \ell, \varepsilon \ell] \subseteq [-m, n]$ due to the third property satisfied by $(\textbf{a}, \textbf{a}')$ and the facts that $B \le \ell^{1 / 2}$ and $\ell > 400 \varepsilon^{-4}$; thus, the quantity on the left side of the second inequality in \eqref{wzxystestimates} is defined. That bound then holds since $\prod_{j = -U}^V \big| w - a_j \big|^{-1} \le 16$ (as $d (w, \mathbb{Z}) \ge \frac{1}{4}$ and the elements of $\textbf{a}$ are mutually distinct) and $\prod_{j = -U}^V \big| z - a_j \big| \le \big( |\Im z| + |\Re z| + 20 \varepsilon^{-2} B \big)^{41B / \varepsilon} \le (24 \varepsilon^{-2} B)^{50B / \varepsilon} \big( |\Im z| + 1 \big)^{50 B / \varepsilon}$, where we have used the facts that $|\Re z| \le 3B$ and $|a_j| \in [-20 \varepsilon^{-2} B, 20 \varepsilon^{-2} B]$ for each $j \in [-U, V]$. Moreover, the first inequality in \eqref{wzxystestimates} holds since $|w - z| \ge \frac{1}{16}$; the third holds since $\varepsilon \le \frac{\beta}{1 - \beta} \le \varepsilon^{-1}$ and $|w - s| \le 6B$; the fourth since $\big| (w - x)_{t + 1} \big| \ge \frac{1}{16}$ and $(z - y + 1)_{s - 1} \le \big( |z| + |y| + B \big)^B \le \big( 5B + |\Im z| \big)^B$; the fifth since $\Re z - \frac{1}{2} \in \mathbb{Z}$ and $|w| \le 5B$; and the sixth since $d (w, \mathbb{Z}) \ge \frac{1}{4}$. 
		
		Now, in order to bound $\big| \mathcal{I} (\textbf{a}; \beta) \big|$ (the analogous estimate on $\big| \mathcal{I} (\textbf{a}'; \beta') \big|$ is very similar), we will first replace the terms $z - a_j$ appearing there with $\Re z - a_j$ for $|j|$ sufficiently large. To that end, fix $U, V \in \mathbb{Z}$ and $M \in \mathbb{R}_{> 0}$ such that $20 \varepsilon^{-1} B \le M \le 20 \varepsilon^{-2} B$; such that $-20 \varepsilon^{-1} B \le -U \le V \le 20 \varepsilon^{-1} B$; and such that $a_{-U - 1} < -M < a_{-U} \le a_V \le M < a_{V + 1}$. The third condition satisfied by $\textbf{a}$ and $\textbf{a}'$ guarantees that such $U, V, M$ exist. Denoting $b_j = a_j - \Re z$ for each $j \in [-m, n]$, we then have that 
		\begin{flalign*}
		\log \Bigg| \displaystyle\prod_{j = -m}^{-U - 1} \displaystyle\frac{a_j - z}{a_j - \Re z} \Bigg| +  & \log \Bigg| \displaystyle\prod_{j = V + 1}^n \displaystyle\frac{a_j - z}{a_j - \Re z} \Bigg| \\
		&  = \displaystyle\frac{1}{2} \displaystyle\sum_{j = -m}^{-U - 1} \log \bigg( \displaystyle\frac{(\Im z)^2 + b_j^2}{b_j^2} \bigg) + \displaystyle\frac{1}{2} \displaystyle\sum_{j = V + 1}^n \log \bigg( \displaystyle\frac{(\Im z)^2 + b_j^2}{b_j^2} \bigg). 
		\end{flalign*} 
		
		\noindent Recalling that $|\Re z| \le 3B$; that $a_{-U - 1} < -M \le -20 \varepsilon^{-1} B$; that $a_{V + 1} > M \ge 20 \varepsilon^{-1} B$; and that $(1 + \varepsilon) |j| \le |a_j| \le \varepsilon^{-1} |j|$ for each $j \in [-m, n]$, it follows upon setting $\zeta = (1 + \varepsilon)^{-1} |\Im z|$ that 
		\begin{flalign}
		\label{qjzqjz2} 
		\log \Bigg| \displaystyle\prod_{j = -m}^{-U} \displaystyle\frac{a_j - z}{a_j - \Re z} \Bigg| +  \log \Bigg| \displaystyle\prod_{j = V}^n \displaystyle\frac{a_j - z}{a_j - \Re z} \Bigg| & \le \displaystyle\sum_{j = B}^{\infty} \log \bigg( \displaystyle\frac{\zeta^2 + j^2}{j^2}  \Bigg) \le \zeta \displaystyle\int_0^{\infty} \log \left( \displaystyle\frac{1}{x^2} + 1 \right) dx = \pi \zeta,	
		\end{flalign}
		
		\noindent where we have used the fact that $\log \big( 1 + x^{-2} \big)$ is decreasing in $x$. 
		
		Next, since $|w - \Re z| \le 8B$ and $|a_j - w| \ge |a_j| - 5B \ge 15B$ for $j \le -U$ or $j \ge V$, we have from a Taylor expansion that
		\begin{flalign*}
		\Bigg| \log \displaystyle\prod_{j = -m}^{-U} & \bigg( \displaystyle\frac{a_j - w}{a_j - \Re z} \bigg) + \log \displaystyle\prod_{j = V}^n \bigg( \displaystyle\frac{a_j - w}{a_j - \Re z} \bigg) - (\Re z - w) \displaystyle\sum_{j = -m}^{-U} \displaystyle\frac{1}{a_j - \Re z} - (\Re z - w) \displaystyle\sum_{j = V}^n \displaystyle\frac{1}{a_j - \Re z} \Bigg| \\
		& \le 3 \big| w - \Re z \big|^2 \left( \displaystyle\sum_{j = -m}^{-U} \displaystyle\frac{1}{|a_j - \Re z|^2} + \displaystyle\sum_{j = V}^n \displaystyle\frac{1}{|a_j - \Re z|^2} \right) \le 192 B^2 \displaystyle\sum_{j = 15B}^{\infty} \displaystyle\frac{1}{j^2} \le 14 B.
		\end{flalign*} 
		
		\noindent Moreover, since $\Re z \le 3B$, the second condition imposed on $\textbf{a}$ and $\textbf{a}'$ implies that 
		\begin{flalign*}
		\Bigg| (\Re z - w) \displaystyle\sum_{j = -m}^{-U} \displaystyle\frac{1}{a_j - \Re z} + (\Re z - w) \displaystyle\sum_{j = V}^n \displaystyle\frac{1}{a_j - \Re z} \Bigg| = |w - \Re z| \mathfrak{D} (\textbf{a} - \Re z; M, Y_0) < 8 B (\log M)^{-c},
		\end{flalign*}
		
		\noindent where $Y_0 = \max \{ -a_{-m}, a_n \}$ and we have used the fact that $|w - \Re z| \le 8B$. Thus, 
		\begin{flalign}
		\label{qjwqjz} 
		\Bigg| \log \displaystyle\prod_{j = -m}^{-U} \bigg( \displaystyle\frac{a_j - \Re z}{a_j - w} \bigg) + \log \displaystyle\prod_{j = V}^n \bigg( \displaystyle\frac{a_j - \Re z}{a_j - w} \bigg) \Bigg| & \le 22 B.
		\end{flalign}
		
		\noindent Combining \eqref{qjzqjz2}, \eqref{qjwqjz}, the fact that $1 - \frac{1}{1 + \varepsilon} \ge \frac{\varepsilon}{2}$, the six bounds listed in \eqref{wzxystestimates}, and applying the entirely analogous reasoning for $\textbf{a}'$, we deduce that
		\begin{flalign}
		\label{estimatei1}  
		\max \Big\{ \big| \mathcal{I} (\textbf{a}; \beta) \big|, \big| \mathcal{I} (\textbf{a}'; \beta') \big| \Big\} \le \varepsilon^{-450 B / \varepsilon} B^{60 B / \varepsilon} \big( |\Im z| + 1 \big)^{60 B / \varepsilon} e^{- \varepsilon |\Im z|}.
		\end{flalign}
		
		\noindent Since the contours in \eqref{kxtys} can be taken such that $w$ and $z$ satisfy \eqref{wz} everywhere on these contours and such that the contour for $w$ in \eqref{kxtys} is of length at most $40 B$, the first bound in \eqref{kxtys12} follows from integrating \eqref{estimatei1} over $w$ and $z$ (and the fact that $\binom{t - s}{x - y}, \frac{t!}{(s - 1)!} \le B^B$).  
		
		To establish the second bound in \eqref{kxtys12}, fix some $0 < r < \frac{\varepsilon \ell}{10}$ (to be specified later) and observe by integrating \eqref{estimatei1} over $w$ and $z$ such that $|\Im z| > r$ and setting $\xi = |\Im z|$ that 
		\begin{flalign}
		\label{kqbeta2} 
		\begin{aligned} 
		\big| K_{\textbf{a}; \beta} (x, t; y, s) - K_{\textbf{a}'; \beta'} (x, t; y, s) \big| & \le \varepsilon^{-500 B / \varepsilon} B^{70 B / \varepsilon} \displaystyle\int_r^{\infty}  \big( \xi + 1 \big)^{60 B / \varepsilon} e^{- \varepsilon \xi} d \xi \\
		& \qquad + 80 r B^{B + 1} \displaystyle\sup_{w, z} \big| \mathcal{I} (\textbf{a}; \beta; w; z) - \mathcal{I} (\textbf{a}'; \beta'; w; z) \big|,
		\end{aligned} 
		\end{flalign}
		
		\noindent where $w, z \in \mathbb{C}$ are taken over all pairs of complex numbers satisfying \eqref{wz}. 
		
		To estimate the second term on the right side of \eqref{kqbeta2}, let us fix $U_0, V_0 \in \mathbb{Z}$ and $M_0 \in \mathbb{R}$ such that $\varepsilon \ell \le M_0 \le \ell$; such that $U_0, V_0 \in [\varepsilon \ell, \ell]$; and such that $a_{-U_0 - 1} < -M_0 \le a_{-U_0} \le a_{V_0} \le M_0 < a_{V_0 + 1}$. Then, set 
		\begin{flalign*}
		\mathcal{J} (\textbf{a}; \beta) = \mathcal{J} (\textbf{a}; \beta; w; z) = \displaystyle\frac{(z - y + 1)_{s - 1}}{(w - x)_{t + 1}} \displaystyle\frac{1}{w - z} \displaystyle\frac{\sin (\pi w)}{\sin (\pi z)} \displaystyle\prod_{j = -U_0}^{V_0} \displaystyle\frac{z - a_j}{w - a_j},
		\end{flalign*}
		
		\noindent and so the first assumption on $\textbf{a}$ and $\textbf{a}'$ implies that 
		\begin{flalign}
		\label{jiidentity} 
		\begin{aligned} 
		& \mathcal{I} (\textbf{a}; \beta) = \mathcal{J} (\textbf{a}; \beta) \left( \displaystyle\frac{1 - \beta}{\beta} \right)^{w - s}  \displaystyle\prod_{j = V_0 + 1}^n \displaystyle\frac{z - a_j}{w - a_j} \displaystyle\prod_{j = -m}^{-U_0 - 1} \displaystyle\frac{z - a_j}{w - a_j}; \\
		& \mathcal{I} (\textbf{a}'; \beta') = \mathcal{J} (\textbf{a}; \beta) \left( \displaystyle\frac{1 - \beta'}{\beta'} \right)^{w - s}  \displaystyle\prod_{j = V_0 + 1}^n \displaystyle\frac{z - a_j'}{w - a_j'} \displaystyle\prod_{j = -m}^{-U_0 - 1} \displaystyle\frac{z - a_j'}{w - a_j'}.
		\end{aligned} 
		\end{flalign} 
		
		\noindent Then, through entirely analogous reasoning as used to deduce \eqref{estimatei1}, we have that 
		\begin{flalign}
		\label{jiestimate} 
		\big| \mathcal{J} (\textbf{a}; \beta) \big| \le \varepsilon^{-450 B / \varepsilon} B^{60 B / \varepsilon} (r + 1)^{60 B / \varepsilon}.
		\end{flalign}
		
		\noindent Now let us estimate the difference between the terms in $\mathcal{I} (\textbf{a}; \beta)$ and $\mathcal{I} (\textbf{a}'; \beta)$ not contained in $\mathcal{J} (\textbf{a}; \beta) = \mathcal{J} (\textbf{a}'; \beta)$. To do this first observe that since $|\beta - \beta'| \le \delta$; $\beta, \beta' \in (\varepsilon, 1 - \varepsilon)$; and $|w - s| \le |w| + s \le 6B$, we have 
		\begin{flalign}
		\label{beta1beta2estimate} 
		\Bigg| \bigg( \displaystyle\frac{1 - \beta}{\beta} \bigg)^{w - s} - \bigg( \displaystyle\frac{1 - \beta'}{\beta'} \bigg)^{w - s} \Bigg| \le  \varepsilon^{-2} \big( |w| + s \big) \delta\displaystyle\max_{v \in [\beta, \beta']} \left| \displaystyle\frac{1 - v}{v}\right|^{\Re (w - s)} \le 6B \varepsilon^{-6B - 2} \delta.
		\end{flalign}
		
		\noindent Next, we bound the quantity 
		\begin{flalign*} 
		\left| \log \displaystyle\prod_{j = -m}^{-U_0 - 1} \displaystyle\frac{a_j - z}{a_j - w} + \log \displaystyle\prod_{j = V_0 + 1}^n \displaystyle\frac{a_j - z}{a_j - w} \right|.
		\end{flalign*} 
		
		\noindent To that end, a Taylor expansion and the fact that $M_0 \ge \varepsilon \ell \ge 2 (3B + r) \ge 2|z|$ together imply 
		\begin{flalign*} 
		& \left| \log \displaystyle\prod_{j = -m}^{-U_0 - 1} \displaystyle\frac{a_j - z}{a_j} + \log \displaystyle\prod_{j = V_0 + 1}^n \displaystyle\frac{a_j - z}{a_j} \right| \\
		 & \qquad \qquad \le |z| \left| \displaystyle\sum_{j = -m}^{-U_0 - 1} \displaystyle\frac{1}{a_j} + \displaystyle\sum_{j = V_0 + 1}^n \displaystyle\frac{1}{a_j} \right| + 2 |z|^2 \left( \displaystyle\sum_{j = -m}^{-U_0 - 1} \displaystyle\frac{1}{a_j^2} + \displaystyle\sum_{j = V_0 + 1}^n \displaystyle\frac{1}{a_j^2} \right) \\
		& \qquad \qquad \le |z| \mathfrak{D} (\textbf{a}; M_0, Y_0) + 4 |z|^2 \displaystyle\sum_{j = M_0 + 1}^{\infty} \displaystyle\frac{1}{j^2} \le |z| ( \log M_0)^{-c} + 4 \varepsilon^{-1} |z|^2 \ell^{-1},
		\end{flalign*} 
		
		\noindent where we recall that $Y_0 = \max \{ - a_{-m}, a_n \}$; the analogous bound holds if $z$ is replaced by $w$, $a_j$ by $a_j'$, or both. Summing over all four such replacements and using the fact that $M_0 \ge  \varepsilon \ell \ge \ell^{1 / 2}$ yields 
		\begin{flalign}
		\label{qjzqjw} 
		\begin{aligned} 
		& \left| \log \displaystyle\prod_{j = -m}^{-U_0} \displaystyle\frac{a_j - z}{a_j - w} + \log \displaystyle\prod_{j = V_0}^n \displaystyle\frac{a_j - z}{a_j - w} \right| \le 4 |z| (\log \ell)^{-c} + 8 \varepsilon^{-1} |z|^2 \ell^{-1}; \\
		& \left| \log \displaystyle\prod_{j = -m}^{-U_0} \displaystyle\frac{a_j' - z}{a_j' - w} + \log \displaystyle\prod_{j = V_0}^n \displaystyle\frac{a_j' - z}{a_j' - w} \right| \le 4 |z| (\log \ell)^{-c} + 8 \varepsilon^{-1} |z|^2 \ell^{-1}.
		\end{aligned} 
		\end{flalign} 
		
		\noindent Let us assume that $r$, $\ell$, and $z$ satisfy $4 |z| (\log \ell)^{-c} + 8 \varepsilon^{-1} |z|^2 \ell^{-1} \le \frac{1}{2}$ (in addition to the previous constraint $r < \frac{\varepsilon \ell}{10}$). Then, the fact that $e^x \le 1 + 2x$ for $x \le 1$, \eqref{jiidentity}, \eqref{jiestimate}, \eqref{beta1beta2estimate}, \eqref{qjzqjw}, and the third estimate in \eqref{wzxystestimates} together yield
		\begin{flalign*}	
		\big| \mathcal{I} (\textbf{a}; \beta) - \mathcal{I} (\textbf{a}'; \beta') \big| \le \varepsilon^{-500 B / \varepsilon} B^{60 B / \varepsilon} (r + 1)^{60 B / \varepsilon} \big( |z| (\log \ell)^{-c} + |z|^2 \ell^{-1} + \delta), 
		\end{flalign*}
		
		\noindent which upon insertion into \eqref{kqbeta2} (using the fact that $|z| \le 3B + |\Im z|$) implies the existence of a constant $C_0 = C_0 (\varepsilon) > 1$ such that 
		\begin{flalign}
		\label{kqestimate3}
		\begin{aligned} 
		\big| K_{\textbf{a}; \beta} (x, t; y, s) - K_{\textbf{a}'; \beta'} (x, t; y, s) \big| & \le  \varepsilon^{-500 B / \varepsilon} B^{130 B / \varepsilon} \Bigg( \displaystyle\int_r^{\infty} (\xi + 1)^{60 B / \varepsilon} e^{- \varepsilon \xi} d \xi \\
		& \quad + (r + 1)^{70 B / \varepsilon} \big( (\log \ell)^{-c} + \ell^{-1} + \delta) \Bigg) \\
		& \le (C_0 B)^{C_0 B} \Big( e^{-\varepsilon r / 2} + (r + 1)^{C_0 B} \big( (\log \ell)^{-c} + \delta \big) \Big).
		\end{aligned} 
		\end{flalign}
		
		\noindent We now deduce the proposition by taking $r = \big( (\log \ell)^{-c} + \delta \big)^{-1 / 2 C_0 B}$ in \eqref{kqestimate3}. 
	\end{proof}

	\subsection{Proof of \Cref{plpprcouple}}
	
	\label{ProofPCouple} 
	
	In this section we establish \Cref{plpprcouple}. To that end, we begin with the following lemma, which under the notation of \Cref{x1x2x} bounds the expected difference $\mathbb{E} \big[ H_{\textbf{R}} (x, y) - H_{\textbf{P}} (x, y) \big]$ if $(x, y)$ in some large neighborhood of $(0, 0)$.

	\begin{prop} 
		
		\label{pqpqinitialexpectation} 
		
		Adopting the notation of \Cref{x1x2x}, there exists a sequence $A = A (\varepsilon) = (A_1, A_2, \ldots ) \subset \mathbb{Z}_{\ge 1}$ of integers tending to $\infty$ such that
		\begin{flalign}
		\label{expectationhrhp} 
		\Big| \mathbb{E} \big[ H_{\textbf{\emph{R}}} (x, y) \big] - \mathbb{E} \big[ H_{\textbf{\emph{P}}} (x, y) \big] \Big| \le A_N^{-5}, \qquad \text{for each $(x, y) \in [-A_N, A_N] \times [0, A_N] \cap \mathbb{Z}^2$.}
		\end{flalign}

	\end{prop}

	\begin{proof}
		
		Condition on $\textbf{q}$, and recall the event $\Gamma = \Gamma_{\ell}$ from \Cref{ait1t2b2}. Since that corollary yields a constant $C_0 = C_0 (\varepsilon) > 1$ for which $\mathbb{P} [\Gamma] \le C_0 \ell^{-5}$, and since $\big| H_{\textbf{W}} (x, y) \big| \le 2 A_N$ for each $\textbf{W} \in \{ \textbf{P}, \textbf{R} \}$ and $(x, y) \in [-A_N, A_N] \times [0, A_N]$, to establish \eqref{expectationhrhp} we may restrict to the event $\Gamma^c$. 
		
		Then $\textbf{q} \in \mathcal{Z} (\ell)$ (recall \Cref{zl}), so \Cref{lrqdefinition} implies that $H_{\textbf{P}} (x, 0) = H_{\textbf{R}} (x, 0)$ for each $x \in [-\ell, \ell]$. Therefore, recalling the set $\mathscr{Y} (\textbf{E})$ from \Cref{ym}, we have by \Cref{heightpaths} that 
		\begin{flalign}
		\label{hrhpy}
		\begin{aligned} 
		\mathbb{E} \big[ H_{\textbf{R}} (x, y) \big] - \mathbb{E} \big[ H_{\textbf{P}} (x, y) \big] & = \Big( \mathbb{E} \big[ H_{\textbf{R}} (x, y) \big] - \mathbb{E} \big[ H_{\textbf{R}} (x, 0) \big] \Big) - \Big( \mathbb{E} \big[ H_{\textbf{P}} (x, y) \big] - \mathbb{E} \big[ H_{\textbf{P}} (x, 0) \big] \Big) \\
		& = \displaystyle\sum_{T = 0}^y \Big( \mathbb{P} \big[ (x, T) \in \mathscr{Y} (\textbf{R}) \big] - \mathbb{P} \big[ (x, T) \in \mathscr{Y} (\textbf{P}) \big] \Big),
		\end{aligned} 
		\end{flalign}
		
		\noindent for any integers $x \in [-\ell, \ell]$ and $y \in [0, \ell]$. 
		
		To bound the right side of \eqref{hrhpy}, recall the set $\mathscr{X}$ from \Cref{xm}. Since $\mathscr{X}$ determines the tiling $\mathscr{M} |_{\mathscr{A}}$, it also determines $\mathscr{Y}$, and so it suffices to bound the total variation distance between the random variables $\mathscr{X} (\textbf{R})$ and $\mathscr{X} (\textbf{P})$ (after restricting to a neighborhood of $(0, 0)$). To that end, recall from \Cref{determinantnonintersecting} that the correlation functions of $\mathscr{X}$ are governed by the kernel $K$ from \eqref{kxtys}. So we begin by bounding $\big| K_{\beta_2; \textbf{r}} (x, t; y, s) - K_{\beta_1; \textbf{p}} (x, t; y, s) \big|$. 
		
		To do this, fix a real number $B \in (1, \ell^{1 / 2})$ and let $\Gamma_B$ denote the event from \Cref{ait1t2b2}, which satisfies $\mathbb{P} [\Gamma_B] \le C_0 B^{-5}$. Then apply \Cref{continuityk}, with the $\varepsilon$ there equal to $\frac{\varepsilon}{6}$ here; the $c$ there equal to $\frac{1}{25000}$ here; the $B$ there equal to $A_N$ here; the $(\textbf{a}, \textbf{a}')$ there equal to $(\textbf{p}, \textbf{r})$ here; and the $(\beta, \beta')$ there equal to $(\beta_1, \beta_2)$ here. Observe that $\textbf{p}$ and $\textbf{r}$ satisfy the three assumptions listed in that proposition, on the event $\Gamma_B^c$. Indeed, \Cref{lrqdefinition} implies that they satisfy the first assumption listed there on $\Gamma_B^c$, since then $\textbf{p}, \textbf{r} \in \mathcal{Z} (B) \subset \mathcal{Z} (\ell)$ (recall \Cref{zl}). That they also satisfy the second and third follows from the first and third parts of \Cref{ait1t2b2}, respectively. 
		
		Then, on $\Gamma_B^c$ this proposition yields a constant $C = C (\varepsilon) > 1$ such that 
		\begin{flalign}
		\label{kbeta1beta2pr}
		\begin{aligned}
		& \max \Big\{ \big| K_{\beta_1; \textbf{p}} (x, t; y, s) \big|, \big| K_{\beta_2; \textbf{r}} (x, t; y, s) \big| \Big\} \le (C B)^{CB}; \\
		& \big| K_{\beta_1; \textbf{p}} (x, t; y, s) - K_{\beta_2; \textbf{r}} (x, t; y, s) \big| \le (CB)^{CB} \big( (\log N)^{-1 / CB} + \delta_N^{1 / CB} \big),
		\end{aligned}
		\end{flalign}
		
		\noindent for any $x, y \in \mathbb{Z}$ and $t, s \in \mathbb{Z}_{\ge 0}$ with $|x|, |y|, t, s \le B$; here, we have used the fact that $\ell = \lfloor N^{1 / 9} \rfloor$ (and have recalled the notation $\delta_N$ from the statement of \Cref{x1x2x}).
		
		Next, for any sets $\textbf{x} = (x_1, x_2, \ldots , x_k) \subset \mathbb{Z}$ and $\textbf{y} = (y_1, y_2, \ldots,  y_k) \subset \mathbb{Z}_{\ge 0}$, let $\textbf{K}_{\textbf{x}, \textbf{y}}$ denote the $k \times k$ matrix whose $(i, j)$ entry is equal to $K_{\beta_1; \textbf{p}} (x_i, y_i; x_j, y_j)$. Similarly, let $\textbf{K}_{\textbf{x}, \textbf{y}}'$ denote the $k \times k$ matrix whose $(i, j)$ entry is equal to $K_{\beta_2; \textbf{r}} (x_i, y_i; x_j, y_j)$. Then, setting $\mathscr{B} = [-B, B] \times [0, B] \subset \mathbb{Z}^2$, it follows from \eqref{kbeta1beta2pr} that
		\begin{flalign}
		\label{determinantkestimate} 
		\displaystyle\max_{(x_j, y_j) \in \mathscr{B}} |\det \textbf{K}_{\textbf{x}, \textbf{y}} - \det \textbf{K}_{\textbf{x}, \textbf{y}}'| \le (12 CB)^{18 CB^3} \big( (\log N)^{-1 / CB} + \delta_N^{1 / CB} \big),
		\end{flalign} 
		
		\noindent for sufficiently large $N$ and on $\Gamma_B^c$. Summing \eqref{determinantkestimate} over all (at most $2^{6B^2}$) pairs of sequences $\textbf{x} = (x_1, x_2, \ldots , x_k) \in [-B, B]$ and $\textbf{y} = (y_1, y_2, \ldots , y_k) \subseteq [0, B]$ such that $(x_i, y_i) \ne (x_j, y_j)$ for $i \ne j$, and using the fact that $\mathbb{P} [\Gamma_B] \le C_1 B^{-5}$, it follows from \Cref{determinantnonintersecting} that the total variation distance between $\mathscr{X} \big( \textbf{P} |_{\mathscr{B}} \big)$ and $\mathscr{X} \big( \textbf{R} |_{\mathscr{B}} \big)$ is at most $C_1 B^{-5} + (24 CB)^{18CB^3} \big( (\log N)^{-1 / CB} + \delta_N^{1 / CB} \big)$.
		
		Since $\mathscr{X}$ determines $\mathscr{Y}$, the same estimate holds for the total variation distance between $\mathscr{Y} \big( \textbf{P} |_{\mathscr{B}} \big)$ and $\mathscr{Y} \big( \textbf{R} |_{\mathscr{B}} \big)$. So, \eqref{hrhpy} implies whenever $(x, y) \in \mathscr{B}$ that 
		\begin{flalign*}
		\mathbb{E} \big[ H_{\textbf{R}} (x, y) \big] - \mathbb{E} \big[ H_{\textbf{P}} (x, y) \big] \le 2 C_1 B^{-4} + (24 CB)^{19 CB^3} \big( (\log N)^{-1 / CB} + \delta_N^{1 / CB} \big),
		\end{flalign*}
		
		\noindent due to the deterministic bound $\big| H_{\textbf{P}} (x, y) \big|, \big| H_{\textbf{R}} (x, y) \big| \le 2B$. Then, we deduce the proposition upon setting $A_N = B = \min \{ \log \log \log N, \log \log \delta_N^{-1} \}$.
	\end{proof}

	Now we can establish \Cref{plpprcouple}.

	\begin{proof}[Proof of \Cref{plpprcouple}]
		
		By \Cref{x1x2x}, there exists a constant $C_1 = C_1 (\varepsilon) > 1$ and a mutual coupling between $(\textbf{P}, \textbf{Q}, \textbf{R})$ under which $\mathbb{P} \big[ \Omega_1 \big] \le C_1 N^{-2}$, where we have denoted the event $\Omega_1^c = \big\{ H_{\textbf{P}} |_{\mathscr{S}} \le H_{\textbf{Q}} |_{\mathscr{S}} \le H_{\textbf{R}} |_{\mathscr{S}} \big\}$ and recalled that $\mathscr{S} = [-\ell^3, \ell^3] \times [0, \ell]$. Additionally, \Cref{pqpqinitialexpectation} yields a sequence $A = (A_1, A_2, \ldots ) \subset \mathbb{Z}_{\ge 1}$ of integers tending to $\infty$ such that $\mathbb{E} \big[ H_{\textbf{R}} (x, y) \big] - \mathbb{E} \big[ H_{\textbf{P}} (x, y) \big] \le A_N^{-5}$, for each $(x, y) \in \mathscr{A} = [-A_N, A_N] \times [0, A_N]$. Denoting $\varsigma_N = 48 A_N^{-3}$ for each $N \in \mathbb{Z}_{\ge 1}$, it follows that $\varsigma = \varsigma (\varepsilon) = (\varsigma_1, \varsigma_2, \ldots ) \subset \mathbb{R}_{> 0}$ is a sequence of real numbers tending to $0$ such that $\mathbb{E} \big[ H_{\textbf{R}} (x, y) \big] - \mathbb{E} \big[ H_{\textbf{P}} (x, y) \big] \le \frac{\varsigma_N}{48 A_N^2}$, for each $(x, y) \in \mathscr{A}$. We may assume that $A_N < N^{1 / 20}$ and that $\varsigma_N > N^{-1 / 20}$. 
		
		Then, since $\mathbb{P}[\Omega_1] \le C_1 N^{-2}$ and since the $1$-Lipschitz properties of $H_{\textbf{P}}$ and $H_{\textbf{R}}$ imply the deterministic estimate $\big| H_{\textbf{R}} (x, y) - H_{\textbf{P}} (x, y) \big| \le 4A_N$ for each $(x, y) \in \mathscr{A}$, it follows that 
		\begin{flalign*} 
		0 & \le \mathbb{E} \big[ H_{\textbf{R}} (x, y) | \Omega_1^c \big] - \mathbb{E} \big[ H_{\textbf{P}} (x, y) | \Omega_1^c \big] \\
		& = \Big( \mathbb{E} \big[ H_{\textbf{R}} (x, y) \big] - \mathbb{E} \big[ H_{\textbf{P}} (x, y) \big] \Big) \mathbb{P}[\Omega_1^c]^{-1} - \bigg( \mathbb{E} \Big[ \big( H_{\textbf{R}} (x, y) - H_{\textbf{P}} (x, y) \big) \textbf{1}_{\Omega_1} \Big] \bigg)  \mathbb{P} [\Omega_1^c]^{-1} \\
		& \le \displaystyle\frac{\varsigma_N}{24 A_N^2} + 8 C_1 A_N N^{-2} \le \displaystyle\frac{\varsigma_N}{12 A_N^2},
		\end{flalign*} 
		
		\noindent for each $(x, y) \in \mathscr{A}$ and sufficiently large $N$. Therefore, since $|\mathscr{A}| \le 6 A_N^2$, we deduce that 
		\begin{flalign*} 
		\displaystyle\sum_{(x, y) \in \mathscr{A}} \mathbb{E} \big[ H_{\textbf{R}} (x, y) - H_{\textbf{P}} (x, y) | \Omega_1^c \big] \le \displaystyle\frac{\varsigma_N}{2}.
		\end{flalign*} 
		
		Thus, a Markov estimate and the fact that $H_{\textbf{P}} |_{\mathscr{A}}  \le H_{\textbf{R}} |_{\mathscr{A}}$ on $\Omega_1^c$ together yield an event $\Omega_2 \subseteq \Omega_1^c$ with $\mathbb{P} [\Omega_2] \le \frac{\varsigma_N}{2}$, such that $H_{\textbf{P}} |_{\mathscr{A}} = H_{\textbf{R}} |_{\mathscr{A}}$ on $\Omega_2^c \cap \Omega_1^c$. Letting $\Omega = \Omega_1 \cup \Omega_2$ then implies $\mathbb{P} [\Omega] \le \frac{\varsigma_N}{2} + C_1 N^{-2} \le \varsigma_N$ for sufficiently large $N$. Furthermore, since $H_{\textbf{P}} |_{\mathscr{A}} \le H_{\textbf{Q}} |_{\mathscr{A}} \le H_{\textbf{R}} |_{\mathscr{A}}$ on $\Omega_1^c$, we must have $H_{\textbf{P}} |_{\mathscr{A}} = H_{\textbf{Q}} |_{\mathscr{A}} = H_{\textbf{R}} |_{\mathscr{A}}$ on $\Omega^c$. Hence $ \mathbb{P} \big[ H_{\textbf{P}} |_{\mathscr{A}} = H_{\textbf{Q}} |_{\mathscr{A}} = H_{\textbf{R}} |_{\mathscr{A}} \big] \ge \mathbb{P} [\Omega^c] \ge 1 - \varsigma_N$, and so $\mathbb{P} \big[ \textbf{P} |_{\mathscr{A}} = \textbf{Q} |_{\mathscr{A}} = \textbf{R} |_{\mathscr{A}} \big] \ge 1 - \varsigma_N$. 
	\end{proof}

	\section{A Scale Reduction Estimate} 
	
	\label{Local}
	
	By the content of \Cref{ConvergenceProof} and \Cref{ProofHHlHr}, in order to establish \Cref{localconverge} it suffices to prove the local law \Cref{heightlocal1}. This will proceed by induction on (the logarithm of) the scale $M$. In particular, after establishing an initial estimate that consists of verifying \eqref{gammahnv0estimate} and \eqref{mhuhv0estimate} when $M \sim \frac{N}{(\log N)^c}$ is sufficiently large, we will show that if the bounds \eqref{gammahnv0estimate} and \eqref{mhuhv0estimate} hold on some scale $M$ then they also hold (with a mild error) on the smaller scale $\frac{M}{8}$. In this section we establish this type of ``scale reduction estimate,'' given by \Cref{estimaten2n} below; using this result, we will later prove \Cref{heightlocal1} in \Cref{EstimateHLocalProof}. 
	
	We begin in \Cref{LawGlobal2} by stating two global laws for lozenge tilings, with effective error rates, and a H\"{o}lder estimate satisfied by maximizers of $\mathcal{E}$ with suitable boundary data, which will be established in \Cref{GlobalLaw}, \Cref{GlobalEstimate2}, and \Cref{ProofGradientEstimateu} below. Next, we state the scale reduction estimate \Cref{estimaten2n} in \Cref{Estimaten9n}; there, we also provide some preliminary definitions and results that will be used in its proof. We then prove \Cref{estimaten2n} in \Cref{Estimaten9nProof}.

	\subsection{Effective Global Laws and a \texorpdfstring{$\mathcal{C}^{2, \alpha}$}{} Estimate} 
	
	\label{LawGlobal2}

	Recall that \Cref{hnh} provided a global law for height functions associated with uniformly random lozenge tilings. For the proof of \Cref{estimaten2n}, it will be useful to have an effective version of that result, which makes the dependence of the error $\varpi$ on $N$ there explicit. In this section we state two different such effective global laws, given by \Cref{heightapproximate} and \Cref{estimateboundaryheight} below, which will be established in \Cref{GlobalLaw} and \Cref{GlobalEstimate2}, respectively. 
	
	The former (\Cref{heightapproximate}) holds in greater generality, in that it applies to arbitrary boundary data subject to \Cref{regularestimate}, while the latter (\Cref{estimateboundaryheight}) only applies to boundary data whose associated limit shape does not exhibit frozen facets. However, the error is smaller in \Cref{estimateboundaryheight}, which will make it useful in the proof of \Cref{estimaten2n}. Therefore, we will also require a condition under which the assumptions in \Cref{estimateboundaryheight} can be verified; this will be given by \Cref{euv1v2estimategradient}, which provides gradient and H\"{o}lder estimates for maximizers of $\mathcal{E}$ under suitable boundary data. 
	
	The first of these two effective global laws is given by the following result, whose proof will be provided in \Cref{ProofApproximateGlobal} below. In what follows, we adopt the notation and assumptions of \Cref{regularestimate}, so that $\nabla \mathcal{H} (N^{-1} v_0) \in \mathcal{T}_{\varepsilon}$, meaning that $N^{-1} v_0$ is in the liquid region for $\mathcal{H}$.

	\begin{thm}
		
		\label{heightapproximate} 	
		
		For any fixed $\varepsilon \in \big( 0, \frac{1}{4} \big)$, there exists a constant $C = C(\varepsilon) > 1$ such that the following holds. Adopting the notation and assumptions of \Cref{regularestimate}, we have that 
		\begin{flalign*}
		\mathbb{P} \Bigg[ \displaystyle\max_{u \in \mathcal{B}_{\varepsilon N} (v_0)} \bigg|  N^{-1} \big( H (u) - H (v_0) \big) - \Big( \mathcal{H} \big( N^{-1} u \big) - \mathcal{H} \big( N^{-1} v_0 \big) \Big)  \bigg| > (\log N)^{- 1/ 400} \Bigg] < C e^{-N}. 
		\end{flalign*}	
		
	\end{thm}

	Before proceeding, let us briefly explain an issue that will arise if one attempts to use only \Cref{heightapproximate} in the proof of \Cref{heightlocal1}. As mentioned previously, that proof will be by induction, reducing the scale $M$ there by factors of $8$ at a time. In particular, to pass from the initial scale of $\frac{N}{(\log N)^c}$ to significantly smaller scales requires of order $\log N$ reductions. The error incurred by each reduction will approximately be given by that of the global law and so, if this error is $(\log N)^{-1 / 400}$ as in \Cref{heightapproximate}, then the total error accumulated will diverge with $N$. 
	
	Thus, although \Cref{heightapproximate} will be used in the proof of \Cref{heightlocal1}, we will also require a variant of that global law with a smaller error estimate, namely one of order $(\log N)^{-1 - c}$ for some $c > 0$. Unfortunately, for general boundary conditions, we do not know how to show such a result. Although one can likely increase the exponent $\frac{1}{400}$ in \Cref{heightapproximate} to some extent, a limitation in the proof of that theorem to directly improving the exponent to exceed $1$ appears in the effective variant of Rademacher's theorem, given by \Cref{twodimensionalapproximatelinear} below. The error there is of order $(\log N)^{-1 / 10}$, and it is not certain to us whether the bound still holds with error $(\log N)^{-1}$. 
	
	However, for certain families of boundaries, a more precise variant of the global limit shape theorem was implicitly established in \cite{LTGDMLS}. In particular, it was essentially shown there that if the normalized boundary height function $\mathfrak{h}$ ``approximately'' gives rise to a global profile exhibiting no frozen facets (that is, its gradient is everywhere in the interior of $\mathcal{T}$), then the limit shape result can be established with a smaller error estimate. Here, the term ``approximately'' refers to the existence of ``barrier functions'' $g_1, g_2: \partial \mathfrak{R} \rightarrow \mathbb{R}$ whose associated maximizers on $\mathfrak{R}$ exhibit no frozen facets, such that $g_1 \le \mathfrak{h} \le g_2$ and $|g_2 - g_1|$ is small. 
	
	To state this result, we first require certain families of (semi)norms on function spaces. To that end, for any integer $d > 0$ and subset $\mathfrak{R} \subseteq \mathbb{R}^d$, let $\mathcal{C} (\mathfrak{R})$ denote the set of real-valued continuous functions on $\mathfrak{R}$. Furthermore, for any $d$-tuple $(\gamma_1, \gamma_2, \ldots , \gamma_d) \in \mathbb{Z}_{\ge 0}^d$, define $|\gamma| = \sum_{i = 1}^d \gamma_i$ and $\partial_{\gamma} = \prod_{i = 1}^d (\partial_i)^{\gamma_i}$, where we have abbreviated $\partial_i = \frac{\partial}{\partial x_i}$ for each $i \in [1, d]$. 
	
	If $\mathfrak{R}$ is open, then for each $k \in \mathbb{Z}_{\ge 1}$ let $\mathcal{C}^k (\mathfrak{R})$ denote the set of $f \in \mathcal{C} (\mathfrak{R})$ such that $\partial_{\gamma} f \in \mathcal{C} (\mathfrak{R})$ for each $\gamma \in \mathbb{Z}_{\ge 0}^d$ with $|\gamma| \le k$. Further let $\mathcal{C}^k (\overline{\mathfrak{R}})$ denote the set of functions $f \in \mathcal{C}^k (\mathfrak{R})$ such that $\partial_{\gamma} f \in \mathcal{C} (\mathfrak{R})$ extends continuously to $\overline{\mathfrak{R}}$, for each $\gamma \in \mathbb{Z}_{\ge 0}^d$ with $|\gamma| \le k$. 
	
	For any $f \in \mathcal{C} (\mathfrak{R})$, $\alpha \in (0, 1]$, and $k \in \mathbb{Z}_{\ge 0}$, additionally define 
	\begin{flalign*}
	& \| f \|_0 = \| f \|_{0; \mathfrak{R}} = \displaystyle\sup_{z \in \mathfrak{R}} \big| f(z) \big|; \qquad \qquad \qquad \qquad [f]_{\alpha} = [f]_{\alpha; \mathfrak{R}} = \displaystyle\sup_{\substack{y, z \in \mathfrak{R} \\ y \ne z}} \displaystyle\frac{\big| f(y) - f(z) \big|}{|y - z|^{\alpha}} \\
	& [f]_k = [f]_{k, 0} = [f]_{k, 0; \mathfrak{R}} = \displaystyle\max_{\substack{\gamma \in \mathbb{Z}_{\ge 0}^d \\ |\gamma| = k}} \| \partial_{\gamma} f \|_0; \qquad \quad [f]_{k, \alpha} = [f]_{k, \alpha; \mathfrak{R}} = \displaystyle\max_{\substack{\gamma \in \mathbb{Z}_{\ge 0}^d \\ |\gamma| = k}} [ \partial_{\gamma} f ]_{\alpha}.
	\end{flalign*} 
	
	Define the H\"{o}lder space $\mathcal{C}^{k, \alpha} (\mathfrak{R})$ to be the set of $f \in \mathcal{C}^k (\mathfrak{R})$ such that $[f]_{k, \alpha; \mathfrak{R}_0} < \infty$, for any compact subset $\mathfrak{R}_0 \subset \mathfrak{R}$. Also define $\mathcal{C}^{k, \alpha} (\overline{\mathfrak{R}}) = \big\{ f \in \mathcal{C}^k (\mathfrak{R}): [f]_{k, \alpha; \mathfrak{R}} < \infty \big\}$. By restriction, we have that $\mathcal{C}^{k, \alpha} (\mathfrak{R}) \subseteq \mathcal{C}^{k, \alpha} (U)$ for any subset $U \subseteq \mathfrak{R}$. Let us define norms on these spaces by setting 
	\begin{flalign*}
	\| f \|_{\mathcal{C}^k (\overline{\mathfrak{R}})} = \| f \|_k = \displaystyle\sum_{j = 0}^k [f]_{j; 0; \mathfrak{R}}; \qquad \| f \|_{\mathcal{C}^{k, \alpha} (\overline{\mathfrak{R}})} = \| f \|_k + [f]_{k, \alpha}. 
	\end{flalign*}
	
	Now we can state the second effective limit shape theorem; its proof will be given in \Cref{EstimateHjk} below. As mentioned above, this result imposes the existence of barrier functions $g_1, g_2$ that bound the boundary height function of a (free) tiling from above and below, respectively, whose associated maximizers of $\mathcal{E}$ exhibit no frozen facets. It additionally assumes global $\mathcal{C}^2$ bounds on these maximizers of order slightly smaller than $\log N$, as well as $\mathcal{C}^{2, \alpha}$ bounds whose exponent $\alpha$ appears in the error estimate. 
	
	In what follows, we restrict to the case when $\mathfrak{R} = \mathcal{B} = \mathcal{B}_1 (0, 0)$ (recall the notation on disks from \eqref{brzdefinition}), as this will suffice for our purposes. We also recall the sets $\mathcal{T}_{\varepsilon}$ and $\mathfrak{G} (h)$ from \eqref{tset2} and \Cref{gh}, respectively.

	\begin{thm}
		
		\label{estimateboundaryheight}

		For fixed real numbers $\varepsilon \in \big( 0, \frac{1}{4} \big)$; $\alpha, \nu \in (0, 1)$; and $D > 0$, there exists a constant $C = C (\varepsilon, \alpha, \nu, D) > 1$ such that the following holds. Fix an integer $N \in \mathbb{Z}_{\ge 1}$; set $R = \mathcal{B}_N \cap \mathbb{T}$; and let $h: \partial R \rightarrow \mathbb{Z}$ denote a boundary height function on $\partial R$. 
		
		Suppose that there exist two functions $g_1, g_2: \partial \mathcal{B} \rightarrow \mathbb{R}$ admitting admissible extensions to $\mathcal{B}$ and satisfying the following three assumptions. In the below, $G_1, G_2 \in \Adm (\mathcal{B})$ denote the maximizers of $\mathcal{E}$ on $\mathcal{B}$ with boundary data $g_1, g_2$, respectively. 
		
		\begin{enumerate}
			
			\item For any $z \in \partial \mathcal{B}$ and $v \in \partial R$ with $|v - N^{-1} z| \le 4 N^{-1}$, we have $g_1 (z) \le N^{-1} h (v) \le g_2 (z)$.
			
			\item For each $z \in \mathcal{B}$ and $i \in \{ 1, 2 \}$, we have $\nabla G_i (z) \in \mathcal{T}_{\varepsilon}$.
			
			\item For each $i \in \{ 1, 2 \}$, we have $\| G_i - G_i (0, 0) \|_{\mathcal{C}^2 (\overline{\mathcal{B}})} \le (\log N)^{1 - \nu}$ and $\| G_i - G_i (0, 0) \|_{\mathcal{C}^{2, \alpha} (\overline{\mathcal{B}})} \le (\log N)^{10}$. 
			
		\end{enumerate} 
		
		\noindent If $H: \mathbb{V}(R) \rightarrow \mathbb{Z}$ denotes a uniformly random element of $\mathfrak{G} (h)$, then 
		\begin{flalign}
		\label{nhg1g2v}
		\begin{aligned}
		& \displaystyle\max_{v \in \mathbb{V}(R)} \mathbb{P} \big[ N^{-1} H (v) < G_1 (N^{-1} v) - 2 N^{-\alpha/5} \big] < C N^{-D}; \\
		& \displaystyle\max_{v \in \mathbb{V}(R)} \mathbb{P} \big[ N^{-1} H (v) > G_2 (N^{-1} v) + 2 N^{-\alpha/5} \big] < C N^{-D}. 
		\end{aligned} 
		\end{flalign}	
		
	\end{thm} 

	\begin{rem} 
		
	\label{g1g2lambda} 

	If $|g_1 - g_2| < \lambda$ then \Cref{h1h2gamma} and \Cref{estimateboundaryheight} together imply that \Cref{hnh} holds, with $\varpi = \lambda + 2 N^{-\alpha / 5}$. In particular, if $\lambda \ll (\log N)^{-1 - c}$, then $\varpi \ll (\log N)^{-1 - c}$. 
	
	\end{rem}

	The special case of \Cref{estimateboundaryheight} when $\| G_i \|_{\mathcal{C}^2 (\overline{\mathcal{B}})}$ is uniformly bounded in $N$ can be deduced from the proof of equation (44) of \cite{LTGDMLS}. However, since the setting in \cite{LTGDMLS} is a bit different from the one here, and since we also must address the situation when $\| G_i \|_{\mathcal{C}^2 (\overline{\mathcal{B}})}$ is not uniformly bounded, we will provide the proof of \Cref{estimateboundaryheight} in \Cref{GlobalEstimate2}, closely following the framework of \cite{LTGDMLS}. 
	
	Given \Cref{estimateboundaryheight}, it will be useful to understand under what conditions will a maximizer of $\mathcal{E}$ on a disk (which we will take to be $\mathcal{B}_{1 / 8}$, instead of $\mathcal{B}$, since this will be the setting for applications) satisfy the second and third assumptions listed there. Such an assumption is provided by the following proposition, whose proof will be given in \Cref{ProofGradientEstimateu} below. This condition essentially imposes for some small $\lambda > 0$ that the boundary data $\mathfrak{h}: \partial \mathcal{B}_{1 / 8} \rightarrow \mathbb{R}$ is the restriction to $\partial \mathcal{B}_{1 / 8}$ of some $\varphi \in \mathcal{C}^2 (\overline{\mathcal{B}}_{1 / 8})$ with $\| \varphi \|_{\mathcal{C}^2 (\overline{\mathcal{B}}_{1 / 8})} \ll \lambda^{-1}$, and that there exists a nearly linear boundary function $g$ on the larger domain $\partial \mathcal{B}$ whose associated maximizer is close to (within $\lambda$ of) $\mathfrak{h}$ on $\partial \mathcal{B}_{1 / 8}$. We refer to the left side of \Cref{ghfigure} for a depiction. 
	
	In the below, we recall the notation on nearly linear functions from \Cref{linearst}.

	\begin{prop}

		\label{euv1v2estimategradient} 
		
		For any fixed real numbers $\varepsilon, \theta \in \big( 0, \frac{1}{2} \big)$ and $\alpha \in \big( 0, \frac{\theta}{4} \big]$, there exists a constant $\delta = \delta (\varepsilon, \theta) \in (0, 1)$ such that the following holds. Let $\mathfrak{h}: \partial \mathcal{B}_{1 / 8} \rightarrow \mathbb{R}$ denote a function admitting an admissible extension to $\mathcal{B}_{1 / 8}$; let $\mathcal{H} \in \Adm (\mathcal{B}_{1 / 8}; \mathfrak{h})$ denote the maximizer of $\mathcal{E}$ on $\mathcal{B}_{1 / 8}$ with boundary data $\mathfrak{h}$; and let $\lambda \in (0, \delta)$ be a real number.
		
		 Suppose that there exists a function $g: \partial \mathcal{B} \rightarrow \mathbb{R}$ admitting an admissible extension to $\mathcal{B}$; a pair $(s, t) \in \mathcal{T}_{\varepsilon}$; and a function $\varphi \in \mathcal{C}^2 (\overline{\mathcal{B}}_{1 / 8})$ satisfying the following three assumptions. In the below, $G \in \Adm (\mathcal{B}; g)$ denotes the maximizer of $\mathcal{E}$ on $\mathcal{B}$ with boundary data $g$. 
		
		\begin{enumerate} 
			
			\item We have $\sup_{z \in \partial \mathcal{B}_{1 / 8}} \big| G (z) - \mathfrak{h} (z) \big| \le \lambda$. 
			
			\item The function $g$ is $\lambda^{2 \theta}$-nearly linear of slope $(s, t)$ on $\partial \mathcal{B}$. 
			
			\item We have $\| \varphi - \varphi (0, 0) \|_{\mathcal{C}^2 (\overline{\mathcal{B}}_{1 / 8})} \le \lambda^{2 \theta - 1}$ and $\varphi |_{\partial \mathcal{B}_{1 / 8}} = \mathfrak{h}$.
			
		\end{enumerate}
		
		\noindent Then, denoting $\mu = \lambda^{1 + \theta / 8}$, we have that
		\begin{flalign}
		\label{derivativehestimates122} 
		\begin{aligned}
		& \displaystyle\sup_{z \in \mathcal{B}_{1 / 8}} \big| \nabla \mathcal{H} (z) - (s, t) \big| \le \lambda^{15 \theta / 16}; \qquad \big\| \mathcal{H} - \mathcal{H} (0, 0) \big\|_{\mathcal{C}^{2, \alpha} (\overline{\mathcal{B}}_{1 / 8 - \mu})} \le \lambda^{\theta / 2 - 1}.
		\end{aligned}
		\end{flalign}
		
	\end{prop}

	\begin{figure}

		\begin{center}

			\begin{tikzpicture}[
			>=stealth,
			auto,
			style={
				scale = .4
			}
			]

			\filldraw[fill = white!75!gray] (0, 0) circle[radius = 3.6];
			\filldraw[fill = white!75!gray] (21.5, 1.5) circle[radius = 2.5];
			
			\draw[black, thick] (0, 0) circle[radius = 8];
			\draw[black, thick ] (0, 0) circle[radius = 3.6];

			\draw[] (0, 8) circle[radius = 0] node[above]{$\mathcal{B}$};
			\draw[] (0, 3.6) circle[radius = 0] node[above = -2, scale = .8]{$\mathcal{B}_{1 / 8}$};
			
			\draw[] (0, -8) circle[radius = 0] node[below, scale = .9]{$g$};
			\draw[] (0, -4.1) circle[radius = 0] node[scale = .8]{$\varphi |_{\partial \mathcal{B}_{1 / 8}} = \mathfrak{h} \approx G|_{\partial \mathcal{B}_{1 / 8}}$};
			\draw[] (5.125, 5.125) circle[radius = 0] node[below, scale = .9]{$G$};
			\draw[] (2.25, 2.25) circle[radius = 0] node[below, scale = .9]{$\mathcal{H}$};
			\draw[] (0, 0) circle[radius = 0] node[scale = .9]{$\nabla \mathcal{H} \approx (s, t)$};

			\draw[black, thick] (20, 0) circle[radius = 8];
			\draw[black, thick, dotted] (20, 0) circle[radius = 5.5];
			\draw[black] (21.5, 1.5) circle[radius = 2.5];
			
			\draw[] (20, 8) circle[radius = 0] node[above]{$\mathcal{B}$};
			\draw[] (20, 5.5) circle[radius = 0] node[above = -2, scale = .9]{$\mathcal{B}_{1 / 5}$};
			
			\filldraw[fill = black] (21.5, 1.5) circle[radius = .1] node[below = 1, scale = .75]{$z_0$};
			\filldraw[fill = black] (20.212, 2.25) circle[radius = .1] node[below = 1, scale = .75]{$z$};
			\draw[->, dashed, black] (21.5, 1.5) -- (23.3, 3.3);
			\draw[<->, black] (20.282, 2.18) -- (21.43, 1.57);

			\filldraw[fill = black] (22.9, 1.9) circle[radius = 0] node[scale = .75]{$\frac{8 \lambda}{\varsigma^{14}}$};

			\end{tikzpicture}
			
		\end{center}

		\caption{\label{ghfigure} Depicted to the left is the setting for \Cref{euv1v2estimategradient}. Depicted to the right is the setting for \Cref{omegadefinition} where, on the complement of $\Gamma_{z_0; z}$, we require that $\Phi$ be approximately linear on the shaded disk. }

	\end{figure}
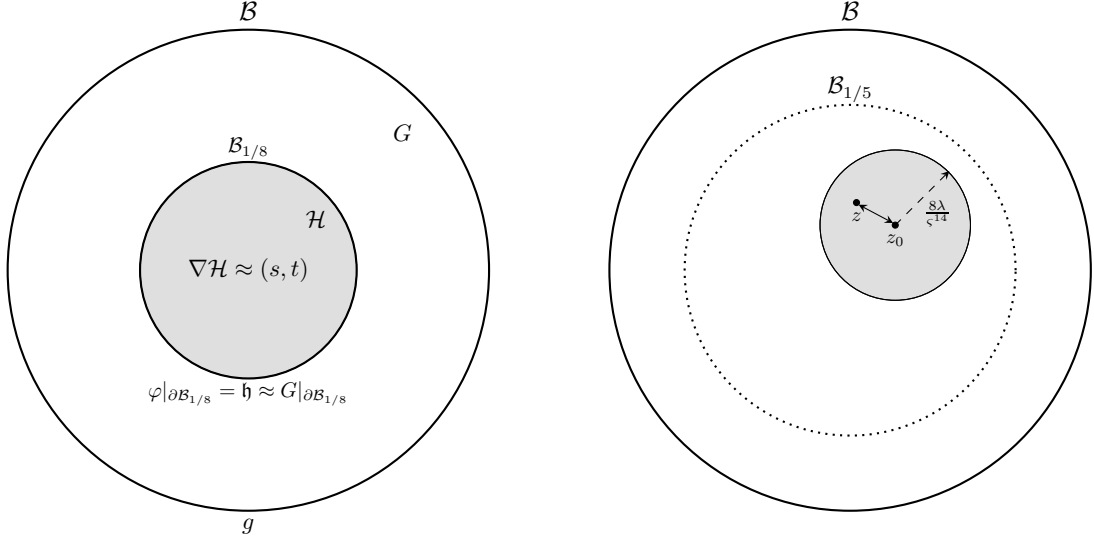

	Observe that the first bound in \eqref{derivativehestimates122} is a global gradient estimate,\footnote{The exponent $\frac{15 \theta}{16}$ there can be replaced by any constant less than $\theta$.} as it applies on all of the domain $\mathcal{B}_{1 / 8}$ of $\mathcal{H}$. Moreover, since $(s, t) \in \mathcal{T}_{\varepsilon}$, it implies for sufficiently small $\delta$ that $\nabla \mathcal{H} (z) \in \mathcal{T}_{\varepsilon / 2}$ for each $z \in \mathcal{B}_{1 / 8}$, and so $\mathcal{H}$ exhibits no frozen facets. The second bound in \eqref{derivativehestimates122} is an interior estimate, and is therefore slightly different from what is imposed in the third assumption of \Cref{estimateboundaryheight}. However, it is ``nearly'' a global estimate, since it applies quite close (of distance $\mu \ll \lambda$) to the boundary, and will therefore eventually suffice for our purposes. 
	
	Additionally, recall that \Cref{estimateboundaryheight} requires that the $\mathcal{C}^2$-norms of the $G_1$ and $G_2$ there are bounded by $(\log N)^{1 - c}$. Since the second estimate in \eqref{derivativehestimates122} bounds this norm by $\lambda^{\theta / 2 - 1}$, we may take $\lambda = (\log N)^{-1 - c}$ (for some small constant $c = c (\theta) > 0$) while still ensuring that $\| \mathcal{H} - \mathcal{H} (0, 0) \|_{\mathcal{C}^2 (\overline{\mathcal{B}}_{1 / 8 - \mu})} < (\log N)^{1 - c}$. This fact will later be useful for showing that the scale reduction estimate \Cref{estimaten2n} below remains valid when its error $\lambda$ is of order $(\log N)^{-1 - c}$.

	\subsection{An Improved Estimate Upon Scale Reduction} 
	
	\label{Estimaten9n}
	
	In this section we state the scale reduction estimate, which is given by \Cref{estimaten2n} and will be established in \Cref{Estimaten9nProof} below. However, before doing so, let us introduce notation for when \Cref{estimateboundaryheight} is applicable and yields a global law whose error is of order $\lambda$.
	
	\begin{definition} 
		
	\label{lambdabounded} 
	
	Suppose we are given an integer $N \in \mathbb{Z}_{\ge 1}$ and four real numbers $\varepsilon \in \big( 0, \frac{1}{4} \big)$, $\rho \in \mathbb{R}_{> 0}$, and $\varsigma, \alpha \in (0, 1)$. Define the simply-connected induced subgraph $R = \mathcal{B}_{\rho N} \cap \mathbb{T} \subset \mathbb{T}$, and further suppose that a boundary height function $h: \partial R \rightarrow \mathbb{Z}$ and a function $g: \partial \mathcal{B}_{\rho} \rightarrow \mathbb{R}$ admitting an admissible extension to $\mathcal{B}_{\rho}$ are also given. For any real number $\lambda \in (0, \rho)$, we say that the parameters $(N; \varepsilon; \rho; \varsigma; \alpha; h; g)$ are \emph{$\lambda$-confined} if they satisfy the following three properties. In the below, $G \in \Adm (\mathcal{B}_{\rho}; g)$ denotes the maximizer of $\mathcal{E}$ on $\mathcal{B}_{\rho}$ with boundary data $g$. 
	
	\begin{enumerate}

		\item For each point $z \in \partial \mathcal{B}_{\rho}$ and vertex $v \in \partial R$ such that $|z - N^{-1} v| \le 4 N^{-1}$, we have that $\big| g (z) - N^{-1} h(v) \big| \le \lambda$.
		
		\item For each $z \in \mathcal{B}_{\rho}$, we have that $\nabla G (z) \in \mathcal{T}_{\varepsilon / 2}$. 
		
		\item We have that $\big\| G - G (0, 0) \big\|_{\mathcal{C}^{2, \alpha} (\overline{\mathcal{B}}_{\rho - \lambda})} < \varsigma \log N$.

	\end{enumerate} 
	
	\end{definition}

	Observe in particular that, if $\varsigma \le (\log N)^{- \nu}$ for some $\nu > 0$, then \Cref{estimateboundaryheight} applies on the slightly smaller domain $\mathcal{B}_{\rho - \lambda}$, when the parameters $(N; \varepsilon; \rho; \varsigma; \alpha; h; g)$ are $\lambda$-confined. By the $1$-Lipschitz property of height functions and \Cref{g1g2lambda}, this can be seen to imply that a uniformly random height function $H: \mathbb{V}(R) \rightarrow \mathbb{Z}$ of $\mathfrak{G} (h)$ (recall \Cref{gh}) likely approximates the maximizer $\mathcal{H}: \mathcal{B}_{\rho} \rightarrow \mathbb{R}$ of $\mathcal{E}$ on $\mathcal{B}_{\rho}$ with boundary data $N^{-1} h$ (within an error of $3 \lambda + 2 N^{-\alpha/5}$). 
	
	Thus, the scale reduction estimate, which we recall should provide bounds for $H$ on $\mathcal{B}_{N / 8}$ given analogous ones on $\mathcal{B}_N$, will essentially state the following. If $(N; \varepsilon; 1; \varsigma; \alpha; h; g)$ are $\lambda$-confined, then with high probability there exists $f: \partial \mathcal{B}_{1 / 8} \rightarrow \mathbb{R}$, which is in a sense ``close to'' the restriction of the maximizer $G \in \Adm (\mathcal{B}; g)$ of $\mathcal{E}$ to $\partial \mathcal{B}_{1 / 8}$, such that $\big( N; \varepsilon; \frac{1}{8}; \kappa; \alpha; H|_{\partial (\mathcal{B}_{N / 8} \cap \mathbb{T})}; f \big)$ are $\mu$-confined, for some $\mu \le \frac{\lambda}{8}$ and $\kappa \le 8^{-1 - \alpha} \varsigma$; these altered values of $\mu$ and $\vartheta$ are to account for scaling by $\frac{1}{8}$. 
	
	The following proposition provides such a statement precisely, under the additional assumption that $g$ is nearly linear (recall \Cref{linearst}). Its proof will be given in \Cref{Estimaten9nProof}.

	\begin{prop}
		
		\label{estimaten2n}
		
		For any fixed real numbers $\varepsilon \in \big( 0, \frac{1}{4} \big)$ and $D > 0$, there exist constants $C_1 = C_1 (\varepsilon) > 1$ and $C_2 = C_2 (\varepsilon, D) > 1$ such that the following holds. Let $N$ be a positive integer; define the domain $R = \mathcal{B}_N \cap \mathbb{T}$; let $h: \partial R \rightarrow \mathbb{Z}$ denote a boundary height function on $R$; and let $H: \mathbb{V}(R) \rightarrow \mathbb{Z}$ denote a uniformly random element of $\mathfrak{G} (h)$. Further fix parameters 
		\begin{flalign*}
		\alpha = \frac{1}{17500}; \qquad \varsigma = (\log N)^{-\alpha}; \qquad \lambda \in \big[  (\log N)^{-1 - \alpha}, \varsigma^{15} \big].
		\end{flalign*}
		
		\noindent Suppose that there exist a pair $(s, t) \in \mathcal{T}_{\varepsilon}$ and a function $g: \partial \mathcal{B} \rightarrow \mathbb{R}$ admitting an admissible extension to $\mathcal{B}$ that satisfy the following two assumptions. 
		
		\begin{enumerate} 
			
			\item The function $g$ is $\varsigma^{10}$-nearly linear with slope $(s, t)$ on $\partial \mathcal{B}$. 
			
			\item The parameters $(N; \varepsilon; 1; \varsigma; \alpha; h; g)$ are $\lambda$-confined. 
			
		\end{enumerate}  
		
		\noindent Let $G \in \Adm (\mathcal{B}; g)$ denote the maximizer of $\mathcal{E}$ on $\mathcal{B}$ with boundary data $g$. Then, there exists an event $\Gamma$ with $\mathbb{P} [\Gamma] \le C_2 N^{-D}$ such that 
		\begin{flalign} 
		\label{g1nvlambda} 
		\displaystyle\max_{u \in \mathbb{V}(R)} \big| G (N^{-1} u) -  N^{-1} H(u) \big| \textbf{\emph{1}}_{\Gamma^c} \le 6 \lambda.
		\end{flalign} 
		
		\noindent Moreover, on $\Gamma^c$, there exists a random (dependent on $H$) function $f: \partial \mathcal{B}_{1 / 8} \rightarrow \mathbb{R}$ admitting an admissible extension to $\mathcal{B}_{1 / 8}$ and satisfying the following two properties. Here, we set $\mu = \varsigma^{\alpha} \lambda = (\log N)^{-\alpha^2} \lambda$ and let $F \in \Adm (\mathcal{B}_{1 / 8}; f)$ denote the maximizer of $\mathcal{E}$ on $\mathcal{B}_{1 / 8}$ with boundary data $f$. 
		
		\begin{enumerate} 
			
			\item We have the bounds 
			\begin{flalign*} 
			\displaystyle\sup_{z \in \mathcal{B}_{1 / 8}} \big| F (z) - G (z) \big| \le 7 \lambda; \qquad \displaystyle\sup_{z \in \mathcal{B}_{1 / 32}} \big| \nabla F (z) - \nabla G (z)  \big| < C_1 \lambda.
			\end{flalign*}

			\item The parameters $\big( N; \varepsilon; \frac{1}{8}; \frac{\varsigma}{64}; \alpha; H |_{\partial (\mathcal{B}_{N / 8} \cap \mathbb{T})}; f \big)$ are $\mu$-confined.

		\end{enumerate}	
	\end{prop}
	
	We refer to the left side of \Cref{bm} for a depiction, where the $M$ there is $N$ here. 

	\begin{rem} 
	
	One would like to inductively use \Cref{estimaten2n} by applying it; then scaling by $\frac{1}{8}$ (that is, replacing $\frac{N}{8}$ with $N$); and repeating. Assuming that such a procedure can be implemented (as will be shown to be the case in \Cref{EstimateHLocalProof}), observe that, since $\mu = (\log N)^{-\alpha^2} \lambda \ll \lambda$, after $\alpha^{-2}$ applications of \Cref{estimaten2n}, $\lambda$ would decrease from an initial value of $(\log N)^{-15 \alpha} \gg (\log N)^{-1 / 400}$ to its minimum possible value of $(\log N)^{-1 - \alpha}$. Then, \eqref{g1nvlambda} would imply that $H$ satisfies a global law with error $5 (\log N)^{-1 - \alpha}$, which (as mentioned after \Cref{heightapproximate}) will be useful for establishing the local law \Cref{heightlocal1}. Thus, \Cref{estimaten2n} also serves as a way of decreasing the error in the global law from $(\log N)^{-c}$ to below $(\log N)^{-1}$. We will explain this in more detail in \Cref{EstimateHLocalProof}.
	
	\end{rem}
	
	Now, let us define the event $\Gamma$ from \Cref{estimaten2n}; it will essentially be the one on which either \eqref{g1nvlambda} does not hold or $H$ is not locally approximately linear. In what follows, we extend $H$ to $\overline{\mathcal{B}}_N$ by linearity on $\mathbb{F}(R)$ and arbitrarily on $\overline{\mathcal{B}}_N \setminus \mathbb{F}(R)$, in such a way that it is $1$-Lipschitz on $\overline{\mathcal{B}}_N$. 
	
	\begin{definition}
		
		\label{omegadefinition} 
		
		Adopt the notation and assumptions of \Cref{estimaten2n}. Define $\Phi: \overline{\mathcal{B}} \rightarrow \mathbb{R}$ by setting $\Phi (z) = N^{-1} H (Nz)$, for each $z \in \overline{\mathcal{B}}$. Then define the event
		\begin{flalign}
		\label{fg1g2lambda} 
		\Gamma_0 = \bigg\{ \displaystyle\sup_{z \in \overline{\mathcal{B}}} \big| \Phi(z) - G(z) \big| > 6 \lambda \bigg\}.
		\end{flalign}
		
		\noindent Moreover, for any $z_0, z \in \mathcal{B}$ such that $\mathcal{B}_{8 \lambda / \varsigma^{14}} (z_0) \subset \mathcal{B}_{1 / 5}$ and $|z - z_0| \le \varsigma \lambda$, define the event
		\begin{flalign*}
		\Gamma_{z_0; z} = \Big\{ \big| \Phi(z) - \Phi(z_0) - (z - z_0) \cdot \nabla G (z_0) \big| > \varsigma^{14} \lambda \Big\}.
		\end{flalign*}
		
			\noindent Set $\Gamma = \Gamma_0 \cup \bigcup \Gamma_{z_0; z}$, where the union is over all $z_0, z \in \mathcal{B}$ with $\mathcal{B}_{8 \lambda / \varsigma^{14}} (z_0) \subset \mathcal{B}_{1 / 5}$ and $z \in \mathcal{B}_{\varsigma \lambda} (z_0)$.
		
	\end{definition}

	We refer to the right side of \Cref{ghfigure} for a depiction. Observe that \eqref{g1nvlambda} holds on $\Gamma^c$, due to \eqref{fg1g2lambda}; we will bound $\mathbb{P} [\Gamma]$ in \Cref{flinearlambda} below. Next, let us describe the function $f$ from \Cref{estimaten2n}; it will be determined by first convolving $\Phi$ with a suitable smooth function $\psi$, supported on scale $\varsigma \lambda$, and then restricting to $\mathcal{B}_{1/8}$ (we will later see that these choices make $f$ of distance $\varsigma \lambda \ll \lambda$ from $\Phi$, so it is closer to $\Phi$ than $g$ is assumed to be, while still having some regularity). This is made precise through the following definition.

	\begin{definition}
		
		\label{psif1f2} 
		
		Adopt the notation and assumptions of \Cref{estimaten2n}. Set $\vartheta = \varsigma \lambda$, and let $\psi: \mathbb{R}^2 \rightarrow \mathbb{R}_{\ge 0}$ denote a smooth, nonnegative function supported on $\overline{\mathcal{B}}_{\vartheta}$, such that 
		\begin{flalign}
		\label{psiestimates} 
		\displaystyle\int_{\mathbb{R}^2} \psi (w) dw = 1; \qquad |\psi|_{\mathcal{C}^0 (\mathbb{R}^2)} \le 10 \vartheta^{-2}; \qquad |\psi|_{\mathcal{C}^1 (\mathbb{R}^2)} \le 20 \vartheta^{-3}; \qquad |\psi|_{\mathcal{C}^2 (\mathbb{R}^2)} \le 60 \vartheta^{-4}.
		\end{flalign}
		
		\noindent Further define the function $\varphi: \mathcal{B}_{1 / 6} \rightarrow \mathbb{R}$ by setting 
		\begin{flalign}
		\label{functions12z}
		& \varphi (z) = \displaystyle\int_{\mathbb{R}^2} \Phi (z - w) \psi (w) dw, \qquad \text{for each $z \in \mathcal{B}_{1 / 6}$.}
		\end{flalign}
		
		\noindent  Denote $f = \varphi |_{\partial \mathcal{B}_{1 / 8}}$, and let $F \in \Adm (\mathcal{B}_{1 / 8}; f)$ denote the maximizer of $\mathcal{E}$ on $\mathcal{B}_{1 / 8}$ with boundary data $f$. 
		
	\end{definition}
	
	We will verify that $f$ satisfies the two properties stated in \Cref{estimaten2n} in \Cref{Estimaten9nProof}. We conclude this section with the following lemma that estimates $\mathbb{P} [\Gamma]$. The bound on $\mathbb{P} [\Gamma_0]$ will follow from \Cref{estimateboundaryheight}, and the bound on each $\mathbb{P} [\Gamma_{z_0; z}]$ will follow from first applying of \Cref{heightapproximate} on disks of radius $8 \varsigma^{-14} \lambda N \ll N$ centered at $N z_0$ and then using a Taylor expansion.

	\begin{lem} 
		
		\label{flinearlambda} 
		
		For any $D > 0$, there exists a constant $C = C (\varepsilon, D) > 1$ such that $\mathbb{P} [\Gamma] < C N^{-D}$. 
		
	\end{lem} 
	
	\begin{proof}
		
		By a union bound and the fact that $\Phi$ is $1$-Lipschitz, it suffices to show the existence of a constant $C = C (\varepsilon, D) > 1$ such that 
		\begin{flalign}
		\label{omega0omegaz0z}
		\mathbb{P} [\Gamma_0] \le C N^{-D}; \qquad \mathbb{P} \Big[ \big| \Phi(z) - \Phi(z_0) - (z - z_0) \cdot \nabla G (z_0) \big|  \textbf{1}_{\Gamma_0^c} > C \varsigma^{15} \lambda \Big] \le C e^{-\lambda N},
		\end{flalign}
		
		\noindent for any fixed $z_0, z \in \mathcal{B}$ such that $\mathcal{B}_{8 \lambda / \varsigma^{14}} (z_0) \subset \mathcal{B}_{1 / 5}$ and $z \in \mathcal{B}_{\varsigma \lambda} (z_0)$. Here, we have implicitly used the fact that it suffices to take a union bound of \eqref{omega0omegaz0z} over $(z_0, z)$ among vertices in the mesh $N^{-10} \cdot (\mathcal{B}_{1/5} \cap \mathbb{T})$ satisfying the above conditions. Indeed, for any pair $(z_0, z)$ satisfying $\mathcal{B}_{8\lambda / \varsigma^{14}} (z_0) \subset \mathcal{B}_{1/5}$ and $z \in \mathcal{B}_{\varsigma \lambda} (z_0)$ that lies outside of this mesh, let $(z_0', z')$ denote the pair inside of this mesh closest to it. Then, assuming \eqref{omega0omegaz0z} for $(z_0', z')$, we have  
		\begin{flalign*}
			\big| \Phi(z) - \Phi(z_0) - (z - z_0) \cdot \nabla G (z_0) \big| & \le \big| \Phi (z') - \Phi (z_0') - (z' - z_0') \cdot \nabla G(z_0') \big| + \big| \Phi (z) - \Phi (z') \big| \\
			& \qquad + \big| \Phi (z_0) - \Phi (z_0') \big| + \big( |z - z'| + |z_0 - z_0'| \big) \big| \nabla G(z_0') \big| \\
			& \qquad + |z - z_0| \cdot \big| \nabla G (z_0) - \nabla G(z_0') \big| \\
			& \le C \varsigma^{15} \lambda + 4 N^{-10} + |z_0 - z_0'| \cdot \| G - G(0, 0) \|_{\mathcal{C}^2 (\mathcal{B}_{1/5})} \\
			& \le C \varsigma^{15} \lambda + 4 N^{-10} + N^{-10} \cdot \varsigma \log N \le \varsigma^{14} \lambda,
		\end{flalign*}  
	
		\noindent where the second bound holds since $\Phi$ is $1$-Lipschitz and $|z - z_0| \le \varsigma \lambda \le 1$, and the third holds from the third part of \Cref{lambdabounded}.
		
		To establish the first bound in \eqref{omega0omegaz0z}, we apply the $1$-Lipschitz properties of $H$ and $G$; \Cref{estimateboundaryheight} on the domain $\mathcal{B}_{(1 - \lambda) N} \cap \mathbb{T}$, where the functions $(g_1, g_2)$ there are $\big(G |_{\partial \mathcal{B}_{1 - \lambda}} - 3 \lambda, G |_{\partial \mathcal{B}_{1 - \lambda}} + 3 \lambda \big)$ here (and whose assumptions are satisfied by the $\lambda$-confinement of $(N; \varepsilon; 1; \varsigma; \alpha; h; g)$ and the Lipschitz properties of $F$ and $G$); and a union bound to deduce that
		\begin{flalign*}
		& \mathbb{P} \Bigg[ \bigcup_{z \in \mathcal{B}_{1 - \lambda}} \big\{ \Phi(z) < G (z) - 3 \lambda - 2 N^{-\alpha/5} \big\} \cup \bigcup_{z \in \mathcal{B}_{1 - \lambda}} \big\{ \Phi(z) > G (z) + 3 \lambda + 2 N^{-\alpha/5} \big\} \Bigg] < C_1 N^{-D},
		\end{flalign*}
		
		\noindent for some constant $C_1 = C_1 (\varepsilon, D) > 1$. This estimate; the $1$-Lipschitz properties of $\Phi$ and $G$ on $\mathcal{B}_1$; the fact that $2 N^{-\alpha/5} < \lambda$ for sufficiently large $N$; and \eqref{fg1g2lambda} together imply the first bound of \eqref{omega0omegaz0z}. 
		
		We now establish the second bound in \eqref{omega0omegaz0z}. Fix $z_0, z \in \mathcal{B}$ with $\mathcal{B}_{8 \lambda / \varsigma^{14}} (z_0) \subset \mathcal{B}_{1 / 5}$ and $z \in \mathcal{B}_{\varsigma \lambda} (z_0)$; for notational convenience, let us assume that $z_0 = (0, 0)$. Further denote $r = \varsigma^{-14} \lambda$; abbreviate $\mathfrak{U} = \mathcal{B}_{8r} = \mathcal{B}_{8r} (z_0)$; and condition on $\Phi |_{\mathcal{B} \setminus \mathfrak{U}}$ (or, equivalently, on $H |_{\mathcal{B}_N \setminus N \mathfrak{U}}$). Then associated with $H |_{N \mathfrak{U}}$ is a free tiling of $N \mathfrak{U} \cap \mathbb{T}$; denoting $\mathscr{A} = N \mathfrak{U} \cap \mathbb{T}$, the law of $H |_{\mathscr{A}}$ is then given by that of a uniformly random element of $\mathfrak{G} ( H|_{\partial \mathscr{A}})$ by the Gibbs property (recall \Cref{InfiniteMeasures}). 
		
		Next, we will rescale by $r^{-1}$ and use the global law \Cref{heightapproximate} to approximate $(rN)^{-1} H |_{\mathscr{A}}$ by a maximizer of $\mathcal{E}$ on $\mathcal{B}_8 = r^{-1} \mathfrak{U}$. Since $r \gg \varsigma \lambda$, a Taylor expansion will then show that $\Phi$ is likely almost linear (with error of order $r^{-1} \varsigma \lambda = \varsigma^{15}$) on $\mathcal{B}_{\varsigma \lambda}$.
		
		To implement this, define the function $\chi: \partial \mathcal{B}_8 \rightarrow \mathbb{R}$ by setting $\chi (z) = (rN)^{-1} H (rNz) = r^{-1} \Phi (rz)$, for each $z \in \partial \mathcal{B}_8$. Then let $\Xi \in \Adm (\mathcal{B}_8; \chi)$ denote the maximizer of $\mathcal{E}$ on $\mathcal{B}_8$ with boundary data $\chi$. In order to apply \Cref{heightapproximate} on $\mathscr{A}$ (with the $\varepsilon$, $\mathfrak{h}$, and $\mathcal{H}$ there equal to $\frac{\varepsilon}{4}$, $\chi$, and $\Xi$ here, respectively), we must first verify that $\nabla \Xi (z) \in \partial \mathcal{T}_{\varepsilon / 4}$ for $|z| < \frac{\varepsilon}{4}$, as imposed in \Cref{regularestimate}. 
		
		To that end, define $G^{(r)}: \mathcal{B}_8 \rightarrow \mathbb{R}$ by setting $G^{(r)} (z) = r^{-1} G (rz)$ for each $z \in \mathcal{B}_8$, and observe (by \Cref{estimatehrho}) that $G^{(r)}$ is a maximizer of $\mathcal{E}$ on $\mathcal{B}_8$. We will first bound $\big| \Xi - G^{(r)} \big|$ on $\mathcal{B}_8$ and then use \Cref{perturbationboundary} to compare $\nabla \Xi$ to $\nabla G$. Since $\nabla G^{(r)} (z) = \nabla G (rz) \in \mathcal{T}_{\varepsilon / 2}$, this will show $\nabla \Xi (z) \in \mathcal{T}_{\varepsilon / 4}$ for $z \in \mathcal{B}$. 
		
		To do this, observe that \eqref{fg1g2lambda} yields $\sup_{z \in \mathcal{B}} \big| \Phi(z) - G (z) \big| \textbf{1}_{\Gamma_0^c} \le 6 \lambda$, and therefore 
		\begin{flalign}
		\label{xizgrzestimateb}
		\sup_{z \in \mathcal{B}_8} \big| \Xi (z) - G^{(r)} (z) \big| \textbf{1}_{\Gamma_0^c} = \sup_{z \in \partial \mathcal{B}_8} \big| \chi (z) - G^{(r)} (z) \big| \textbf{1}_{\Gamma_0^c} \le 6 r^{-1} \lambda = 6 \varsigma^{14},
		\end{flalign}
		
		\noindent where to derive the first equality we used \Cref{h1h2gamma}. 
		
		In order to apply \Cref{perturbationboundary}, we must next show that $\Xi$ and $G^{(r)}$ are both nearly linear. To that end, first observe that since $g$ is $\varsigma^{10}$-nearly linear of slope $(s, t)$ on $\mathcal{B}$, there exists a linear function $\Lambda: \overline{\mathcal{B}} \rightarrow \mathbb{R}$ of slope $(s, t)$ on $\mathcal{B}$ such that $\sup_{z \in \mathcal{B}} \big| g(z) - \Lambda (z) \big| < \varsigma^{10}$. Then, since $(s, t) \in \mathcal{T}_{\varepsilon}$ and $\mathfrak{U} \subset \mathcal{B}_{1 / 4}$, we deduce from \Cref{perturbationboundary} (applied with the $(\mathfrak{h}_1, \mathfrak{h}_2)$ there equal to the $(g, \Lambda |_{\partial \mathcal{B}})$ here) that there exists a constant $C_2 = C_2 (\varepsilon) > 1$ such that $\sup_{z \in \mathfrak{U}} \big| \nabla G (z) - (s, t) \big| < C_2 \varsigma^{10} < C_2 \varsigma$ and $\| G - G (0, 0) \|_{\mathcal{C}^2 (\overline{\mathfrak{U}})} \le C_2$. Next, recall the constants $\delta = \delta \big( \frac{\varepsilon}{4} \big) > 0$ and $C_3 = C_3 \big(\frac{\varepsilon}{4} \big) > 1$ from \Cref{perturbationboundary}, and assume that $N$ is sufficiently large so that $64 C_2 r + 6 C_3 \varsigma^{14} < \min \big\{ \frac{\delta}{2}, \frac{\varepsilon}{4} \big\}$. 
		
		Now denote $(s_0, t_0) = \nabla G (z_0) = \nabla G (0, 0) \in \mathcal{T}_{\varepsilon / 2}$. From a Taylor expansion and the bound $\| G \|_{\mathcal{C}^2 (\overline{\mathfrak{U}})} \le C_2$, we obtain
		\begin{flalign*}
		\displaystyle\sup_{z \in \mathcal{B}_8} \big| G^{(r)} (z) - G^{(r)} (0, 0) - z \cdot (s_0, t_0) \big| & = r^{-1} \displaystyle\sup_{z \in \mathfrak{U}} \big| G (z) - G (0, 0) - z \cdot (s_0, t_0) \big| \\
		& \le C_2 r^{-1} \displaystyle\sup_{z \in \mathfrak{U}} |z - z_0|^2 = 64 C_2 r,
		\end{flalign*}
		
		\noindent where we have used the fact that $\nabla G^{(r)} (0, 0) = \nabla G (0, 0) = (s_0, t_0)$. It follows that $G^{(r)}$ is $\frac{\delta}{2}$-nearly linear of slope $(s_0, t_0)$ on $\mathcal{B}_8$, which together with \eqref{xizgrzestimateb} implies that $\Xi$ is $\delta$-nearly linear with slope $(s_0, t_0)$ on $\mathcal{B}_8$. Therefore, \Cref{perturbationboundary} (with \Cref{estimatehrho}) applies and by \eqref{xizgrzestimateb} yields
		\begin{flalign}
		\label{xigirgradient}
		 \displaystyle\sup_{z \in \mathcal{B}} \big| \nabla G^{(r)} (z) - \nabla \Xi (z) \big| \textbf{1}_{\Gamma_0^c} \le C_3 \varsigma^{14}; \qquad  \| \Xi - \Xi (0, 0) \|_{\mathcal{C}^2 (\overline{\mathcal{B}})} \textbf{1}_{\Gamma_0^c}  \le C_3.
		\end{flalign}
		
		Since $\nabla G^{(r)} (z) \in \mathcal{T}_{\varepsilon / 2}$ for each $z \in \mathcal{B}_2$ (due to the $\lambda$-confinement of $(N; \varepsilon; 1; \varsigma; \alpha; h; g)$), it follows from the first bound in \eqref{xigirgradient} that $\nabla \Xi (z) \in \mathcal{T}_{\varepsilon / 4}$ for each $z \in \mathcal{B}_{\varepsilon / 4} \subset \mathcal{B}$ and $N$ sufficiently large. This implies that \Cref{regularestimate} is satisfied, where the $N$ there is $\lfloor rN \rfloor$ here; the $\varepsilon$ there is $\frac{\varepsilon}{4}$ here; the $R$ there is $\mathscr{A}$ here; the $h$ there is $H |_{\partial \mathscr{A}}$ here; the $\mathcal{H}$ there is $\Xi$ here; and the $v_0$ there is $N z_0 = (0, 0)$ here. Thus \Cref{heightapproximate}, the fact that $\varsigma^{35} = (\log N)^{-1 / 500} > \big( \log \lfloor rN \rfloor \big)^{-1 / 400}$, and the fact that $H$ and $\Xi$ are $1$-Lipschitz together yield a constant $C_4 = C_4 (\varepsilon) > 1$ such that 
		\begin{flalign}
		\label{hnzhnz0xirestimate} 
		\mathbb{P} \Bigg[ \displaystyle\sup_{z \in \mathcal{B}_{\varepsilon r / 4}} \Big| r^{-1} N^{-1} \big( H (Nz) - H (0, 0) \big) - \big( \Xi (r^{-1} z) - \Xi (0, 0) \big) \Big| \textbf{1}_{\Gamma_0^c}  > \varsigma^{35} \Bigg] < C_4 e^{- rN}. 
		\end{flalign}
		
		\noindent Moreover, a Taylor expansion and the second bound in \eqref{xigirgradient} together imply 
		\begin{flalign}
		\label{xirestimate}
		\big| \Xi ( r^{-1} z) - \Xi (0, 0) - r^{-1} z \cdot \nabla \Xi (0, 0) \big| \textbf{1}_{\Gamma_0^c} \le C_3 r^{-2} |z|^2 \le C_3 \varsigma^{30},
		\end{flalign}
		
		\noindent whenever $|z - z_0| = |z| \le \varsigma \lambda = \varsigma^{15} r$. Combining the first statement of \eqref{xigirgradient}, \eqref{hnzhnz0xirestimate}, \eqref{xirestimate}, and the facts that $r = \varsigma^{-14} \lambda$ and $\nabla G^{(r)} (0, 0) = \nabla G (0, 0)$ imply the existence of a constant $C_5 = C_5 (\varepsilon) > 1$ such that 
		\begin{flalign}
		\label{nhnznz0g1}
		\mathbb{P} \Bigg[ \displaystyle\sup_{z \in \mathcal{B}_{\varsigma \lambda}} \Big| N^{-1} \big( H (N z) - H (0, 0) \big) - z \cdot \nabla G (0, 0) \Big| \textbf{1}_{\Gamma_0^c} > C_5 \varsigma^{29} r \Bigg] < C_4 e^{- rN},
		\end{flalign}
		
		\noindent whenever $|z| \le \varsigma \lambda = \varsigma^{15} r$. So, we deduce the second bound in \eqref{omega0omegaz0z} from \eqref{nhnznz0g1} and the facts that $\Phi (z) = N^{-1} H (Nz)$ and $r = \varsigma^{-14} \lambda$.
	\end{proof}

	\subsection{Proof of \Cref{estimaten2n}}
	
	\label{Estimaten9nProof}
	
	In this section we establish \Cref{estimaten2n}; throughout this section, we adopt the notation of \Cref{Estimaten9n}. It suffices to show that $f$ satisfies the two properties listed in \Cref{estimaten2n} on the event $\Gamma^c$. To that end, we begin with the following lemma, whose second statement verifies that it satisfies the first one there.

	\begin{lem}
		
		\label{gradientfigiomega} 
		
		There exists a constant $C = C(\varepsilon) > 1$ such that the following two statements hold for sufficiently large $N$. First, we have that
		\begin{flalign}
		\label{gigradientst}
		\displaystyle\sup_{z \in \mathcal{B}_{1 / 4}} \big| \nabla G (z) - (s, t) \big| < C \varsigma. 
		\end{flalign}
		
		\noindent Second, we have that
		\begin{flalign}
		\label{figigradientst}
		\displaystyle\sup_{z \in \mathcal{B}_{1 / 8}} \big| G (z) - F (z) \big| \textbf{\emph{1}}_{\Gamma^c} \le 7 \lambda; \qquad \displaystyle\sup_{z \in \mathcal{B}_{1 / 32}} \big| \nabla F (z) - \nabla G (z) \big| \textbf{\emph{1}}_{\Gamma^c}  < C \lambda.
		\end{flalign}
		
	\end{lem} 
	
	\begin{proof}
		
		To establish \eqref{gigradientst}, recall the constant $\delta = \delta (\varepsilon)$ from \Cref{perturbationboundary} and assume that $N$ is sufficiently large so that $7 \lambda + \varsigma < \delta$. The first property imposed on $g$ in \Cref{estimaten2n} yields a linear function $\Lambda: \overline{\mathcal{B}} \rightarrow \mathbb{R}$ of slope $(s, t)$ such that $\sup_{z \in \partial \mathcal{B}} \big| g (z) - \Lambda (z) \big| < \varsigma^{10} < \varsigma < \delta$. Then applying \Cref{perturbationboundary} (and in particular the first bound in \eqref{perturbationboundaryestimate}), with the $(\mathfrak{h}_1, \mathfrak{h}_2)$ there equal to $(g, \Lambda |_{\partial \mathcal{B}})$ here, we obtain \eqref{gigradientst}.
		
		Next, to establish \eqref{figigradientst}, we restrict to $\Gamma^c$ and first bound $\big| \varphi (z) - G (z) \big| \textbf{1}_{\Gamma^c} $ for $z \in \partial \mathcal{B}_{1 / 8}$. The definition \eqref{functions12z} of $\varphi$ and the facts that $H$ is $1$-Lipschitz; that $\int_{\mathbb{R}^2} \psi (w) dw = 1$; and that $\psi$ is nonnegative and supported on $\overline{\mathcal{B}}_{\vartheta}$ together imply that $\sup_{z \in \partial \mathcal{B}_{1 / 8}} \big| \varphi (z) - \Phi(z) \big| \le \vartheta < \lambda$. Moreover, \eqref{fg1g2lambda} yields $\sup_{z \in \mathcal{B}} \big| \Phi (z) - G (z) \big| \textbf{1}_{\Gamma^c} \le 6 \lambda$, and so $\sup_{z \in \partial \mathcal{B}_{1 / 8}} \big| \varphi (z) - G (z) \big| \textbf{1}_{\Gamma^c} \le 7 \lambda$. Hence, the first estimate in \eqref{figigradientst} follows from \Cref{h1h2gamma} and the fact that $f = \varphi |_{\partial \mathcal{B}_{1 / 8}}$.
		
		Now let us establish the second statement of \eqref{figigradientst}. To do this first observe that, by applying \Cref{h1h2gamma} with the $(\mathcal{H}_1, \mathcal{H}_2)$ there equal to $(G, \Lambda)$ here, we obtain that $G$ is $\varsigma^{10}$-nearly linear on $\mathcal{B}$. This and the bounds $\sup_{z \in \partial \mathcal{B}_{1 / 8}} \big| F (z) - G (z) \big| \textbf{1}_{\Gamma^c} \le 7 \lambda$ and $7 \lambda + \varsigma < \delta$ together imply that $F$ is $\delta$-nearly linear on $\mathcal{B}_{1 / 8}$. Therefore, we may apply \Cref{perturbationboundary} (and \Cref{estimatehrho}) to the pair $(F, G)$, and so the first bound in \eqref{perturbationboundaryestimate} and the first bound in \eqref{figigradientst} together yield 	  
		\begin{flalign*}
		\displaystyle\sup_{z \in \mathcal{B}_{1 / 32}} \big| \nabla F (z) - \nabla G (z) \big| \textbf{1}_{\Gamma^c}  <  C \displaystyle\sup_{z \in \partial \mathcal{B}_{1 / 8}} \big| F (z) - G (z) \big| \textbf{1}_{\Gamma^c} \le 7 C \lambda,
		\end{flalign*}
		
		\noindent for some constant $C = C(\varepsilon) > 1$. This yields the second estimate in \eqref{figigradientst}. 
	\end{proof}

	Next, we must verify that the parameters $\big( N; \varepsilon; \frac{1}{8}; \frac{\varsigma}{64}; \alpha; H|_{\partial (\mathcal{B}_{N / 8} \cap \mathbb{T})}; f \big)$ satisfy the three conditions of being $\mu$-confined listed in \Cref{lambdabounded}. To establish the first, observe for sufficiently large $N$ and any $z \in \partial \mathcal{B}_{1 / 8}$ that we (deterministically) have that 
	\begin{flalign}
	\label{f1f2h}
	\displaystyle\sup_{z \in \partial \mathcal{B}_{1 / 8}} \big| N^{-1} H (Nz) - f(z) \big| = \displaystyle\sup_{z \in \partial \mathcal{B}_{1 / 8}} \big| \Phi (z) - \varphi (z) \big| \le  \mu - 4 N^{-1},
	\end{flalign}
	
	\noindent which follows from the definitions \eqref{functions12z} of $\varphi$ and the facts that $\int_{\mathbb{R}^2} \psi (w) dw = 1$; that $\psi$ is nonnegative and supported on $\overline{\mathcal{B}}_{\vartheta}$; that $H$ is $1$-Lipschitz; and that $\vartheta = \varsigma \lambda <  \mu - 4 N^{-1}$ (since $\mu = \varsigma^{\alpha} \lambda$). Again since $H$ is $1$-Lipschitz, this verifies the first property in \Cref{lambdabounded} (whose $g$ and $\lambda$ there are $f$ and $\mu$ here, respectively). 
		
	That $f$ satisfies the second and third conditions in \Cref{lambdabounded} will follow from a suitable application of \Cref{euv1v2estimategradient}, to which end it suffices to verify the conditions of that proposition (whose $\mathfrak{h}$ and $\mathcal{H}$ are $f$ and $F$ here, respectively). The following lemma will be used to verify the third one there. 
	
	Before stating it, let us mention that by the $1$-Lipschitz property of $\Phi$ and rescaling by $\vartheta^{-1}$, one can quickly deduce that the second derivatives of $\varphi$ are bounded by $\vartheta^{-1}$. However, since $\vartheta^{-1} \gg \lambda^{- 1}$, this would not verify the third assumption in \Cref{euv1v2estimategradient}. The below lemma will improve this $\vartheta^{-1}$ estimate by powers of $\varsigma$, through use of the fact that we are restricting to the complements of the events $\Gamma_{z_0; z}$ (on which $\Phi$ is not nearly linear) from \Cref{omegadefinition}.

	\begin{lem} 
		
		\label{estimatealpha21mu} 
		
		For $N$ sufficiently large, we have that $\big\| \varphi - \varphi (0, 0) \big\|_{\mathcal{C}^2 (\overline{\mathcal{B}}_{1 / 6})} \textbf{\emph{1}}_{\Gamma^c} < \varsigma^{11} \lambda^{-1}$.
		
	\end{lem} 
	
	\begin{proof}
		
		To estimate $\big\| \varphi - \varphi (0, 0) \big\|_{\mathcal{C}^2 (\overline{\mathcal{B}}_{1 / 6})}$, we must bound $\big\| \varphi - \varphi (0, 0) \big\|_{\mathcal{C}^1 (\overline{\mathcal{B}}_{1 / 6})}$ and $\big| \partial_i \partial_j \varphi (z_0) \big|$ for any  $i, j \in \{ x, y \}$ and $z_0 \in \mathcal{B}_{1 / 6}$. To do the former, observe that the definition \eqref{functions12z} of $\varphi$; the fact that $\Phi$ is $1$-Lipschitz; and the fact that the function $\psi$ from \Cref{psif1f2} is positive, supported on $\overline{\mathcal{B}}_{\vartheta}$, and satisfies $\int_{\mathbb{R}^2} \psi (w) dw = 1$ together imply for sufficiently large $N$ that
		\begin{flalign}
		\label{derivativefunctionestimate1}
		\begin{aligned}
		\big\| \varphi - \varphi (0, 0) \big\|_{\mathcal{C}^1 (\overline{\mathcal{B}}_{1 / 6})} & = \displaystyle\sup_{z \in \mathcal{B}_{1 / 6}} \big| \nabla \varphi (z) \big| + \displaystyle\sup_{z \in \mathcal{B}_{1 / 6}} \big| \varphi (z) - \varphi (0, 0) \big|  \\
		& \le  \displaystyle\sup_{z \in \mathcal{B}_{1 / 5}} \big| \nabla \Phi(z) \big| + \displaystyle\sup_{z \in \mathcal{B}_{1 / 6}} \big| \Phi (z) - \Phi (0, 0) \big| + 2 \vartheta \le 2.
		\end{aligned}
		\end{flalign}
		
		Thus, it remains to estimate the second derivatives of $\varphi$. To that end, it will be useful to rescale by $\vartheta^{-1}$. So, define $\psi^{(\vartheta)}: \mathbb{R}^2 \rightarrow \mathbb{R}^2$ by setting $\psi^{(\vartheta)} (z) = \vartheta^2 \psi (\vartheta z)$, for each $z \in \mathbb{R}^2$. Further fix $z_0 \in \mathcal{B}_{1 / 6}$, and define $\Phi^{(\vartheta)}: \mathcal{B}_{1 / 2 \vartheta} \rightarrow \mathbb{R}$ and $\varphi^{(\vartheta)}: \mathcal{B}_{1 / 6 \vartheta} \rightarrow \mathbb{R}$ by setting 
		\begin{flalign*}
		\Phi^{(\vartheta)} (z) = \vartheta^{-1} \big( \Phi( \vartheta z) - \Phi(z_0) - (\vartheta z - z_0) \cdot \nabla G (z_0) \big); \qquad \varphi^{(\vartheta)} (z) = \int_{\mathbb{R}^2} \Phi^{(\vartheta)} (z - w) \psi^{(\vartheta)} (w) dw,
		\end{flalign*}
		
		\noindent for each $z \in \mathcal{B}_{1 / 2 \vartheta}$ and each $z \in \mathcal{B}_{1 / 6 \vartheta}$, respectively. We then have for any $i, j \in \{ x, y \}$ that $\big| \partial_i \partial_j \varphi (z_0) \big| = \vartheta^{-1} \big| \partial_i \partial_j \varphi^{(\vartheta)} (\vartheta^{-1} z_0) \big|$, and so it suffices to estimate the latter quantity. 
		
		In view of the fourth statement of \eqref{psiestimates}, the fact that $\partial_i \partial_j \psi^{(\vartheta)} (\vartheta^{-1} z) = \vartheta^4 \partial_i \partial_j \psi (z)$ for any $z \in \mathbb{R}^2$, and the fact that $\psi^{(\vartheta)}$ is supported on $\overline{\mathcal{B}}$, we have that 
		\begin{flalign}
		\label{derivativefunctionestimate2} 
		\begin{aligned} 
		\big| \partial_i \partial_j \varphi (z_0) \big| \textbf{1}_{\Gamma^c} = \vartheta^{-1} \big| \partial_i \partial_j \varphi^{(\vartheta)} (\vartheta^{-1} z_0) \big| \textbf{1}_{\Gamma^c} & = \vartheta^{-1} \textbf{1}_{\Gamma^c} \left| \displaystyle\int_{\mathbb{R}^2} \Phi^{(\vartheta)} (\vartheta^{-1} z_0 - w) \partial_i \partial_j \psi^{(\vartheta)} (w) dw \right| \\
		& \le 60 \vartheta^{-1} \textbf{1}_{\Gamma^c} \displaystyle\int_{\mathcal{B}} \big| \Phi^{(\vartheta)} (\vartheta^{-1} z_0 - w) \big| dw \le 240 \vartheta^{-2} \varsigma^{14} \lambda.
		\end{aligned} 
		\end{flalign}
		
		\noindent To deduce the last estimate of \eqref{derivativefunctionestimate2} we used the fact that 
		\begin{flalign*} 
			\big| \Phi^{(\vartheta)} (\vartheta^{-1} z_0 - w) \big| \textbf{1}_{\Gamma^c} \le  \vartheta^{-1} \big( \Phi (z_0 - \vartheta w) - \Phi (z_0) + \vartheta w \cdot \nabla G (z_0) \big) \textbf{1}_{\cap_{z \in \mathcal{B}_{\vartheta} (z_0)} \Gamma_{z_0; z}^c} \le \vartheta^{-1} \varsigma^{14} \lambda,
		\end{flalign*} 
	
		\noindent whenever $w \in \mathcal{B}$, where the first bound holds since $\Gamma^c \subseteq \bigcap_{z \in \mathcal{B}_{\vartheta} (z_0)} \Gamma_{z_0; z}^c$ and the second holds from \Cref{omegadefinition} for $\Gamma_{z_0; z}$. We now obtain the lemma from \eqref{derivativefunctionestimate1} and \eqref{derivativefunctionestimate2} (applied to each $z_0 \in \mathcal{B}_{1 / 6}$), since $\vartheta = \varsigma \lambda$. 
	\end{proof}

	Now we can establish \Cref{estimaten2n}. 
	
	\begin{proof}[Proof of \Cref{estimaten2n}] 
		
		By \Cref{flinearlambda}, there exists a constant $C_2 = C_2 (\varepsilon, D) > 1$ such that $\mathbb{P}[\Gamma] \le C_2 N^{-D}$. Furthermore, the definition \eqref{fg1g2lambda} of $\Gamma_0 \subseteq \Gamma$ implies that \eqref{g1nvlambda} holds on $\Gamma^c$. Therefore, it suffices to show that $f$ from \Cref{psif1f2} satisfies the two properties listed in \Cref{estimaten2n}. The first is verified by \eqref{figigradientst}. 
		
		It remains to show that the parameters $\big( N; \varepsilon; \frac{1}{8}; \frac{\varsigma}{64}; \alpha; H |_{\partial (\mathcal{B}_{N / 8} \cap \mathbb{T})}; f \big)$ are $\mu$-confined, to which end we must verify that they satisfy the three properties listed in \Cref{lambdabounded}, after restricting to the event $\Gamma^c$. The estimate \eqref{f1f2h} and the $1$-Lipschitz property of $\Phi$ together imply that they satisfy the first one listed there. 
		 
		 To verify that they also satisfy the second and third, we apply \Cref{euv1v2estimategradient}, with the $(\varepsilon, \theta, \alpha)$ there equal to $\big( \varepsilon, \frac{1}{3750}, \frac{1}{17500} \big)$ here; the $(\mathfrak{h}, \mathcal{H})$ there equal to $(f, F)$ here; the $\lambda$ there equal to $7 \lambda$ here; and the $\big( g, (s, t), \varphi \big)$ there equal to $\big( g, (s, t), \varphi \big)$ here. Let us verify that these parameters satisfy the three assumptions listed in \Cref{euv1v2estimategradient}. 
		 
		 The first bound in \eqref{figigradientst} implies that $\sup_{z \in \partial \mathcal{B}_{1 / 8}} \big| f (z) - G(z) \big| \le 7 \lambda$ on $\Gamma^c$; this verifies the first condition in \Cref{euv1v2estimategradient}. The second follows from the fact that $g$ is assumed to be $\varsigma^{10}$-nearly linear with slope $(s, t) \in \mathcal{T}_{\varepsilon}$ on $\partial \mathcal{B}$ and that $\varsigma^{10} = (\log N)^{-1/1750} < (\log N)^{-1 / 1800} \le (7 \lambda)^{2 \theta}$ (since $\lambda \ge (\log N)^{-1 - \alpha}$). Moreover, the third follows from the fact that $f = \varphi |_{\partial \mathcal{B}_{1 / 8}}$ (recall \Cref{psif1f2}), \Cref{estimatealpha21mu}, and the fact that $\varsigma^{11} \lambda^{-1} \le \varsigma \lambda^{2 \theta - 1} \le (7 \lambda)^{2 \theta - 1}$.
		
		Thus, \Cref{euv1v2estimategradient} applies, where the $\mu$ there is equal to $(7 \lambda)^{1 + \theta / 8} \le \varsigma^{\alpha} \lambda = \mu$ here (the first inequality holds since $\lambda \le \varsigma^{15}$). Then combining the first estimate in \eqref{derivativehestimates122} with the fact that $(s, t) \in \mathcal{T}_{\varepsilon}$ yields $\nabla F (z) \in \mathcal{T}_{\varepsilon / 2}$ for each $z \in \mathcal{B}_{1 / 8}$ and $N$ sufficiently large. This verifies that $F$ satisfies the second property of \Cref{lambdabounded}.
		
		 Moreover, combining the second estimate in \eqref{derivativehestimates122} with the bound $\lambda \ge (\log N)^{-1 - \alpha} = \varsigma (\log N)^{-1}$ implies for sufficiently large $N$ that $64 \| F \|_{\mathcal{C}^{2, \alpha} (\overline{\mathcal{B}}_{1 / 8} - \mu)} \le 64 (7 \lambda)^{\theta / 2 - 1} < (\log N)^{1 - \theta / 2} \varsigma^{-1} \le \varsigma \log N$ (since $(\log N)^{-\theta / 2} = (\log N)^{-1 / 7500} < (\log N)^{-1 / 8750} = \varsigma^2$); hence, $F$ satisfies the third condition in \Cref{lambdabounded}. This verifies that the parameters $\big( N; \varepsilon; \frac{1}{8}; \frac{\varsigma}{64}; \alpha; H |_{\partial (\mathcal{B}_{N / 8} \cap \mathbb{T})}; f \big)$ are $\mu$-confined and therefore establishes the proposition. 
	\end{proof}

	\section{Proof of the Local Law}
	
	\label{EstimateHLocalProof}
	
	In this section we establish the local law \Cref{heightlocal1}. As mentioned previously, this will proceed by induction on (the logarithm of) the scale $M$, using \Cref{estimaten2n} to ensure that estimates on some scale $M$ are essentially retained on the smaller scale $\frac{M}{8}$. We begin in \Cref{LocalEstimate} by introducing notation for this multi-scale analysis, and then we establish \Cref{heightlocal1} in \Cref{LocalProof}.

	\subsection{Domains, Events, and Height Functions}
	
	\label{LocalEstimate}

	In this section we introduce several events and functions that will be used in the proof of \Cref{heightlocal1}; see \Cref{rkrk1figure} for a depiction. We begin with the following definition for the sequence of domains on which we will apply \Cref{estimaten2n}. Throughout this section we adopt the notation (and in particular recall $N$, $v_0$, $\varepsilon$, and $\mathcal{H}$) of \Cref{regularestimate} and also recall the notation on disks from \eqref{brzdefinition}.
	
	\begin{definition}
		
		\label{nlambdar}
		
		Fix $\alpha = \frac{1}{17500}$. Let $\varsigma_0 = (\log N)^{-\alpha}$, and define the integers $N_1, K \in \mathbb{Z}_{\ge 1}$ by setting $N_1 = 8^K$ to be the unique integral power of $8$ in the interval $[\varsigma_0^{21} N, 8 \varsigma_0^{21} N)$. For each $k \in [1, K] \cap \mathbb{Z}$, further define 
		\begin{flalign*} 
		N_k = 8^{1 - k} N_1; \qquad \varsigma_k = (\log N_k)^{-\alpha}; \qquad R_k = \mathcal{B}_{N_k} (v_0) \cap \mathbb{T}. 
		\end{flalign*}
	\end{definition} 

	The next definition introduces the error parameter $\lambda_k$ that we will use when we apply \Cref{estimaten2n} on the domain $R_k$. 

	\begin{definition} 
		
	\label{lambdak} 
	
	 For each $k \in [1, K] \cap \mathbb{Z}$, define $\lambda_k \in \mathbb{R}_{> 0}$ inductively, by setting 
	 \begin{flalign*} 
	 \lambda_1 = \varsigma_0^{20}, \qquad \text{and} \qquad \lambda_k = \max \big\{ 8 \varsigma_{k - 1}^{\alpha} \lambda_{k - 1}, (\log N_k)^{-1 - \alpha} \big\}, \quad \text{for $k > 1$}.
	 \end{flalign*}  
	 
	 \noindent Further let $C_3 > 1$ denote a uniform (independent of $N$ and $k$) constant such that 
	 \begin{flalign}
	 \label{kestimatelambda}
	 \lambda_k \le C_3 \varsigma_0^{20}, \quad \text{if $1 \le k < 3 (\log \log N)^2$}; \qquad \qquad \lambda_k \le C_3 (\log N_k)^{-1 - \alpha}, \quad \text{if $2^{35} \le k \le K$}. 
	 \end{flalign}

	\end{definition} 
	
	The constant $C_3$ from \Cref{lambdak} is quickly verified to exist, since $\varsigma_k = (\log N_k)^{-\alpha}$ and $\alpha = \frac{1}{17500} > 2^{-15}$. 
	
	Next, we will inductively apply \Cref{estimaten2n} on the domains $R_k$ from \Cref{nlambdar} to obtain sequences of functions $\{ G_k \}, \{ F_k \}$ and events $\{ \Gamma_k \}, \{ \Omega_k \}$. The following definition provides some of the initial elements (namely, $G_1$ and $\Omega_0$) of these sequences. 

	\begin{definition}
		
	\label{g1gamma0} 
	
	Define the (linear) function $G_1 \in \Adm (\mathcal{B})$ by setting 
		\begin{flalign}
		\label{g1kzg2kz}
		\begin{aligned}
		& G_1 (z) = z \cdot \nabla \mathcal{H} (N^{-1} v_0) + N_1^{-1} \mathbb{E} \big[ H (v_0) \big], \qquad \text{for each $z \in \mathcal{B}$}.
		\end{aligned}
		\end{flalign}
		
		\item Additionally, define the event 
		\begin{flalign*}
		\Omega_0 = \bigg\{ \displaystyle\max_{u \in \mathbb{V}(R_1)} \Big| N_1^{-1} H(u) - G_1 \big( N_1^{-1} (u - v_0) \big) \Big| > \displaystyle\frac{\lambda_1}{2} \bigg\}. 
		\end{flalign*}
	
\end{definition}

	For $k \ge 1$, the event $\Omega_k$ will be the one on which the assumptions of \Cref{estimaten2n} are not satisfied on $R_k$ (assuming $G_k$ is given). This is made precise through the following definition. In what follows we recall $\lambda$-confinement from \Cref{lambdabounded}.

	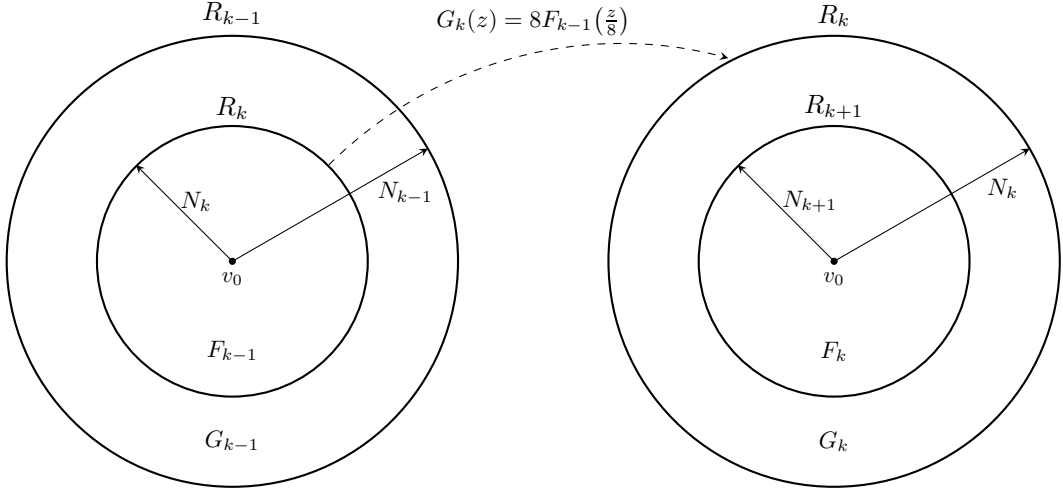
\begin{figure}

		\begin{center}

			\begin{tikzpicture}[
			>=stealth,
			auto,
			style={
				scale = .4
			}
			]

			\draw[thick] (0, 0) circle[radius = 7.5];
			\draw[thick] (0, 0) circle[radius = 4.5];
			
			\filldraw[fill = black] (0, 0) circle[radius = .1] node[below = 1, scale = .85]{$v_0$};
			
			\draw[] (0, 7.5) circle[radius = 0] node[above]{$R_{k - 1}$};
			\draw[] (0, 4.5) circle[radius = 0] node[above = -1]{$R_k$};
			\draw[] (-2, 2) circle[radius = 0] node[right, scale = .9]{$N_k$};
			\draw[] (5.75, 2.28) circle[radius = 0] node[scale = .9]{$N_{k - 1}$};
			\draw[] (0, -6) circle[radius = 0] node[scale = .9]{$G_{k - 1}$};
			\draw[] (0, -3) circle[radius = 0] node[scale = .9]{$F_{k - 1}$};
			
			\draw[->] (0, 0) -- (-3.2, 3.2);
			\draw[->] (0, 0) -- (6.5, 3.75);

			\draw[thick] (20, 0) circle[radius = 7.5];
			\draw[thick] (20, 0) circle[radius = 4.5];
			
			\draw[->, black, dashed] (3.2, 3.2) arc (135:75:13.75);
			\draw[] (10, 8) circle[radius = 0] node[scale = .9]{$G_k (z) = 8 F_{k - 1} \big( \frac{z}{8} \big)$};
			
			\draw[->] (20, 0) -- (16.8, 3.2);
			\draw[->] (20, 0) -- (26.5, 3.75);
			
			\draw[] (20, 7.5) circle[radius = 0] node[above]{$R_k$};
			\draw[] (20, 4.5) circle[radius = 0] node[above = -1]{$R_{k + 1}$};
			\draw[] (18, 2) circle[radius = 0] node[right, scale = .9]{$N_{k + 1}$};
			\draw[] (25.6, 2.43) circle[radius = 0] node[scale = .9]{$N_k$};
			\draw[] (20, -6) circle[radius = 0] node[scale = .9]{$G_k$};
			\draw[] (20, -3) circle[radius = 0] node[scale = .9]{$F_k$};
			
			\filldraw[fill = black] (20, 0) circle[radius = .1] node[below = 1, scale = .85]{$v_0$};

			\end{tikzpicture}
			
		\end{center}

		\caption{\label{rkrk1figure} Depicted above is a diagram for some of the parameters in \Cref{LocalEstimate}.}

	\end{figure}

	\begin{definition} 
		
	\label{omegakdefinition} 
	
	Fix $k \in [1, K] \cap \mathbb{Z}$, and suppose $G_k \in \Adm (\mathcal{B})$ is given. Set $(s_k, t_k) = \nabla G_k (0, 0) \in \overline{\mathcal{T}}$; let $g_k = G_k |_{\partial \mathcal{B}}$; and define the event $\Omega_k$ on which at least one of the following two assumptions is not satisfied.
	
	\begin{enumerate}

		\item \label{2g1g2} The function $G_k$ is $\varsigma_k^{10}$-nearly linear of slope $(s_k, t_k)$ on $\mathcal{B}$. 
		
		\item \label{3g1g2} The parameters $\big( N_k; \frac{\varepsilon}{2}; 1; \varsigma_k; \alpha; H |_{\partial R_k}; g_k \big)$ are $\lambda_k$-confined. 
		
	\end{enumerate}

	\end{definition} 
	
	Observe that, unlike in \Cref{estimaten2n}, \Cref{omegakdefinition} does not impose that $(s_k, t_k) \in \mathcal{T}_{\varepsilon / 2}$. This will be established separately, as a consequence of \Cref{estimatesktk} below. We also did not impose that $G_k$ be a maximizer of $\mathcal{E}$, although this will eventually follow from its definition (see \Cref{gke}).

	The next definition essentially sets $\Gamma_k$ to be the event $\Gamma$ from \Cref{estimaten2n} applied on $R_k$, and also introduces the event $\Gamma = \Gamma (v_0; M)$ from \Cref{heightlocal1} when $M \in \{ N_1, N_2, \ldots , N_k \}$. 

	\begin{definition} 
		
	\label{gammakdefinition}
	
	Fix $k \in [1, K] \cap \mathbb{Z}$; suppose $G_k \in \Adm (\mathcal{B})$ is given; and let $C_0 = C_0 (\varepsilon) > 1$ denote the constant $C_1 (\varepsilon)$ from \Cref{estimaten2n}. Then let $\Gamma_k$ denote the event on which at least one of the following two conditions is not satisfied. 
	 
	 \begin{enumerate} 
	 	
	 	\item The event $\bigcap_{j = 0}^k \Omega_j^c$ holds, and we have the bound
		\begin{flalign}
		\label{g1knkvlambdak} 
		\displaystyle\max_{u \in \mathbb{V}(R_k)} \Big| N_k^{-1} H(u) - G_k \big( N_k^{-1} (u - v_0) \big)  \Big| \le 6 \lambda_k.
		\end{flalign}
	
		\item There exists a function $f_k: \partial \mathcal{B}_{1 / 8} \rightarrow \mathbb{R}$ admitting an admissible extension to $\mathcal{B}_{1 / 8}$ satisfying the following two properties. In the below, we let $F_k \in \Adm (\mathcal{B}_{1 / 8}; f_k)$ denote the maximizer of $\mathcal{E}$ on $\mathcal{B}_{1 / 8}$ with boundary data $f_k$. 
	
	\begin{enumerate}

		\item \label{3f1f2} We have the bounds 
		\begin{flalign}
		\label{ikgikfestimate}
		\displaystyle\sup_{z \in \mathcal{B}_{1 / 8}} \big| F_k (z) - G_k (z) \big| \le 7 \lambda_k; \qquad \displaystyle\sup_{z \in \mathcal{B}_{1 / 32}} \big| \nabla F_k (z) - \nabla G_k (z)  \big| < C_0 \lambda_k. 
		\end{flalign}

		\item The parameters $\big( N_k; \frac{\varepsilon}{2}; \frac{1}{8}; \frac{\varsigma}{64}; \alpha; H|_{\partial R_{k + 1}}; f_k \big)$ are $\frac{\lambda_{k + 1}}{8}$-confined. 
	
	\end{enumerate}	
	
	\end{enumerate}
	
	\noindent For each $k \in [1, K] \cap \mathbb{Z}$, define the event $\Gamma (v_0; N_k) = \bigcup_{j = 1}^k \Gamma_j$. Further set $\Gamma (v_0; N_0) = \Omega_0$. 
	
	\end{definition}

	The previous definitions have assumed that $G_k$ are given, and so we must now explain how to set them; this will be done recursively (recall $G_1$ is provided by \eqref{g1kzg2kz}). Assuming that $G_{k - 1}$ is given, one first defines the function $F_k$ through \Cref{gammakdefinition} and then defines $G_{k + 1}$ by rescaling $F_k$.

	\begin{definition}
		
		\label{g1kg2k}
		
		Fix $k \in [2, K] \cap \mathbb{Z}$, and suppose $G_{k - 1}$ is given. On $\Gamma (v_0; N_{k - 1})^c$, set $F_{k - 1} \in \Adm (\mathcal{B}_{1 / 8})$ as in part 2 of \Cref{gammakdefinition}; on the complementary event $\Gamma (v_0; N_{k - 1})$, set $F_{k - 1} \in \Adm (\mathcal{B}_{1 / 8})$ arbitrarily. Then define $G_k \in \Adm (\mathcal{B})$ by setting $G_k (z) = 8 F_{k - 1} \big( \frac{z}{8} \big)$, for each $z \in \mathcal{B}$. 
		
	\end{definition} 
	
	\begin{rem}
		
	\label{gke}
	
	Observe by induction and the scale-invariance of the variational principle (recall \Cref{estimatehrho}) that, on $\Gamma (v_0; N_{k - 1})^c$, $G_k$ is a maximizer of $\mathcal{E}$ on $\mathcal{B}$ with boundary data $g_k = G_k |_{\partial \mathcal{B}}$. 
	\end{rem}

	Under this notation, we will establish \Cref{heightlocal1} in \Cref{LocalProof} below. We conclude this section with the following lemma that bounds $\mathbb{P} [\Omega_0]$.

	\begin{lem}
		
		\label{probabilityevent0}
		
		There exists a constant $C = C(\varepsilon) > 1$ such that $\mathbb{P} [\Omega_0] < C \exp (-N^{1 / 2})$. 
	\end{lem}
	
	\begin{proof} 
		
		Observe that since $R_1 = \mathcal{B}_{N_1} (v_0) \cap \mathbb{T}$, since $\varsigma_0^{42} < (\log N)^{-1 / 410}$, and since $\nabla \mathcal{H} (z) \in \mathcal{T}_{\varepsilon}$ for each $z \in \mathcal{B}_{\varepsilon} (N^{-1} v_0)$ (recall \Cref{regularestimate}), we have from \Cref{heightapproximate} that there exists a constant $C_4 = C_4 (\varepsilon) > 1$ such that 
		\begin{flalign}
		\label{huhwestimateprobability}
		\mathbb{P} \left[ \displaystyle\max_{u \in \mathbb{V}(R_1)} \Big| N^{-1} \big( H(u) - H(v_0) \big) - \big( \mathcal{H} (N^{-1} u) - \mathcal{H} (N^{-1} v_0) \big) \Big| > \varsigma_0^{42} \right] < C_4 e^{-N}. 
		\end{flalign}
		 
		\noindent Next, since the diameter of $R$ is at most $2N$, \Cref{hvuexpectation} implies that 
		\begin{flalign} 
		\label{hv0expectationhv0} 
		\mathbb{P} \bigg[ \Big| H(v_0) - \mathbb{E} \big[ H(v_0) \big] \Big| > N^{5 / 6} \bigg] \le 2 \exp \left( - \displaystyle\frac{N^{2 / 3}}{128} \right).
		\end{flalign} 
		
		\noindent Moreover, \Cref{derivativeshestimate}, the fact that $\nabla \mathcal{H} (z) \in \mathcal{T}_{\varepsilon}$ for each $z \in \mathcal{B}_{\varepsilon} (N^{-1} v_0)$, and the fact that $\frac{N_1}{N} \in [\varsigma_0^{21}, 8 \varsigma_0^{21}]$ together yield a constant $C_5 = C_5 (\varepsilon) > 1$ such that 
		\begin{flalign}
		\label{huhwestimategradient}
		\displaystyle\max_{u \in \mathbb{V}(R_1)} \Big| \big( \mathcal{H} (N^{-1} u) - \mathcal{H} (N^{-1} v_0) \big) - N^{-1} (u - v_0) \cdot \nabla \mathcal{H} (N^{-1} v_0) \Big| \le C_5 \varsigma_0^{42}.
		\end{flalign}
		
		\noindent Now the lemma follows from \eqref{huhwestimateprobability}, \eqref{huhwestimategradient}, \eqref{hv0expectationhv0}, the fact that $\lambda_1 = \varsigma_0^{20}$, the fact that $\frac{N_1}{N} \in [\varsigma_0^{21}, 8 \varsigma_0^{21}]$, and the explicit form \eqref{g1kzg2kz} for $G_1$. 
	\end{proof}

	\subsection{Proof of \Cref{heightlocal1}} 
	
	\label{LocalProof}
	
	In this section we establish \Cref{heightlocal1}; throughout this section, we recall the notation introduced in \Cref{LocalEstimate}. We begin by verifying that the gradients  $(s_k, t_k) = \nabla G_k (0, 0)$ are in $\mathcal{T}_{\varepsilon / 2}$. The following lemma indicates that this is the case by showing that their distance to $\nabla \mathcal{H} (N^{-1} v_0) \in \mathcal{T}_{\varepsilon}$ is small, upon restricting to the event $\Gamma (v_0; N_{k - 1})^c$. It further shows off of this event that the gradient of $G_{k - 1}$ is approximately constant (with error $(\log N_k)^{-1 - c_0}$, where $c_0 = \frac{1}{18000}$) on $\mathcal{B}_{1 / 4}$, for $k$ sufficiently large. In what follows, the phase, ``sufficiently large $N$,'' means that there exists a constant $C = C(\varepsilon) > 1$ such that $N > C$ (that is, ``sufficiently large'' only depends on $\varepsilon$).    
	
	\begin{lem}
		
		\label{estimatesktk}
		
		Setting $c_0 = \frac{1}{18000}$, the following two statements hold.
		
		\begin{enumerate}
			
			\item For any $k \in [1, K] \cap \mathbb{Z}$ with $N_k$ sufficiently large, we have that 
			\begin{flalign}
			\label{sktk1} 
			\big| (s_k, t_k) - \nabla \mathcal{H} (N^{-1} v_0) \big| \textbf{\emph{1}}_{\Gamma (v_0; N_{k - 1})^c} < (\log N_k)^{-c_0}.
			\end{flalign} 
			
			\item For any $k \in [1, K] \cap \mathbb{Z}$ with $k \ge 2 (\log \log N)^2$, we have for $N$ and $N_k$ sufficiently large that   
			\begin{flalign}
			\label{sktkg1k1} 
			\displaystyle\sup_{z \in \mathcal{B}_{1 / 4}} \big| \nabla G_{k - 1} (z) - (s_k, t_k) \big| \textbf{\emph{1}}_{\Gamma (v_0; N_{k - 1})^c} \le (\log N_k)^{-1 - c_0}. 
			\end{flalign}
			
		\end{enumerate} 
		
	\end{lem}
	
	\begin{proof}
		
		Observe from \Cref{g1kg2k} that $\nabla G_{j + 1} (z) = \nabla F_j \big( \frac{z}{8} \big)$, for each $j \in [1, K - 1]$ and $z \in \mathcal{B}$. Together with the fact that $\nabla G_1 = \nabla \mathcal{H} (N^{-1} v_0)$ (recall \eqref{g1kzg2kz}), this yields
		\begin{flalign}
		\label{sktkestimate} 
		\begin{aligned}
		\big| \nabla G_k (z) - \nabla \mathcal{H} (N^{-1} v_0) \big| & \le \displaystyle\sum_{j = 1}^{k - 1} \big| \nabla G_{j + 1} (8^{j - k + 1} z) - \nabla G_j (8^{j - k} z) \big| \\
		& = \displaystyle\sum_{j = 1}^{k - 1} \big| \nabla F_j (8^{j - k} z) - \nabla G_j (8^{j - k} z) \big|, \qquad \quad \text{for each $z \in \mathcal{B}$}. 
		\end{aligned}
		\end{flalign}
		
		\noindent Furthermore, the second statement of \eqref{ikgikfestimate} gives $\big| \nabla G_j (0, 0) - \nabla F_j (0, 0) \big| \textbf{1}_{\Gamma (v_0; N_{k - 1})^c} \le C_0 \lambda_j$, for each $j \in [1, k - 1]$. Together with \eqref{kestimatelambda}, \eqref{sktkestimate}, and the fact that $(s_k, t_k) = \nabla G_k (0, 0)$, this yields
		\begin{flalign*}
		\big| (s_k, t_k) - \nabla \mathcal{H} (N^{-1} v_0) \big| \textbf{1}_{\Gamma (v_0; N_{k - 1})^c} \le C_0 \displaystyle\sum_{j = 1}^{k - 1} \lambda_j & \le 2^{35} C_0 C_3 \varsigma_0^{20} + C_0 C_3 \displaystyle\sum_{j = 1}^{k - 1} (\log N_j)^{-1 - \alpha} \\
		& < 2^{35} C_0 C_3 \varsigma_0^{20} + C_0 C_3 \alpha^{-1} (\log N_k)^{-\alpha} \\
		& < (\log N_k)^{-c_0},
		\end{flalign*}
		
		\noindent for sufficiently large $N_k$, which verifies \eqref{sktk1}.
		
		Now set $\ell = \ell_k = \big\lceil (\log \log N_k)^2 \big\rceil$ and assume that $k \ge 2 (\log \log N)^2$. Let us first approximate $\nabla G_{k - 1} (z) \approx \nabla G_{k - \ell} (8^{1 - \ell} z) \approx \nabla G_{k - \ell} (0, 0)$ for each $z \in \mathcal{B}_{1 / 4}$. To that end, observe that from the fact that $\nabla G_{k - j} (z) = \nabla F_{k - j - 1} \big( \frac{z}{8} \big)$ for each $j \in [1, k - 1]$ and $z \in \mathcal{B}$; the second bound in \eqref{ikgikfestimate}; \eqref{kestimatelambda}; and the fact that $k - \ell - 1 > 2^{35}$ for sufficiently large $N$ that
		\begin{flalign}
		\label{gk1gk}
		\begin{aligned}
		\displaystyle\sup_{z \in \mathcal{B}_{1 / 4}} & \big| \nabla G_{k - 1} (z) - \nabla G_{k - \ell} (8^{1 - \ell} z) \big|  \textbf{1}_{\Gamma (v_0; N_{k - 1})^c} \\
		& \le \displaystyle\sum_{j = 1}^{\ell - 1} \displaystyle\sup_{z \in \mathcal{B}_{1 / 4}} \big| \nabla G_{k - j} (8^{1 - j} z) - \nabla G_{k - j - 1} ( 8^{-j} z) \big|  \textbf{1}_{\Gamma (v_0; N_{k - 1})^c} \\
		& = \displaystyle\sum_{j = 1}^{\ell - 1} \displaystyle\sup_{z \in \mathcal{B}_{1 / 4}} \big| \nabla F_{k - j - 1} (8^{- j} z) - \nabla G_{k - j - 1} ( 8^{-j} z) \big|  \textbf{1}_{\Gamma (v_0; N_{k - 1})^c} \\
		& = \displaystyle\sum_{j = 1}^{\ell - 1} \displaystyle\sup_{z \in \mathcal{B}_{1 / 32}} \big| \nabla F_{k - j - 1} (8^{1 - j} z) - \nabla G_{k - j - 1} (8^{1 - j} z) \big|  \textbf{1}_{\Gamma (v_0; N_{k - 1})^c} \\
		& \le C_0 \displaystyle\sum_{j = 1}^{\ell - 1} \lambda_{k - j - 1} \le C_0 C_3 \displaystyle\sum_{j = 1}^{\ell - 1} (\log N_{k - j - 1})^{-1 - \alpha} \le C_0 C_3 \ell (\log N_k)^{-1 - \alpha}. 
		\end{aligned} 
		\end{flalign}
	
		\noindent Moreover, since the parameters $\big( N_{k - \ell}; \frac{\varepsilon}{2}; 1; \varsigma_{k - \ell}; \alpha; H |_{\partial R_{k - \ell}}; g_{k - \ell} \big)$ are $\lambda_k$-confined on $\Gamma (v_0; N_{k - 1})^c$, we have from the third property listed in \Cref{lambdabounded} that $\| G_{k - \ell} - G_{k - \ell} (0, 0) \|_{\mathcal{C}^2 (\overline{\mathcal{B}}_{1 / 2})} \le \log N_{k- \ell}$. Therefore, since $\ell \ge (\log \log N_k)^2$, we have for sufficiently large $N$ that
		\begin{flalign*}
		\displaystyle\sup_{z \in \mathcal{B}_{1 / 4}} & \big| \nabla G_{k - \ell} (8^{1 - \ell} z) - (s_{k - \ell}, t_{k - \ell}) \big| \textbf{1}_{\Gamma (v_0; N_{k - 1})^c} \\
		& \le \displaystyle\sup_{z \in \mathcal{B}_{8^{1 - \ell}}} \big| \nabla G_{k - \ell} (z) - \nabla G_{k - \ell} (0, 0) \big| \textbf{1}_{\Gamma (v_0; N_{k - 1})^c} \\
			& \le 8^{1 - \ell} \| G_{k - \ell} - G_{k - \ell} (0, 0) \|_{\mathcal{C}^2 (\overline{\mathcal{B}}_{1 / 2})} \textbf{1}_{\Gamma (v_0; N_{k - 1})^c} \le 8^{1 - \ell} \log N_{k - \ell} \le (\log N_k)^{-2},
		\end{flalign*}
		
		\noindent which together with \eqref{gk1gk} implies 
		\begin{flalign}
		\label{gk1zsktk}
		\begin{aligned}
		\displaystyle\sup_{z \in \mathcal{B}_{1 / 4}} \big| \nabla G_{k - 1} (z) - (s_{k - \ell}, t_{k - \ell}) \big|  \textbf{1}_{\Gamma (v_0; N_{k - 1})^c} & \le 2 C_0 C_3 \ell (\log N_k)^{-1 - \alpha}.
		\end{aligned} 
		\end{flalign}
		
		\noindent Applying \eqref{gk1zsktk} twice (once with general $z \in \mathcal{B}_{1 / 4}$ and once with $z = (0, 0)$); using the fact that $(s_k, t_k) = \nabla G_k (0, 0) = \nabla F_{k - 1} (0, 0)$; and applying the second bounds in both \eqref{ikgikfestimate} and \eqref{kestimatelambda} (with the fact that $k - 1 \ge 2^{35}$), we deduce  
		\begin{flalign*} 
		& \displaystyle\sup_{z \in \mathcal{B}_{1 / 4}} \big| \nabla G_{k - 1} (z) - (s_k, t_k) \big| \textbf{1}_{\Gamma (v_0; N_{k - 1})^c} \\
		& \quad \le \displaystyle\sup_{z \in \mathcal{B}_{1 / 4}} \big| \nabla G_{k - 1} (z) - (s_{k - \ell}, t_{k - \ell}) \big| \textbf{1}_{\Gamma (v_0; N_{k - 1})^c} + \big| (s_{k - 1}, t_{k - 1}) - (s_{k - \ell}, t_{k - \ell}) \big| \textbf{1}_{\Gamma (v_0; N_{k - 1})^c} \\
		& \qquad \quad 	+ \big| (s_k, t_k) - (s_{k - 1}, t_{k - 1}) \big| \textbf{1}_{\Gamma (v_0; N_{k - 1})^c} \\
		& \quad \le 4 C_0 C_3 \ell (\log N_k)^{-1 - \alpha} + \big| \nabla F_{k - 1} (0, 0) - \nabla G_{k - 1} (0, 0) \big| \textbf{1}_{\Gamma (v_0; N_{k - 1})^c} \\
		& \quad \le 4 C_0 C_3 \ell (\log N_k)^{-1 - \alpha} + C_0 \lambda_{k - 1} \le (\log N_k)^{-1 - c_0}, 
		\end{flalign*} 
		
		\noindent for sufficiently large $N$ and $N_k$. This yields \eqref{sktkg1k1}. 	
	\end{proof}

	Next, we require the following proposition indicating that $\Omega_k \subseteq \Gamma (v_0; N_{k - 1})$. This essentially implies that the function $G_k$ is ``well-behaved'' if the functions $F_j$ are for each $j \in [1, k - 1]$, which will be useful for applying \Cref{estimaten2n} on $R_k$ (after having applied it on all previous $R_j$). 
	
	\begin{prop} 
		
		\label{omegakgammakprobability}
		
		For sufficiently large $N_k$, we have that $\Omega_k \subseteq \Gamma (v_0; N_{k - 1})$.

	\end{prop}

	\begin{proof}
		
		Since $G_1$ is linear of slope $\nabla G_1 = \nabla \mathcal{H} (N^{-1} v_0) \in \mathcal{T}_{\varepsilon}$, one can quickly verify using \Cref{g1gamma0} and \Cref{omegakdefinition} that $\Omega_1 \subseteq \Omega_0 = \Gamma (v_0; N_0)$. So, in what follows, let us assume that $k > 1$ and restrict to the event $\Gamma (v_0; N_{k - 1})^c$. To show that $\Omega_k^c$ holds, we must verify that $g_k$ and $G_k$ satisfy the two properties listed in \Cref{omegakdefinition}. 
		
		Let us first show that the parameters $\big(N_k; \frac{\varepsilon}{2}; 1; \varsigma_k; \alpha; H|_{\partial R_k}; g_k \big)$ are $\lambda_k$-confined; this will follow from rescaling. Indeed, recall from the second part of the second condition of \Cref{gammakdefinition} that the parameters $\big( N_{k - 1}; \frac{\varepsilon}{2}; \frac{1}{8}; \frac{\varsigma_{k - 1}}{64}; \alpha; H |_{\partial R_k}; f_{k - 1} \big)$ are $\frac{\lambda_k}{8}$-confined on the event $\Gamma (v_0; N_{k - 1})^c$. From this and the facts that $N_{k - 1} = 8 N_k$ and $G_k (z) = 8 F_{k - 1} \big( \frac{z}{8} \big)$ for each $z \in \mathcal{B}$, we deduce that $\big| g_k (z) - N_k^{-1} H(v) \big| = 8 \big| f_{k - 1} \big( \frac{z}{8} \big) - N_{k - 1}^{-1} H (v) \big| \le \lambda_k$ whenever $|z - N_k^{-1} v| = 8 \big| \frac{z}{8}  - N_{k - 1}^{-1} v \big| \le 32 N_{k - 1}^{-1} = 4 N_k^{-1}$. Hence, $\big(N_k; \frac{\varepsilon}{2}; 1; \varsigma_k; \alpha; H|_{\partial R_k}; g_k \big)$ satisfy the first condition listed in \Cref{lambdabounded}. 
		
		They also satisfy the second, since $\nabla G_k (z) =  \nabla F_{k - 1} \big( \frac{z}{8} \big) \in \mathcal{T}_{\varepsilon / 4}$ for any $z \in \mathcal{B}$. That they satisfy the third for $N_k$ sufficiently large follows from the estimate 
		\begin{flalign*} 
		\big\| G_k - G_k (0, 0) \big\|_{\mathcal{C}^{2, \alpha} (\overline{\mathcal{B}}_{1 - \lambda_k})} \le 8^{-1 - \alpha} \big\| F_{k - 1} - F_{k - 1} (0, 0) \big\|_{\mathcal{C}^{2, \alpha} (\overline{\mathcal{B}}_{1 / 8 - \lambda_k / 8})} & \le 8^{-3 - \alpha} \varsigma_{k - 1} \log N_{k - 1} \\
		& \le \varsigma_k \log N_k.
		\end{flalign*} 
		
		\noindent Hence, the parameters $\big(N_k; \frac{\varepsilon}{2}; 1; \varsigma_k; \alpha; H|_{\partial R_k}; g_k \big)$ are $\lambda_k$-confined, and so to show that $\Omega_k^c$ holds it remains to verify that $G_k$ is $\varsigma_k^{10}$-nearly linear with slope $(s_k, t_k) = \nabla G_k (0, 0) = \nabla F_{k - 1} (0, 0)$ on $\mathcal{B}$. 
		
		To do this, first define the linear function $\Psi: \mathcal{B}_{1 / 8} \rightarrow \mathbb{R}$ by setting $\Psi (z) = z \cdot (s_k, t_k) + G_{k - 1} (0, 0)$ for each $z \in \mathcal{B}_{1 / 8}$. Then define $\Lambda_k: \mathcal{B} \rightarrow \mathbb{R}$ by setting $\Lambda_k (z) = 8 \Psi \big( \frac{z}{8} \big)$, for each $z \in \mathcal{B}_{1 / 8}$. The first statement of \eqref{ikgikfestimate} and the identity $G_{k - 1} (0, 0) = \Psi (0, 0)$ then give 
		\begin{flalign}
		\label{gikl}
		\begin{aligned}
		 \displaystyle\sup_{z \in \mathcal{B}} \big| G_k (z) - \Lambda_k (z) \big| = 8 \displaystyle\sup_{z \in \mathcal{B}_{1 / 8}} \big| F_{k - 1} (z) - \Psi (z) \big|  &  \le 8 \displaystyle\sup_{z \in \mathcal{B}_{1 / 8}} \big| G_{k - 1} (z) -  \Psi (z) \big| + 7 \lambda_{k - 1} \\
		& \le \displaystyle\sup_{z \in \mathcal{B}_{1 / 8}} \big| \nabla G_{k - 1} (z) - (s_k, t_k) \big| + 7 \lambda_{k - 1}.
		\end{aligned}	
		\end{flalign}
		
		Denoting $\ell = 2 \big\lceil (\log \log N_k)^2 \big\rceil$, we will bound the right side of \eqref{gikl} separately in the cases when $2 \le k \le \ell$ and $\ell < k \le K$. In the latter case, it follows from \eqref{kestimatelambda}, \eqref{sktkg1k1}, and \eqref{gikl} that
		\begin{flalign*} 
		\sup_{z \in \mathcal{B}} \big| G_k (z) - \Lambda_k (z) \big| \le (\log N_k)^{-1 - c_0} + 7 \lambda_{k - 1} \le \varsigma_k^{10},
		\end{flalign*} 
		
		\noindent for $N_k$ sufficiently large. Hence, $G_k$ is $\varsigma_k^{10}$-nearly linear $(s_k, t_k)$ when $k > \ell$. 
		
		So let us assume instead that $2 \le k \le \ell$, in which case it follows for sufficiently large $N$ from \eqref{kestimatelambda}, \eqref{sktkestimate} (with the $k$ there equal to $k - 1$ here), and the second bound in \eqref{ikgikfestimate} that
		\begin{flalign}
		\label{gjk1hnv0} 
		\begin{aligned} 
		\displaystyle\sup_{z \in \mathcal{B}_{1 / 4}} \big| \nabla G_{k - 1} (z) - \nabla \mathcal{H} (N^{-1} v_0) \big| & \le \displaystyle\sum_{j = 1}^{k - 2} \displaystyle\sup_{z \in \mathcal{B}_{1 / 4}} \big| \nabla G_j (8^{j - k + 1} z) - \nabla F_j ( 8^{j - k + 1} z) \big| \\ 
		& = \displaystyle\sum_{j = 1}^{k - 2} \displaystyle\sup_{z \in \mathcal{B}_{1 / 32}} \big| \nabla G_j (8^{j - k + 2} z) - \nabla F_j ( 8^{j - k + 2} z) \big| \\
		& \le C_0 \displaystyle\sum_{j = 1}^{\ell} \lambda_j \le C_0 C_3 \ell \varsigma_0^{20} < \varsigma_0^{15}.
		\end{aligned} 
		\end{flalign}
		
		\noindent From two applications of \eqref{gjk1hnv0} (one with general $z \in \mathcal{B}_{1 / 8}$ and one with $z = (0, 0)$), \eqref{kestimatelambda}, the fact that $(s_k, t_k) = \nabla G_k (0, 0) = \nabla F_{k - 1} (0, 0)$, and the second bound in \eqref{ikgikfestimate}, we deduce that
		\begin{flalign*}
		\displaystyle\sup_{z \in \mathcal{B}_{1 / 8}} \big| \nabla G_{k - 1} (z) - (s_k, t_k) \big| & \le \displaystyle\sup_{z \in \mathcal{B}_{1 / 8}} \big| \nabla G_{k - 1} (z) - \nabla \mathcal{H} (N^{-1} v_0) \big| + \big| \nabla G_{k - 1} (0, 0) - \nabla \mathcal{H} (N^{-1} v_0) \big| \\
		& \qquad  + \big| \nabla G_{k - 1} (0, 0) - (s_k, t_k) \big| \\
		& \le 2 \varsigma_0^{15} + \big| \nabla G_{k - 1} (0, 0) - \nabla F_{k - 1} (0, 0) \big| \le 2 \varsigma_0^{15} + C_0 \lambda_{k - 1}, 
		\end{flalign*}
		
		\noindent  for $N$ sufficiently large. Since \eqref{kestimatelambda} implies $2 \varsigma_0^{15} + (C_0 + 7) \lambda_{k - 1} < \varsigma_k^{10}$ for $N$ sufficiently large, we deduce from \eqref{gikl} that $G_k$ is $\varsigma_k^{10}$-nearly linear when $k \le \ell$. Thus, $\Omega_k^c$ holds on $\Gamma (v_0; N_{k - 1})^c$.
	\end{proof} 
	
	Using \Cref{estimaten2n}, we can establish the following corollary that bounds the probability of $\Gamma_k \cap \Gamma (v_0; N_{k - 1})^c$, namely, the event on which $F_k$ is not well-behaved given that the $\{ F_j \}_{j < k}$ are.

	\begin{cor}
		
		\label{probabilitygammakgammav0nk1} 
		
		For any fixed real number $D > 0$, there exists a constant $C = C(\varepsilon, D) > 1$ such that $\mathbb{P} \big[ \Gamma_k \cap \Gamma (v_0; N_{k - 1})^c \big] < C N_k^{-D}$ for any $k \in [1, K] \cap \mathbb{Z}$. 
	\end{cor}
	
	\begin{proof} 
		
		By \Cref{omegakgammakprobability}, we may assume that $N_k$ is sufficiently large so that $\Omega_k \subseteq \Gamma (v_0; N_{k - 1})$; thus, it suffices to bound $\mathbb{P} \big[ \Gamma_k \cap \Gamma (v_0; N_{k - 1})^c \cap \Omega_k^c \big]$. This will proceed through an application of \Cref{estimaten2n}, with the $N, \varepsilon, (s, t), h, g, \lambda$ there equal to $N_k, \frac{\varepsilon}{2}, (s_k, t_k), H |_{\partial R_k}, g_k, \lambda_k$ here, respectively. To verify that these parameters satisfy the assumptions of that proposition on $\Gamma (v_0; N_{k - 1})^c \cap \Omega_k^c$, first observe that \Cref{lambdak} yields $\varsigma_k^{15} \le \lambda_k \le (\log N_k)^{-1 - \alpha}$ for sufficiently large $N_k$. Additionally, on $\Gamma (v_0; N_{k - 1})^c$, \Cref{gke} implies that $G_k$ is the maximizer of $\mathcal{E}$ on $\mathcal{B}$ with boundary data $g_k$. Moreover, \eqref{sktk1} and the fact that $\nabla \mathcal{H} (N^{-1} v_0) \in \mathcal{T}_{\varepsilon}$ (recall \Cref{regularestimate}) together imply that $(s_k, t_k) \in \mathcal{T}_{\varepsilon / 2}$ on $\Gamma (v_0; N_{k - 1})^c$, for $N_k$ sufficiently large. Furthermore, on $\Omega_k^c$, the parameters $\big( N_k; \frac{\varepsilon}{2}; 1; \varsigma_k; \alpha; H |_{\partial R_k}; g_k \big)$ are $\lambda_k$-confined and $g_k$ is $\varsigma_k^{10}$-nearly linear with slope $(s_k, t_k)$ on $\partial \mathcal{B}$.  
		
		Thus \Cref{estimaten2n} applies upon restricting to $\Gamma (v_0; N_{k - 1})^c \cap \Omega_k^c$. Since $\varsigma_{k - 1}^{\alpha} \lambda_{k - 1} \le \frac{\lambda_k}{8}$ (recall \Cref{lambdak}), that proposition yields a constant $C = C (\varepsilon) > 1$ and an event $\widetilde{\Gamma}_k$ with $\mathbb{P}[\widetilde{\Gamma}_k] < C N_k^{-D}$, such that \eqref{g1knkvlambdak} and the second condition in \Cref{gammakdefinition} both hold on $\widetilde{\Gamma}_k^c \cap \Gamma (v_0; N_{k - 1})^c \cap \Omega_k^c$. Since we also have $\Gamma (v_0; N_{k - 1})^c \cap \Omega_k^c = \Gamma (v_0; N_{k - 1})^c \cap \bigcap_{j = 0}^k \Omega_j^c$ (as $\bigcup_{j = 0}^{k - 1} \Omega_j \subseteq \Gamma_{k - 1} \subseteq \Gamma (v_0; N_{k - 1})$), it follows that $\widetilde{\Gamma}_k^c \cap \Gamma (v_0; N_{k - 1})^c \cap \Omega_k^c \subseteq \bigcap_{j = 0}^k \Omega_j^c$. 
		
		Therefore, $\widetilde{\Gamma}_k^c \cap \Gamma (v_0; N_{k - 1})^c \cap \Omega_k^c$ satisfies the two conditions listed in \Cref{gammakdefinition}, and so $\Gamma_k \subseteq \widetilde{\Gamma}_k \cup \Gamma (v_0; N_{k - 1}) \cup \Omega_k$. It follows that $\mathbb{P} \big[ \Gamma_k \cap \Gamma (v_0; N_{k - 1})^c \cap \Omega_k^c \big] \le  \mathbb{P} \big[ \widetilde{\Gamma}_k \big] < C N_k^{-D}$, and so $\mathbb{P} \big[ \Gamma_k \cap \Gamma (v_0; N_{k - 1})^c \big] \le C N_k^{-D}$ by \Cref{omegakgammakprobability}. 
	\end{proof}

	Now we can establish \Cref{heightlocal1}. 
	
	\begin{proof}[Proof of \Cref{heightlocal1}]
		
		If $k \in [1, K]$ denotes the integer such that $M \in (N_{k + 1}, N_k]$, then define the event $\Gamma = \Gamma (v_0; M) = \Gamma (v_0; N_k)$ (from \Cref{gammakdefinition}) and the pair $(s, t) = \big( s (v_0; M), t (v_0; M) \big) = (s_k, t_k) = \nabla G_k (0, 0)$. 
		
		Let us first bound $\mathbb{P} [\Gamma]$. To that end, observe since \Cref{gammakdefinition} implies that $\Gamma (v_0; N_k) = \Omega_0 \cup \bigcup_{j = 1}^k \big( \Gamma_j \cap \Gamma (v_0; N_{j - 1})^c \big)$, \Cref{probabilityevent0} and \Cref{probabilitygammakgammav0nk1} together yield a constant $C_4 = C_4 (\varepsilon, D) > 1$ such that 
		\begin{flalign*} 
		\mathbb{P} \big[ \Gamma (v_0; N_k) \big] \le \displaystyle\sum_{j = 1}^k \mathbb{P} \big[ \Gamma_j \cap \Gamma (v_0; N_{j - 1})^c \big] + \mathbb{P} [\Omega_0] \le C_4 \displaystyle\sum_{j = 1}^k N_j^{-D} + C_4 \exp (-N^{1 / 2}),
		\end{flalign*} 
		
		\noindent and so $\mathbb{P} [\Gamma] \le C_2 M^{-D}$, for some constant $C_2 = C_2 (\varepsilon, D) > 1$. This verifies the first bound in \eqref{gammahnv0estimate}; the second follows from \eqref{sktk1} and the fact that $\Gamma (v_0; N_{k - 1}) \subseteq \Gamma (v_0; N_k)$.  
		
		It remains to verify \eqref{mhuhv0estimate}. To that end, set $c_0 = \frac{1}{18000}$ (as in \Cref{estimatesktk}); recall that $c = \frac{1}{20000}$; and observe since $\Gamma (v_0; N_{k - 1}) \subseteq \Gamma (v_0; N_k) = \Gamma$ that, for any vertex $u \in \mathcal{B}_M (v_0) \cap \mathbb{T}$ and $N_k$ sufficiently large, we have  
		\begin{flalign*}
		\Big|  M^{-1} & \big( H(u) - H(v_0) \big) - M^{-1} (u - v_0) \cdot (s, t) \Big| \textbf{1}_{\Gamma^c} \\ 
		& \le 64 \bigg| N_{k - 1}^{-1} \big( H(u) - H(v_0) \big) - \Big( G_{k - 1} \big( N_{k - 1}^{-1} (u - v_0) \big) - G_{k - 1} (0, 0) \Big) \bigg| \textbf{1}_{\Gamma^c} \\
		& \qquad + 64 \Big| G_{k - 1} \big( N_{k - 1}^{-1} (u - v_0) \big) - G_{k - 1} (0, 0) -  N_{k - 1}^{-1} (u - v_0) \cdot (s_k, t)  \Big| \textbf{1}_{\Gamma^c} \\
		& \le 128 \displaystyle\max_{w \in \mathbb{V}(R_{k - 1})} \Big| N_{k - 1}^{-1} H(w) - G_k \big( N_{k - 1}^{-1} (w - v_0)  \big) \Big| \textbf{1}_{\Gamma^c} + 8 \displaystyle\sup_{z \in \mathcal{B}_{1 / 8}} \big| \nabla G_{k - 1} (z) - (s_k, t_k) \big| \textbf{1}_{\Gamma^c} \\
		& \le 768 \lambda_{k - 1} + 8 (\log N_k)^{-1 - c_0} \le 768 C_3 (\log N_{k - 1})^{-1 - \alpha} + 8 (\log N_k)^{-1 - c_0} \le (\log M)^{-1 - c}.
		\end{flalign*} 
		
		\noindent Here, to deduce the first and second estimates we used the fact that $\frac{N_{k - 1}}{64} = N_{k + 1} < M \le N_k = \frac{N_{k - 1}}{8}$; to deduce the third we used \eqref{g1knkvlambdak}, the fact that $k \ge 2 (\log \log N)^2$ (since $M \le N \exp \big( -5 (\log \log N)^2 \big)$), and \eqref{sktkg1k1}; to deduce the fourth we used \eqref{kestimatelambda}; and to deduce the fifth we used the fact that $M$ is sufficiently large. This implies that $\Gamma$ and $(s, t)$ satisfy \eqref{gammahnv0estimate} and \eqref{mhuhv0estimate}.
	\end{proof}

	\section{The Global Law Under General Boundary Data} 
	
	\label{GlobalLaw}
	
	In this section we establish the global law \Cref{heightapproximate}, which holds under general boundary data subject to \Cref{regularestimate}. Its proof will closely follow that of Theorem 1.1 of \cite{VPDT}, except for an effective variant of Rademacher's theorem (provided by \Cref{twodimensionalapproximatelinear} below) that gives rise to the $(\log N)^{-c}$ type error bound. As in \cite{VPDT}, we will proceed by counting the number of height functions approximating some given $F$. To that end, we will select a triangulation of the domain $R$ such that $F$ is approximately linear on (the boundaries of) most triangles in this triangulation; estimate the number of free tilings on each triangle with given, nearly linear, boundary data; and then sum over all triangles and determine which height function $F$ is asymptotically most common. 
	
	We begin in \Cref{EstimateTriangle} by providing an effective estimate for the number of free tilings of a large triangle, whose boundary height function is approximately linear of some given slope $(s, t)$; this will proceed as in Sections 3 and 4 of \cite{VPDT}, where a similar count was given but effective error bounds were not provided. Next, the ``approximate linearity'' of $F$ on ``most triangles'' in the triangulation was justified in \cite{VPDT} through Rademacher's theorem. However, since this result does not immediately appear to provide an explicit error estimate, we establish a variant of this result in \Cref{AlmostLinearEstimate} with an effective error bound of order $(\log N)^{-c}$ that is sufficient for our purposes. Using these results, we then prove \Cref{heightapproximate} in \Cref{ProofApproximateGlobal}.

	\subsection{Estimates for the Number of Tilings of a Triangle}

	\label{EstimateTriangle}

	In this section we establish the following result, which is an effective analog of Corollary 4.2 of \cite{VPDT} that estimates the number of tilings of a triangle, whose boundary height function is nearly linear of a given slope. In what follows, we recall the set $\mathfrak{G} (h)$ from \Cref{gh} and the surface tension $\sigma$ from \eqref{lsigma}.

	\begin{prop}
		
		\label{numbersigmat}
		
		There exists a constant $C > 1$ such that the following holds. Fix an integer $N \in \mathbb{Z}_{\ge 1}$; a real number $\varsigma \in \big( N^{-1 / 75}, 1 \big)$; a pair $(s, t) \in \overline{\mathcal{T}}$; and a linear function $g: \mathbb{R}^2 \rightarrow \mathbb{R}$ of slope $(s, t)$. Let $T = T_N \subset \mathbb{T}$ denote the either one of the triangles $\big\{ (x, y) \in \mathbb{Z}_{\ge 0}^2: 0 \le x \le y \le N \big\}$ or $\big\{ (x, y) \in \mathbb{Z}_{\ge 0}^2: 0 \le y \le x \le N \big\}$, and let $h: \partial T \rightarrow \mathbb{Z}$ denote a boundary height function such that $\max_{z \in \partial T} \big| h(z) - g(z) \big| < \varsigma N$. If $\big| \mathfrak{G} (h)\big| \ne 0$, then 
		\begin{flalign}
		\label{ghestimate}
		\sigma (s, t) - C \varsigma^{1 / 3} \le 2 N^{-2} \log \big| \mathfrak{G} (h) \big| \le \sigma (s, t) + C \varsigma^{1 / 3}.
		\end{flalign}
		
		\noindent Moreover,  
		\begin{flalign}
		\label{ghestimatesum}
		\sigma (s, t) - C \varsigma^{1 / 3} \le 2 N^{-2} \log \left( \displaystyle\sum_{h} \big| \mathfrak{G} (h) \big| \right) \le \sigma (s, t) + C \varsigma^{1 / 3}.
		\end{flalign}
		
		\noindent where $h$ is summed over all $h: \partial T \rightarrow \mathbb{Z}$ such that $\max_{z \in \partial T} \big| h(z) - g(z) \big| < \varsigma N$.

	\end{prop}

	To establish \Cref{numbersigmat}, we first approximately enumerate tilings of an $N \times N$ torus. This result is stated as \Cref{numbersigma}; as in Sections 7 and 8 of \cite{VPDT}, its proof (to appear in \Cref{ProofstPn} below) will be largely facilitated by an exact determinantal identity for this specific count. We then use comparison and concentration estimates to asymptotically express counts of torus tilings in terms of counts of triangle free tilings whose boundary height functions approximately coincide, thereby establishing \Cref{numbersigmat}. 
	
	We begin by introducing some notation on torus tilings. To that end, for each $N \in \mathbb{Z}_{\ge 1}$, let $P_N$ denote the $N \times N$ discrete torus on the triangular lattice, on which we identify vertices $(x, y), (x', y') \in \mathbb{T}$ if and only if $N$ divides both $x - x'$ and $y - y'$. For any tiling $\mathscr{M}$ of $P_N$, let $\mathcal{N}_1 (\mathscr{M})$, $\mathcal{N}_2 (\mathscr{M})$, and $\mathcal{N}_3 (\mathscr{M})$ denote the numbers of type $1$, type $2$, and type $3$ lozenges in $\mathscr{M}$, respectively (recall \Cref{TilingsHeight}). Since $P_N$ has $2 N^2$ faces, we have $\mathcal{N}_1 (\mathscr{M}) + \mathcal{N}_2 (\mathscr{M}) + \mathcal{N}_3 (\mathscr{M}) = N^2$.
	
	For any $\omega \in \mathbb{R}_{> 0}$ and $(s, t) \in \overline{\mathcal{T}}$, let 
	\begin{flalign}
	\label{estomegapn} 
	\mathfrak{E} (s, t; \omega; P_N) = \bigg\{ \mathscr{M} \in \mathfrak{E} (P_N): \quad \displaystyle\frac{\mathcal{N}_1 (\mathscr{M})}{N^2} \in [s - \omega, s + \omega]; \quad \displaystyle\frac{\mathcal{N}_2 (\mathscr{M})}{N^2} \in [t - \omega, t + \omega] \bigg\}.
	\end{flalign} 
	
	Recalling the function $\sigma$ from \eqref{lsigma}, it was shown as Theorem 9.2 of \cite{VPDT} (see also Theorem 8 of \cite{OD}) that $N^{-2} \log \big| \mathfrak{E} (s, t; \omega; P_N) \big| \approx \sigma (s, t)$, for large $N$, small $\omega$, and any $(s, t) \in \overline{\mathcal{T}}$. The following proposition makes that result effective, by approximating $N^{-2} \log \big| \mathfrak{E} (s, t; \omega; P_N) \big|$ to within $N^{-1 / 75}$ if $\omega = N^{-1 / 70}$. Its proof will be given in \Cref{EstimatePn2} below, closely following that of Theorem 9.2 of \cite{VPDT}.
	
	\begin{prop}		
		\label{numbersigma}
		
		For sufficiently large $N$ and any $(s, t) \in \overline{\mathcal{T}}$, we have that
		\begin{flalign}
		\label{steestimate}
		\sigma (s, t) - N^{-1 / 75} \le N^{-2} \log \big| \mathfrak{E} (s, t; N^{-1 / 70}; P_N) \big| \le \sigma (s, t) + N^{-1 / 75}.
		\end{flalign}
		
	\end{prop}

	Given \Cref{numbersigma}, to establish \Cref{numbersigmat} it will be useful to compare the number of height functions with similar boundary data on a given domain. To that end, we have the following lemma, which appears as Proposition 3.6 of \cite{VPDT}; it essentially states that the number of tilings of a domain, whose boundary height function is nearly linear of given slope, is approximately independent of the specific boundary conditions. Although it was stated in \cite{VPDT} with the parameter $\varsigma$ below as a small constant, it can be quickly verified from the proofs there that one can in fact take $\varsigma > N^{-1 / 4}$, as indicated below.

	\begin{lem}[{\cite[Proposition 3.6]{VPDT}}]
		
		\label{perturbfunctionboundary}
		
		For any $\varepsilon > 0$, there exists a constant $C = C (\varepsilon) > 1$ such that the following holds. Fix an integer $N \in \mathbb{Z}_{\ge 1}$; a real number $\varsigma \in (N^{-1 / 4}, 1)$; and a simply-connected, convex domain $R \subset \mathbb{T}$ of area at least $\varepsilon N^2$ and diameter at most $\varepsilon^{-1} N$. Further let $h_1, h_2: \partial R \rightarrow \mathbb{Z}$ denote two boundary height functions on $R$, such that $\big| \mathfrak{G} (h_i) \big| \ne 0$ for each $i \in \{ 1, 2 \}$. If there exists a linear function $g: \mathbb{R}^2 \rightarrow \mathbb{R}$ such that $\max_{z \in \partial R} \big| h_i (z) - g (z) \big| < \varsigma N$ for each $i \in \{ 1, 2 \}$, then $N^{-2} \big| \log \big| \mathfrak{G} (h_1) \big| - \log \big| \mathfrak{G} (h_2) \big| \big| < C \big| \varsigma^{1 / 2} \log \varsigma \big|$. 
		
	\end{lem}

	Next, observe that any lozenge tiling of the $N \times N$ torus $P_N$ gives rise to a free tiling $\mathscr{M}$ (recall the terminology from below \Cref{domainabc}) of the $N \times N$ square $S_N = [0, N] \times [0, N] \subset \mathbb{T}$, which is associated with a (unique) height function $H : \mathbb{V}(S_N) \rightarrow \mathbb{Z}$ such that $H(0, 0) = 0$. The following lemma, which is similar to Proposition 3.4 of \cite{VPDT}, provides a concentration estimate for this height function.

	\begin{lem}
		
	\label{heightestimatep}
	
	Fix an integer $N \in \mathbb{Z}_{\ge 1}$; a real number $\varsigma \in (0, 1)$; and a pair $(s, t) \in \overline{\mathcal{T}}$. Let $\mathscr{M} \in \mathfrak{E} (s, t; \varsigma; P_N)$ be sampled uniformly at random, and let $H: \mathbb{V}(S_N) \rightarrow \mathbb{Z}$ denote the height function on $S_N = [0, N] \times [0, N]$ associated with $\mathscr{M}$, such that $H(0, 0) = 0$. Then, for any $r \in \mathbb{R}_{> 0}$, we have  
	\begin{flalign*} 
	\displaystyle\max_{(x, y) \in [0, N]^2} \mathbb{P} \Big[ \big| H (x, y) - s x - ty \big| > 8 r \big( |x| + |y| \big)^{1 / 2} + 2 \varsigma N \Big] \le 2 e^{-2r^2}.
	\end{flalign*}

	\end{lem}
	
	\begin{proof}
		
		This bound will follow from a suitable application of the concentration estimate \Cref{hvuexpectation}, to which end we must first evaluate the expectation $\mathbb{E} \big[ H (x, y) \big] - \mathbb{E} \big[ H (0, 0) \big]$. So, for any $(x, y) \in \mathbb{V}(S_N)$, let $p_1 (x, y)$ and $p_2 (x, y)$ denote the probabilities that there is a type $1$ lozenge in $\mathscr{M}$ with center $\big( x + \frac{1}{2}, y \big)$ and that there is a type $2$ lozenge in $\mathscr{M}$ with center $\big( x, y + \frac{1}{2} \big)$, respectively. Then, $p_1 (x, y) = \mathbb{E} \big[ \mathbb{H} (x + 1, y) \big] - \mathbb{E} \big[ H(x, y) \big]$ and $p_2 (x, y) = \mathbb{E} \big[ H(x, y + 1)\big] - \mathbb{E} \big[ H(x, y) \big]$. 
		
		Next, since the uniform measure on $\mathfrak{E} (s, t; \varsigma; P_N)$ is translation-invariant, $p_1 = p_1 (x, y)$ and $p_2 = p_2 (x, y)$ are independent of $(x, y)$; so, $\mathbb{E} \big[ H (x, y) \big] - \mathbb{E} \big[ H (0, 0) \big] = p_1 x + p_2 y$, for any $(x, y) \in \mathbb{V}(S_N)$. Moreover, since $\mathscr{M} \in \mathfrak{E} (s, t; \varsigma; P_N)$, we have that $p_1 \in [s - \varsigma, s + \varsigma]$ and $p_2 \in [t - \varsigma, t + \varsigma]$, and thus $\big| \mathbb{E} \big[ H(x, y) \big] - sx - ty \big| \le \varsigma (x + y)$. Inserting this bound into \eqref{hxyh00m} then yields the lemma.\footnote{Although \Cref{hvuexpectation} was not stated to hold with respect to the uniform measure on $\mathfrak{E} (s, t; \varsigma; P_N)$, with $H(0, 0)$ conditioned to equal $0$, it can be quickly verified from the proof of Proposition 22 of \cite{LSRT} (very similar to that of \Cref{walksexpectationnear}) that it does.}
	\end{proof}

	Now we can establish \Cref{numbersigmat}. 
	
	\begin{proof}[Proof of \Cref{numbersigmat}]
		
		First observe that \eqref{ghestimate} and \eqref{ghestimatesum} are equivalent (if a larger value of $C$ is taken in \eqref{ghestimatesum}), due to \Cref{perturbfunctionboundary} and the fact that there are at most $(2 \varsigma N + 1)^{4N}$ possible boundary height functions $h: \partial T \rightarrow \mathbb{Z}$ for which $\max_{z \in \partial T} \big| h(z) - g(z) \big| < \varsigma N$. Thus, we only establish \eqref{ghestimatesum}. Again by \Cref{perturbfunctionboundary} (and the H\"{o}lder continuity of $\sigma$ stated in \Cref{concavesigmat}), we may assume that $g(0, 0), g(N, N) \in \mathbb{Z}$. As in the proof of Corollary 4.2 of \cite{VPDT}, we will prove \eqref{ghestimatesum} by comparing $\big| \mathfrak{G} (h) \big|$ to $\big| \mathfrak{E} (s, t; \omega; P_N) \big|$for suitable $\omega$. 
		
		To that end, define the two triangles $T^{(1)} = T_N^{(1)} = \big\{ (x, y) \in \mathbb{Z}_{\ge 0}^2: 0 \le x \le y \le N \big\} \subset \mathbb{T}$ and $T^{(2)} = T_N^{(2)} = \big\{ (x, y) \in \mathbb{Z}_{\ge 0}^2: 0 \le y \le x \le N \big\} \subset \mathbb{T}$. Then, fix an index $i \in \{ 1, 2 \}$; set $j = 3 - i = \{ 1, 2 \} \setminus \{ i \}$; and suppose that $h^{(i)}: \partial T^{(i)} \rightarrow \mathbb{Z}$ is a boundary height function satisfying $\max_{z \in \partial T^{(i)}} \big| h^{(i)} (z) - g(z) \big|< \varsigma N$. Then, the boundary height function $h^{(j)}: \partial T^{(j)} \rightarrow \mathbb{Z}$ defined by setting $h^{(j)} (x, y) = g(0, 0) + g (N, N) - h^{(i)} (N - x, N - y)$ for each $(x, y) \in \partial T^{(j)}$ also satisfies $\max_{z \in \partial T^{(j)}} \big| h^{(j)} (z) - g(z) \big|< \varsigma N$. Therefore,  
		\begin{flalign}
		\label{gt1gt2}
		\displaystyle\sum_{h^{(1)}} \Big| \mathfrak{G} \big( h^{(1)} \big) \Big| = \displaystyle\sum_{h^{(2)}} \Big| \mathfrak{G} \big( h^{(2)} \big) \Big|,
		\end{flalign}
		
		\noindent where, for each $i \in \{ 1, 2 \}$, $h^{(i)}: \partial T^{(i)} \rightarrow \mathbb{Z}$ is summed over all boundary height functions such that $\max_{z \in \partial T^{(i)}} \big| h^{(i)} (z) - g(z) \big| < \varsigma N$.
		
		Now let $\mathfrak{F} = \mathfrak{F} (s, t; \varsigma; N) \subset \mathfrak{E} (P_N)$ denote the subset of tilings $\mathscr{M}$ of $P_N$ such that 
		\begin{flalign*}
		\displaystyle\max_{x, y \in [0, N]} \big| H (x, y) - H(0, 0) - sx - ty \big| \le \displaystyle\frac{\varsigma N}{2},
		\end{flalign*}
		
		\noindent where $H: \mathbb{V}(S_N) \rightarrow \mathbb{Z}$ denotes the height function associated with $\mathscr{M}$ such that $H(0, 0) = 0$ and $S_N = [0, N] \times [0, N]$. Then for sufficiently large $N$ we have that 
		\begin{flalign*}
		\displaystyle\frac{\Big| \mathfrak{E} \big( s, t; \frac{\varsigma}{6}, P_N \big) \Big|}{2} \le |\mathfrak{F}| \le \big| \mathfrak{E} (s, t; \varsigma, P_N) \big|.
		\end{flalign*}
		
		\noindent Here, the first estimate follows from \Cref{heightestimatep} and the fact that $\varsigma > N^{- 1 / 75}$. The second follows from the definition \eqref{estomegapn} of $\mathfrak{E} (s, t; \omega; P_N)$ and the fact that $H (x + 1, y) - H(x, y) = 1$ (or $H (x, y + 1) - H(x, y) = 1$) if and only if $\big( x + \frac{1}{2}, y \big)$ is the center of a type $1$ lozenge (or $\big( x, y + \frac{1}{2} \big)$ is that of a type $2$ lozenge, respectively).
		
		Next, observe that any element of $\mathfrak{F}$ (after translating the associated height function, if necessary) gives rise to a pair of tilings $\mathscr{M}^{(1)}$ on $T^{(1)}$ and $\mathscr{M}^{(2)}$ on $T^{(2)}$, whose boundary height functions $h^{(i)}: \partial T^{(i)} \rightarrow \mathbb{Z}$ satisfy $\max_{z \in \partial T} \big| h^{(i)} (z) - g(z) \big| < \varsigma N$, for each $i \in \{ 1, 2 \}$. Thus,  
		\begin{flalign}
		\label{h1h2s}
		\displaystyle\sum_{h^{(1)}} \Big| \mathfrak{G} \big( h^{(1)} \big) \Big| \displaystyle\sum_{h^{(2)}} \Big| \mathfrak{G} \big( h^{(2)} \big) \Big| \ge |\mathfrak{F}| \ge \displaystyle\frac{\Big| \mathfrak{E} (s, t; \frac{\varsigma}{6}, P_N \big) \Big|}{2},
		\end{flalign}
		
		\noindent where $h^{(i)}: \partial T^{(i)} \rightarrow \mathbb{Z}$ is again summed over all functions for which $\max_{z \in \partial T} \big| h^{(i)} (z) - g(z) \big| < \varsigma N$, for each $i \in \{ 1, 2 \}$. Thus, \eqref{h1h2s}; \eqref{gt1gt2}; \Cref{numbersigma}; the fact that $\varsigma > N^{-1 / 75} > 6 N^{-1 / 70}$; and the existence of a constant $C > 1$ such that 
		\begin{flalign*}
		\displaystyle\max_{|s - s_0| \le \varsigma} \displaystyle\max_{|t - t_0| \le \varsigma} \sigma (s, t) - \sigma (s_0, t_0) \le C \varsigma^{1 / 2},
		\end{flalign*}
		
		\noindent holds for any $(s_0, t_0) \in \mathcal{T}$ and $\varsigma \in (0, 1)$ (recall \Cref{concavesigmat}) together imply the first bound in \eqref{ghestimatesum}. 
		
		To establish the second bound in \eqref{ghestimatesum}, fix some $\mathscr{M} \in \mathfrak{F}$. As mentioned above, $\mathscr{M}$ gives rise to a pair of tilings $\mathscr{M}^{(1)}$ on $T^{(1)}$ and $\mathscr{M}^{(2)}$ on $T^{(2)}$, whose boundary height functions $h^{(i)}: \partial T^{(i)} \rightarrow \mathbb{Z}$ satisfy $\max_{z \in \partial T} \big| h^{(i)} (z) - g(z) \big| < \varsigma N$, for each $i \in \{ 1, 2 \}$. For fixed such $h^{(1)}$ and $h^{(2)}$, any pair of tilings $\mathscr{M}^{(1)}$ and $\mathscr{M}^{(2)}$ on $T^{(1)}$ and $T^{(2)}$ with boundary height functions $h^{(1)}$ and $h^{(2)}$, respectively, gives rise to a unique tiling $\mathscr{M} \in \mathfrak{E} (s, t; \varsigma; P_N)$, and so 
		\begin{flalign}
		\label{gh1gh2stpn}
		\Big| \mathfrak{G} \big( h^{(1)} \big) \Big| \Big| \mathfrak{G} \big( h^{(2)} \big) \Big| \le \big| \mathfrak{E} (s, t; \varsigma; P_N) \big|. 
		\end{flalign}
		
		\noindent Therefore, the upper bound in \eqref{ghestimatesum} follows from \eqref{gh1gh2stpn}; \Cref{perturbfunctionboundary}; \eqref{gt1gt2}; the fact that there are at most $(2 \varsigma N + 1)^{4N}$ possibilities for each $h^{(i)}$ satisfying $\max_{z \in \partial T} \big| h^{(i)} (z) - g(z) \big| < \varsigma N$; \Cref{numbersigma}; the fact that $\varsigma > N^{- 1 / 75} > N^{-1 / 70}$; and the existence of a constant $C > 1$ such that 
		\begin{flalign*}
		\sigma (s_0, t_0) - \displaystyle\min_{|s - s_0| \le \varsigma} \displaystyle\min_{|t - t_0| \le \varsigma} \sigma (s, t) \le C \varsigma^{1 / 2},
		\end{flalign*}
		
		\noindent for any $(s_0, t_0) \in \mathcal{T}$ and $\varsigma \in (0, 1)$. This establishes \eqref{ghestimatesum} and therefore the proposition. 
	\end{proof}

\subsection{Triangulations With Approximately Linear Boundary Data}

\label{AlmostLinearEstimate}

In order to establish \Cref{heightapproximate}, we will estimate the number of tilings of $R$ whose height function approximates a given one $F: \mathbb{V}(R) \rightarrow \mathbb{Z}$. To that end, we first select a triangulation of $R$ so that $F$ is nearly linear on the boundaries of most triangles, and then apply \Cref{numbersigmat} to approximately count the number of tilings of each triangle. The existence of such a triangulation in \cite{VPDT} was a consequence of Rademacher's theorem. In this section we establish an effective variant of Rademacher's theorem that gives rise to explicit error estimates. 

To state this result, we first require the following definition that quantifies the approximate linearity of a function.

\begin{definition}
	
	\label{rlinear}
	
	Fix $a < b \in \mathbb{R}$ and $\varsigma \in \mathbb{R}_{> 0}$. A function $f: [a, b] \rightarrow \mathbb{R}$ is \textit{$\varsigma$-linear on $[a, b]$} if
	\begin{flalign*}
	\displaystyle\sup_{x \in [a, b]} \displaystyle\frac{\big|f(x) - f_{[a, b]} (x) \big|}{b - a} \le \varsigma, \quad \text{where} \quad f_{[a, b]} (x) = \displaystyle\frac{(b - x) f(a) + (x - a) f(b)}{b - a},
	\end{flalign*}
	
	\noindent is the linear function on $\mathbb{R}$ with $f_{[a, b]} (a) = f(a)$ and $f_{[a, b]} (b) = f(b)$. Moreover, if $T \subset \mathbb{R}^2$ is a triangle, then a function $f: \partial T \rightarrow \mathbb{R}$ is \textit{$\varsigma$-linear on $\partial T$} if is $\varsigma$-linear on each of the three sides of $\partial T$. 
	
\end{definition}

The following proposition now provides the effective variant of Rademacher's theorem that will be required for our purposes. 

\begin{prop}
	\label{twodimensionalapproximatelinear}
	
	There exists a constant $C > 1$ such that the following holds. Fix integers $n, v \in \mathbb{Z}_{\ge 1}$ with $C < v < 2v < n$, and set $N = 2^n$. Define the triangle $T_N = \big\{ (x, y) \in \mathbb{R}_{\ge 0}^2: 0 \le x \le y \le N  \big\} \subset \mathbb{R}^2$, and let $F: T_N \rightarrow \mathbb{R}$ denote an $1$-Lipschitz function. 
	
	Then, there exists an integer $m \in [v, 2v]$ with the following property. Set $M = 2^m$ and let $\{ \mathscr{F}_i \}_{i \in \mathcal{I}}$ denote the set of faces of $2^{n - m} \mathbb{T} = M^{-1} N \mathbb{T}$ that are entirely contained in $T_N$, for some index set $\mathcal{I}$. Then there exists a subset $\mathcal{J} \subseteq \mathcal{I}$, with $|\mathcal{J}| \ge |\mathcal{I}| - M^2 (\log M)^{- 1 / 10}$, such that $F$ is $(\log M)^{-1 / 10}$-linear on $\partial \mathscr{F}_j$, for each $j \in \mathcal{J}$.
\end{prop}

To establish \Cref{twodimensionalapproximatelinear}, we first introduce the following weaker version of $\varsigma$-linearity that will be a bit more convenient to analyze. 

\begin{definition} 
	
\label{semilineardefinition}

Fix $a, b \in \mathbb{R}$ with $a < b$. A function $f: [a, b] \rightarrow \mathbb{R}$ is \textit{$\varsigma$-semilinear on $[a, b]$} if
\begin{flalign*}
\displaystyle\frac{1}{b - a} \Bigg| f \bigg( \frac{a + b}{2} \bigg) - \bigg( \displaystyle\frac{f(a) + f(b)}{2} \bigg) \Bigg| \le \varsigma.
\end{flalign*}

\end{definition}

We next require a certain family of dyadic intervals. 

\begin{definition} 

\label{ikpqik} 

Fix $a, b \in \mathbb{R}$ with $a < b$. For any integers $i, k \in \mathbb{Z}_{\ge 0}$ such that $i \le 2^k$, set $p_{i, k} = p_{i, k; a, b} = a + i 2^{-k} (b- a)$. Further define the dyadic interval $Q_{i, k} = Q_{i, k; a, b} = [p_{i, k}, p_{i + 1, k}]$, for each $i \in [0, 2^k - 1]$.

\end{definition} 

Now we have the following lemma indicating that, if a $1$-Lipschitz function is $\varsigma$-semilinear on sufficiently many dyadic scales, then it is nearly $2 \varsigma$-linear.

\begin{lem} 
	
	\label{semilinearlinear}
	
	Let $m \in \mathbb{Z}_{\ge 0}$ denote an integer; $\varsigma \in \mathbb{R}_{> 0}$ denote a real number; and $f: [a, b] \rightarrow \mathbb{R}$ denote a $1$-Lipschitz function. If $f$ is $\varsigma$-semilinear on $Q_{i, k}$ for all integers $k \in [0, m]$ and $i \in [0, 2^k - 1]$, then $f$ is $(2 \varsigma + 2^{- m})$-linear on $[a, b]$. 
	
\end{lem}

\begin{proof} 
	
	Observe that the function $f_{[a, b]} (x)$ from Definition \ref{rlinear} is $1$-Lipschitz since $f$ is $1$-Lipschitz. Let $x_0 \in [a, b]$ be such that $\big| f(x_0) - f_{[a, b]} (x_0) \big|$ is maximized; there exists some integer $h \in [0, 2^m - 1]$ such that $x_0 \in Q_{h, m}$. Setting $p = \frac{p_{h, m} + p_{h + 1, m}}{2} = p_{2h + 1, m + 1}$ (which is the midpoint of the interval $Q_{h, m}$), we deduce from the $1$-Lipschitz properties of $f$ and $f_{[a, b]}$ that 
	\begin{flalign}
	\label{fxfpij}
	\big| f(x_0) - f_{[a, b]} (x_0) \big| \le \big| f (p) - f_{[a, b]} (p) \big| + 2^{- m} (a - b). 
	\end{flalign}
	
	\noindent Since $f$ is $\varsigma$-semilinear on each $Q_{i, k}$, we obtain for each $k \in [0, m]$ and $i \in [0, 2^k - 1]$ that
	\begin{flalign*}
	\big| f (p_{2i + 1, k + 1}) - f_{[a, b]} (p_{2i + 1, k + 1}) \big| & \le \displaystyle\frac{1}{2} \Big( \big| f (p_{i, k}) - f_{[a, b]} (p_{i, k}) \big| + \big| f (p_{i + 1, k}) - f_{[a, b]} (p_{i + 1, k}) \big| \Big) \\
	& \qquad + \Bigg|  f (p_{2i + 1, k + 1}) - \bigg( \displaystyle\frac{f (p_{i, k}) + f (p_{i + 1, k})}{2} \bigg) \Bigg| \\
	& \le \displaystyle\max_{0 \le j \le 2^k - 1} \big| f (p_{j, k}) - f_{[a, b]} (p_{j, k}) \big|  + 2^{-k} (a - b) \varsigma,
	\end{flalign*}
	
	\noindent from which we deduce upon maximizing over $i$ that
	\begin{flalign*}
	\displaystyle\max_{0 \le i \le 2^{k + 1} - 1} \big| f(p_{i, k + 1}) - f_{[a, b]} (p_{i, k + 1}) \big| \le \displaystyle\max_{0 \le j \le 2^k - 1}  \big| f(p_{j, k}) - f_{[a, b]} (p_{j, k}) \big| + 2^{-k} (a - b) \varsigma.
	\end{flalign*}
	
	\noindent Summing over all $k \in [0, m]$ and using the fact that $p = p_{2h + 1, m + 1}$, we obtain that
	\begin{flalign*}
	\big| f (p) - f_{[a, b]} (p) \big| \le  (a - b) \varsigma \sum_{k = 0}^m 2^{-k} < 2 (a - b) \varsigma,
	\end{flalign*}
	
	\noindent which upon insertion into \eqref{fxfpij} yields the lemma. 
\end{proof}

Now, to establish \Cref{twodimensionalapproximatelinear}, we will restrict the function $F$ to lines parallel to the sides of $T_N$ and exhibit many small scales on which these $1$-Lipschitz, one-dimensional restrictions are approximately linear. In view of \Cref{semilinearlinear}, we would therefore like to show that one-dimensional Lipschitz functions are often semilinear on small scales. The following definition introduces sets on which a function $f$ is not linear or semilinear. 

\begin{definition}
	
	\label{notlinearsets} 
	
	Fix $a, b \in \mathbb{R}$ with $a < b$, and let $f: [a, b] \rightarrow \mathbb{R}$ denote a function. For any $k \in \mathbb{Z}_{\ge 1}$ and $\varsigma \in \mathbb{R}_{> 0}$, let $\mathcal{A} (\varsigma; k) = \mathcal{A}_f (\varsigma; k)$ denote the set of indices $i \in [0, 2^k - 1]$ for which $f$ is not $\varsigma$-linear on $Q_{i, k}$. Similarly, let $\mathcal{S} (\varsigma; k) = \mathcal{S}_f (\varsigma; k)$ denote the set of indices $i \in [0, 2^k - 1]$ for which that $f$ is not $\varsigma$-semilinear on $Q_{i, k}$. 
	
	Furthermore, for any $\varsigma, \vartheta \in \mathbb{R}_{> 0}$ and integers $n > m > 0$, let $\mathcal{X} (\varsigma; \vartheta; m, n) = \mathcal{X}_f (\varsigma, \vartheta; m, n)$ denote the set of indices $k \in [m, n - 1]$ for which $\big| \mathcal{A} (\varsigma; k) \big| \ge \vartheta 2^k$. Similarly, let $\mathcal{Y} (\varsigma; \vartheta; m, n) = \mathcal{Y}_f (\varsigma; \vartheta; m, n)$ denote the set of $k \in [m, n - 1]$ for which $\big| \mathcal{S} (\varsigma; k) \big| \ge \vartheta 2^k$.
	
\end{definition}

In particular, $\mathcal{X}_f$ (or $\mathcal{Y}_f$) denotes the set of scales in $k \in [m, n)$ for which $f$ is not $\varsigma$-linear (or $\varsigma$-semilinear, respectively) on a $\vartheta$-proportion of the dyadic intervals $Q_{i, k}$. In order to establish \Cref{twodimensionalapproximatelinear}, we will show that $|\mathcal{X}_f|$ is small for $1$-Lipschitz $f$. The following lemma bounds $|\mathcal{Y}_f|$; an estimate on $|\mathcal{X}_f|$ is then deduced as a corollary below.

\begin{lem}
	
	\label{ylambda} 
	
	If is $f$ is $1$-Lipschitz, then $\big| \mathcal{Y} (\varsigma, \vartheta; m, n) \big| \le (\varsigma^2 \vartheta)^{-1}$ for any integers $n > m > 0$. 
\end{lem}

\begin{proof}
	
	By scaling, we may let $[a, b] = [0, 1]$. For each $k \in \mathbb{Z}_{\ge 1}$ and $i \in [0, 2^k - 1]$, denote 
	\begin{flalign*}
	c_{i, k} = 2^k \bigg( f \Big( \displaystyle\frac{i + 1}{2^k} \Big) - f \Big( \displaystyle\frac{i}{2^k} \Big) \bigg); \qquad d_{i, k} = 2^k \bigg( f \Big( \displaystyle\frac{i}{2^k} \Big) - 2 f \Big( \displaystyle\frac{2i + 1}{2^{k + 1}} \Big) + f \Big( \displaystyle\frac{i + 1}{2^k} \Big) \bigg).
	\end{flalign*}
	
	\noindent Then, $c_{2i + 1, k + 1} = c_{i, k}+ d_{i, k}$ and $c_{2i, k + 1} = c_{i, k} - d_{i, k}$, so that
	\begin{flalign}
	\label{aikcik}
	2 c_{i, k}^2 + 2 d_{i, k}^2 = c_{2i, k + 1}^2 + c_{2i + 1, k + 1}^2.
	\end{flalign}
	
	\noindent Repeated application of \eqref{aikcik} yields 
	\begin{flalign*}
	\displaystyle\sum_{k = 0}^{n - 1} 2^{n - k} \displaystyle\sum_{i = 0}^{2^k - 1} d_{i, k}^2 + 2^n c_{0, 0}^2 = \displaystyle\sum_{i = 0}^{2^n - 1} c_{i, n}^2 \le 2^n,
	\end{flalign*}
	
	\noindent where in the last inequality we used the fact that $c_{i, n} \le 1$ (since $f$ is $1$-Lipschitz). It follows that 
	\begin{flalign}
	\label{weightedsumbik}
	\displaystyle\sum_{k = 0}^{n - 1} 2^{-k} \displaystyle\sum_{i = 0}^{2^k - 1} d_{i, k}^2 \le 1. 
	\end{flalign}
	
	 Now, for each $k \in \mathbb{Z}_{\ge 1}$ and $\varsigma \in \mathbb{R}_{> 0}$, let $s_k = 2^{-k} \big| \mathcal{S}_f (\varsigma; k) \big|$. Then, since $d_{i, k} \ge \varsigma$ whenever $i \in \mathcal{S}_f (\varsigma, k)$, \eqref{weightedsumbik} implies that $\sum_{k = m}^{n - 1} s_k \le \sum_{k = 0}^{n - 1} s_k \le \varsigma^{-2}$. Hence, there are at most $(\varsigma^2 \vartheta)^{-1}$ indices $k \in [m, n - 1]$ for which $s_k \ge \vartheta$, and so we deduce that $\big| \mathcal{Y} (\varsigma, \vartheta; m, n) \big| \le (\varsigma^2 \vartheta)^{-1}$. 
\end{proof}

\begin{cor}
	
	\label{onedimensionallinearscales}
	
	If $f$ is $1$-Lipschitz, then $\big| \mathcal{X} (\varsigma; \vartheta; m, n ) \big| \le 1944 (\varsigma^5 \vartheta)^{-1}$ for any integers $n > m > 0$. 
	
\end{cor}

\begin{proof}
	
	We will proceed by first comparing $\big| \mathcal{X} (\varsigma; \vartheta; m, n) \big|$ to $\big| \mathcal{Y} (\varsigma; \vartheta; m, n) \big|$ (with slightly altered parameters), and then applying \Cref{ylambda}. Since the $1$-Lipschitz property of $f$ implies that it is $1$-linear on $[a, b]$, we may assume $\varsigma < 1$. Then, let $r \in \mathbb{Z}_{> 1}$ denote the integer for which $2^{-r} \le \frac{\varsigma}{3} < 2^{1 - r}$. 
	
	Now, suppose that $i \in \mathcal{A} (\varsigma; k)$, for some $k \in \mathbb{Z}_{\ge 1}$. Then, Lemma \ref{semilinearlinear} yields the existence of indices $j = j(i) \in [k, k + r]$ and $h = h(i) \in \big[ i 2^{j - k}, (i + 1) 2^{j - k} - 1 \big]$ such that $h \in \mathcal{S}_f \big( \frac{\varsigma}{3}, j \big)$; since the pair $(h, j)$ determines $i$, it follows that $\big| \mathcal{A} (\varsigma; k) \big| \le \sum_{j = k}^{k + r} \big| \mathcal{S} \big( \frac{\varsigma}{3}; j \big) \big|$. Thus, if $k \in \mathcal{X} (\varsigma; \vartheta; m, n)$, then 
	\begin{flalign*}
	\displaystyle\sum_{j = k}^{k + r} \Bigg|  \mathcal{S} \bigg( \frac{\varsigma}{3}; j \bigg) \Bigg| \ge \vartheta 2^k.
	\end{flalign*}
	
	\noindent This implies the existence of an index $j \in [k, k + r]$ for which 
	\begin{flalign}
	\label{sestimate}
	\Bigg| \mathcal{S} \bigg( \frac{\varsigma}{3}; j \bigg) \Bigg| \ge \displaystyle\frac{\vartheta 2^k}{r + 1} \ge \displaystyle\frac{\vartheta 2^{j - r}}{r + 1} \ge \displaystyle\frac{\varsigma^2 \vartheta 2^j}{36},
	\end{flalign}
	
	\noindent where in the third inequality we used the fact that $r + 1 \le 2^r \le 6 \varsigma^{-1}$.
	
	Thus, with any $k \in \mathcal{X} (\varsigma; \vartheta; m, n)$ we can associate an index $j = j(k) \in \mathcal{Y} \big( \frac{\varsigma}{3}; \frac{\varsigma^2 \vartheta}{36}; m, n + r \big)$ such that $k \le j \le k + r$. Since there are at most $r + 1$ such $j$ for any fixed $k$, we deduce that 
	\begin{flalign*}
	\big| \mathcal{X} (\varsigma; \vartheta; m, n) \big| \le (r + 1) \left| \mathcal{Y} \Big( \frac{\varsigma}{3}; \displaystyle\frac{\varsigma^2 \vartheta}{36}; m, n + r \Big) \right| \le (r + 1) \left( \displaystyle\frac{\varsigma^4 \vartheta}{324} \right)^{-1} \le 1944 (\varsigma^5 \vartheta)^{-1},
	\end{flalign*}
	
	\noindent where in the second inequality we used Lemma \ref{ylambda} to estimate $|\mathcal{Y}|$, and in the third we again used the fact that $r + 1 \le 2^r \le 6 \varsigma^{-1}$. This implies the lemma.
\end{proof}

Now we can establish \Cref{twodimensionalapproximatelinear}. 

\begin{proof}[Proof of \Cref{twodimensionalapproximatelinear}]
	
	Extend $F$ to an $1$-Lipschitz function on $\mathbb{R}^2$. Then, for any line $\ell \subset \mathbb{R}^2$ parallel to one of the three lines $\{ y = 0 \}$, $\{ x = 0 \}$, or $\{ y = x + 1 \}$, we define a function $f_{\ell}$ (which will essentially be the restriction of $F$ to $\ell$) as follows. First, if $\ell = \{ y = \eta \}$ for some $\eta \in \mathbb{R}$, then define $f_{\ell}: [0, N] \rightarrow \mathbb{R}$ by setting $f_{\ell} (z) = F (z, \eta)$ for each $z \in [0, N]$. Second, if $\ell = \{ x = \xi \}$ for some $\xi \in \mathbb{R}$, then define $f_{\ell}: [0, N] \rightarrow \mathbb{R}$ by setting $f_{\ell} (z) = F (\xi, z)$. Third, if $\ell = \{ y = x + \zeta \}$ for some $\zeta \in \mathbb{R}$, then define $f_{\ell}: [0, N \sqrt{2}] \rightarrow \mathbb{R}$ by setting $f_{\ell} (z) = F \big( \frac{z}{\sqrt{2}}, \zeta + \frac{z}{\sqrt{2}} \big)$. In either of these three cases, define the sets $\mathcal{X}_{\ell} = \mathcal{X}_{f_{\ell}} \big( v^{-1 / 9}; v^{-1 / 9}; v, 2v \big)$ and $\mathcal{A}_{\ell} (k) = \mathcal{A}_{f_{\ell}} \big( v^{-1 / 9}; k \big)$, for any $k \in [v, 2v)$, where we recall $\mathcal{X}$ and $\mathcal{A}$ from \Cref{notlinearsets}.
	
	Next, for each integer $j \in [0, 2^{2v}]$, define the lines 
	\begin{flalign*} 
	\ell_j^{(1)} = \{ y = j 2^{n - 2v} \}; \qquad \ell_j^{(2)} = \{ x = j 2^{n - 2v} \}; \qquad \ell_j^{(3)} = \{ y = x + j 2^{n - 2v} \}.
	\end{flalign*} 
	
	\noindent Then, for any fixed integer $k \in [v, 2v)$, all edges of $T_N \cap 2^{n - k} \mathbb{T}$ are among the $Q_{i, k}$-subintervals of the lines in the union
	\begin{flalign*} 
	\bigcup_{h = 0}^{2^k} \big( \ell_{2^{2v - k} h}^{(1)}  \cup \ell_{2^{2v - k} h}^{(2)} \cup \ell_{2^{2v - k} h}^{(3)} \big). 	
	\end{flalign*}

	Now set $\mathcal{X}_j^{(i)} = \mathcal{X}_{\ell_j^{(i)}}$ for each $i \in \{ 1, 2, 3\}$ and $j \in [0, 2^{2v}]$. Let us show that there exist an integer $m \in [v, 2v)$ and ``small'' integers $w_1, w_2, w_3 \in [0, 2^{2v - m}]$ such that
	\begin{flalign}
	\label{x1x2x3m} 
	m \notin \mathcal{X}_{2^{2v - m} h + w_1}^{(1)} \cup \mathcal{X}_{2^{2v - m} h + w_2}^{(2)} \cup \mathcal{X}_{2^{2v - m} h + w_3}^{(3)},
	\end{flalign} 
	
	\noindent holds for ``most'' $h \in [0, 2^m)$. 
	
	To that end, for any integers $i \in \{1, 2, 3 \}$; $k \in [v, 2v)$; and $j \in [0, 2^{2v - k})$, define the sets 
	\begin{flalign}
	\label{wjz}
	\begin{aligned} 
	\mathcal{W}_i (j; k) & = \Big\{ h \in [0, 2^k): k \in \mathcal{X}_{2^{2v - k} h + j}^{(i)} \Big\}; \quad  \mathfrak{J}_i (k) = \Big\{ j \in [0, 2^{2v - k}): \big| \mathcal{W}_i (j; k) \big| \ge v^{-1 / 9} 2^k \Big\}; \\
	& \qquad \qquad \qquad \qquad \quad \mathcal{Z}_i = \Big\{ k \in [v, 2v):  \big| \mathfrak{J}_i (k) \big| \ge v^{-1 / 9} 2^{2v - k} \Big\}. 
	\end{aligned} 
	\end{flalign} 
	
	We claim that $\mathcal{Z}_1 \cup \mathcal{Z}_2  \cup \mathcal{Z}_3 \ne [v, 2v)$; the integer $m$ in \eqref{x1x2x3m} will then be an element of $[v, 2v) \setminus (\mathcal{Z}_1 \cup \mathcal{Z}_2 \cup \mathcal{Z}_3)$. To that end, first observe that since $g_{\ell}$ is $1$-Lipschitz for any line $\ell$, setting $\varsigma = v^{-1 / 9} = \vartheta$ in \Cref{onedimensionallinearscales} and summing yields 
	\begin{flalign}
	\label{xsum}
	\displaystyle\sum_{i \in \{ 1, 2, 3 \}} \displaystyle\sum_{j = 0}^{2^{2v}} \big| \mathcal{X}_j^{(i)} \big| \le 5832 v^{2 / 3} (2^{2v} + 1). 
	\end{flalign}
	
	Then \eqref{wjz} implies that $| \mathcal{Z}_1| + | \mathcal{Z}_2 | + | \mathcal{Z}_3| \le 11664 v^{8 / 9} < v - 1$, for sufficiently large $v$. Indeed, there would otherwise exist at least $11664 v^{2 / 3} 2^{2v} > 5832 v^{2 /3 } (2^{2v} + 1)$ integer quadruples $(i, h, j, k) \in \{ 1, 2 , 3 \} \times [0, 2^k) \times [0, 2^{2v - k}) \times [v, 2v)$ such $k \in \mathcal{X}_{2^{2v - k} h + j}^{(i)}$, which would contradict \eqref{xsum}. So, for sufficiently large $v$ there exists an integer $m \in [v, 2v) \setminus (\mathcal{Z}_1 \cup \mathcal{Z}_2 \cup \mathcal{Z}_3)$, which therefore satisfies $\big| \mathfrak{J}_1 (m) \big| + \big| \mathfrak{J}_2 (m) \big| + \big| \mathfrak{J}_3 (m) \big| < 3 v^{-1 / 9} 2^{2v - m}$. So, if we define the interval $\mathcal{I}$ and sets $\mathcal{K}_i$ by
	\begin{flalign*}
	 \mathcal{I} = \big[ 0, 4 v^{-1 / 9} 2^{2v - m} \big] \cap \mathbb{Z}; \qquad \mathcal{K}_i = \mathcal{I} \setminus \big( \mathfrak{J}_i (m) \cap \mathcal{I} \big), \quad \text{for each $i \in \{ 1, 2 , 3 \}$},
	\end{flalign*}
	
	\noindent then there exist $w_1 \in \mathcal{K}_1$, $w_2 \in \mathcal{K}_2$, and $w_3 \in \mathcal{K}_3$. For this choice of $(w_1, w_2, w_3)$ and $m$, each $w_i \notin \mathfrak{J}_i (m)$, and so \eqref{x1x2x3m} holds for all but at most $3 v^{-1 / 9} 2^m$ indices $h \in [0, 2^m)$. 
	
	Next, set $\mathcal{W}_i = \mathcal{W}_i (w_i; m)$ for each $i \in \{ 1, 2, 3 \}$ and denote $\mathcal{W} = \mathcal{W}_1 \cup \mathcal{W}_2 \cup \mathcal{W}_3$. Then, \eqref{x1x2x3m} implies that $\big| \mathcal{A}_{\ell_{2^{2v - m} h + w_i}}^{(i)} (m) \big| \le v^{-1 / 9} 2^m$ holds, for any $i \in \{ 1, 2, 3 \}$ and $h \in [0, 2^m) \setminus \mathcal{W}$. This and the fact that $|\mathcal{W}| \le 3 v^{-1 / 9} 2^m$ together imply that, for all but $12 v^{-1 / 9} 2^{2m} = 12 v^{-1 / 9} M^2$ pairs $(h_1, h_2) \in [0, 2^m) \times [0, 2^m)$, $F$ is $v^{-1 / 9}$-linear on all three intervals 
	\begin{flalign}
	\label{intervalsf}
	\begin{aligned}
	&  \Big[ (2^{n - m} h_1 +  2^{n - 2v} w_1, 2^{n - m} h_2), \big(2^{n - m} h_1 + 2^{n - 2v} w_1, 2^{n - m} (h_2 + 1) \big) \Big]; \\
	&  \Big[  \big( 2^{n - m} h_1, 2^{n - m} (h_2 + 1) + 2^{n - 2v} w_2 \big), \big(2^{n - m} (h_1 + 1), 2^{n - m} (h_2 + 1) +  2^{n - 2v} w_2 \big)  \Big]; \\ 
	&  \Big[ (2^{n - m} h_1, 2^{n - m} h_2 + 2^{n - 2v} w_3), \big(2^{n - m} (h_1 + 1), 2^{n - m} (h_2 + 1) + 2^{n - 2v} w_3 \big) \Big]. \\
	\end{aligned} 
	\end{flalign} 
	
	\noindent We refer to \Cref{triangleslines} for a depiction of these intervals. 
	
	Now, observe that if $\varsigma, \rho \in \mathbb{R}_{> 0}$ are real numbers; $I = [a, b] \subset \mathbb{R}$ is some finite interval; $G: I \rightarrow \mathbb{R}$ is a $\varsigma$-linear function; and $H: I \rightarrow \mathbb{R}$ is a function satisfying $\sup_{z \in I} \big| G(z) - H (z) \big| \le \rho (b - a)$, then $H$ is $(\varsigma + 2 \rho)$-linear on $I$.  Combining this with the facts that $F$ is $1$-Lipschitz; that $F$ is $v^{-1 / 9}$-linear on the three intervals \eqref{intervalsf} for all but $12 v^{-1 / 9} M^2$ pairs $(h_1, h_2)$; and that $\max \big\{ |w_1|, |w_2|, |w_3| \big\} \le 4 v^{-1 / 9} 2^{2v - m}$, we deduce that $F$ is $9 v^{-1 / 9}$-linear on all three intervals 
	 \begin{flalign}
	 \label{t1boundary} 
	 \begin{aligned}	
	&  \Big[ (2^{n - m} h_1, 2^{n - m} h_2), \big(2^{n - m} h_1, 2^{n - m} (h_2 + 1) \big) \Big]; \\
	&  \Big[  \big( 2^{n - m} h_1, 2^{n - m} (h_2 + 1) \big), \big(2^{n - m} (h_1 + 1), 2^{n - m} (h_2 + 1) \big)  \Big]; \\ 
	&  \Big[ (2^{n - m} h_1, 2^{n - m} h_2), \big(2^{n - m} (h_1 + 1), 2^{n - m} (h_2 + 1) \big) \Big],
	\end{aligned} 
	\end{flalign}  
	
	\noindent for all but $12 v^{-1 / 9} M^2$ pairs $(h_1, h_2) \in [0, 2^m) \times [0, 2^m)$. Applying similar reasoning to the intervals 
	\begin{flalign} 
	\label{t2boundary} 
	\begin{aligned} 
	&  \Big[ (2^{n - m} h_1, 2^{n - m} h_2), \big(2^{n - m} (h_1 + 1), 2^{n - m} h_2 \big) \Big]; \\
	&  \Big[  \big( 2^{n - m} (h_1 + 1), 2^{n - m} h_2 \big), \big(2^{n - m} (h_1 + 1), 2^{n - m} (h_2 + 1) \big)  \Big]; \\ 
	&  \Big[ (2^{n - m} h_1, 2^{n - m} h_2), \big(2^{n - m} (h_1 + 1), 2^{n - m} (h_2 + 1) \big) \Big],
	\end{aligned} 
	\end{flalign} 
	
	\noindent and observing that the family of triples of intervals of types \eqref{t1boundary} and \eqref{t2boundary} include all boundaries of the faces of $T_N \cap 2^{n - m} \mathbb{T}$, and recalling that $2^m = M$ and $v \le m \le 2v$, we deduce for sufficiently large $v$ that there are at most $24 v^{-1 / 9} M^2 \le (\log M)^{-1 / 10} M^2$ faces $\mathscr{F}$ of $T_N \cap 2^{n - m} \mathbb{T}$ such that $F$ is not $(\log M)^{-1 / 10}$-linear on $\partial \mathscr{F}$. 
\end{proof}

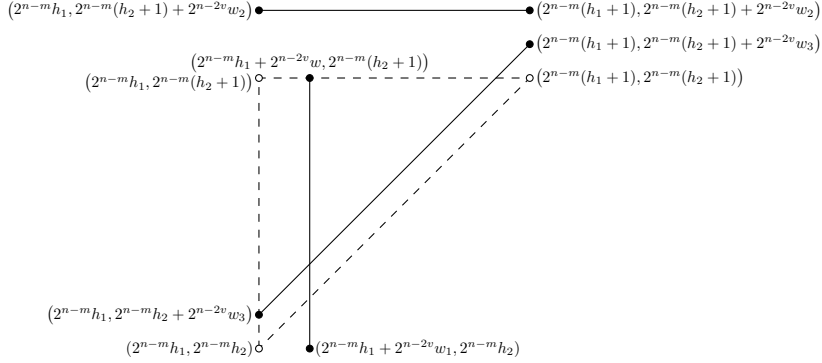
\begin{figure}

	\begin{center}

		\begin{tikzpicture}[
		>=stealth,
		auto,
		style={
			scale = .9
		}
		]

		\draw[-, black] (2.75, 0) -- (2.75, 4); 
		\draw[-, black] (2, 5) -- (6, 5);
		\draw[-, black] (2, .5) -- (6, 4.5);
		
		\draw[-, dashed, black] (2, 0) -- (2, 4) -- (6, 4) -- (2, 0);
		
		\filldraw[fill=white] (2, 0) circle[radius = .05] node[left, scale = .6]{$(2^{n - m} h_1, 2^{n - m} h_2)$};
		\filldraw[fill=white] (2, 4) circle[radius = .05] node[below = 2, left = 0, scale = .6]{$\big( 2^{n - m} h_1, 2^{n - m} (h_2 + 1) \big)$};
		\filldraw[fill=white] (6, 4) circle[radius = .05] node[right, scale = .6]{$\big( 2^{n - m} (h_1 + 1), 2^{n - m} (h_2 + 1) \big)$};

		\filldraw[fill=black] (2.75, 0) circle[radius = .05] node[right, scale = .6]{$( 2^{n - m} h_1 + 2^{n - 2v} w_1, 2^{n - m} h_2)$};
		\filldraw[fill=black] (2.75, 4) circle[radius = .05] node[above, scale = .6]{$\big( 2^{n - m} h_1 + 2^{n - 2v} w, 2^{n - m} (h_2 + 1) \big)$};
		\filldraw[fill=black] (2, 5) circle[radius = .05] node[left, scale = .6]{$\big( 2^{n - m} h_1, 2^{n - m} (h_2 + 1) + 2^{n - 2v} w_2 \big)$};
		\filldraw[fill=black] (6, 5) circle[radius = .05] node[right, scale = .6]{$\big( 2^{n - m} (h_1 + 1), 2^{n - m} (h_2 + 1) + 2^{n - 2v} w_2 \big)$};
		\filldraw[fill=black] (2, .5) circle[radius = .05] node[left, scale = .6]{$\big( 2^{n - m} h_1, 2^{n - m} h_2 + 2^{n - 2v} w_3 \big)$};
		\filldraw[fill=black] (6, 4.5) circle[radius = .05] node[right, scale = .6]{$\big( 2^{n - m} (h_1 + 1), 2^{n - m} (h_2 + 1) + 2^{n - 2v} w_3 \big)$};

		\end{tikzpicture}
		
	\end{center}

	\caption{\label{triangleslines}  The solid and dashed intervals above are those listed in \eqref{intervalsf} and \eqref{t1boundary}, respectively. The dashed ones form a face of $2^{n - m} \mathbb{T}$. }

\end{figure}

\subsection{Proof of \Cref{heightapproximate}}

\label{ProofApproximateGlobal}

In this section we establish \Cref{heightapproximate} by estimating the number of tilings of $R$ whose height functions approximate a given one $F$ (where we recall the notation from \Cref{regularestimate}). To that end, we begin with a lower bound on this quantity when $F = \mathcal{H}$ is the maximizer of $\mathcal{E}$ with boundary data $\mathfrak{h}$. 

In what follows, for any domain $R \subset \mathbb{T}$; boundary height function $h: \partial R \rightarrow \mathbb{Z}$; height function $F: R \rightarrow \mathbb{R}$; and real number $U \in \mathbb{R}_{> 0}$, we let $\mathfrak{G} (F; h; U)$ denote the set of height functions $H: \mathbb{V}(R) \rightarrow \mathbb{Z}$ satisfying $H |_{\partial R} = h$ and $\max_{v \in \mathbb{V}(R)} \big| H(v) - F(v) \big| \le U$. 

\begin{prop}
	
\label{tilingsh}

Adopt the notation of \Cref{heightapproximate}, and define $\mathcal{H}^{(N)}: R \rightarrow \mathbb{R}$ by setting $\mathcal{H}^{(N)} (v) = N \mathcal{H} (N^{-1} v)$ for each $v \in \mathbb{V}(R)$. Then, there exists a constant $C = C(\varepsilon) > 1$ such that
\begin{flalign*}
	N^{-2} \log \Big| \mathfrak{G} \big(\mathcal{H}^{(N)}; h; (\log N)^{-1 / 21} N \big) \Big| \ge  \mathcal{E} (\mathcal{H}) - C (\log N)^{-1 / 63}.
\end{flalign*}
\end{prop}

\begin{proof}
	
	For any real number $A > 0$, recall the triangle $T_A = \big\{ (x, y) \in \mathbb{R}_{\ge 0}^2 :  0 \le x \le y \le A \big\}$ from \Cref{twodimensionalapproximatelinear}. Since $R = \mathcal{B}_N \cap \mathbb{T}$, we have that $R \subset T_{4 N} - (N, N)$. Define $F: \mathbb{V}(R) \rightarrow \mathbb{Z}$ by setting $F(u) = \big\lfloor \mathcal{H}^{(N)} (u) \big\rfloor$ for each $u \in \mathbb{V}(R)$; then $F$ is a height function on $R$, since $\mathcal{H} \in \Adm (\mathfrak{R}; \mathfrak{h})$. Let $G: \mathbb{R}^2 \rightarrow \mathbb{R}$ denote a $1$-Lipschitz extension of $F$ to $\mathbb{R}^2$.
	
	Apply \Cref{twodimensionalapproximatelinear}, where the $f$ there is equal to the $G$ here, and the $n$ and $v$ there are such that the $N$ here satisfies $4 N \le 2^n < 8 N$ and $N^{1 / 4} \le 2^v < 2 N^{1 / 4}$. This yields an integer $m \in [v, 2v]$ such that the following holds for sufficiently large $N$. Letting $M = 2^m \in [N^{1 / 4}, 4 N^{1 / 2}]$ and $\big\{ \mathscr{F}_i \big\}_{\mathcal{I}}$ (for some index set $\mathcal{I}$) denote the faces of $2^{n - m} \mathbb{T}$ that are contained entirely inside $R$, there exists a subset $\mathcal{J} \subseteq \mathcal{I}$ such that $|\mathcal{J}| \ge \mathcal{I} - 2 M^2 (\log N)^{- 1 / 10}$ and such that $G$ is $2 (\log N)^{-1 / 10}$-linear on $\partial \mathscr{F}_j$ for each $j \in \mathcal{J}$ (recall $4 \log M \ge \log N$). For any $j \in \mathcal{J}$, define $g_j: \partial \mathscr{F}_j \rightarrow \mathbb{Z}$ by setting $g_j = G |_{\partial \mathscr{F}_j}$; here, we view $\partial \mathscr{F}_j \subset \mathbb{T}$ (instead of as a subset of $\mathbb{R}^2$). Further set $\rho = (\log N)^{-1 / 21}$. 
	
	We claim that any height function $H \in \mathfrak{G} (h)$ such that $H |_{\partial \mathscr{F}_j} = g_j$ for each $j \in \mathcal{J}$ satisfies $\max_{v \in \mathbb{V}(R)} \big| H(v) - \mathcal{H}^{(N)} (v) \big| \le \rho N$ for sufficiently large $N$. Indeed, for any $u \in \mathbb{V}(R)$, we have that $\big| H(v) - \mathcal{H}^{(N)} (v) \big| \le \big| H(u) - G (u) \big| + 2 |v - u|$, since $H$ and $G$ are $1$-Lipschitz. Thus, if $v$ is in the interior of some $\mathscr{F}_j$ with $j \in \mathcal{J}$, then $\big| H(v) - \mathcal{H}^{(N)} (v) \big| \le \big| H (u) - G (u) \big| + 2 |v - u| \le 32 N^{3 / 4} \le \rho N$, for any $u \in \partial \mathscr{F}_j$ (since then $H(u) = G(u)$ and the diameter of $\mathscr{F}_j$ is at most $16 N^{3 / 4}$).
	
	If instead $v$ is not in the interior of some $\mathscr{F}_j$ with $j \in \mathcal{J}$, then let $\mathfrak{B} = \mathcal{B}_{\rho N / 2} (v)$. Then, for sufficiently large $N$, $\mathfrak{B}$ either contains a vertex in $\partial R \cup \bigcup_{j \in \mathcal{J}} \partial \mathscr{F}_j$, or it contains at least $\rho^2 M^2$ faces of $2^{n - m} \mathbb{T}$ that are either not among $\{ \mathscr{F}_j \}_{j \in \mathcal{J}}$ or not entirely contained in $R$; indeed, this can be quickly deduced from the facts that the area of each such face is $\frac{N^2}{2M^2}$ and that of $\mathfrak{B}$ is $\frac{\pi \rho^2 N^2}{4}$. The facts that $|\mathcal{J}| \ge \mathcal{I} - 2 M^2 (\log N)^{- 1 / 10}$; that there are at most $12 M$ faces in $2^{n - m} \mathbb{T}$ that are not entirely contained in $R$ (since $R = \mathcal{B}_N \cap \mathbb{T}$); that $M \ge N^{1 / 4}$; and that $\rho = (\log N)^{-1 / 21}$ together imply that the latter cannot occur. So, there exists some $u \in \partial R \cup \bigcup_{j \in \mathcal{J}} \partial \mathscr{F}_j$ such that $2 |v - u| \le \rho N$, and we again obtain $\big| H(v) - \mathcal{H}^{(N)} (v) \big| \le \big| H(u) - G(u) \big| + \rho N = \rho N$.
	
	Hence, recalling the set $\mathfrak{G} (h)$ from \Cref{gh}, we have that
	\begin{flalign}
	\label{gproductestimate}
	\big| \mathfrak{G} \big( \mathcal{H}^{(N)}; h; \rho N \big) \big| \ge \displaystyle\prod_{j \in \mathcal{J}} \big| \mathfrak{G} (g_j) \big|.
	\end{flalign}
	
	For each $j \in \mathcal{J}$, let $(s_j, t_j) \in \overline{\mathcal{T}}$ denote the slope of the plane passing through $\big(v, F(v) \big) \in \mathbb{R}^3$, for each of the three corners $v$ of $\mathscr{F}_j$. Stated alternatively, $(s_j, t_j)$ is the unique pair such that $F(v) - F(u) = (s_j, t_j) \cdot (v - u)$, for any two of the three corners $v, u$ of the triangle $\mathscr{F}_j$. Then applying \eqref{ghestimate} with the $N$ and $\varsigma$ there equal to $\frac{N}{M}$ and $\rho = (\log N)^{-1 / 21}$ here, respectively, we deduce from \eqref{gproductestimate} that
	\begin{flalign}
	\label{gnhestimatesum}
	\begin{aligned}
	N^{-2} \log \big| \mathfrak{G} \big( \mathcal{H}^{(N)}; h; (\log N)^{-1 / 21} N \big) \big| & \ge (2 M^2)^{-1} \displaystyle\sum_{j \in \mathcal{J}} \sigma (s_j, t_j) - C_1 M^{-2} (\log N)^{-1 / 63} |\mathcal{J}|, \\
	& \ge (2 M^2)^{-1} \displaystyle\sum_{j \in \mathcal{J}} \sigma (s_j, t_j) - 8 C_1 (\log N)^{-1 / 63},
	\end{aligned} 
	\end{flalign}
	
	\noindent for some constant $C_1 > 1$. To deduce the second inequality in \eqref{gnhestimatesum}, we used the bound $|\mathcal{J}| \le |\mathcal{I}| \le 2 M^2 N^{-2} |R| \le 8 M^2$.

	Next, the facts that $\sigma$ is concave and uniformly H\"{o}lder continuous with exponent $\frac{1}{2}$ (recall \Cref{concavesigmat}); that the area of $N^{-1} \mathscr{F}_j$ is $(2 M^2)^{-1}$; and that $\mathcal{H}^{(N)}$ is $2 (\log N)^{-1 / 10}$-linear on each $\mathscr{F}_j$ together yields a constant $C_2 > 1$ such that, for any $j \in \mathcal{J}$, we have 
	\begin{flalign}
	\label{sigmasjtjestimate}
	2 M^2 \displaystyle\int_{N^{-1} \mathscr{F}_j} \sigma \big( \nabla \mathcal{H} (z) \big) dz \le \displaystyle\max_{s, t} \sigma (s, t) \le \sigma (s_j, t_j) + C_2 (\log N)^{-1 / 20},
	\end{flalign} 
	
	\noindent where in the second term in \eqref{sigmasjtjestimate} the maximum is taken over all pairs $(s, t) \in \overline{\mathcal{T}}$ such that $s_j - 2 (\log N)^{-1 / 10} \le s \le s_j + 2 (\log N)^{-1 / 10}$ and $t_j - 2 (\log N)^{-1 / 10} \le t \le t_j + 2 (\log N)^{-1 / 10}$. 
	
	The lemma now follows from inserting \eqref{sigmasjtjestimate} into \eqref{gnhestimatesum} and using the facts that $\sigma$ is uniformly bounded on $\overline{\mathcal{T}}$; that $|\mathcal{I}| - |\mathcal{J}| \le 2 M^2 (\log N)^{-1 / 10}$; and that $\big| \mathcal{B}_N \setminus \bigcup_{i \in \mathcal{I}} \mathscr{F}_i \big| \le \frac{6 N^2}{M} \le 6 N^{7 / 4}$ (which holds since $R = \mathcal{B}_N \cap \mathbb{T}$ and $M \ge N^{1 / 4}$).
\end{proof}

Next, we provide an upper bound on the size of $\mathfrak{G}$, although the precise form of this set will be slightly different from that in \Cref{tilingsh}. More specifically, for some finite domain $R \subset \mathbb{T}$; boundary height function $h: \partial R \rightarrow \mathbb{Z}$; integer $K \in \mathbb{Z}_{\ge 1}$; and function $F_K: K \mathbb{T} \rightarrow \mathbb{Z}$, let $\mathfrak{G}_K (F_K; h)$ denote the set of height functions $H: \mathbb{V}(R) \rightarrow \mathbb{Z}$ such that $H |_{\partial R} = h$ and $H |_{K \mathbb{T}} = F_K$. This is the set of height functions in $\mathfrak{G} (h)$ whose values on the rescaled lattice $K \mathbb{T}$ are determined by $F_K$. 

The following lemma upper bounds $\big| \mathfrak{G}_K (F_K; h) \big|$ essentially by $\mathcal{E} (\mathcal{H}_K)$, for some function $\mathcal{H}_K$.

\begin{prop}
	
	\label{tilingsh2}
	
	Adopt the notation of \Cref{heightapproximate}; let $F: \mathbb{T} \rightarrow \mathbb{Z}$ denote a height function on $R$ such that $F |_{\partial R} = h$; let $K = 2^k$ denote the integral power of $2$ such that $N^{1 / 8} \le K < 2 N^{1 / 8}$; and set $F_K = F |_{K \mathbb{T}}$. Then there exists a constant $C > 1$ and a function $\mathcal{H}_K \in \Adm (\mathfrak{R}; \mathfrak{h})$ such that $ \big| \mathcal{H}_K (N^{-1} v) - N^{-1} F (v) \big| < N^{-1 / 8}$ for each $v \in \mathbb{V}(R)$, and 
	\begin{flalign*}
	N^{-2} \log \big| \mathfrak{G}_K (F_K; h ) \big| \le  \mathcal{E} (\mathcal{H}_K) + C (\log N)^{-1 / 33}.
	\end{flalign*}
\end{prop}

\begin{proof}
	
	As in the proof of \Cref{tilingsh2}, let $G: \mathbb{R}^2 \rightarrow \mathbb{R}$ denote an arbitrary $1$-Lipschitz extension of $F$ to $\mathbb{R}^2$ and let $n \in \mathbb{Z}_{\ge 1}$ be such that $4 N \le 2^n < 8 N$.  Applying \Cref{twodimensionalapproximatelinear} then yields the existence of $m \in \mathbb{Z}_{\ge 1}$ such that $M = 2^m \in [N^{1 / 4}, 4 N^{1 / 2}]$ and the following holds for sufficiently large $N$. Letting $\big\{ \mathscr{F}_i \big\}_{\mathcal{I}}$ denote the faces of $2^{n - m} \mathbb{T}$ contained inside $R$, there exists a subset $\mathcal{J} \subseteq \mathcal{I}$ such that $|\mathcal{J}| \ge \mathcal{I} - 2 M^2 (\log N)^{- 1 / 10}$ and $G$ is $2 (\log N)^{-1 / 10}$-linear on $\partial \mathscr{F}_j$ for each $j \in \mathcal{J}$. 
	
	Set $\mathcal{H}_K (N^{-1} v) = N^{-1} F (v)$, for each $v \in (2^{n - m} \mathbb{T} \cap R) \cup \partial R$; extend $\mathcal{H}_K$ to the interior of each $N^{-1} \mathscr{F}_i$ by linearity for each $i \in \mathcal{I}$; and extend $\mathcal{H}_K$ arbitrarily (so that it remains admissible on $\mathfrak{R}$) to $\mathfrak{R} \setminus \bigcup_{i \in \mathcal{I}} N^{-1} \mathscr{F}_i$. Observe here that $\mathcal{H}_K (N^{-1} v)$ is not necessarily equal to $N^{-1} F(v)$ for $v \in K \mathbb{T} \cap R$, since $K < \frac{N}{M} = 2^{n - m}$.
	
	We claim that the proposition holds for this choice of $\mathcal{H}_K$. To show that it satisfies the first property listed there, observe since $F$ and $\mathcal{H}_K$ are $1$-Lipschitz that $\big| \mathcal{H}_K (N^{-1} v) - N^{-1} F_K (v) \big| \le \big| \mathcal{H}_K (N^{-1} u) - N^{-1} F (u) \big| + 2 N^{-1} |v - u|$ for any $v, u \in \mathbb{V}(R)$. Since any vertex $v \in \mathbb{V}(R)$ is of distance at most $16 M^{-1} N$ from either $\partial R$ or one of the three vertices of some $\mathcal{F}_i$, it follows that $\big| \mathcal{H}_K (N^{-1} v) - N^{-1} F_K (v) \big| \le 32 M^{-1} < N^{-1 / 8}$. This verifies the first condition on $\mathcal{H}_K$. 
	
	To establish the second, let $(s_j, t_j) \in \overline{\mathcal{T}}$ denote the slope of the plane passing through $\big( v, F (v) \big) \in \mathbb{R}^3$ at the three corners $v$ of $\mathscr{F}_j$, for each $j \in \mathcal{J}$ (as in the proof of \Cref{tilingsh}). Then observe for sufficiently large $N$ that any $H \in \mathfrak{G}_K (F_K; h)$ must be $(\log N)^{-1 / 11}$-nearly linear of slope $(s_j, t_j)$ on $\partial \mathscr{F}_j$ for each $j \in \mathcal{J}$, since $G$ is $(\log N)^{-1 / 10}$-nearly linear of slope $(s_j, t_j)$ on each $\partial \mathscr{F}_j$; since $H |_{K \mathbb{T}} = F_K = G |_{K \mathbb{T}}$; since $K \big( \frac{N}{M} \big)^{-1} \le 8 N^{-1 / 8}$; and since $H$ and $G$ are $1$-Lipschitz. Then, there exists some constant $C > 1$ such that  
	\begin{flalign}
	\label{gkfkh} 
	\begin{aligned} 
	N^{-2} \log \big| \mathfrak{G}_K (F_K; h) \big| & \le (2 M^2)^{-2} \displaystyle\sum_{j \in \mathcal{J}} \sigma (s_j, t_j) + 2 N^{-2} \Bigg| R \setminus \bigcup_{j \in \mathcal{J}} \mathscr{F}_j \Bigg| + C (\log N)^{-1 / 33}  \\
	 & \le \displaystyle\int_{\mathfrak{R}} \sigma \big( \nabla \mathcal{H}_K (z) \big) dz +  2 N^{-2} \Bigg| R \setminus \bigcup_{i \in \mathcal{I}} \mathscr{F}_i \Bigg| + 2 C (\log N)^{-1 / 33}  \\
	& \le \displaystyle\int_{\mathfrak{R}} \sigma \big( \nabla \mathcal{H}_K (z) \big) dz + 3 C (\log N)^{-1 / 33}. 
	\end{aligned} 
	\end{flalign}
	
	\noindent Here, to deduce the first bound, we used the second estimate in \eqref{ghestimatesum}; the fact that any $H \in \mathfrak{G}_K (F_K; h)$ is $(\log N)^{-1 / 11}$-nearly linear of slope $(s_j, t_j)$ on $\partial \mathscr{F}_j$, for each $j \in \mathcal{J}$; and the fact that any face of $R$ can be in one of three possible tiles. To deduce the second, we used the facts that $\mathcal{H}_K$ is linear of slope $(s_j, t_j)$ on $\mathscr{F}_j$ for each $j \in \mathcal{J}$; that $\sigma$ is nonnegative; and that that $|\mathcal{J}| \ge |\mathcal{I}| - 2 M^2 (\log N)^{-1 / 11}$. To deduce the third, we used the fact that $R = \mathcal{B}_N \cap \mathbb{T}$ (which implies that there are at most $12 M$ faces of $2^{n - m} \mathbb{T}$ that do not lie entirely inside $R$). The proposition now follows from \eqref{gkfkh}. 	
\end{proof}

Next, we require the following lemma, which provides an effective bound on the distance from a function $\mathcal{F}$ to the maximizer $\mathcal{H}$ of $\mathcal{E}$, given that $\mathcal{E} (\mathcal{F}) \approx \mathcal{E} (\mathcal{H})$.

\begin{lem}
	
	\label{nearmaximum1}

	For any $\varepsilon \in \big(0, \frac{1}{4} \big)$, there exists a constant $C = C (\varepsilon) > 1$ such that the following holds. Let $\mathfrak{R} \subset \mathbb{R}^2$ denote a simply-connected closed subset, and let $\mathfrak{h}: \partial \mathfrak{R} \rightarrow \mathbb{R}$ denote a function admitting an admissible extension to $\mathfrak{R}$. Also let $\mathcal{H} \in \Adm (\mathfrak{R}; \mathfrak{h})$ denote the maximizer of $\mathcal{E}$ on $\mathfrak{R}$ with boundary data $\mathfrak{h}$, and $\mathcal{F} \in \Adm (\mathfrak{R}; \mathfrak{h})$ denote another admissible extension of $\mathfrak{h}$ to $\mathfrak{R}$. 
	
	Let $\delta \in (0, 1)$; suppose that $\mathcal{E} (\mathcal{F}) \ge \mathcal{E} (\mathcal{H}) - \delta$; and let $z_0 \in \mathfrak{R}$ be a point such that $\mathcal{B}_{\varepsilon} (z_0) \subset \mathfrak{R}$ and $\nabla \mathcal{H} (w) \in \mathcal{T}_{\varepsilon}$, for each $w \in \mathcal{B_{\varepsilon}} (z_0)$. Then, 
	\begin{flalign*}
	 \displaystyle\sup_{w \in \mathcal{B}_{\varepsilon} (z_0)} \Big| \big( \mathcal{H} (w) - \mathcal{H} (z_0) \big) - \big( \mathcal{F} (w) - \mathcal{F} (z_0) \big)  \Big| < C \delta^{1 / 6}.
	\end{flalign*}

\end{lem}

\begin{proof}

	Fix $w_0 \in \mathcal{B}_{\varepsilon} (z_0)$, and set $\omega = \frac{|w_0 - z_0|}{2}$. For notational convenience, let us assume that $z_0 = (0, 0)$ and that $w_0 = (2 \omega, 0)$. Recall from \Cref{concavesigmat} that $\sigma$ is uniformly concave on $\mathcal{T}_{\varepsilon}$, meaning that there exists a constant $c = c(\varepsilon) > 0$ such that 
	\begin{flalign*}
	\sigma \left( \displaystyle\frac{s_1 + s_2}{2}, \displaystyle\frac{t_1 + t_2 }{2} \right) \ge \displaystyle\frac{\sigma (s_1, t_1) + \sigma (s_2, t_2)}{2} + c \big( (s_1 - s_2)^2 + (t_1 - t_2)^2 \big),
	\end{flalign*}
	
	\noindent for any $(s_1, t_1) \in \mathcal{T}_{\varepsilon}$ and $(s_2, t_2) \in \overline{\mathcal{T}}$. Therefore, letting $\mathcal{G} = \frac{\mathcal{F} + \mathcal{H}}{2}$, it follows from integration that 
	\begin{flalign*}
	c \displaystyle\int_{\mathcal{B}_{\omega} (0, \omega)} \big| \nabla \mathcal{H} (z) - \nabla \mathcal{F} (z) \big|^2 dz & \le \displaystyle\int_{\mathfrak{R}} \bigg( \sigma \big( \nabla \mathcal{G} (z) \big) - \displaystyle\frac{1}{2} \Big( \sigma \big( \nabla \mathcal{H} (z) \big) + \sigma \big( \nabla \mathcal{F} (z) \big) \Big) \bigg) dz \\
	& = \mathcal{E} (\mathcal{G}) - \displaystyle\frac{\mathcal{E} (\mathcal{H}) + \mathcal{E} (\mathcal{F}) }{2} \le \displaystyle\frac{\delta}{2},
	\end{flalign*}
	
	\noindent where we have used the fact that $\mathcal{E} (\mathcal{G}) \le \mathcal{E} (\mathcal{H}) \le \mathcal{E} (\mathcal{F}) + \delta$. Hence, 
	\begin{flalign}
	\label{gradientfh}
	\displaystyle\int_{\mathcal{B}_{\omega} (0, \omega)} \big| \nabla \mathcal{H} (z) - \nabla \mathcal{F} (z) \big| dz \le \omega \left( \displaystyle\frac{\pi \delta}{2c} \right)^{1 / 2}.
	\end{flalign}
	
	Next, we will exhibit a path from $w_0$ to $z_0$ of length at most $\pi \omega$, along which $| \nabla \mathcal{H} - \nabla \mathcal{F}|$ is typically bounded above by a multiple of $\omega^{-1} \delta^{1 / 2}$. To that end, for each $\rho \in [-1, 1]$, let us define the function $\gamma_{\rho}: [0, \pi] \rightarrow \mathcal{B}_{\omega}$ by setting $\gamma_{\rho} (\theta) = (\omega \cos \theta + \omega, \rho \omega \sin \theta)$, for each $\theta \in [0, \pi]$, so in particular $\gamma_{\rho} (0) = w_0$ and $\gamma_{\rho} (\pi) = z_0$. Then, it follows from changing variables that 
	\begin{flalign*}
	\displaystyle\int_{\mathcal{B}_{\omega} (0, \omega)} \big| \nabla f (z) \big| dz = \omega^2 \displaystyle\int_{-1}^1 \displaystyle\int_0^{2 \pi} \Big| \nabla f \big( \gamma_{\rho} (\theta) \big) \Big| (\sin \theta)^2 d \theta d \rho,
	\end{flalign*}
	
	\noindent for any function $f: \mathcal{B}_{\omega} \rightarrow \mathbb{R}$. In particular, taking $f = \nabla \mathcal {H} - \nabla \mathcal{F}$ and applying \eqref{gradientfh} yields
	\begin{flalign*}
	\omega \displaystyle\int_{-1}^1 \displaystyle\int_0^{\pi} \Big| \nabla \mathcal{H} \big( \gamma_{\rho} (\theta) \big) - \nabla \mathcal{F} \big( \gamma_{\rho} (\theta) \big)  \Big| (\sin \theta)^2 d \theta d \rho \le \left( \displaystyle\frac{\pi \delta}{2c} \right)^{1 / 2}.
	\end{flalign*}
	
	\noindent Therefore, there exists some $\rho \in [-1, 1]$ such that $\nabla \mathcal{H} (z)$ and $\nabla \mathcal{F} (z)$ are defined for almost all $z \in \gamma_{\rho}$ and  
	\begin{flalign}
	\label{hfgammaestimate} 
	\omega \displaystyle\int_0^{\pi} \Big| \nabla \mathcal{H} \big( \gamma_{\rho} (\theta) \big) - \nabla \mathcal{F} \big( \gamma_{\rho} (\theta) \big) \Big| (\sin \theta)^2 d \theta \le \left( \displaystyle\frac{\pi \delta}{2c} \right)^{1 / 2}.
	\end{flalign}
	
	\noindent Thus, since $\mathcal{H}, \mathcal{F} \in \Adm (\mathfrak{R})$ are $1$-Lipschitz on $\mathcal{B}_{\omega} (0, \omega)$ and $\gamma_{\rho}$ is $1$-Lipschitz on $[0, \pi]$, we obtain 
	\begin{flalign*}
	\Big| \big( & \mathcal{H} (w_0) - \mathcal{H} (z_0) \big) - \big( \mathcal{F} (w_0) - \mathcal{F} (z_0) \big) \Big| \\
	& \le \bigg| \Big( \mathcal{H} \big( \gamma_{\rho} (\pi - \delta^{1 / 6}) \big) - \mathcal{H} \big( \gamma_{\rho} (\delta^{1 / 6}) \big) \Big) - \Big( \mathcal{F} \big( \gamma_{\rho} ( \pi - \delta^{1 / 6} ) \big) - \mathcal{F} \big( \gamma_{\rho} ( \delta^{1 / 6} ) \big) \Big) \bigg| + 4 \delta^{1 / 6} \\ 
	& \le \displaystyle\int_{\delta^{1 / 6}}^{\pi - \delta^{1 / 6}} \bigg| \Big( \nabla \mathcal{H} \big( \gamma_{\rho} (\theta) \big) - \nabla \mathcal{F} \big( \gamma_{\rho} (\theta) \big) \Big) \cdot \gamma_{\rho}' (\theta) \bigg| d \theta + 4 \delta^{1 / 6} \\
	& \le 4 \delta^{-1 / 3} \omega \displaystyle\int_{\delta^{1 / 6}}^{\pi - \delta^{1 / 6}} \Big| \nabla \mathcal{H} \big( \gamma_{\rho} (\theta) \big) - \nabla \mathcal{F} \big( \gamma_{\rho} (\theta) \big) \Big| (\sin \theta)^2 d \theta + 4 \delta^{1 / 6}  \le \big( 6 c^{-1 / 2} + 4 \big) \delta^{1 / 6},
	\end{flalign*}	
	
	\noindent where in the third bound we used the facts that $\big| \gamma_{\rho}' (\theta) \big| \le \omega$ and that $\sin \theta \ge \frac{\delta^{1 / 6}}{2}$ whenever $\theta \in [\delta^{1 / 6}, \pi - \delta^{1 / 6}]$, and in the fourth bound we used \eqref{hfgammaestimate}. 
\end{proof}

Now we can establish \Cref{heightapproximate}. 

\begin{proof}[Proof of \Cref{heightapproximate}]
	
	Let $\mathfrak{V} = \mathfrak{V} (h) = \mathfrak{V} (h; R)$ denote the set of height functions $H: \mathbb{V}(R) \rightarrow \mathbb{Z}$ such that $H |_{\partial R} = h$ and 
	\begin{flalign*}
	\displaystyle\max_{u \in \mathcal{B}_{\varepsilon N} (v_0)} \bigg| N^{-1} \big( H (u) - H(v_0) \big) - \Big( \mathcal{H} \big( N^{-1} u \big) - \mathcal{H} \big( N^{-1} v_0 \big) \Big) \bigg| > (\log N)^{-1 / 400}. 
	\end{flalign*}
	
	It suffices to estimate $\mathbb{P} [H \in \mathfrak{V}]$, to which end we will bound $|\mathfrak{V}|$ using \Cref{tilingsh2} and \Cref{nearmaximum1}. So, let $K$ be as in \Cref{tilingsh2}, and observe that $\big| K \mathbb{T} \cap R \big| \le 4 K^{-2} N^2$. Thus, the number of restrictions $H_K$ to $K \mathbb{T}$ of some height function $H: \mathbb{V}(R) \rightarrow \mathbb{Z}$ such that $F |_{\partial R} = h$ is at most $(2N)^{4 K^{-2} N^2} \le C_1 \exp ( N^{15 / 8})$, for some constant $C_1 = C_1 (\varepsilon) > 1$ (since $H$ must be $1$-Lipschitz). 
	
	It follows from \Cref{tilingsh2} that there exists a constant $C_2 = C_2 (\varepsilon)> 1$ such that
	\begin{flalign}
	\label{destimate}
	|\mathfrak{V}| \le \exp \Bigg( N^2 \bigg( \displaystyle\sup_{\mathcal{H}_K} \mathcal{E} (\mathcal{H}_K) + C_2 (\log N)^{-1 / 33} + C_2 N^{- 1 / 8} \bigg) \Bigg),
	\end{flalign} 
	
	\noindent where the supremum is taken over all $\mathcal{H}_K \in \Adm (\mathfrak{R}; \mathfrak{h})$ such that
	\begin{flalign}
	\label{hkyhkzhyhz}
	\begin{aligned}
	\displaystyle\max_{z \in \mathcal{B}_{\varepsilon} (N^{-1} v_0)} \Big| \big( \mathcal{H}_K (z) - \mathcal{H}_K (N^{-1} v_0) \big) - \big( \mathcal{H} (z) - \mathcal{H} (N^{-1} v_0) \big) \Big|  & > (\log N)^{-1 / 400} - 3 N^{-1 / 8}.
	\end{aligned}
	\end{flalign}
	
	For sufficiently large $N$, \Cref{nearmaximum1} implies that $\sup_{\mathcal{H}_K} \mathcal{E} (\mathcal{H}_K) \le \mathcal{E} (\mathcal{H}) - (\log N)^{- 1 / 66}$, where the supremum is again taken over all $\mathcal{H}_K \in \Adm (\mathfrak{R}; \mathfrak{h})$ satisfying \eqref{hkyhkzhyhz}. Inserting this into \eqref{destimate}, we obtain 
	\begin{flalign}
	\label{destimate2}
	|\mathfrak{V}| \le \exp \Big( N^2 \big( \mathcal{E} (\mathcal{H}) - (\log N)^{-1 / 65} \big) \Big).
	\end{flalign} 
	
	\noindent Moreover, since \Cref{tilingsh} yields a constant $C_3 = C_3 (\varepsilon) > 1$ such that
	\begin{flalign*}
	\Big| \mathfrak{G} \big( \mathcal{H}^{(N)}; h; (\log N)^{-1 / 21} N \big) \Big| \ge \exp \Big( N^2 \big( \mathcal{E} (\mathcal{H}) - C_3 (\log N)^{-1 / 63} \big) \Big),
	\end{flalign*}
	
	\noindent it follows from \eqref{destimate2} that 
	\begin{flalign}
	\label{probabilityhv}
	\mathbb{P} [H \in \mathfrak{V}] \le |\mathfrak{V}| \Big| \mathfrak{G} \big( \mathcal{H}^{(N)}; h; (\log N)^{-1 / 21} N \big) \Big|^{-1} \le \exp \Big( N^2 \big(  C_3 (\log N)^{-1 / 63} - (\log N)^{-1 / 65} \big) \Big).
	\end{flalign}
	
	\noindent The theorem now follows from the fact that the right side of \eqref{probabilityhv} is bounded by $C_4 e^{-N}$, for some constant $C_4 = C_4 (\varepsilon) > 1$.
\end{proof}

\section{The Global Law Without Facets} 

\label{GlobalEstimate2}

In this section we establish \Cref{estimateboundaryheight}, following Sections 6 and 7 of \cite{LTGDMLS}. As in that work, this will proceed by locally comparing a uniformly random height function on $R$ to one on a hexagon. So, we begin in \Cref{MonotonicityDomains} by introducing notation and recalling results from \cite{LTGDMLS,ARTP} on hexagonal tilings. Next, in \Cref{BoundaryProof} we establish \Cref{estimateboundaryheight} assuming a certain inductive statement given by \Cref{estimategammakm} below. We then establish \Cref{estimategammakm} in \Cref{EstimateHjk}.

\subsection{Results on Hexagonal Domains}

\label{MonotonicityDomains}

 In this section we recall certain estimates on tilings of $A \times B \times C$ hexagons (recall \Cref{domainabc}) due to \cite{LTGDMLS,ARTP}. To that end, recall that we seek an estimate on the rate of convergence of the height function associated with a uniformly random lozenge tiling of a domain to its global law. In the special case when the domain is an $A \times B \times C$ hexagon, essentially optimal error estimates were established in \cite{ARTP}. Before stating this result, we require the following definition that provides the limiting height profile associated with a typical tiling of such a hexagon.

\begin{definition} 

\label{habc}

Fix real numbers $a, b, c > 0$, and let $\mathfrak{X}_{a, b, c} \subset \mathbb{R}^2$ denote the hexagon with vertex set $\big\{ (0, 0), (a, 0), (0, b), (a + c, c), (c, b + c), (a + c, b + c)\big\}$. Further let $\mathfrak{L}_{a, b, c}$ denote the interior of the ellipse inscribed in the hexagon $\mathcal{X}_{a, b, c}$. 

Additionally define $\mathfrak{h}_{a, b, c}: \partial \mathfrak{X}_{a, b, c} \rightarrow \mathbb{R}^2$ by setting
\begin{flalign*} 
& \mathfrak{h}_{a, b, c} (x, y) = 0, \quad \text{if $0 \in \{ x, y \}$}; \qquad \qquad \mathfrak{h}_{a, b, c} (x, y) = y, \qquad \text{if $x = y + a$}; \\
& \mathfrak{h}_{a, b, c} (x, y) = x, \quad \text{if $b + x = y$}; \qquad \qquad \mathfrak{h}_{a, b, c} (x, y) = c, \qquad \text{if either $x = a + c$ or $y = b + c$},
\end{flalign*}

\noindent whenever $(x, y) \in \partial \mathfrak{X}_{a, b, c}$. In particular, $\mathfrak{h}_{a, b, c}$ denotes the unique normalized boundary height function (recall \Cref{TilingsHeight}) associated with a tiling of the $\lfloor aN \rfloor \times \lfloor b N \rfloor \times \lfloor cN \rfloor$ hexagon, such that $\mathfrak{h}_{a, b, c} (0, 0) = 0$; see the right side of \Cref{tilinghexagon} and \Cref{domainl} for a depiction. 

Let $\mathfrak{H}_{a, b, c}: \mathfrak{X}_{a, b, c} \rightarrow \mathbb{R}$ denote the maximizer of $\mathcal{E}$ on $\mathfrak{X}_{a, b, c}$ with boundary data $\mathfrak{h}_{a, b, c}$. 
	
\end{definition} 

\begin{rem} 
	
\label{habcderivatives} 

Fix $\vartheta > 0$. Then, it is known (see, for instance, Theorem 1.1 of \cite{STP}) that $\mathfrak{H}_{a, b, c} (z)$ is uniformly smooth on the domain $\big\{ z \in \mathfrak{L}_{a, b, c}: d (z, \partial \mathfrak{L}_{a, b, c}) > \vartheta \big\} \subset \mathfrak{X}_{a, b, c}$, for any triple $(a, b, c) \in \mathbb{R}_{> 0}^3$ such that $a + b + c = 1$ and $\min \{ a, b, c \} \ge \vartheta$. 
	
\end{rem}

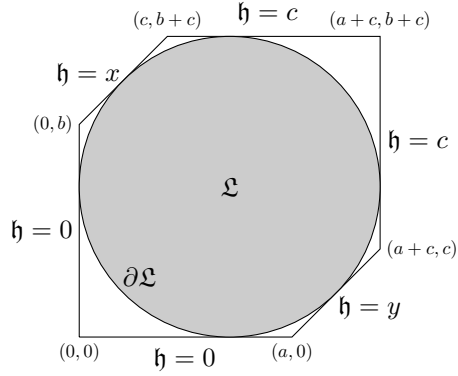
\begin{figure}

	\begin{center}

		\begin{tikzpicture}[
		>=stealth,
		auto,
		style={
			scale = 2
		}
		]
		
		\draw[-, black] (0, 0) -- (1.4142, 0) -- (2, .5858) -- (2, 2) -- (.5858, 2) -- (0, 1.4142) -- (0, 0);
		
		\filldraw[fill=white!60!gray] (1, 1)  circle [radius = 1] node[]{$\mathfrak{L}$};
		
		\draw (.4, .4)  circle [radius = 0] node[]{$\partial\mathfrak{L}$};
		
		\draw (.7, -.15)  circle [radius = 0] node[]{$\mathfrak{h} = 0$};
		\draw (-.25, .7)  circle [radius = 0] node[]{$\mathfrak{h} = 0$};
		
		\draw (1.25, 2.15)  circle [radius = 0] node[]{$\mathfrak{h} = c$};
		\draw (2.25, 1.3)  circle [radius = 0] node[]{$\mathfrak{h} = c$};

		\draw (.06, 1.75)  circle [radius = 0] node[]{$\mathfrak{h} = x$};
		\draw (1.93, .2)  circle [radius = 0] node[]{$\mathfrak{h} = y$};

		\draw (0, 0)  circle [radius = 0] node[below, scale = .7]{$(0, 0)$};
		\draw (1.4142, 0)  circle [radius = 0] node[below, scale = .7]{$(a, 0)$};
		\draw (2, .5858)  circle [radius = 0] node[right, scale = .7]{$(a + c, c)$};
		\draw (2, 2)  circle [radius = 0] node[above, scale = .7]{$(a + c, b + c)$};
		\draw (.5858, 2)  circle [radius = 0] node[above, scale = .7]{$(c, b + c)$};
		\draw (0, 1.4142)  circle [radius = 0] node[left, scale = .7]{$(0, b)$};
		
		\end{tikzpicture}
		
	\end{center}

	\caption{\label{domainl} Depicted above is a hexagon $\mathcal{X}_{a, b, c}$; its inscribed ellipse $\mathfrak{L} = \mathfrak{L}_{a, b, c}$ (shaded); and the associated boundary height function $\mathfrak{h} = \mathfrak{h}_{a, b, c}$. }

\end{figure}

The following lemma now states an effective global law for hexagons tilings, with error $N^{\delta - 1}$ for any $\delta > 0$. It originally appeared as Lemma 5.6 of \cite{ARTP}, although it was stated in its below form as equation (13) of \cite{LTGDMLS}. 

\begin{lem}[{\cite[Lemma 5.6]{ARTP}}]
	
\label{estimatedomainabc} 

For any fixed real numbers $\vartheta, \delta \in (0, 1)$ and $D > 0$, there exists a constant $C = C(\vartheta, \delta, D) > 1$ such that the following holds. Fix an integer $N \in \mathbb{Z}_{\ge 1}$ and real numbers $a, b, c \in (\vartheta, 1)$; set $A = \lfloor a N \rfloor$, $B = \lfloor b N \rfloor$, and $C = \lfloor cN \rfloor$. Abbreviate the $A \times B \times C$ hexagonal domain $\mathcal{X} = \mathcal{X}_{A, B, C}$, and let $h: \partial \mathcal{X} \rightarrow \mathbb{Z}$ denote the associated boundary height function with $h(0, 0) = 0$. Let $v \in \mathcal{X}$ be a vertex such that $N^{-1} v \in \mathfrak{L}_{a, b, c}$ and $d (N^{-1} v, \partial \mathfrak{L}_{a, b, c}) \ge \vartheta$. 

If $H: \mathcal{X} \rightarrow \mathbb{Z}$ denotes a uniformly random element of $\mathfrak{G} (h)$, then 
\begin{flalign*}
\mathbb{P} \Big[ \big| N^{-1} H (v) - \mathfrak{H}_{a, b, c} (N^{-1} v) \big| > N^{\delta - 1} \Big] < C N^{-D}.
\end{flalign*}
\end{lem}

Next we require a way of ``locally comparing'' a tiling of a domain, whose associated global profile $\mathcal{H}: \mathfrak{R} \rightarrow \mathbb{R}$ exhibits no facets, to that of a hexagon. To that end, following \cite{LTGDMLS}, we will match the gradient and Hessian of $\mathcal{H}$ at any given point in $\mathfrak{R}$ to those of some hexagonal global profile $\mathfrak{H}_{a, b, c}$ at some point in $\mathfrak{L}_{a, b, c}$. 

More explicitly, we may express the gradient and Hessian of $\mathcal{H}$ at some fixed $z \in \mathfrak{R}$ by a \emph{local quadruple} of parameters $(s, t, \xi, \eta) \in \mathcal{T} \times \mathbb{R}^2$. Here, $(s, t) = \nabla \mathcal{H} (z) \in \mathcal{T}$; $\xi = \partial_x^2 \mathcal{H} (z)$; and $\eta = \partial_y^2 \mathcal{H} (z)$. Observe that the final component $\partial_x \partial_y \mathcal{H} (z) = \partial_y \partial_x \mathcal{H} (z)$ of the Hessian is determined by $(s, t, \xi, \eta)$ from the Euler-Lagrange equation \eqref{aijh} satisfied by $\mathcal{H}$ (recall \Cref{haijequations}).

The following lemma, which appears as Theorem 4.4 of \cite{LTGDMLS}, essentially states that the local quadruple $(s, t, \xi, \eta)$ associated with a point of any global profile $\mathcal{H}$ can be equated with that at some point in the inscribed ellipse $\mathfrak{L}_{a, b, 1 - a - b} \subset \mathfrak{X}$ of an $a \times b \times (1 - a - b)$ hexagon $\mathfrak{X} = \mathfrak{X}_{a, b, 1 - a - b}$. To state it, we first define the domain
\begin{flalign*}
\mathfrak{Y} = \big\{ (a, b, x, y) \in \mathbb{R}^4: \quad (a, b) \in \mathcal{T}; \quad (x, y) \in \mathfrak{L}_{a, b, 1 - a - b} \big\}.
\end{flalign*}

\noindent Here, we interpret the first two components $(a, b)$ of a quadruple $(a, b, x, y) \in \mathfrak{Y}$ as parameters defining the $a \times b \times (1 - a - b)$ hexagon $\mathfrak{X}_{a, b, 1 - a - b}$. The pair $(x, y)$ is then viewed as a point in the interior of the inscribed ellipse $\mathfrak{L}_{a, b, c}$. 

The following result appears at Theorem 4.4 of \cite{LTGDMLS} and essentially states that local quadruple appearing in a maximizer of $\mathcal{E}$ ``arises'' from an element of $\mathfrak{Y}$.

\begin{lem}[{\cite[Theorem 4.4]{LTGDMLS}}]
	
\label{abcderivative}

Define the map $\Psi: \mathfrak{Y} \rightarrow \mathcal{T} \times \mathbb{R}^2$ as follows. For any $(a, b, x, y) \in \mathfrak{Y}$, set $\Psi (a, b, x, y) = (s, t, \xi, \eta)$, where 
\begin{flalign*}
\nabla \mathfrak{H} (x, y) = (s, t); \qquad \partial_x^2 \mathfrak{H} (x, y) = \xi; \qquad \partial_y^2 \mathfrak{H} (x, y) = \eta,
\end{flalign*}

\noindent and we have abbreviated $\mathfrak{H} = \mathfrak{H}_{a, b, 1 - a - b}$. Then, $\Psi$ is a diffeomorphism. 

\end{lem}

\subsection{Proof of \Cref{estimateboundaryheight}}

\label{BoundaryProof}

In this section we establish \Cref{estimateboundaryheight} assuming \Cref{estimategammakm} below. Since the proofs of the two estimates in \eqref{nhg1g2v} are very similar, we only detail that of the second. Following Sections 6 and 7 of \cite{LTGDMLS}, this will proceed by introducing a sequence of functions $\Phi_0 \ge \Phi_1 \ge \ldots \ge \Phi_J$ such that $\Phi_J \approx G_2$ and $N^{-1} H \le \Phi_0$ holds deterministically. Then, we will use the results from \Cref{MonotonicityDomains} to show that $N^{-1} H \le \Phi_j$ implies $N^{-1} H \le \Phi_{j + 1}$ with high probability, for each $1 \le j \le J - 1$. By induction, this will yield the upper bound $N^{-1} H \le \Phi_J \approx G_2$.  

To implement this procedure, let us abbreviate $g_2 = g$ and $G_2 = G$; we recall throughout that $\nabla G(z) \in \mathcal{T}_{\varepsilon}$ for each $z \in \mathcal{B}$, that is $G$ does not have any frozen regions in $\mathcal{B}$. Define $\Phi_0: \mathcal{B} \rightarrow \mathbb{R}$ by setting $\Phi_0 (z) = G(z) + 3$, for each $z \in \mathbb{R}$. Since $H$ and $G$ are both $1$-Lipschitz, the first assumption on $G$ listed in \Cref{estimateboundaryheight} yields $N^{-1} H (v) \le \Phi_0 (N^{-1} v)$ deterministically, for each $v \in \mathbb{V}(R)$ and sufficiently large $N$. 

Following equation (46) of \cite{LTGDMLS}, we also define $\omega \in \mathbb{R}$, $J \in \mathbb{Z}_{\ge 0}$, and $\Phi_j: \mathcal{B} \rightarrow \mathbb{R}$ for each $j \in [0, J]$ by setting
\begin{flalign}
\label{omegajfunction}
\omega = N^{-\alpha / 5}; \qquad J = \big\lfloor 3 \omega^{-1} \big\rfloor - 1; \qquad \Phi_j (z) = G(z) - j \omega + 3,
\end{flalign} 

\noindent for each $z \in \mathcal{B}$. Defining for any $j \in [0, J]$ and $v \in \mathbb{V}(R)$ the event
\begin{flalign}
\label{omegajv}
\Omega_j (v) = \big\{ N^{-1} H(v) > \Phi_j (N^{-1} v) \big\},
\end{flalign}

\noindent we have the following proposition, which is similar to equation (46) of \cite{LTGDMLS} and states that $N^{-1} H \le \Phi_j$ implies $N^{-1} H \le \Phi_{j + 1}$ with high probability. 

\begin{prop}
	
	\label{estimategammaj} 
	
	There exists a constant $C = C(\varepsilon, \alpha, \nu, D) > 1$ such that the following holds. For any integer $j \in [1, J]$ and vertex $v \in \mathbb{V}(R)$, we have that
	\begin{flalign*}
	\mathbb{P} \bigg[ \Omega_j (v) \cap \bigcap_{u \in \mathbb{V}(R)} \Omega_{j - 1} (u)^c  \bigg] \le C N^{-D}.
	\end{flalign*}
\end{prop}

Assuming \Cref{estimategammaj}, we can quickly establish \Cref{estimateboundaryheight}.

\begin{proof}[Proof of \Cref{estimateboundaryheight}]
	
	Since $\Phi_J (z) \le G(z) + 2 \omega + \varkappa$ and $N^{-1} H (v) \le \Phi_0 (N^{-1} v)$ holds deterministically, a union bound and \Cref{estimategammaj} together yield a constant $C = C(\varepsilon, \alpha, \nu, D) > 1$ such that 
	\begin{flalign*}
	& \displaystyle\max_{v \in \mathbb{V}(R)} \mathbb{P} \big[ N^{-1} H(v) > G_2 (N^{-1} v) + 2 \omega \big]  \\
	& \qquad \le \displaystyle\max_{v \in \mathbb{V}(R)} \mathbb{P} \big[ \Omega_J (v) \big]  \le \displaystyle\sum_{j =  1}^J \mathbb{P} \Bigg[ \bigcup_{v \in \mathbb{V}(R)} \bigg( \Omega_j (v) \cap \bigcap_{u \in \mathbb{V}(R)} \Omega_{j - 1} (u)^c \bigg) \Bigg] \le C J |R| N^{-D - 3} \le 12 C N^{-D},
	\end{flalign*}
	
	\noindent since $J \le 3N$ and $|R| \le 4N^2$. Since $2 \omega = 2 N^{-\alpha/5}$, this establishes the second estimate in \eqref{nhg1g2v}; the proof of the former is entirely analogous and is therefore omitted. 
\end{proof}

\noindent We will establish \Cref{estimategammaj} also through an inductive procedure. To that end, following equation (50) of \cite{LTGDMLS}, define $\rho, \zeta \in \mathbb{R}$ and $\psi: \mathcal{B} \rightarrow \mathbb{R}$ by setting
\begin{flalign}
\label{rhozetapsi}
\rho = N^{-1 / 4}; \qquad \zeta = (\log N)^{1 - \nu / 2}; \qquad \psi (x, y) = \Upsilon - e^{\zeta x} - e^{\zeta y}, 
\end{flalign}

\noindent for each $(x, y) \in \mathcal{B}$, where $\Upsilon \in \mathbb{R}$ is chosen such that $\inf_{z \in \mathcal{B}} \psi (z) = 2$.

Now let $K \in \mathbb{Z}_{\ge 1}$ denote the largest integer for which $\big(1 - \frac{K}{N} \big) \sup_{z \in \mathcal{B}} \psi (z) \le \frac{1}{2}$. Following equation (52) of \cite{LTGDMLS}, define for any $j \in [0, J]$ and $k \in [0, K]$ the function $\varphi_{j; k}: \mathcal{B} \rightarrow \mathbb{R}$ by setting
\begin{flalign}
\label{functionkj}
\varphi_{j; k} (z) = G(z) - (j + 1) \omega +  \left( 1 - \frac{k}{N} \right) \omega \psi (z) + 3 \rho + 3, \quad \text{for each $z \in \mathcal{B}$.}
\end{flalign}

\noindent Similar to in \eqref{omegajv}, for any vertex $v \in \mathbb{V}(R)$ and integer pair $(j, k) \in [0, J] \times [0, K]$, we define the event 
\begin{flalign}
\label{omegakj}
\Omega_{j; k} (v) = \big\{ N^{-1} H (v) > \varphi_{j; k} (N^{-1} v) \big\}.
\end{flalign} 

\noindent The following proposition, which is analogous to Proposition 6.7 of \cite{LTGDMLS}, indicates that $N^{-1} H \ge \varphi_{j; k - 1}$ implies $N^{-1} H \ge \varphi_{j; k}$ with high probability. Its proof will be given in \Cref{EstimateHjk} below. 	
 
\begin{prop}

\label{estimategammakm} 

There exists a constant $C = C(\varepsilon, \alpha, \nu, D) > 1$ such that the following holds. For any pair $(j, k) \in [0, J] \times [1, K]$ and vertex $v \in \mathbb{V}(R)$, we have that
\begin{flalign*}
\mathbb{P} \bigg[ \Omega_{j; k} (v) \cap \bigcap_{u \in \mathbb{V}(R)} \Omega_{j; k - 1} (u)^c  \bigg] \le C N^{-D}.
\end{flalign*}
\end{prop}

Assuming \Cref{estimategammakm}, we now can establish \Cref{estimategammaj}.

\begin{proof}[Proof of \Cref{estimategammaj}]
	
	First observe that since $\big(1 - \frac{K}{N} \big) \sup_{z \in \mathcal{B}} \psi (z) \le \frac{1}{2} < 2 \le \inf_{z \in \mathcal{B}} \psi (z)$ and $\rho \le \frac{\omega}{6}$ for sufficiently large $N$, we have for any $j \in [1, J]$ and $z \in \mathcal{B}$ that 
	\begin{flalign} 
	\label{functionjk0k} 
	\varphi_{j; K} (z) \le \Phi_j (z) \le \Phi_{j - 1} (z) \le \varphi_{j; 0} (z). 
	\end{flalign}
	
	\noindent Therefore, \Cref{estimategammakm}, \eqref{functionjk0k}, and a union bound together yield a constant $C = C (\varepsilon, \alpha, \nu, D) > 1$ such that 
	\begin{flalign*}
	\displaystyle\max_{v \in \mathbb{V}(R)} \mathbb{P} \bigg[ \Omega_j (v) \cap \bigcap_{u \in \mathbb{V}(R)} \Omega_{j - 1} (u)^c \bigg] & \le \displaystyle\max_{v \in \mathbb{V}(R)} \mathbb{P} \bigg[ \Omega_{j; K} (v) \cap \bigcap_{u \in \mathbb{V}(R)} \Omega_{j; 0} (u)^c  \bigg] \\
	& \le \displaystyle\sum_{k =  1}^K \mathbb{P} \Bigg[ \bigcup_{v \in \mathbb{V}(R)} \bigg( \Omega_{j; k} (v) \cap \bigcap_{u \in \mathbb{V}(R)} \Omega_{j; k - 1} (u)^c \bigg) \Bigg] \\
	& \le C K |R| N^{-D - 3} \le 4 C N^{-D},
	\end{flalign*}
	
	\noindent where we have used the facts that $K \le N$ and $|R| \le 4N^2$. 
\end{proof}

\subsection{Proof of \Cref{estimategammakm}}

\label{EstimateHjk}

In this section we establish \Cref{estimategammakm}, following the proof of Proposition 6.7 given in Section 7 of \cite{LTGDMLS}. To that end, recall the notation of \Cref{BoundaryProof}, fix some pair $(j, k) \in [0, J] \times [1, K]$ and vertex $v \in \mathbb{V}(R)$, and define 
\begin{flalign}
\label{functionomegakappa} 
\varphi = \varphi_{j; k - 1}; \qquad \Omega = \bigcup_{u \in \mathbb{V}(R)} \Omega_{j; k - 1} (u); \qquad \kappa = \frac{N - k + 1}{N}.
\end{flalign} 

\noindent Throughout this section, we restrict to the event $\Omega^c$. 

\begin{rem} 
	
\label{zetakappadelta} 

By the third property listed in \Cref{estimateboundaryheight} and the definition \eqref{rhozetapsi} of $\zeta$ (which implies that $e^{\zeta}$ grows slower than any polynomial in $N$), for any $\delta > 0$ there exists a constant $C = C (\delta) > 1$ such that 
\begin{flalign}
\label{gzetadeltakappa} 
\big\| G - G(0, 0) \big\|_{\mathcal{C}^2 (\overline{\mathcal{B}})} \le (\log N)^{-\nu / 2} \zeta; \qquad \zeta \le e^{\zeta} \le C N^{\delta}; \qquad  \displaystyle\frac{1}{C N^{\delta}} \le 1 - \displaystyle\frac{K}{N} \le \kappa \le 1.
\end{flalign}

\end{rem} 

Now, for $N$ sufficiently large, observe that since $H$ and $G$ are $1$-Lipschitz; since $N^{-1} h(u) \le G(N^{-1} u) + \frac{4}{N} < G (N^{-1} u) + \rho$ for any $u \in \partial R$ (by the first assumption on $G$ listed in \Cref{estimateboundaryheight} and the fact that $G$ is $1$-Lipschitz); and since $(j + 1) \omega \le 3$ for each $z \in \mathcal{B}$ and $j \in [1, J]$, the bound
\begin{flalign}
\label{hnvrho} 
N^{-1} H(v) \le \varphi_{j; k} (N^{-1} v), \qquad \text{holds deterministically if $d (v, \partial R) \le \rho N$}.
\end{flalign} 

\noindent Thus, we will assume in what follows that $d(v, \partial R) > \rho N$, so $\mathcal{B}_{\rho N} (v) \subset \mathcal{B}_{\rho N} (v) \subset \mathbb{F}(R)$. 

In this case, let us briefly outline how we will establish \Cref{estimategammakm}; in the below, we recall the notation of \Cref{domainabc} and \Cref{habc}. As mentioned previously, this will proceed by locally comparing the tiling of $R$ around $v$ with a uniformly random tiling of a suitable $\lfloor aN \rfloor \times \lfloor bN \rfloor \times \lfloor cN \rfloor$ hexagon around some point $N z_0$. More specifically, we will exhibit a triple $(a, b, c) \in \mathbb{R}_{> 0}$ and a point $z_0 \in \mathfrak{L}_{a, b, c} \subset \mathfrak{X}_{a, b, c}$ such that $\varphi |_{\partial \mathcal{B}_{\rho} (N^{-1} v)} < \mathfrak{H}_{a, b, c} |_{\partial \mathcal{B}_{\rho} (z_0)}$ and $\varphi_{j; k} (N^{-1} v) > \mathfrak{H}_{a, b, c} (z_0)$ (after suitably globally shifting $\mathfrak{H}_{a, b, c}$, if necessary). Since the normalized height function $N^{-1} H$ is bounded above by $\varphi$, \Cref{monotoneheightcouple} and \Cref{estimatedomainabc} will then imply that $N^{-1} H(v) \le \mathfrak{H}_{a, b, c} (z_0) < \varphi_{j; k} (N^{-1} v)$ with high probability, thereby establishing \Cref{estimategammakm}.  

The triple $(a, b, c) \in \mathbb{R}_{> 0}^3$ and point $z_0 \in \mathfrak{L}_{a, b, c}$ will be determined (through \Cref{abcderivative}) by imposing $\nabla \mathfrak{H}_{a, b, c} (z_0) = \nabla \varphi (N^{-1} v)$ and fixing the second derivatives of $\mathfrak{H}_{a, b, c}$ at $z_0$ in a specific way. To make this more precise, we require some additional notation. 

For any pair $(s, t) \in \mathcal{T}$, function $F \in \mathcal{C}^2 (\mathcal{B})$, and point $z \in \mathcal{B}$, define the vectors $\textbf{p} (s, t) \in \mathbb{R}^4$ and $\textbf{q}_F (z) \in \mathbb{R}^4$ by setting
\begin{flalign}
\label{pq} 
\textbf{p} (s, t) = \big( \mathfrak{a}_{xx} (s, t), \mathfrak{a}_{xy} (s, t), \mathfrak{a}_{yx} (s, t), \mathfrak{a}_{yy} (s, t) \big); \quad \textbf{q}_F (z) = \big( \partial_x^2 F(z), \partial_x \partial_y F(z), \partial_y \partial_x F(z), \partial_y^2 F(z) \big),
\end{flalign} 

\noindent where we recall the functions $\mathfrak{a}_{jk}$ from \eqref{aijst}. Then the partial differential equation \eqref{aijh} is equivalent to stipulating that $\textbf{p} \big( \nabla F(z) \big) \cdot \textbf{q}_F (z) = 0$ holds for each $z \in \mathcal{B}$; in particular, this holds for $F = G$.

We next establish the below lemma, whose first part shows that $\varphi$ does not exhibit facets and whose second provides an estimate on $\textbf{p}$ and $\textbf{q}$ that will be useful for what follows. 

\begin{lem} 
	
	\label{pgradientestimateqg}
	
	There exists a constant $C_0 = C_0 (\varepsilon, \alpha, \nu) > 1$ such that the following two statements hold for $N > C_0$. 
	
	\begin{enumerate} 
		
	\item For any $z \in \mathcal{B}$, we have that $\nabla \varphi (z) \in \mathcal{T}_{\varepsilon / 2}$. 
	
	\item For any $z \in \mathcal{B}$, we have that
	\begin{flalign}
	\label{pqgpqgestimate}
	\left| \displaystyle\frac{\textbf{\emph{p}} \big( \nabla \varphi (z) \big) \cdot \textbf{\emph{q}}_G ( z)}{\textbf{\emph{p}} \big( \nabla \varphi (z) \big) \cdot \textbf{\emph{q}}_{\psi} ( z)} \right| < \kappa \omega (\log N)^{-\nu / 4}.
	\end{flalign}
	
	\end{enumerate} 
	
\end{lem} 

\begin{proof} 
	
	Fix $z = (x, y) \in \mathcal{B}$. From the definitions \eqref{functionkj} of $\varphi$ and \eqref{rhozetapsi} of $\psi$, we find that
	\begin{flalign}
	\label{gradientgestimatez}
	\big| \nabla \varphi (z)  - \nabla G (z) \big| = \kappa \omega \big| \nabla \psi (z) \big| \le 2 \kappa \omega \zeta \displaystyle\max\{ e^{\zeta x}, e^{\zeta y} \}.
	\end{flalign}
	
	 Next, from the second bound in \eqref{gzetadeltakappa}; the fact that $\max \big\{ x, y \big\} \le 1$; and the identity $\omega = N^{-\alpha / 5}$, we deduce the existence of a constant $C_1 = C_1 (\alpha, \nu) > 1$ with $2 \kappa \omega \zeta \displaystyle\max\{ e^{\zeta x}, e^{\zeta y} \} \le C_1 \omega^{1 / 2}$. Using the fact that $\nabla G (z) \in \mathcal{T}_{\varepsilon}$ and taking $N$ sufficiently large so that $C_1 \omega^{1 / 2} < \frac{\varepsilon}{2}$, we obtain $\nabla \varphi (z) \in \mathcal{T}_{\varepsilon / 2}$ from \eqref{gradientgestimatez}; this implies the first statement of the lemma.
	
	To deduce \eqref{pqgpqgestimate} observe that, since $\textbf{p} \big( \nabla G(z) \big) \cdot \textbf{q}_G (z) = 0$ and each $\mathfrak{a}_{ij}$ is uniformly Lipschitz on $\mathcal{T}_{\varepsilon / 2}$ (recall \eqref{aijst}), there exists a constant $C_2 = C_2 (\varepsilon, \alpha) > 1$ such that 
	\begin{flalign*}
	\Big| \textbf{p} \big( \nabla \varphi (z) \big) \cdot \textbf{q}_G (z)  \Big| = \bigg| \Big( \textbf{p} \big( \nabla \varphi (z) \big) - \textbf{p} \big( \nabla G (z) \big) \Big) \cdot \textbf{q}_G (z) \bigg| \le C_2 \big| \nabla \varphi (z) - \nabla G (z) \big| \big| \textbf{q}_G (z) \big|.
	\end{flalign*}
	
	\noindent Thus, since $\big| \textbf{q}_G (z) \big| \le 2 \| G - G (0, 0) \|_{\mathcal{C}^2 (\overline{\mathcal{B}})}$, it follows from \eqref{gradientgestimatez} and the first bound in \eqref{gzetadeltakappa} that
	\begin{flalign}
	\label{pestimategradient}
	\Big| \textbf{p} \big( \nabla \varphi (z) \big) \cdot \textbf{q}_G (z)  \Big| \le 4 C_2 \kappa \omega \zeta^2 (\log N)^{- \nu / 2} \displaystyle\max \big\{ e^{\zeta x}, e^{\zeta y} \big\}.
	\end{flalign}	
	
	Moreover, the fact that $\nabla \varphi (z) \in \mathcal{T}_{\varepsilon / 2}$ and the explicit form \eqref{aijst} for the $\mathfrak{a}_{jk}$ together yield a constant $c_1 = c_1 (\varepsilon, \alpha) > 0$ for which $\min \big\{ \mathfrak{a}_{xx} \big( \nabla \varphi (z) \big), \mathfrak{a}_{yy} \big( \nabla \varphi (z) \big) \big\} > c_1$. Thus, since $\textbf{q}_{\psi} (z) = \big( - \zeta^2 e^{\zeta x}, 0, 0, - \zeta^2 e^{\zeta y} \big)$, we deduce that
	\begin{flalign}
	\label{pestimategradientpsi}
	\Big| \textbf{p} \big( \nabla \varphi (z) \big) \cdot \textbf{q}_{\psi} (z)  \Big| \ge c_1 \zeta^2 \displaystyle\max \big\{ e^{\zeta x}, e^{\zeta y} \big\}.
	\end{flalign}
	
	\noindent Thus, \eqref{pqgpqgestimate} follows from \eqref{pestimategradient}, \eqref{pestimategradientpsi}, and the fact that $N$ is sufficiently large.
\end{proof}

Now, similar to in equation (72) of \cite{LTGDMLS}, define the vector $\textbf{q}_{\Theta} (z) \in \mathbb{R}^4$ (here, we view $\Theta$ as an index and do not necessarily require that it be a function in $\mathcal{C}^2 (\mathcal{B})$) by setting 
\begin{flalign} 
\label{thetaz} 
\textbf{q}_{\Theta} (z) =  \textbf{q}_G (z) - \displaystyle\frac{ \textbf{p} \big( \nabla \varphi (z) \big) \cdot \textbf{q}_G (z) }{\textbf{p} \big( \nabla \varphi (z) \big) \cdot \textbf{q}_{\psi} (z)} \textbf{q}_{\psi} (z), \qquad \text{for each $z \in \mathcal{B}$},
\end{flalign}

\noindent where the fact that $\textbf{p} \big( \nabla \varphi (z) \big) \cdot \textbf{q}_{\psi} (z) \ne 0$ for each $z \in \mathcal{B}$ was established as \eqref{pqgpqgestimate} (see also \eqref{pestimategradientpsi}). Observe that $\textbf{q}_{\Theta}$ satisfies the linear variant of \eqref{aijh} given by $\textbf{p} \big( \nabla \varphi (z) \big) \cdot \textbf{q}_{\Theta} (z) = 0$.

The lemma below establishes the existence of a triple $(a, b, c) \in \mathbb{R}_{> 0}^3$ and a point $z_0 \in \mathfrak{L}_{a, b, c}$ such that $\nabla \varphi (N^{-1} v) = \nabla \mathfrak{H}_{a, b, c} (z_0)$ and the second derivatives of $\mathfrak{H}_{a, b, c}$ at $z_0$ are given by $\textbf{q}_{\Theta} (N^{-1} v)$. 

\begin{lem} 
	
\label{abchabctheta} 

There exists a constant $C = C(\varepsilon, \alpha, \nu) > 1$; a triple $(a, b, c) \in \mathbb{R}_{> 0}^3$; and a point $z_0 \in \mathfrak{L}_{a, b, c}$ such that the following holds. Assuming $N > C$ and abbreviating $\mathfrak{H} = \mathfrak{H}_{a, b, c}$, we have 
\begin{flalign}
\label{abczetaestimates}
a + b + c = \displaystyle\frac{1}{\zeta}; \qquad \qquad \qquad \displaystyle\min \{ a, b, c \} > \displaystyle\frac{1}{C \zeta}; \qquad \qquad \quad d (z_0, \partial \mathfrak{L}_{a, b, c}) > \displaystyle\frac{1}{C \zeta},
\end{flalign}

\noindent and 
\begin{flalign}
\label{abczetaestimates1}
 \nabla \varphi (N^{-1} v) = \nabla \mathfrak{H} (z_0); \qquad \textbf{\emph{q}}_{\Theta} (N^{-1} v) = \textbf{\emph{q}}_{\mathfrak{H}} (z_0).
\end{flalign}

\end{lem} 

\begin{proof} 
	
This will follow from a suitable application of \Cref{abcderivative} and rescaling by $\zeta^{-1}$. However, it will first be useful to estimate the second derivatives of $\Theta$. So, for sufficiently large $N$, observe that 
\begin{flalign}
	\label{derivativesthetaestimate} 
	\begin{aligned}
\displaystyle\sup_{z \in \mathcal{B}} \big| \textbf{q}_{\Theta} (z) \big| & \le 2 \big\| G - G (0, 0) \big\|_{\mathcal{C}^2 (\overline{\mathcal{B}})} + 2 \left|  \displaystyle\frac{ \textbf{p} \big( \nabla \varphi (z) \big) \cdot \textbf{q}_G (z) }{\textbf{p} \big( \nabla \varphi (z) \big) \cdot \textbf{q}_{\psi} (z)} \right| \displaystyle\sup_{z \in \mathcal{B}} \displaystyle\max_{i, j \in \{ x, y \}} \big| \partial_i \partial_j \psi (z) \big| \\
& \le 2 (\log N)^{1 - \nu} + 2 \kappa \omega \displaystyle\sup_{z \in \mathcal{B}} \displaystyle\max_{i, j \in \{ x, y \}} \big| \partial_i \partial_j \psi (z) \big| \\
& \le 2 (\log N)^{1 - \nu} + 4 \kappa \omega \zeta^2 e^{\zeta} \le (\log N)^{1 - \nu/2} = \zeta,
\end{aligned} 
\end{flalign}

\noindent where to deduce the second inequality we used the first bound in \eqref{gzetadeltakappa} and \eqref{pqgpqgestimate}; to deduce the third we used the explicit form \eqref{rhozetapsi} for $\psi$; and to deduce the fourth we used the second bound in \eqref{gzetadeltakappa} and the fact that $\omega = N^{-\alpha/5}$.

Now observe for any $z \in \mathfrak{X}_{a, b, c}$ and $i, j \in \{ x, y \}$ that   
\begin{flalign*}
\nabla \mathfrak{H}_{a, b, c} (z) = \nabla \mathfrak{H}_{\zeta a, \zeta b, \zeta c} ( \zeta z ); \qquad \partial_i \partial_j \mathfrak{H}_{a, b, c} (z) = \zeta \partial_i \partial_j \mathfrak{H}_{\zeta a, \zeta b, \zeta c} ( \zeta z ).
\end{flalign*}

\noindent Therefore, \eqref{derivativesthetaestimate}; the fact that $\nabla \varphi (N^{-1} v) \in \mathcal{T}_{\varepsilon / 2}$; and \Cref{abcderivative} (scaled by $\zeta^{-1}$) together imply the existence of a constant $C = C (\varepsilon, \alpha, \nu) > 1$, real numbers $a, b, c > 0$, and a point $z_0 \in \mathfrak{L}_{a, b, c}$ such that \eqref{abczetaestimates} and the identities 
\begin{flalign}
\label{derivativeshtheta1} 
\nabla \varphi (N^{-1} v) = \nabla \mathfrak{H} (z_0); \qquad q_{\Theta} (1; N^{-1} v) = \partial_x^2 \mathfrak{H} (z_0); \qquad \qquad q_{\Theta} (4; N^{-1} v) = \partial_y^2 \mathfrak{H} (z_0),
\end{flalign} 

\noindent all hold; here, we have abbreviated $\mathfrak{H} = \mathfrak{H}_{a, b, c}$ and set $\textbf{q}_{\Theta} (z) = \big( q_{\Theta} (1; z), q_{\Theta} (2; z), q_{\Theta} (3; z), q_{\Theta} (4; z) \big)$. It remains to establish the last identity in \eqref{abczetaestimates1}.

To that end, observe from the definition \eqref{thetaz} of $\Theta$ and the fact that $\mathfrak{H}$ is a maximizer of $\mathcal{E}$ that $\textbf{p} \big( \nabla \varphi (N^{-1} v) \big) \cdot \textbf{q}_{\Theta} (N^{-1} v) = 0 = \textbf{p} \big( \nabla \mathfrak{H} (z_0) \big) \cdot \textbf{q}_{\mathfrak{H}} (z_0)$. Therefore, it follows from \eqref{derivativeshtheta1} that $q_{\Theta} (2; z) = q_{\Theta} (3; z) =\partial_x \partial_y \mathfrak{H} (z_0) =\partial_y \partial_x \mathfrak{H} (z_0)$ (where the first equality follows from the definition \eqref{thetaz} of $\textbf{q}_{\Theta}$, and the fact that the second and third coordinates of $\textbf{q}_G (z)$ are equal, as are those of $\textbf{q}_{\psi} (z)$). Thus, $\textbf{q}_{\Theta} (N^{-1} v) = \textbf{q}_{\mathfrak{H}} (z_0)$. 
\end{proof} 

Under the notation of \Cref{abchabctheta}, we have the following lemma that is similar to equation (88) of \cite{LTGDMLS} and indicates that the relative height of $\varphi (N^{-1} v)$ with respect to $\varphi |_{\mathcal{B}_{\rho} (N^{-1} v)}$ is larger than that of $\mathfrak{H} (z_0)$ with respect to $\mathfrak{H} |_{\mathcal{B}_{\rho} (z_0)}$. 

\begin{lem}

\label{functionhuwvz}

Adopting the notation of \Cref{abchabctheta}, there exists a constant $C_0 = C_0 (\varepsilon, \alpha, \nu) > 1$ such that the following holds. Let $u \in \mathcal{B}_{\rho N} (v)$ denote a vertex of $R$, and set $w_0 = z_0 + N^{-1} (u - v) \in \mathcal{B}_{\rho} (z_0)$. If $N > C_0$, then   
\begin{flalign*}
\varphi (N^{-1} v) - \varphi (N^{-1} u) \ge \mathfrak{H} (z_0) - \mathfrak{H} (w_0) + \omega^2 |w_0 - z_0|^2 .
\end{flalign*}
\end{lem} 

\begin{proof}

	For any function $f \in \mathcal{C}^2 (\mathcal{B})$ and point $z \in \mathcal{B}$, denote the Hessian $\textbf{H}_f (z)$ of $f$ at $z$ by 
	\begin{flalign*} 
	\textbf{H}_f (z) = \left[ \begin{array}{cc} \partial_x^2 f(z) & \partial_x \partial_y f(z) \\ \partial_y \partial_x f(z) & \partial_y^2 f(z) \end{array} \right].
	\end{flalign*} 
	
	\noindent Then, a Taylor expansion and the identity $w_0 - z_0 = N^{-1} (u - v)$ yield a constant $C_1 > 0$ such that
	\begin{flalign}
	\label{huwvz}
	\begin{aligned}
	& \Big| \varphi (N^{-1} u) - \mathfrak{H} (w_0) - \big( \varphi (N^{-1} v) - \mathfrak{H} (z_0) \big) - \big( \nabla \varphi (N^{-1} v) - \nabla \mathfrak{H} (z_0) \big) \cdot (w_0 - z_0) \\
	& \quad - (w_0 - z_0) \cdot \big( \textbf{H}_{\varphi} (N^{-1} v) - \textbf{H}_{\mathfrak{H}} (z_0) \big) (w_0 - z_0) \Big| \\
	& \qquad \qquad \qquad \le C_1 |w_0 - z_0|^{2 + \alpha}  \big( \| \varphi - \varphi (0, 0) \|_{\mathcal{C}^{2, \alpha} (\overline{\mathcal{B}})} + \| \mathfrak{H} \|_{\mathcal{C}^{2, \alpha} (\overline{\mathcal{B}}_{\rho} (z_0))} \big).
	\end{aligned}
	\end{flalign}
	
	To estimate the right side of \eqref{huwvz}, observe that the fact that $\mathfrak{H}_{a, b, c} (z) = \zeta^{-1} \mathfrak{H}_{a \zeta, b \zeta, c \zeta} (\zeta z)$; \eqref{abczetaestimates}; and \Cref{habcderivatives} (with scaling by $\zeta$) together yield a constant $C_2 = C_2 (\varepsilon, \nu) > 1$ such that $\| \mathfrak{H} \|_{\mathcal{C}^{2, \alpha} (\overline{\mathcal{B}}_{\rho} (z_0))} \le C_2 \zeta^{1 + \alpha}$. Additionally, we have for sufficiently large $N$ that
	\begin{flalign}
	\label{estimatederivativesfunction1}  
	\begin{aligned}
	\| \varphi - \varphi (0, 0) \|_{\mathcal{C}^{2, \alpha} (\overline{\mathcal{B}})} & \le \| G - G (0,0) \|_{\mathcal{C}^{2, \alpha} (\overline{\mathcal{B}})} + \kappa \omega \| \psi - \psi (0, 0) \|_{\mathcal{C}^{2, \alpha} (\overline{\mathcal{B}})} \\
	& \le (\log N)^{10} + 2 \kappa \omega \zeta^3 e^{\zeta} \le 2 (\log N)^{10}.
	\end{aligned}
	\end{flalign}
	
	\noindent Here, to deduce the first estimate, we used \eqref{functionkj}; to deduce the second, we used the third assumption imposed on $G$ in \Cref{estimateboundaryheight} and the explicit form \eqref{rhozetapsi} for $\psi$; and to deduce the third we used the second bound in \eqref{gzetadeltakappa} and the fact that $\omega = N^{-\alpha / 5}$.
	
	Hence, from \eqref{huwvz}; \eqref{estimatederivativesfunction1}; \eqref{abczetaestimates1}; and the fact that $\zeta^{1 + \alpha} < (\log N)^{10}$, we deduce the existence of a constant $C_3 = C_3 (\varepsilon, \alpha, \nu) > 1$ such that 
	\begin{flalign}
	\label{huwvz2} 
	\begin{aligned}
	\Big| \varphi (N^{-1} u) - \mathfrak{H} (w_0) - \big( \varphi (N^{-1} v) - \mathfrak{H} (z_0) \big) - (w_0 - z_0) \cdot \big( \textbf{H}_{\varphi} ( & N^{-1} v) - \textbf{H}_{\Theta} (N^{-1} v) \big) (w_0 - z_0) \Big| \\
	& \qquad \le C_3 |w_0 - z_0|^{2 + \alpha} (\log N)^{10}.
	\end{aligned} 
	\end{flalign}
	
	\noindent Next, observe from the definitions \eqref{functionkj} and \eqref{thetaz} of $\varphi$ and $\Theta$, respectively, that
	\begin{flalign}
	\label{hnvhthetanvequation} 
	\textbf{H}_{\varphi} (N^{-1} v) - \textbf{H}_{\Theta} (N^{-1} v) & = \left( \displaystyle\frac{\textbf{p} \big( \nabla \varphi (N^{-1} v) \big) \cdot \textbf{q}_G ( N^{-1} v)}{\textbf{p} \big( \nabla \varphi (N^{-1} v) \big) \cdot \textbf{q}_{\psi} ( N^{-1} v)} + \kappa \omega \right) \textbf{H}_{\psi} (N^{-1} v).
	\end{flalign}
	
	\noindent Now since \eqref{pqgpqgestimate} implies for sufficiently large $N$ that 
	\begin{flalign*}
	 \displaystyle\frac{\textbf{p} \big( \nabla \varphi (N^{-1} v) \big) \cdot \textbf{q}_G ( N^{-1} v)}{\textbf{p} \big( \nabla \varphi (N^{-1} v) \big) \cdot \textbf{q}_{\psi} ( N^{-1} v)} + \kappa \omega \ge \kappa \omega \big( 1 -  (\log N)^{- \nu / 4} \big) > \displaystyle\frac{\kappa \omega}{2},
	\end{flalign*}
	
	\noindent and since the explicit form \eqref{rhozetapsi} for $\psi$ implies that $\textbf{H}_{\psi} (N^{-1} v)$ is negative definite with eigenvalues bounded above by $-\zeta^2 e^{-\zeta}$, we deduce from \eqref{huwvz2} and \eqref{hnvhthetanvequation} that
	\begin{flalign}
	\label{huwvz3} 
	\begin{aligned}
	 \varphi (N^{-1} u) - \mathfrak{H} (w_0) - \big( \varphi (N^{-1} v) - \mathfrak{H} (z_0) \big) & \le C_3 |w_0 - z_0|^{2 + \alpha} (\log N)^{10} - \displaystyle\frac{\kappa \omega \zeta^2 e^{-\zeta} |w_0 - z_0|^2}{2} \\
	 & \le - \displaystyle\frac{\kappa \omega \zeta^2 e^{-\zeta} |w_0 - z_0|^2}{4} < - \omega^2 |w_0 - z_0|^2,
	\end{aligned} 
	\end{flalign} 
	
	\noindent for sufficiently large $N$. Here, to deduce the second and third inequalities in \eqref{huwvz3} we used the second and third bounds in \eqref{gzetadeltakappa} and the facts that $|w_0 - z_0| \le \rho = N^{-1 / 4}$ and $\omega = N^{-\alpha / 5}$. 
	
	The lemma then follows from \eqref{huwvz3}. 
\end{proof}

Now we can establish \Cref{estimategammakm}. 

\begin{proof}[Proof of \Cref{estimategammakm}]
	
By \eqref{hnvrho}, we may assume that $d (v, \partial R) > \rho N$. Adopt the notation of \Cref{abchabctheta}, and denote the $aN \times bN \times cN$ hexagon from \Cref{domainabc} by $\mathcal{X} = N \mathfrak{X}_{a, b, c}$, where we assume for notational convenience that $aN, bN, cN, \frac{N}{\zeta} \in \mathbb{Z}_{\ge 1}$ and $N z_0 \in \mathbb{Z}^2$. Let $\chi: \partial \mathcal{X} \rightarrow \mathbb{Z}$ denote the boundary height function on $\partial \mathcal{X}$ determined by setting its value at $(0, 0)$ to be
\begin{flalign*} 
\chi (0, 0) =  N \varphi (N^{-1} v) - N \mathfrak{H} (z_0) - N \left\lfloor \frac{\omega^2 \rho^2}{2} \right\rfloor,
\end{flalign*} 

\noindent and setting $\chi (u) = N \mathfrak{h}_{a, b, c} (N^{-1} u) + \chi (0, 0)$ for any $u \in \partial \mathcal{X}$ (where we recall $\mathfrak{h}_{a, b, c}$ from \Cref{habc}). Let $\Xi: \mathcal{X} \rightarrow \mathbb{Z}$ denote a uniformly random height function in $\mathfrak{G} (\chi)$ (recall \Cref{gh}). On the event $\Omega^c = \bigcap_{u \in \mathbb{V}(R)} \Omega_{j; k - 1} (u)^c$, we will show that $H |_{\partial \mathcal{B}_{N \rho} (v)}$ is likely bounded above by $\Xi |_{\mathcal{B}_{N \rho} (z_0)}$. 

To that end, apply \Cref{estimatedomainabc}, with the $N$ there equal to $\frac{N}{\zeta}$ here and the $a, b, c$ there equal to $a \zeta$, $b \zeta$, $c \zeta$ here, respectively. Using \eqref{abczetaestimates}; the fact that $\mathfrak{H}_{a \zeta, b \zeta, c \zeta} (z) = \zeta \mathfrak{H} \big( \frac{z}{\zeta} \big)$; the second estimate in \eqref{gzetadeltakappa}; and a union bound, we deduce for any $\delta, D > 0$ that there exists a constant $C_1 = C_1 (\varepsilon, \alpha, \nu, \delta, D) > 1$ such that
\begin{flalign*}
\mathbb{P} \Bigg[ \max_{u \in \mathcal{B}_{\rho N} (N z_0) \cap \mathbb{T}} \bigg| N^{-1} \Xi (u) - \mathfrak{H} (N^{-1} u) - \varphi (N^{-1} v) + \mathfrak{H} (z_0) + \frac{ \omega^2 \rho^2}{2} \bigg| > N^{\delta - 1} \Bigg] < C_1 N^{-D - 2}.
\end{flalign*}

\noindent In particular, setting $u = Nz_0$ and $u \in \partial \mathcal{B}_{\rho} (Nz_0) \cap \mathbb{T}$, we obtain that
\begin{flalign}
\label{xiuhuestimate}
\mathbb{P} \left[  N^{-1} \Xi (Nz_0) > \varphi (N^{-1} v) - \frac{ \omega^2 \rho^2}{2} + N^{\delta - 1} \right] < C_1 N^{-D - 2},
\end{flalign}

\noindent and 
\begin{flalign}
\label{xiuhuestimate2}
\displaystyle\max_{u \in \partial \mathcal{B}_{\rho N} (N z_0) \cap \mathbb{T}} \mathbb{P} \Bigg[ N^{-1} \Xi (u) < \varphi (N^{-1} v) - \mathfrak{H} (z_0) + \mathfrak{H} (N^{-1} u)  - \frac{ \omega^2 \rho^2}{2} - N^{\delta - 1} \Bigg] < C_1 N^{-D - 2},
\end{flalign}

\noindent respectively. Now, by the third bound in \eqref{gzetadeltakappa} and the facts that $\rho = N^{-1 / 4}$ and $\omega = N^{-\alpha / 5} > N^{-1 / 5}$, we have for $\delta = \frac{1}{20}$ that $ \omega^2 \rho^2 > 4 N^{\delta - 1}$, if $N$ is sufficiently large. Thus, taking $\delta = \frac{1}{20}$ in \eqref{xiuhuestimate2} and then applying \Cref{functionhuwvz}, yields 
\begin{flalign}
\label{xiuhuestimate3}
\displaystyle\max_{u \in \partial \mathcal{B}_{\rho N} (N z_0) \cap \mathbb{T}} \mathbb{P} \Bigg[ N^{-1} \Xi (u) < \varphi \big( N^{-1} u + N^{-1} v - z_0 \big) + \frac{ \omega^2 \rho^2}{4} \Bigg] < C_1 N^{-D - 2},
\end{flalign}

\noindent and so $\Xi (u) |_{\mathcal{B}_{\rho N} (z_0)}$ is likely larger than $N \varphi |_{\mathcal{B}_{\rho} (v)}$. Since $N^{-1} H$ is bounded above by $\varphi$ on the event $\Omega^c = \bigcap_{u \in \mathbb{V}(R)} \Omega_{j; k - 1} (u)^c$, \eqref{xiuhuestimate3} implies that 
\begin{flalign}
\label{xiuhuestimate4}
\displaystyle\max_{u \in \partial \mathcal{B}_{\rho N} (N z_0)} \mathbb{P} \Bigg[ \Xi (u)  \textbf{1}_{\Omega^c} < H (u - Nz_0 + v) \textbf{1}_{\Omega^c}  \Bigg] < C_1 N^{-D - 2},
\end{flalign}

\noindent meaning that $H |_{\mathcal{B}_{N \rho} (v)}$ is likely bounded above by $\Xi |_{\partial \mathcal{B}_{\rho N} (N z_0)}$ on $\Omega^c$. 

 Thus, restricting to the tiling of $R$ to $\mathcal{B}_{\rho N} (v) \cap \mathbb{T}$; restricting the tiling of $\mathcal{X}$ to $\mathcal{B}_{\rho N} (Nz_0) \cap \mathbb{T}$; applying \Cref{monotoneheightcouple}; and using a union bound in \eqref{xiuhuestimate4} over all $u \in \partial \mathcal{B}_{\rho N} (Nz_0) \cap \mathbb{T}$, we deduce the existence of a coupling between $\Xi$ and $H$ such that $\mathbb{P} \big[ \Xi (Nz_0) \textbf{1}_{\Omega^c} < H (v) \textbf{1}_{\Omega^c}  \big] < 4 C_1 N^{-D}$. This, \eqref{xiuhuestimate}, the fact that $\omega^2 \rho^2 > 4 N^{\delta - 1}$, and a union bound together imply
\begin{flalign*}
\mathbb{P} \Bigg[ H (v)  \textbf{1}_{\Omega^c} > \bigg( \varphi (N^{-1} v) - \displaystyle\frac{ \omega^2 \rho^2}{4} \bigg)  \textbf{1}_{\Omega^c}  \Bigg] < 5 C_1 N^{-D},
\end{flalign*}

\noindent from which we deduce the proposition since $\varphi (N^{-1} v) - \varphi_{j; k} (N^{-1} v) = N^{-1} \omega \psi (N^{-1} v) < \frac{\omega^2 \rho^2}{4}$ for sufficiently large $N$. In the latter statement, the first equality follows from \eqref{functionkj} and the second estimate from the facts that $\psi (N^{-1} v) \le \Upsilon \le 2 e^{\zeta} + 2 \le 2 e^{\zeta + 1}$; that $\zeta = (\log N)^{1 - \nu / 2}$; that $\omega = N^{-\alpha / 5} > N^{-1 / 5}$; and that $\rho = N^{-1 / 4}$. 
\end{proof}

\appendix

\section{Proofs of \texorpdfstring{\Cref{derivativeshestimate}}{} and \texorpdfstring{\Cref{perturbationboundary}}{}} 

\label{ProofEstimateu}

In this section we establish \Cref{derivativeshestimate} and \Cref{perturbationboundary}. We begin in \Cref{EstimateEquation} by first recalling several known estimates about uniformly elliptic partial differential equations and then deducing \Cref{derivativeshestimate} as a consequence. We next prove \Cref{perturbationboundary} in \Cref{BoundaryPerturbationVariational}.

\subsection{Proof of \texorpdfstring{\Cref{derivativeshestimate}}{}} 

\label{EstimateEquation}

In this section we first recall several results from \cite{EDSO} about regularity properties of uniformly elliptic partial differential equations, and then use them to establish \Cref{derivativeshestimate}. 

To that end, let us fix an integer $d > 0$ and an open subset $\mathfrak{R} \subseteq \mathbb{R}^d$. Following Section 6.1 of \cite{EDSO}, we define variants of the norms from \Cref{LawGlobal2} where we now weight by the distance from a point to the boundary of $\mathfrak{R}$. First, for any $z, z' \in \mathfrak{R}$, set 
\begin{flalign*} 
d_z = d_z (\mathfrak{R}) = d(z, \partial \mathfrak{R}); \qquad d_{z, z'} = d_{z, z'} (\mathfrak{R}) = \min \{ d_z, d_{z'} \}.
\end{flalign*} 

\noindent Additionally, for any function $f \in \mathcal{C}^{k, \alpha} (\mathfrak{R})$; integers $k, m \in \mathbb{Z}_{\ge 0}$; and real number $\alpha \in (0, 1]$, set  
\begin{flalign*}
& [f]_k^{(m)} = [f]_{k, 0}^{(m)} = [f]_{k, 0; \mathfrak{R}}^{(m)} = \displaystyle\max_{\substack{\gamma \in \mathbb{Z}_{\ge 0}^d \\ |\gamma| = k}} \displaystyle\sup_{z \in \mathfrak{R}} d_z^{k + m} \big| \partial_{\gamma} f (z) \big|; \\
& [f]_{k, \alpha}^{(m)} = [f]_{k, \alpha; \mathfrak{R}}^{(m)} = \displaystyle\sup_{\substack{y, z \in \mathfrak{R} \\ y \ne z}} \displaystyle\sup_{\substack{\gamma \in \mathbb{Z}_{\ge 0}^d \\ |\gamma| = k}} d_{y, z}^{k + m + \alpha} \displaystyle\frac{\big| \partial_{\gamma} f(y) - \partial_{\gamma} f(z) \big|}{|y - z|^{\alpha}};\\
\| f \|_k^{(m)} = \| f \|_{k, 0}^{(m)} & = \| f \|_{k, 0; \mathfrak{R}}^{(m)} = \displaystyle\sum_{j = 0}^k [f]_{j, 0; \mathfrak{R}}^{(m)}; \qquad  \big\| f \|_{k, \alpha} = \big\| f \big\|_{k, \alpha; \mathfrak{R}}^{(m)}  = \| f \|_k^{(m)} + [f]_{k, \alpha}^{(m)},
\end{flalign*}

\noindent which give points closer to $\partial \mathfrak{R}$ ``smaller weight.'' Here, we recall from \Cref{LawGlobal2} that $\partial_{\gamma} = \prod_{i = 1}^d (\partial_i)^{\gamma_i}$ for any $d$-tuple $(\gamma_1, \gamma_2, \ldots , \gamma_d) \in \mathbb{Z}_{\ge 0}^d$. 

Observe (see equation (6.11) of \cite{EDSO}) that, for any $f, g \in \mathcal{C}^{0, \alpha} (\mathfrak{R})$, these norms satisfy 
\begin{flalign}
\label{fgestimatefg}
\| fg \|_{0, \alpha}^{(m)} \le \max \big\{ \| f \|_{0, \alpha}^{(m)} \| g \|_0, \| f \|_0 \| g \|_{0, \alpha}^{(m)} \big\}.
\end{flalign}

The following lemma, which appears as Theorem 6.2 of \cite{EDSO}, provides the \emph{interior Schauder estimates} for solutions to uniformly elliptic partial differential equations with H\"{o}lder continuous coefficients. Although these bounds hold in all dimensions $d \ge 2$, we only state them for $d = 2$. Throughout, we recall from \eqref{brzdefinition} that $\mathcal{B} = \mathcal{B}_1$ is a disk of radius $1$ centered at $(0, 0)$ (as such, it has smooth boundary, and so the results mentioned below from \cite{EDSO} will apply).

\begin{lem}[{\cite[Theorem 6.2]{EDSO}}]
	
	\label{aijuestimates}
	
	For fixed real numbers $B \in \mathbb{R}_{> 0}$ and $\alpha \in (0, 1)$, there exists a constant $C = C(B, \alpha) > 1$ such that the following holds. For each $j, k \in \{ x, y \}$, let $a_{jk}, b_j \in \mathcal{C}^{0, \alpha} (\mathcal{B})$ denote functions such that $a_{xy} = a_{yx}$ and
	\begin{flalign*}
	\displaystyle\max_{j, k \in \{ x, y \}} \| a_{jk} \|_{0, \alpha}^{(0)} < B; \qquad \displaystyle\max_{j \in \{ x, y \}} \| b_j \|_{0, \alpha}^{(1)} < B; \qquad \displaystyle\inf_{z \in \mathcal{B}} \displaystyle\sum_{j, k \in \{ x, y \}} a_{jk} (z) \xi_i \xi_j \ge B^{-1} |\xi|^2,
	\end{flalign*}
	
	\noindent all hold for any $\xi = (\xi_x, \xi_y) \in \mathbb{R}^2$. Let $g \in \mathcal{C}^{0, \alpha} (\mathcal{B})$ denote a function, and suppose that $F \in \mathcal{C}^{2, \alpha} (\mathcal{B})$ satisfies the elliptic partial differential equation 
	\begin{flalign}
	\label{aijubiuf}
	\displaystyle\sum_{j, k \in \{ x, y\}} a_{jk} (z) \partial_j \partial_k F (z) + \displaystyle\sum_{j \in \{ x, y \}} b_j (z) \partial_j F (z) = g, \quad \text{for each $z \in \mathcal{B}$.}
	\end{flalign}
	
	\noindent Then, $\| F \|_{2, \alpha}^{(0)} \le C \big( \| F \|_{0; \mathcal{B}} + \| g \|_{0, \alpha}^{(2)} \big)$.
	
\end{lem}

We will be interested in properties of solutions to certain families of non-linear, uniformly elliptic partial differential equations. The following result, which appears as Theorem 10.1 of \cite{EDSO}, provides a comparison principle for such solutions. 

\begin{lem}[{\cite[Theorem 10.1]{EDSO}}]
	
	\label{aijcomparison} 
	
	Fix a constant $B > 0$ and, for each $j, k \in \{ x, y \}$, let $a_{jk} \in \mathcal{C}^1 (\mathbb{R}^2)$ denote functions such that $a_{xy} = a_{yx}$ and
	\begin{flalign}
	\label{xijxikbsum}
	\displaystyle\inf_{z \in \mathbb{R}^2} \displaystyle\sum_{j, k \in \{x, y\}} a_{jk} (z) \xi_j \xi_k \ge B^{-1} |\xi|^2, \qquad \text{for any $\xi = (\xi_x, \xi_y) \in \mathbb{R}^2$.}
	\end{flalign}
	
	\noindent Moreover, for each $i \in \{ 1, 2 \}$, let $f_i: \partial \mathcal{B} \rightarrow \mathbb{R}$ denote a continuous function, and suppose that $F_i: \overline{\mathcal{B}} \rightarrow \mathbb{R}$ with $F_i \in \mathcal{C}^2 (\mathcal{B}) \cap \mathcal{C} (\overline{\mathcal{B}})$ is a solution to the partial differential equation
	\begin{flalign}
	\label{ajkfiequation1} 
	\displaystyle\sum_{j, k \in \{ x, y \}} a_{jk} \big( \nabla F (z) \big) \partial_j \partial_k F (z) = 0, \quad \text{for each $z \in \mathcal{B}$},
	\end{flalign}
	
	\noindent with boundary data $F_i |_{\partial \mathcal{B}} = f_i$. If $f_1 (z) \ge f_2 (z)$ for each $z \in \partial \mathcal{B}$, then $F_1 (z) \ge F_2 (z)$ for each $z \in \mathcal{B}$. 
	
\end{lem}

The following lemma provides interior regularity estimates for solutions to equations of the form \eqref{ajkfiequation1}. This result is known (for instance, see Chapter 3.6 of \cite{EDE}), but we  quickly outline its proof. 

\begin{lem} 
	
\label{aijgradientfestimate} 

For a fixed integer $m \in \mathbb{Z}_{> 1}$ and real numbers $r, B \in \mathbb{R}_{> 0}$, there exists a constant $C (r, B, m) > 1$ such that the following holds. For each $j, k \in \{ x, y \}$, let $a_{jk} \in \mathcal{C}^{m - 1} (\mathbb{R}^2)$ denote functions such that $a_{xy} = a_{yx}$ and
\begin{flalign}
\label{aijxiestimate}
\displaystyle\max_{j, k \in \{ x, y \}} \| a_{jk} \|_{\mathcal{C}^{m - 1} (\mathbb{R}^2)} \le B; \qquad \displaystyle\inf_{z \in \mathbb{R}^2} \displaystyle\sum_{j, k \in \{x, y\}} a_{jk} (z) \xi_j \xi_k \ge B^{-1} |\xi|^2,
\end{flalign}

\noindent for any $\xi = (\xi_x, \xi_y) \in \mathbb{R}^2$. Then for any continuous function $f \in \mathcal{C} (\partial \mathcal{B})$ such that $\| f \|_0 \le B$, there exists a unique $F \in \mathcal{C}^2 (\mathcal{B}) \cap \mathcal{C} (\overline{\mathcal{B}})$ satisfying the partial differential equation \eqref{ajkfiequation1} with boundary data $ F|_{\partial \mathcal{B}} = f$. Moreover, $F \in \mathcal{C}^m (\mathcal{B})$ and $\| F \|_{\mathcal{C}^m (\overline{\mathcal{B}}_{1 - r})} \le C$. 	
	
\end{lem}  

\begin{proof}[Proof (Outline)]
	
	For any $\alpha \in (0, 1)$, the existence of a function $F \in \mathcal{C}^{2, \alpha} (\mathcal{B})$ satisfying \eqref{ajkfiequation1} follows from Theorem 12.5 of \cite{EDSO}; its uniqueness is a consequence of the comparison principle \Cref{aijcomparison} for equations of the form \eqref{ajkfiequation1}.  Therefore, it remains to estimate $\| F \|_{\mathcal{C}^m (\overline{\mathcal{B}}_{1 - r})}$. Bounds of this type are known to follow from an initial estimate on $[F]_{\mathcal{C}^{1, \alpha} (\mathcal{B}_{1 - r})}$ and an inductive application of the Schauder estimate \Cref{aijuestimates}. Let us briefly outline this procedure. 
	
	To that end, for each $j, k \in \{ x, y \}$ define the measurable function $\widetilde{a}_{jk}: \mathcal{B} \rightarrow \mathbb{R}$ by setting $\widetilde{a}_{jk} (z) = a_{jk} \big( \nabla F (z) \big)$ for each $z \in \mathcal{B}$. Then, the $\widetilde{a}_{jk}$ satisfy $\| \widetilde{a}_{jk} \|_0 \le B$ and \eqref{xijxikbsum}. Since $F \in \mathcal{C}^2 (\mathcal{B})$ satisfies \eqref{ajkfiequation1}, it also satisfies the linear equation 
	\begin{flalign}
	\label{ajkfjkidentity}
	\sum_{j, k \in \{x, y \}} \widetilde{a}_{jk} (z) \partial_j \partial_k F (z) = 0, \quad \text{for each $z \in \mathcal{B}$}.
	\end{flalign}
	
	Therefore, Theorem 12.4 of \cite{EDSO} yields constants $\alpha = \alpha (B) > 0$ and $C_1 = C_1 (B) > 1$ such that $[F]_{1, \alpha}^{(0)} < C_1 \| F \|_0$. By the comparison principle \Cref{aijcomparison} for equations of the form \eqref{ajkfiequation1}, we have that $\| F \|_0 \le \| f \|_0 \le B$, and so $[F]_{1, \alpha}^{(0)} \le C_1 B$. 
	
	Now since $\| a_{jk} \|_{\mathcal{C}^1 (\mathbb{R}^2)} \le B$ and $\widetilde{a}_{jk} (z) = a_{jk} \big( \nabla F(z) \big)$, it follows that $\| \widetilde{a}_{jk} \|_{\alpha}^{(0)} = [ \widetilde{a}_{jk}]_{\alpha}^{(0)} + \| \widetilde{a}_{jk} \|_0 \le [F]_{1, \alpha}^{(0)} \| a_{jk}\|_{\mathcal{C}^1 (\mathbb{R}^2)} + B \le C_1 B^2 + B$. Together with the interior Schauder estimate \Cref{aijuestimates} and the fact that $\| F \|_0 \le B$, this yields a constant $C_2 = C_2 (B) > 1$ such that $\| F \|_{2, \alpha}^{(0)} \le C_2$; thus, $\| F \|_{\mathcal{C}^2 (\overline{\mathcal{B}}_{1 - r / m})} \le m^3 r^{-3} C_2$. 
	
	To bound the higher derivatives of $F$, one differentiates \eqref{ajkfjkidentity} with respect to some $i \in \{ x, y \}$. Denoting $G = \partial_i F$, this yields
	\begin{flalign}
	\label{gequationajk}
	\displaystyle\sum_{j, k \in \{ x, y \}} \widetilde{a}_{jk} (z) \partial_j \partial_k G (z) + \displaystyle\sum_{j \in \{ x, y \}} \widetilde{b}_j (z) \partial_j G (z) = - \displaystyle\sum_{\substack{j, k \in \{ x, y \} \\ j, k \ne i}} \widetilde{c}_{jk} (z) \partial_j \partial_k F(z),
	\end{flalign}
	
	\noindent for some coefficients $\{ \widetilde{b}_j \}$ and $\{ \widetilde{c}_{jk} \}$ obtained by taking (order $1$) derivatives of the $\big\{ \widetilde{a}_{ij} \}$; observe that \eqref{gequationajk} is an equation of the form \eqref{aijubiuf}. One can quickly verify using \eqref{fgestimatefg} that
	\begin{flalign*} 
	 \| \partial_i \widetilde{a}_{jk} \|_{0, \alpha; \mathcal{B}_{1 - r / m}}^{(0)} \le 2 \| a_{jk} \|_{\mathcal{C}^1 (\mathbb{R}^2)} \| F \|_{\mathcal{C}^2 (\overline{\mathcal{B}}_{1 - r / m})} + 2 \| a_{jk} \|_{\mathcal{C}^2 (\mathbb{R}^2)}  \| F \|_{\mathcal{C}^{2, \alpha} (\overline{\mathcal{B}}_{1 - r / m})},
	\end{flalign*} 
	
	\noindent for any $i, j, k \in \{ x, y \}$. So, the estimates 
	\begin{flalign} 
	\label{ajkderivativesfiderivatives} 
	\| a_{jk} \|_{\mathcal{C}^2 (\mathbb{R}^2)} \le B; \qquad \| F_i \|_{\mathcal{C}^{2, \alpha} (\overline{\mathcal{B}}_{1 - r / m})} < m^3 r^{-3} C_2,
	\end{flalign}
	
	\noindent yield a constant $C_3 = C_3 (r, B, m) > 1$ such that $\| \partial_i \widetilde{a}_{jk} \|_{0, \alpha; \mathcal{B}_{1 - r / m}}^{(0)} \le C_3$. Combining this with \eqref{gequationajk}, \Cref{aijuestimates}, the second bound in \eqref{ajkderivativesfiderivatives}, and \eqref{fgestimatefg}, we deduce the existence of a constant $C_4 = C_4 (r, B, m) > 1$ such that $\| G \|_{2, \alpha; \mathcal{B}_{1 - r / m}}^{(0)} \le C_4$. Recalling that $\partial_i F = G$ and ranging over $i \in \{ x, y \}$ then yields $\| F \|_{\mathcal{C}^{3, \alpha} (\overline{\mathcal{B}}_{1 - 2r / m})} < m^4 r^{-4} C_4 + \| F \|_0 \le m^4 r^{-4} C_4 + B$. 
	
	Repeating this procedure $m - 3$ additional times, we deduce the existence of a constant $C = C(r, B, m) > 1$ such that $\| F \|_{\mathcal{C}^{m, \alpha} (\overline{\mathcal{B}}_{1 - r})} \le C$. This implies the lemma.
\end{proof}

We can now deduce \Cref{derivativeshestimate}.

\begin{proof}[Proof of \Cref{derivativeshestimate}]
	
	By suitably shifting $\mathfrak{h}$, we may assume throughout this proof that $\mathcal{H} (0, 0) = 0$, and so $\| \mathfrak{h} \|_0 \le 1$ as $\mathcal{H}$ is $1$-Lipschitz on $\mathcal{B}$. Then since $\nabla \mathcal{H} (z) \in \mathcal{T}_{\varepsilon}$ for each $z \in \mathcal{B}$, \Cref{haijequations} implies that $\mathcal{H} \in \mathcal{C}^2 (\mathcal{B}) \cap \mathcal{C} (\overline{\mathcal{B}})$ satisfies the partial differential equation \eqref{aijh} (where we recall that the $\mathfrak{a}_{jk}$ there are defined by \eqref{aijst}). 
	
	Now, recall from \Cref{concavesigmat} that the function $\sigma$ from \eqref{lsigma} is uniformly concave and smooth on $\mathcal{T}_{\varepsilon / 2}$. Thus, since $\mathfrak{a}_{jk} = \partial_j \partial_k \sigma$ for each $j, k \in \{ x, y \}$, there exist a constant $B = B (\varepsilon) > 0$ and functions $a_{jk} \in \mathcal{C}^1 (\mathbb{R}^2)$ satisfying $a_{xy} = a_{yx}$ and the following three properties. 
	
	\begin{enumerate}
		\item  For any $j, k \in \{ x, y \}$ and $(s, t) \in \mathcal{T}_{\varepsilon / 2}$, we have that $a_{jk} (s, t) = \mathfrak{a}_{jk} (s, t)$. 
		
		\item For any $j, k \in \{ x, y \}$, we have that $\| a_{jk} \|_{\mathcal{C}^2 (\mathbb{R}^2)} < B$. 
		
		\item The bound in \eqref{xijxikbsum} holds. 
	
	\end{enumerate}
	
	Then the first of these properties, together with \Cref{haijequations} and the fact that $\nabla \mathcal{H} (z) \in \mathcal{T}_{\varepsilon}$ for each $z \in \mathcal{B}$, implies that $\mathcal{H}$ satisfies the equation \eqref{ajkfiequation1}. This, together with the $m = 2$ case of \Cref{aijgradientfestimate} and the second and third properties listed above, yields the lemma. 
\end{proof}

\subsection{Proof of \Cref{perturbationboundary}}

\label{BoundaryPerturbationVariational}

In this section we establish \Cref{perturbationboundary}. To that end, we first recall the following result, which appears as Theorem 4.3 of \cite{MCFARS} and provides an interior approximation for the gradient of a nearly linear (recall \Cref{linearst}) maximizer of $\mathcal{E}$.

\begin{prop}[{\cite[Theorem 4.3]{MCFARS}}]
	
	\label{interiorfunctional}
	
	For any fixed $\varepsilon \in \big( 0, \frac{1}{4} \big)$, there exists a $\delta = \delta (\varepsilon) > 0$ such that the following holds. Let $\mathfrak{h}: \partial \mathcal{B} \rightarrow \mathbb{R}$ denote a function admitting an admissible extension to $\mathcal{B}$, and let $\mathcal{H} \in \Adm (\mathcal{B}; \mathfrak{h})$ denote the maximizer of $\mathcal{E}$ on $\mathcal{B}$ with boundary data $\mathfrak{h}$. If there exists a pair $(s, t) \in \mathcal{T}_{\varepsilon}$ such that $\mathcal{H}$ is $\delta$-nearly linear of slope $(s, t)$ on $\mathcal{B}$, then $\sup_{z \in \mathcal{B}_{1 / 2}} \big| \nabla \mathcal{H} (z) - (s, t) \big| < \varepsilon$. 
	
\end{prop}

To establish \Cref{perturbationboundary}, we will first use \Cref{interiorfunctional} to show that $\nabla \mathcal{H}_i (z) \in \mathcal{T}_{\varepsilon / 2}$ for each $i \in \{ 1, 2 \}$ and $z \in \mathcal{B}_{1 / 2}$. Then \Cref{haijequations} will imply that $\mathcal{H}_1$ and $\mathcal{H}_2$ both satisfy the same uniformly elliptic, non-linear partial differential equation on $\mathcal{B}_{1 / 2}$. So, we must bound $\big| \nabla \mathcal{H}_1 (z) - \nabla \mathcal{H}_2 (z) \big|$ assuming that the $\mathcal{H}_i$ satisfy the same uniformly elliptic partial differential equation. This will be facilitated through the following proposition. 

\begin{prop}
	
	\label{ellipticperturbation} 
	
	For any $r \in (0, 1)$ and $B\in \mathbb{R}_{> 0}$, there exist constants $\delta = \delta (r, B) \in (0, 1)$ and $C = C (r, B) > 1$ such that the following holds. For each $j, k \in \{ x, y \}$ let $a_{jk} \in \mathcal{C}^4 (\mathbb{R}^2)$ denote functions satisfying $a_{xy} = a_{yx}$ and 
	\begin{flalign}
	\label{aijxiestimate2}
	\displaystyle\max_{j, k \in \{ x, y \}} \| a_{jk} \|_{\mathcal{C}^4 (\mathbb{R}^2)} \le B; \qquad \displaystyle\inf_{z \in \mathbb{R}^2} \displaystyle\sum_{j, k \in \{x, y\}} a_{jk} (z) \xi_j \xi_k \ge B^{-1} |\xi|^2,
	\end{flalign}
	
	\noindent for any $\xi = (\xi_x, \xi_y) \in \mathbb{R}^2$. For each $i \in \{ 1, 2 \}$, let $f_i \in \mathcal{C} (\partial \mathcal{B})$ denote a continuous function such that $\sup_{z \in \partial \mathcal{B}} \big| f_i (z) \big| \le B$. Additionally, for each $i \in \{ 1, 2 \}$, let $F_i \in \mathcal{C}^5 (\mathcal{B}) \cap \mathcal{C} (\overline{\mathcal{B}})$ denote the solution to the equation
	\begin{flalign}
	\label{ajkfiequation}
	\displaystyle\sum_{j, k \in \{ x, y \}} a_{jk} \big( \nabla F_i (x) \big) \partial_j \partial_k F_i (z) = 0, \quad \text{for each $z \in \mathcal{B}$},
	\end{flalign} 

	\noindent with boundary data $F_i |_{\partial \mathcal{B}} = f_i$. If $\sup_{z \in \partial \mathcal{B}} \big| f_1 (z) - f_2 (z) \big| < \delta$, then 
	\begin{flalign}
	\label{estimatef1f2elliptic}
	\sup_{z \in \mathcal{B}_{1 - r}} \big| \nabla F_1 (z) - \nabla F_2 (z) \big| \le C \sup_{z \in \partial \mathcal{B}} \big| f_1 (z) - f_2 (z) \big|. 
	\end{flalign}
	
\end{prop}

To establish \Cref{ellipticperturbation}, we will use the following lemma, which is an interpolation estimate that bounds derivatives of a function in terms of its higher (and lower) derivatives. Although the proof of \Cref{fjderivativesestimate} is very similar to that of Lemma 6.32 of \cite{EDSO}, the precise form of the two statements are slightly different; so, we will provide its proof.  

\begin{lem} 
	
	\label{fjderivativesestimate} 
	
	For any integers $k > j > 0$ and real number $r \in (0, 1)$, there exists a constant $C = C(k, r) > 1$ such that the following holds. Let $F \in \mathcal{C}^k (\mathcal{B})$ denote a function, and suppose that $A, B \in \mathbb{R}_{> 0}$ are such that $A \ge  \| F \|_{0; \mathcal{B}_{1 - r / 2}}$ and $B \ge [F]_{k, 0; \mathcal{B}_{1 - r / 2}}$. If $r > C \big( \frac{A}{B} \big)^{1 / k}$, then 
	\begin{flalign}
	\label{gradientf1f2estimate1}
	[F]_{j, 0; \mathcal{B}_{1 - r}} \le C A^{1 - j / k} B^{j / k}. 
	\end{flalign}

\end{lem} 

\begin{proof}
	
	Since the proof of this lemma is similar to that of Lemma 6.32 of \cite{EDSO}, we only establish it in the case $(j, k) = (1, 2)$; the proof for general $k > j \ge 1$ is very similar. 
	
	To that end, set $t_0 = \big( \frac{A}{B} \big)^{1 / 2}$; let $r \in (2 t_0, 1)$; and assume to the contrary that there exists some $i \in \{ x, y \}$ and $z_0 \in \mathcal{B}_{1 - r}$ such that $\big| \partial_i F (z_0) \big| > 3 (AB)^{1 / 2}$. For notational convenience, we assume $i = x$ and $\partial_i F (z_0) > 3 (AB)^{1 / 2}$, as the proofs in the alternative cases $i = y$ or $\partial_i F (z_0) < - 3 (AB)^{1 / 2}$ are entirely analogous; further set $w = (1, 0) \in \mathbb{R}^2$. Then the facts that $[F]_{2, 0; \mathcal{B}_{1 - r / 2}} \le B$ and $r > 2 t_0$ together imply that $\big| \partial_x F (z_0 + t w) - \partial_x F (z_0) \big| \le Bt \le (AB)^{1 / 2}$ for each $t \in [0, t_0]$. 
	
	Thus, $\partial_x F (z_0 + t w) > 2 (AB)^{1 / 2}$ whenever $t \in [0, t_0]$, and so integration over $t$ yields $\big| F(z_0) - F(z_0 + t_0 w) \big| > 2 A$. This contradicts the fact that $\| F \|_{0; \mathcal{B}_{1 - r / 2}} \le A$, which verifies \eqref{gradientf1f2estimate1}.
\end{proof}

Now we can establish \Cref{ellipticperturbation}.

\begin{proof}[Proof of \Cref{ellipticperturbation}]
	
	Denoting 
	\begin{flalign*} 
	\varsigma = \sup_{z \in \partial \mathcal{B}} \big| f_1 (z) - f_2 (z) \big|,
	\end{flalign*} 
	
	\noindent we deduce from the comparison principle \Cref{aijcomparison} for partial differential equations of the form \eqref{ajkfiequation} (see also \Cref{h1h2gamma}) that 
	\begin{flalign}
	\label{f1f2estimate}
	\sup_{z \in \mathcal{B}} \big| F_1 (z) - F_2 (z) \big| < \varsigma; \qquad \displaystyle\max_{i \in \{ 1, 2 \}} \displaystyle\sup_{z \in \mathcal{B}} \big| F_i (z) \big| \le B,
	\end{flalign}
	
	\noindent where the second bound follows from the fact that $\big| f_i (z) \big| \le B$, for each $i \in \{1, 2 \}$ and $z \in \partial \mathcal{B}$. 
	
	Moreover, the $m = 5$ case of \Cref{aijgradientfestimate} yields a constant $C_1 = C_1 (r, B) > 1$ such that 
	\begin{flalign}
	\label{fi3bestimate}
	\displaystyle\max_{i \in \{ 1, 2 \}} \| F_i \|_{\mathcal{C}^5 (\overline{\mathcal{B}}_{1 - r / 8})} < C_1.
	\end{flalign}
	
	\noindent Thus \eqref{gradientf1f2estimate1} (with the $F$ there equal to $F_1 - F_2$ here; the $A$ there equal to $\varsigma$ here; the $B$ there equal to $C_1$ here; and the $(j, k)$ there equal to either $(1, 5)$ or $(2, 5)$), together with \eqref{f1f2estimate} and \eqref{fi3bestimate}, yields a constant $C_2 = C_2 (r, B) > 1$ such that
	\begin{flalign}
	\label{gradientf1f2estimate2}
	\begin{aligned}
	\displaystyle\sup_{z \in \mathcal{B}_{1 - r / 4}} \big| \nabla F_1 (z) - \nabla F_2 (z) \big| & \le C_2 \varsigma^{4 / 5}; \qquad \| F_1 - F_2 \|_{\mathcal{C}^2 (\overline{\mathcal{B}}_{1 - r / 4})} \le C_2 \varsigma^{3 / 5},
	\end{aligned}
	\end{flalign}
	
	\noindent whenever $r > C_2 \varsigma^{1 / 5}$. To improve the first bound in \eqref{gradientf1f2estimate2} to \eqref{estimatef1f2elliptic}, define $g: \mathcal{B} \rightarrow \mathbb{R}$ by setting
	\begin{flalign*} 
	g(z) = \varsigma^{-1 / 2} \big( F_2 (z) - F_1 (z) \big), \qquad \text{for each $z \in \overline{\mathcal{B}}$.}
	\end{flalign*} 
	
	\noindent Then the first bound in \eqref{f1f2estimate} and \eqref{gradientf1f2estimate2} together yield
	\begin{flalign}
	\label{estimateg1}
	\| g \|_{0; \mathcal{B}_{1 - r / 4}} \le \varsigma^{1 / 2}; \qquad \| g \|_{\mathcal{C}^1 (\overline{\mathcal{B}}_{1 - r / 4})} \le 2 C_2 \varsigma^{3 / 10}; \qquad \| g \|_{\mathcal{C}^2 (\overline{\mathcal{B}}_{1 - r / 4})} \le C_2 \varsigma^{1 / 10}. \qquad .
	\end{flalign}
	
	\noindent Next, setting $i = 2$ and $F_2 = F_1 + \varsigma^{1 / 2} g$ in \eqref{ajkfiequation}, we obtain 
	\begin{flalign}
	\label{2ajkfiequation}
	\displaystyle\sum_{1 \le j, k \le 2} a_{jk} \big( \nabla F_1 (z) + \varsigma^{1 / 2} \nabla g(z) \big) \big( \partial_j \partial_k F_1 (z) + \varsigma^{1 / 2} \partial_j \partial_k g (z) \big) = 0.
	\end{flalign}
	
	\noindent Then, subtracting the $i = 1$ case of \eqref{ajkfiequation} from \eqref{2ajkfiequation}, we deduce that
	\begin{flalign}
	\label{ghequationmunukappa}
	\begin{aligned} 
	& \displaystyle\sum_{j, k \in \{ x, y \} } \Big( \mu_{jk} (z) \partial_j \partial_k g (z) + \nu_{jk} (z) \kappa_{jk} (z) \cdot \nabla g(z) \Big) \\
	& \qquad \quad = - \varsigma^{1 / 2} \displaystyle\sum_{j, k \in \{ x, y \} 	} \Big(  \big( \nabla g(z) \cdot \kappa_{jk} (z) \big) \partial_j \partial_k g (z) + \nu_{jk} (z) h_{jk} (z) + \varsigma^{1 / 2} h_{jk} (z) \partial_j \partial_k g (z) \Big),
	\end{aligned} 
	\end{flalign} 
	
	\noindent where for each $j, k \in \{ x, y \}$, we have defined $\mu_{jk}, \nu_{jk}, h_{jk} \in \mathcal{C} (\mathcal{B})$ and $\kappa_{jk}: \mathcal{B} \rightarrow \mathbb{R}^2$ by setting
	\begin{flalign}
	\label{hdefinition}
	\begin{aligned}
	& \mu_{jk} (z) =  a_{jk} \big( \nabla F_1 (z) \big); \qquad \kappa_{jk} (z) = \nabla a_{jk} \big( \nabla F_1 (z) \big); \qquad \nu_{jk} (z) = \partial_j \partial_k F_1 (z);  \\
	& h_{jk} (z) = \varsigma^{-1} \Big( a_{jk} \big( \nabla F_1 (z) + \varsigma^{1 / 2} \nabla g (z) \big) - a_{jk} \big( \nabla F_1 (z) \big) - \varsigma^{1 / 2} \nabla g (z) \cdot \nabla a_{jk} \big( \nabla F_1 (z) \big) \Big), \\
	\end{aligned} 
	\end{flalign} 
	
	\noindent for each $z \in \mathcal{B}$. Now, observe that \eqref{fgestimatefg}, \eqref{fi3bestimate}, and the fact that $\| a_{jk} \|_{\mathcal{C}^4 (\mathbb{R}^2)} \le B$ together imply the existence of a constant $C_3 = C_3 (r, B) > 1$ such that 
	\begin{flalign}
	\label{munukappa} 
	\| \mu_{jk} \|_{\mathcal{C}^{0, 1 / 2} (\overline{\mathcal{B}}_{1 - r / 4})} \le C_3; \qquad \| \kappa _{jk} (z) \|_{\mathcal{C}^{0, 1 / 2} (\overline{\mathcal{B}}_{1 - r / 4})} \le C_3; \qquad \| \nu_{jk} (z) \|_{\mathcal{C}^{0, 1 / 2} (\overline{\mathcal{B}}_{1 - r / 4})} \le C_3.
	\end{flalign} 
	
	Since the $a_{jk}$ satisfy \eqref{aijxiestimate2}, it follows from \eqref{ghequationmunukappa}, \eqref{fgestimatefg}, and the $\alpha = \frac{1}{2}$ case of the interior Schauder estimate \Cref{aijuestimates} that there exists a constant $C_4 = C_4 (r, B) > 1$ such that
	\begin{flalign*}
	\| g \|_{2, 1 / 2; \mathcal{B}_{1 - r / 4}}^{(0)} & \le C_4 \| g \|_{0; \mathcal{B}_{1 - r / 4}} + C_4 \varsigma^{1 / 2} \displaystyle\max_{j, k \in \{ x, y \} } \Big\|  \big( \nabla g(z) \cdot \kappa_{jk} (z) \big) \partial_j \partial_k g (z) \Big\|_{0, 1 / 2; \mathcal{B}_{1 - r / 4}}^{(2)} \\
	& \qquad  + C_4 \varsigma^{1 / 2} \displaystyle\max_{j, k \in \{ x, y \} }  \big\| \nu_{jk} (z) h_{jk} (z) \big\|_{0, 1 / 2; \mathcal{B}_{1 - r / 4}}^{(2)} \\
	& \qquad + C_4 \varsigma  \displaystyle\max_{j, k \in \{ x, y \} } \big\| h_{jk} (z) \partial_j \partial_k g (z) \big\|_{0, 1 / 2; \mathcal{B}_{1 - r / 4}}^{(2)}.
	\end{flalign*} 
	
	\noindent Thus, by \eqref{fgestimatefg} and \eqref{munukappa}, there exists a constant $C_5 = C_5 (r, B) > 1$ such that
	\begin{flalign}
	\label{gestimategh}
	\begin{aligned}
	\| g \|_{2, 1 / 2; \mathcal{B}_{1 - r / 4}}^{(0)} & \le C_5 \varsigma^{1 / 2} \Big( \| g \|_{2, 1 / 2; \mathcal{B}_{1 - r / 4}}^{(2)} \| g \|_{\mathcal{C}^1 (\overline{\mathcal{B}}_{1 - r / 4})} + \| g \|_{1, 1 / 2; \mathcal{B}_{1 - r / 4}}^{(2)} \| g \|_{\mathcal{C}^2 (\overline{\mathcal{B}}_{1 - r / 4})} \\
	& \qquad \qquad \qquad + \| g \|_{\mathcal{C}^2 (\overline{\mathcal{B}}_{1 - r / 4})} \| g \|_{\mathcal{C}^1 (\overline{\mathcal{B}}_{1 - r / 4})} + \| h \|_{0; \mathcal{B}_{1 - r / 4}} + \| h \|_{0, 1 / 2; \mathcal{B}_{1 - r / 4}}^{(2)} \Big) \\
	& \quad + C_5 \varsigma \Big( \| h \|_{0; \mathcal{B}_{1 - r / 4}} \| g \|_{2, 1 / 2; \mathcal{B}_{1 - r / 4}}^{(2)} + \| h \|_{0, 1 / 2; \mathcal{B}_{1 - r / 4}}^{(2)} \| g \|_{\mathcal{C}^2 (\overline{\mathcal{B}}_{1 - r / 4})} \Big)+ C_5 \| g \|_{0; \mathcal{B}_{1 - r / 4}}.
	\end{aligned} 
	\end{flalign} 
	
	\noindent Now, \eqref{estimateg1} yields a constant $C_6 = C_6 (r, B) > 1$ such that 
	\begin{flalign}
	\label{ghestimates}
	\begin{aligned}
	\max \big\{ \| g \|_{\mathcal{C}^1 (\overline{\mathcal{B}}_{1 - r / 4})}, \| g \|_{1, 1 / 2; \mathcal{B}_{1 - r / 4}}^{(2)}, \| g & \|_{\mathcal{C}^2 (\overline{\mathcal{B}}_{1 - r / 4})}, \| h \|_{0; \mathcal{B}_{1 - r / 4}}, \| h \|_{0, 1 / 2; \mathcal{B}_{1 - r / 4}}^{(2)} \big\} \\
	& \quad \le \max \big\{ \| g \|_{\mathcal{C}^2 (\overline{\mathcal{B}}_{1 - r / 4})}, \| h \|_{\mathcal{C}^1 (\overline{\mathcal{B}}_{1 - r / 4})} \big\} <  C_6,
	\end{aligned}
	\end{flalign}
	
	\noindent where the bound on $\| h \|_{\mathcal{C}^1 (\overline{\mathcal{B}}_{1 - r / 4})}$ follows from differentiating the definition \eqref{hdefinition} of $h$; \eqref{fi3bestimate}; \eqref{estimateg1}; the fact that $\| a_{jk} \|_{\mathcal{C}^4 (\mathbb{R}^2)} \le B$ for each $j, k \in \{ x, y \}$; and a Taylor expansion.
	
	Inserting the first bound of \eqref{estimateg1}; \eqref{ghestimates}; and the fact that $\| g \|_{2, 1 / 2; \mathcal{B}_{1 - r / 4}}^{(0)} \ge \| g \|_{2, 1 / 2; \mathcal{B}_{1 - r / 4}}^{(2)}$ into \eqref{gestimategh} then yields a constant $C_7 = C_7 (r, B) > 1$ such that 
	\begin{flalign*}
	(1 - C_7 \varsigma^{1 / 2}) \| g \|_{2, 1 / 2; \mathcal{B}_{1 - r / 4}}^{(0)} \le C_7 \varsigma^{1 / 2},
	\end{flalign*}
	
	\noindent whenever $r > C_2 \varsigma^{1 / 5}$. Thus, if $\delta < \min \big\{ \frac{r^5}{C_2^5}, \frac{1}{4 C_7^2}, 1 \big\}$ and $\varsigma < \delta$, we obtain that
	\begin{flalign}
	\label{f1f2estimate2}
	[ F_1 - F_2 ]_{2, 0; \mathcal{B}_{1 - r / 2}} \le 64 r^{-3} \| F_1 - F_2 \|_{2, 1 / 2; \mathcal{B}_{1 - r / 4}}^{(0)} = 64 r^{-3} \varsigma^{1 / 2} \| g \|_{2, 1 / 2; \mathcal{B}_{1 - r / 4}}^{(0)} \le 128 C_7 r^{-3} \varsigma.
	\end{flalign}
	
	\noindent The proposition now follows from inserting \eqref{f1f2estimate2} and the first bound of \eqref{f1f2estimate} into the $(j, k) = (1, 2)$ case of \eqref{gradientf1f2estimate1}.
\end{proof}

We can now deduce \Cref{perturbationboundary}.

\begin{proof}[Proof of \Cref{perturbationboundary}]

	Since the function $\mathfrak{h}_i$ is $\delta$-linear with slope $(s, t)$ on $\partial \mathcal{B}$ for each $i \in \{ 1, 2 \}$, there exists a linear function $\Lambda_i: \overline{\mathcal{B}} \rightarrow \mathbb{R}$ of slope $(s, t)$ such that $\sup_{z \in \partial \mathcal{B}} \big| \mathfrak{h}_i (z) - \Lambda_i (z) \big| < \delta$. Thus, \Cref{h1h2gamma} implies that $\sup_{z \in \mathcal{B}} \big| \mathcal{H}_i (z) - \Lambda_i (z) \big| < \delta$. 
	
	We may assume throughout this proof that $\mathcal{H}_1 (0, 0) = 0 = \mathcal{H}_2 (0, 0)$; then, the fact that $\mathcal{H}_1$ and $\mathcal{H}_2$ are both $1$-Lipschitz on $\mathcal{B}$ implies that $\| \mathfrak{h}_1 \|_0, \| \mathfrak{h}_2 \|_0 \le 1$. Moreover, since $\Lambda_1$ and $\Lambda_2$ have the same slope $(s, t)$, it follows that $\big| \Lambda_1 (z) - \Lambda_2 (z) \big| = \big| \Lambda_1 (0, 0) - \Lambda_2 (0, 0) \big| \le 2 \delta$ for each $z \in \overline{\mathcal{B}}$. Denoting $\Lambda = \Lambda_1$, we then obtain $\sup_{z \in \mathcal{B}} \big| \mathcal{H}_i (z) - \Lambda (z) \big| < 3 \delta$ for each $i \in \{ 1, 2 \}$. 
	
	Now, for each $i \in \{1, 2 \}$, set 
	\begin{flalign*} 
	 F_i = \mathcal{H}_i |_{\overline{\mathcal{B}}_{1 / 2}}; \qquad f_i = \mathcal{H}_i |_{\partial \mathcal{B}_{1 / 2}}; \qquad  \varsigma = \sup_{z \in \partial \mathcal{B}} \big| \mathfrak{h}_1 (z) - \mathfrak{h}_2 (z) \big|,
	\end{flalign*} 
	
	\noindent so in particular $F_i \in \Adm (\mathcal{B}_{1 / 2}; f_i)$ is the maximizer of $\mathcal{E}$ on $\mathcal{B}_{1 / 2}$ with boundary data $f_i$. Then, \Cref{h1h2gamma} yields $\sup_{z \in \partial \mathcal{B}_{1 / 2}} \big| f_1 (z) - f_2 (z) \big| \le \sup_{z \in \mathcal{B}} \big| \mathcal{H}_1 (z) - \mathcal{H}_2 (z) \big| \le \varsigma$. 
	
	Next, assume that the $3 \delta$ here is less than the $\delta \big( \frac{\varepsilon}{2} \big)$ of \Cref{interiorfunctional}. Then, \Cref{interiorfunctional} and the fact that $\Lambda$ is of slope $(s, t) \in \mathcal{T}_{\varepsilon} \subset \mathcal{T}_{\varepsilon / 2}$ together imply that $\nabla F_i (z) = \nabla \mathcal{H}_i (z) \in \mathcal{T}_{\varepsilon / 2}$, for each $i \in \{1, 2\}$ and $z \in \mathcal{B}_{1 / 2}$. Then the second estimate in \eqref{perturbationboundaryestimate} follows from \Cref{derivativeshestimate} (and \Cref{estimatehrho}), applied to the $F_i$. 
	
	To establish the first, we will proceed as in the proof of \Cref{derivativeshestimate}. Specifically, \Cref{haijequations} implies that each $F_i \in \mathcal{C}^2 (\mathcal{B}_{1 / 2})$ satisfies \eqref{aijh} on $\mathcal{B}_{1 / 2}$ (where we recall that the $\mathfrak{a}_{jk}$ there are given by \eqref{aijst}), with boundary data $f_i$. Moreover, similarly to as indicated in the proof of \Cref{derivativeshestimate}, there exist a constant $B = B (\varepsilon) > 0$ and functions $a_{jk} \in \mathcal{C}^4 (\mathbb{R}^2)$ satisfying $a_{xy} = a_{yx}$ and the following three properties. 
	
	\begin{enumerate}
		\item  For any $j, k \in \{ x, y \}$ and $(s, t) \in \mathcal{T}_{\varepsilon / 4}$, we have that $a_{jk} (s, t) = \mathfrak{a}_{ij} (s, t)$. 
		
		\item For any $j, k \in \{ x, y \}$, we have that $\| a_{jk} \|_{\mathcal{C}^4 (\mathbb{R}^2)} < B$. 
		
		\item The bound \eqref{xijxikbsum} holds. 
		
	\end{enumerate}

	Then, $F_1$ and $F_2$ satisfy the equation \eqref{ajkfiequation} (for each $z \in \mathcal{B}_{1 / 2}$), with boundary data $f_1$ and $f_2$, respectively. Thus, for each $i \in \{ 1, 2 \}$, \Cref{aijgradientfestimate} implies that $F_i \in \mathcal{C}^5 (\mathcal{B}_{1 / 2})$ and \Cref{h1h2compare} implies that $\sup_{z \in \partial \mathcal{B}_{1 / 2}} \big| f_i (z) \big| \le \sup_{z \in \partial \mathcal{B}} \big| \mathfrak{h}_i (z) \big| \le 1$. So, \Cref{ellipticperturbation} (with \Cref{estimatehrho}) yields constants $C = C(\varepsilon) > 1$ and $\delta_0 = \delta_0 (\varepsilon) > 0$ such that 
	\begin{flalign*}
	\displaystyle\sup_{z \in \mathcal{B}_{1 / 4}} \big| \nabla F_1 (z) - \nabla F_2 (z) \big| \le C \displaystyle\sup_{z \in \partial \mathcal{B}_{1 / 2}} \big| f_1 (z) - f_2 (z) \big| \le C \varsigma,
	\end{flalign*}
	
	\noindent whenever $\varsigma < \delta_0$. Thus, we deduce the first bound in \eqref{perturbationboundaryestimate} by additionally imposing $\delta < \delta_0$.
\end{proof}

\section{Proof of \texorpdfstring{\Cref{euv1v2estimategradient}}{}}

\label{ProofGradientEstimateu}

In this section we establish \Cref{euv1v2estimategradient}. To that end, we begin in \Cref{EstimateLinear} by establishing a (likely known) boundary variant of the Cordes-Nirenberg estimate; this is a global $\mathcal{C}^{1, \alpha}$ bound on solutions to linear elliptic partial differential equations, whose coefficients are assumed to be nearly constant but not necessarily continuous. We then use this estimate to prove \Cref{euv1v2estimategradient} in \Cref{GlobalE}.

\subsection{A H\"{o}lder Estimate for Solutions of Elliptic Equations} 

\label{EstimateLinear}

In this section we establish the following lemma, which will be used in the proof of \Cref{euv1v2estimategradient} in \Cref{GlobalE} below. It can be viewed as a two-dimensional boundary variant of the Cordes-Nirenberg estimate (the interior version of which is given, for instance, by Theorem 5.21 of \cite{EPE}), as it bounds the $\mathcal{C}^{1, \alpha}$-norm of a solution to a linear, nearly constant coefficient, uniformly elliptic partial differential equation, for any $\alpha \in (0, 1)$. Since this statement essentially follows from the results in Chapter 12 of \cite{EDSO} on two-dimensional elliptic partial differential equations, we only outline its proof.

\begin{lem}
	
	\label{derivativefestimate1} 
	
	For any fixed real numbers $B > 1$ and $\alpha \in (0, 1)$, there exist constants $\delta = \delta (B, \alpha) > 0$ and $C = C (B, \alpha) > 1$ such that the following holds. For each $j, k \in \{ x, y \}$, let $A_{jk} \in \mathbb{R}$ denote real numbers such that 
	\begin{flalign}
	\label{ajkb} 
	\displaystyle\max_{j, k \in \{ x, y \}} |A_{jk}| < B; \qquad A_{xy} = A_{yx}; \qquad \sum_{j, k \in \{ x, y \}} A_{jk} \xi_j \xi_j \ge B^{-1} |\xi|^2, 
	\end{flalign} 
	
	\noindent for any $\xi = (\xi_x, \xi_y) \in \mathbb{R}^2$. Further fix measurable functions $a_{jk}: \mathcal{B} \rightarrow \mathbb{R}$ such that $a_{xy} = a_{yx}$ and $\sup_{z \in \mathcal{B}} \big| a_{jk} (z) - A_{jk} \big| \le \delta$, for each $j, k \in \{ x, y \}$; a real number $M \in \mathbb{R}_{> 1}$; and a function $\varphi \in \mathcal{C}^2 (\overline{\mathcal{B}})$ such that $\| \varphi \|_{\mathcal{C}^2 (\overline{\mathcal{B}})} \le M$. If $F \in \mathcal{C}^2 (\overline{\mathcal{B}})$ denotes a solution to the equation 
	\begin{flalign}
	\label{ajkfequationlinear}
	\displaystyle\sum_{j, k \in \{x, y \}} a_{jk} (z) \partial_j \partial_k F (z) = 0, \qquad \text{for each $z \in \mathcal{B}$},
	\end{flalign}  
	
	\noindent with boundary data $F |_{\partial \mathcal{B}} = \varphi |_{\partial \mathcal{B}}$, then $[F]_{1, \alpha; \mathcal{B}} \le CM$.
	
\end{lem}

\begin{proof}[Proof (Outline)]
	
	Throughout this proof outline, we recall the bounded slope condition from equation (12.41) of \cite{EDSO} and the notion of a $(K, K')$-quasiconformal mapping from equation (12.2) of \cite{EDSO}. Let us fix a sufficiently small constant $\delta_0 = \delta_0 (\alpha) \in (0, 1)$, to be specified later, and suppose that $\delta < \delta_0$. By the linearity of \eqref{ajkfequationlinear}, we may also assume $M = \delta_0$ and, through a linear change of variables in $F$, we may further assume $A_{jk} = \textbf{1}_{j = k}$ for each $j, k \in \{ x, y \}$. 
	
	Then, it is quickly verified (see Remark (4) preceding Lemma 12.6 of \cite{EDSO}) that $\varphi |_{\partial \mathcal{B}}$ satisfies the bounded slope condition on $\partial \mathcal{B}$ with constant $4 \| \varphi \|_{\mathcal{C}^2 (\overline{\mathcal{B}})} \le 4 \delta_0$. Therefore, since $F$ satisfies \eqref{ajkfequationlinear} with boundary data $\varphi |_{\partial \mathcal{B}}$, Lemma 12.6 of \cite{EDSO} implies that $[F]_{1, 0; \mathcal{B}} \le 4 \delta_0$. Moreover, the maximum principle (see Theorem 9.1 of \cite{EDSO}) for equations of the form \eqref{ajkfequationlinear} yields $\| F \|_{0; \mathcal{B}} \le \| \varphi \|_0 \le \delta_0$. Hence, $\| F \|_{\mathcal{C}^1 (\overline{\mathcal{B}})} \le 5 \delta_0$. 
	
	Next, define $G: \overline{\mathcal{B}} \rightarrow \mathbb{R}$ by setting $G (z) = F (z) - \varphi (z)$, for each $z \in \overline{\mathcal{B}}$. Then 
	\begin{flalign}
	\label{gdelta0}
	\| G \|_{\mathcal{C}^1 (\overline{\mathcal{B}})} \le \| F \|_{\mathcal{C}^1 (\overline{\mathcal{B}})} + \| \varphi \|_{\mathcal{C}^1 (\overline{\mathcal{B}})} \le 6 \delta_0,
	\end{flalign} 
	
	\noindent and $G$ satisfies  
	\begin{flalign}
	\label{gequationlinear}
	\displaystyle\sum_{j, k \in \{ x, y \}} a_{jk} (z) \partial_j \partial_k G (z) = - \displaystyle\sum_{j, k \in \{ x, y \}} a_{jk} (z) \partial_j \partial_k \varphi (z), \qquad \text{for each $z \in \mathcal{B}$}.
	\end{flalign} 
	
	\noindent Let $\kappa = \kappa (\delta_0) \in \big( 0, \frac{1}{4} \big)$ denote a small constant (dependent only on $\delta_0$, and therefore only on $\alpha$), to be specified later. Since 
	\begin{flalign}
	\label{f1alphag} 
	[F]_{1, \alpha; \mathcal{B}} \le [G]_{1, \alpha; \mathcal{B}} + \| \varphi \|_{\mathcal{C}^2 (\overline{\mathcal{B}})} \le \displaystyle\sup_{z_0 \in \mathcal{B}} [G]_{1, \alpha; \mathcal{B} \cap \mathcal{B}_{\kappa} (z_0)} + \delta_0,
	\end{flalign} 
	
	\noindent it suffices to bound $[G]_{1, \alpha; \mathcal{B} \cap \mathcal{B}_{\kappa} (z_0)}$, for any $z_0 \in \mathcal{B}$. 
	
	Thus, fix some point $z_0 \in \mathcal{B}$, and first assume that $\mathcal{B}_{2 \kappa} (z_0) \subset \mathcal{B}$. As in the beginning of Section 12.2 of \cite{EDSO}, define the complex function $\mathcal{G}: \mathcal{B} \rightarrow \mathbb{C}$ by setting $\mathcal{G} (z) = \partial_x G (z) - \textbf{i} \partial_y G (z)$, for each $z \in \mathcal{B}$. Then, recall that $G$ satisfies \eqref{gequationlinear}; observe that the right side of \eqref{gequationlinear} is bounded by $4 (1 + \delta_0) \| \varphi \|_{\mathcal{C}^2 (\overline{\mathcal{B}})} \le 8 \delta_0$; and observe that both eigenvalues of the $2 \times 2$ matrix $\textbf{A} (z) = \big[ a_{jk} (z) \big]$ are in the interval $[1 - 3 \delta_0, 1 + 3 \delta_0]$ for each $z \in \mathcal{B}$, since $\big| a_{jk} (z) - \textbf{1}_{j = k} \big| \le \delta_0$ for each $j, k \in \{ x, y \}$. Therefore, it follows from the beginning of Section 12.2 of \cite{EDSO} (see equation (12.19) there and below, with the $(\gamma, \mu, \varepsilon)$ there equal to $(1 + 6 \delta_0, 16 \delta_0, \delta_0)$ here) that $\mathcal{G}$ is a $\big( 1 + 10 \delta_0, 65 \delta_0 \big)$-quasiconformal mapping, for $\delta_0$ sufficiently small.
	
	Therefore, since $\| \mathcal{G} \|_0 \le 2 \| G \|_{\mathcal{C}^1 (\overline{\mathcal{B}})} \le 12 \delta_0$ (recall \eqref{gdelta0}), Theorem 12.3 of \cite{EDSO} yields a constant $C_1 = C_1 (\delta_0) > 1$ such that $[G]_{1, \alpha_0; \mathcal{B}}^{(0)} \le C_1$, whenever $\alpha_0 \le 1 - 5 \delta_0^{1 / 2}$. Therefore, selecting $\delta_0 < \frac{(1 - \alpha)^2}{25}$ and recalling that $\mathcal{B}_{2 \kappa} (z_0) \subset \mathcal{B}$, we deduce that
	\begin{flalign}
	\label{g1estimate}
	[G]_{1, \alpha; \mathcal{B}_{\kappa} (z_0)} \le \kappa^{-1} [\mathcal{G}]_{0, \alpha; \mathcal{B}}^{(0)} \le \kappa^{-1} C_1. 
	\end{flalign}
	
	\noindent This bounds the right side of \eqref{f1alphag} if $\mathcal{B}_{2 \kappa} (z_0) \subset \mathcal{B}$. 
	
	So, let us assume instead that $d (z_0, \partial \mathcal{B}) \le 2 \kappa$, in which case we proceed as in Remark (4) following Theorem 12.3 of \cite{EDSO} (see also the end of Section 12.2 of \cite{EDSO}). More specifically, we first map $\mathcal{B}_{\kappa} (z_0) \cap \mathcal{B}$ to an open subset $\mathfrak{A}$ of the upper half-plane $\mathbb{H}$; next define $\mathcal{G}$ on $\mathfrak{A}$ and reflect it into the real axis boundary of $\mathfrak{A}$; and then apply Theorem 12.3 of \cite{EDSO} to this extension of $\mathcal{G}$. In what follows, we assume for notational convenience that $z_0 = (0, \omega - 1)$, for some $\omega \in \big( 0, 2 \kappa \big)$. 
	
	Then, for $\kappa$ sufficiently small, there exists a domain $\mathfrak{A} \subset \mathbb{H} \subset \mathbb{R}^2$ and a uniformly smooth diffeomorphism $\psi: \mathcal{B} \rightarrow \mathfrak{A}$ such that the following three conditions hold (see \Cref{bdomaina} for a depiction). First, $\psi (z_0) = \omega \textbf{i}$. Second, $\psi$ maps $\big\{ (x, y): |x| \le \kappa \big\} \cap \partial \mathcal{B}$ to the real interval $[ -\kappa, \kappa] \subset \mathbb{R}$. Third, 
	\begin{flalign*} 
	\psi \big( \mathcal{B}_{\kappa} (z_0) \cap \mathcal{B} \big) \subseteq \big\{ x + \textbf{i} y: x \in \big[ -\kappa, \kappa ], y \in [0, 2 \kappa] \big\} \subset \mathfrak{A}.
	\end{flalign*}
	
	\noindent Denoting the Jacobian of $\psi$ by  $\textbf{J} (z) = \big[ J_{ik} (z) \big]$, we may select $\mathfrak{A}$ and $\psi$ (satisfying the above three properties) so that $\sup_{z \in \mathcal{B}} \big| J_{ik} (z) - \textbf{1}_{i = k} \big| < \delta_0$ for each $i, k \in \{ 1, 2 \}$, if $\kappa = \kappa (\delta_0) \in \big( 0, \frac{1}{4} \big)$ is sufficiently small.
	
	Now, define $\widetilde{G}: \mathfrak{A} \rightarrow \mathbb{R}$ by setting $\widetilde{G} (z) = G \big( \psi^{-1} (z) \big)$ for each $z \in \mathfrak{A}$. Then, \eqref{gdelta0}; \eqref{gequationlinear}; the fact that $\big| J_{ik} - \textbf{1}_{i = k} \big| < \delta_0$; and the boundedness of $\| \psi^{-1} \|_{\mathcal{C}^2 (\overline{\mathfrak{A}})}$ together imply the existence a uniform constant $C_2 > 1$ and functions $\widetilde{a}_{jk}: \mathfrak{A} \rightarrow \mathbb{R}$ and $\widetilde{g}: \mathfrak{A} \rightarrow \mathbb{R}$ satisfying the following two properties. First,
	\begin{flalign}
	\label{equationga} 
	\displaystyle\sum_{j, k \in \{ x, y \}} \widetilde{a}_{jk} (z) \partial_j \partial_k \widetilde{G} (z) = \widetilde{g}, 
	\end{flalign}
	
	\noindent for each $z \in \mathfrak{A}$, with boundary data $\widetilde{G} |_{\partial \mathfrak{A}} = 0$. Second, for each $z \in \mathfrak{A}$, we have the bounds 
	\begin{flalign}
	\label{estimatesag}
	\displaystyle\max_{j, k \in \{ x, y \}} \big| \widetilde{a}_{jk} (z) - \textbf{1}_{j = k} \big| < C_2 \delta_0; \qquad \big| \widetilde{g} (z) \big| \le C_2 \delta_0.
	\end{flalign}
	
	Similarly to as above, define $\mathcal{G}: \mathfrak{A} \rightarrow \mathbb{C}$ by setting $\mathcal{G} (z) = \partial_x \widetilde{G} (z) - \textbf{i} \partial_y \widetilde{G} (z)$ for each $z \in \mathfrak{A}$. As previously, \eqref{equationga}; \eqref{estimatesag}; and the beginning of Section 12.2 of \cite{EDSO} together yield a uniform constant $C_3 > 1$ such that $\mathcal{G}$ is a $(1 + C_3 \delta_0, C_3 \delta_0)$-quasiconformal mapping on $\mathfrak{A}$. Moreover, by \eqref{gdelta0} and the boundedness of $\| \psi^{-1} \|_{\mathcal{C}^2 (\overline{\mathfrak{A}})}$, there exists a uniform constant $C_4 > 1$ such that $\| \mathcal{G} \|_{0; \mathfrak{A}} < C_4 \delta_0$. 
	
	Next, define $\widetilde{\mathfrak{A}} \subset \mathbb{C}$ by setting 
	\begin{flalign*} 
	\widetilde{\mathfrak{A}} = \mathfrak{A} \cup \mathcal{S}, \quad \text{where} \quad \mathcal{S} = \big\{ x + \textbf{i} y: x \in [ -\kappa, \kappa \big], y \in [-2 \kappa, 0] \big\}.
	\end{flalign*} 
	
	\noindent We refer to the right side of \Cref{bdomaina} for a depiction. Extend $\mathcal{G}$ to $\widetilde{\mathfrak{A}}$ by reflection, that is, we set $\mathcal{G} (x + \textbf{i} y) = - \partial_x \widetilde{G} (x - \textbf{i} y) - \textbf{i} \partial_y \widetilde{G} (x - \textbf{i} y)$ whenever $x + \textbf{i} y \in \mathcal{S}$; equivalently, $\Re \mathcal{G} (x + \textbf{i} y) = - \Re \mathcal{G} (x - \textbf{i} y)$ and $\Im \mathcal{G} (x + \textbf{i} y) = \Im \mathcal{G} (x - \textbf{i} y)$. Since $\widetilde{G} |_{\partial \mathfrak{A}} = 0$ and $G \in \mathcal{C}^2 (\overline{\mathcal{B}})$, it can quickly be deduced that $\mathcal{G}$ admits a Lipschitz extension to $[-\kappa, \kappa]$ (and therefore to $\widetilde{\mathfrak{A}}$). 
	
	Then, since $\mathcal{G}$ is $(1 + C_3 \delta_0, C_3 \delta_0)$-quasiconformal on $\mathfrak{A}$ and since $ \partial_x \Re \mathcal{G} (z) =  -\partial_x \Re \mathcal{G} (\overline{z})$; $\partial_y \Re \mathcal{G} (z) = \partial_y \Re \mathcal{G} (\overline{z})$; $\partial_x \Im \mathcal{G} (z) = \partial_x \Im \mathcal{G} (\overline{z})$; and $\partial_y \Im \mathcal{G} (z) =  -\partial_y \Im \mathcal{G} (\overline{z})$ all hold for any $z = x + \textbf{i} y \in \mathcal{S}$, it follows that $\mathcal{G}$ is $(1 + C_3 \delta_0, C_3 \delta_0)$-quasiconformal almost everywhere on $\widetilde{\mathfrak{A}}$. This, the fact that $\mathcal{G}$ is Lipschitz on $\widetilde{\mathfrak{A}}$, Theorem 12.3 of \cite{EDSO} (see also Appendix A1 of \cite{CQEM} for the case when $\mathcal{G}$ is not differentiable everywhere), and the fact that $\| \mathcal{G} \|_{0; \widetilde{\mathfrak{A}}} = \| \mathcal{G} \|_{0; \mathfrak{A}} < C_4 \delta_0$ then together yield a constant $C_5 = C_5 (\delta_0) > 1$ such that $[ \mathcal{G} ]_{0, \alpha_0; \widetilde{\mathfrak{A}}}^{(0)} \le C_5$ whenever $\alpha_0 \le 1 - (2 C_3 \delta_0)^{1 / 2}$. Selecting $\delta_0 < \frac{(1 - \alpha)^2}{4 C_3^2}$ and using the fact that $d \big( z_0, \widetilde{\mathfrak{A}} \big) > \kappa \sqrt{2}$, it follows that 
	\begin{flalign*}
	[\widetilde{G}]_{1, \alpha; \mathcal{B}_{\kappa} (z_0) \cap \mathfrak{A}} \le [ \mathcal{G} ]_{0, \alpha; \mathcal{B}_{\kappa} (z_0)} \le \kappa^{-1} [ \mathcal{G} ]_{0, \alpha; \widetilde{\mathfrak{A}}}^{(0)} \le \kappa^{-1} C_5.
	\end{flalign*} 
	
	\noindent This, with the uniform boundedness of $\| \psi^{-1} \|_{\mathcal{C}^2 (\overline{\mathfrak{A}})}$, yields a constant $C_6 = C_6 (\alpha) > 1$ such that $[G]_{1, \alpha; \mathcal{B}_{\kappa} (z_0) \cap \mathcal{B}} \le C_6$. This bound, \eqref{f1alphag}, and \eqref{g1estimate} together imply the lemma. 
\end{proof}

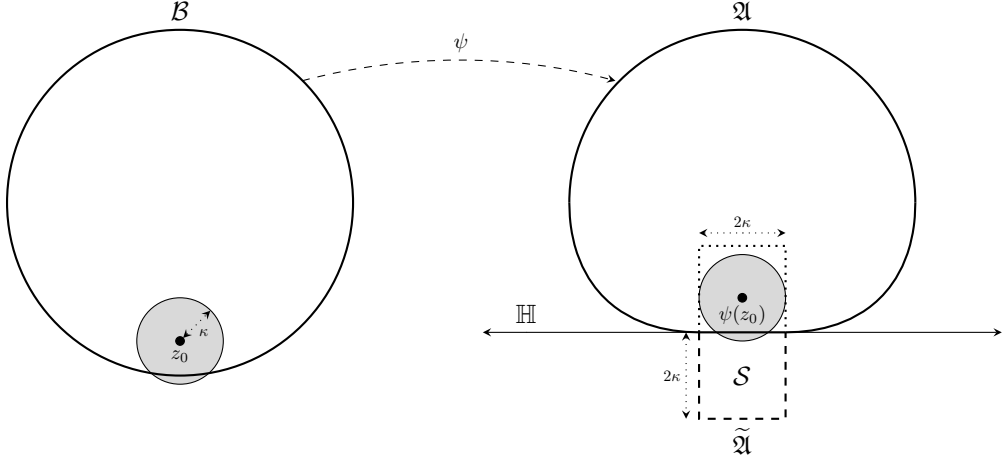
\begin{figure}[t]
	
	\begin{center}
		
		\begin{tikzpicture}[
		>=stealth,
		scale = 1.15
		]

		\filldraw[fill=white!70!gray] (3.5, .65) circle[radius = .5];
		\filldraw[fill=white!70!gray] (10, 1.15) circle[radius = .5];
		
		\draw[<->, dotted] (3.535, .685) -- (3.85, 1);
		
		\draw (3.775, .75) circle[radius = 0] node[scale = .6]{$\kappa$};
		
		\draw[-, black, thick] (3.5, 2.25) circle[radius = 2];

		\draw[-, black, thick] (10.5, .75) -- (10.65, .755179) -- (10.8, .771539) -- (10.95, .800633) -- (11.1, .844602) -- (11.25, .906534) -- (11.4, .991154) -- (11.475, 1.04424) -- (11.55, 1.10638) -- (11.625, 1.17971) -- (11.7, 1.26752) -- (11.775, 1.37532) -- (11.85, 1.51375) -- (11.91, 1.6634) -- (11.94, 1.76299) -- (11.97, 1.8985) -- (11.982, 1.97482) -- (11.988, 2.02385) -- (11.994, 2.08877) -- (12, 2.25);
		
		\draw[-, black, thick] (9.5, .75) -- (10.5, .75);
		
		\draw[thick] (8, 2.25) arc (180:0:2);
		
		\draw[-, black, thick] (8, 2.25) -- (8.006, 2.08877) -- (8.012, 2.02385) -- (8.018, 1.97482) -- (8.03, 1.8985) -- (8.06, 1.76299) -- (8.09, 1.6634) -- (8.15, 1.51375) -- (8.225, 1.37532) -- (8.3, 1.26752) -- (8.375, 1.17971) -- (8.45, 1.10638) -- (8.525, 1.04424) -- (8.6, .991154) -- (8.75, .906534) -- (8.9, .844602) -- (9.05, .800633) -- (9.2, .771539) -- (9.35, .755179) -- (9.5, .75); 
		
		\draw[dotted, thick] (9.5, .75) -- (9.5, 1.75) -- (10.5, 1.75) -- (10.5, .75);
		
		\draw[dashed, thick] (9.5, .75) -- (9.5, -.25) -- (10.5, -.25) -- (10.5, .75);
		
		\draw[<->, black] (7, .75) -- (13, .75);

		\draw[->, dashed] (4.91, 3.66) arc (105:75:7);

		\draw[<->, dotted] (9.35, .75) -- (9.35, -.25); 
		
		\draw[<->, dotted] (9.5, 1.9) -- (10.5, 1.9);

		\filldraw[fill = black] (3.5, 4.25) circle[radius = 0] node[above]{$\mathcal{B}$};
		
		\filldraw[fill = black] (10, 4.25) circle[radius = 0] node[above]{$\mathfrak{A}$};
		
		\filldraw[fill = black] (6.75, 4.1) circle[radius = 0] node[scale = .8]{$\psi$};
		
		\filldraw[fill = black] (10, -.25) circle[radius = 0] node[below]{$\widetilde{\mathfrak{A}}$};

		\filldraw[fill = black] (10, .25) circle[radius = 0] node[]{$\mathcal{S}$};
		
		\filldraw[fill = black] (7.5, .75) circle[radius = 0] node[above]{$\mathbb{H}$};

		\filldraw[fill = black] (3.5, .65) circle[radius = .05] node[below = 1, scale = .75]{$z_0$};
		
		\filldraw[fill = black] (10, 1.15) circle[radius = .05] node[below, scale = .75]{$\psi (z_0)$};
		
		\filldraw[fill = black] (9.35, .25) circle[radius = 0] node[left, scale = .6]{$2 \kappa$};
		\filldraw[fill = black] (10, 1.9) circle[radius = 0] node[above, scale = .6]{$2 \kappa$};

		\end{tikzpicture}
		
	\end{center}	
	
	\caption{\label{bdomaina} Under $\psi$, the disk $\mathcal{B}$ is mapped to the domain $\mathfrak{A}$. Then $\widetilde{\mathfrak{A}}$ is obtained from $\mathfrak{A}$ by augmenting the dashed square $\mathcal{S}$. } 
\end{figure} 

We conclude this section with the following analog of \Cref{derivativefestimate1} for solutions to a non-linear, elliptic partial differential equation of the form \eqref{ajkfiequation1}, whose coefficients $a_{jk}$ are nearly constant. 

\begin{cor}
	
	\label{uestimategradient} 
	
	For any fixed real numbers $B > 1$ and $\alpha \in (0, 1)$, there exist constants $\delta = \delta (B, \alpha) > 0$ and $C = C (B, \alpha) > 1$ such that the following holds. For each $j, k \in \{ x, y \}$, let $A_{jk} \in \mathbb{R}$ satisfy \eqref{ajkb}. Furthermore, fix measurable functions $a_{jk}: \mathbb{R}^2 \rightarrow \mathbb{R}$ such that $a_{xy} = a_{yx}$ and $\sup_{z \in \mathbb{R}^2} \big| a_{jk} (z) - A_{jk} \big| < \delta$ for each $j, k \in \{ x, y \}$.  

	Let $M \in \mathbb{R}_{> 1}$, and let $\varphi \in \mathcal{C}^2 (\overline{\mathcal{B}})$ denote a function such that $\| \varphi \|_{\mathcal{C}^2 (\overline{\mathcal{B}})} \le M$. If $F \in \mathcal{C}^2 (\overline{\mathcal{B}})$ satisfies equation \eqref{ajkfiequation1}, with boundary data $F |_{\partial \mathcal{B}} = \varphi |_{\partial \mathcal{B}}$, then $[F]_{1, \alpha; \mathcal{B}} < CM$.
	
\end{cor}

\begin{proof}
	
	As in the proof of \Cref{aijgradientfestimate}, define the measurable function $\widetilde{a}_{jk}: \mathcal{B} \rightarrow \mathbb{R}$ for any $j, k \in \{ x, y \}$ by setting $\widetilde{a}_{jk} (z) = a_{jk} \big( \nabla F (z) \big)$, for each $z \in \mathcal{B}$. Then $\sup_{z \in \mathcal{B}} \big| \widetilde{a}_{jk} (z) - A_{jk} \big| < \delta$, and $F \in  \mathcal{C}^2 (\overline{\mathcal{B}})$ satisfies the linear equation \eqref{ajkfjkidentity}, with boundary data $F |_{\partial \mathcal{B}} = \varphi |_{\partial \mathcal{B}}$. Thus the corollary follows from \Cref{derivativefestimate1}.
\end{proof}

\subsection{Global Gradient Estimates for Minimizers of \texorpdfstring{$\mathcal{E}$}{}}

\label{GlobalE}

In this section we establish \Cref{euv1v2estimategradient}. To that end, we begin with the following proposition that establishes the first bound in \eqref{derivativehestimates122}. 

\begin{prop}
	
	\label{gradientfe}
	
	Under the notation of \Cref{euv1v2estimategradient}, $\sup_{z \in \mathcal{B}_{1 / 8}} \big| \nabla \mathcal{H} (z) - (s, t) \big| < \lambda^{15 \theta / 16}$. 
	
\end{prop}

\begin{proof}
	
	Through a shift, we may suppose that $\varphi (0, 0) = 0$. Recall the functions $\mathfrak{a}_{jk}: \mathcal{T} \rightarrow \mathbb{R}$ from \eqref{aijst}, and define the real numbers $A_{jk} = \mathfrak{a}_{jk} (s, t)$ for each $j, k \in \{ x, y \}$. Then, since $(s, t) \in \mathcal{T}_{\varepsilon}$, there exists some $B = B (\varepsilon) \in \mathbb{R}_{> 1}$ such that \eqref{ajkb} holds. 
	
	Fix $\alpha = 1 - \frac{\theta}{20}$, and let $\delta_0 = \delta_0 (\varepsilon, \theta)$ and $C_0 = C_0 (\varepsilon, \theta)$ denote the constants $\delta (B, \alpha)$ and $C (B, \alpha)$ from \Cref{uestimategradient}, respectively; further fix $\gamma = \frac{\delta_0}{2}$. By the uniform continuity of the $\mathfrak{a}_{jk}$ on $\mathcal{T}_{\varepsilon / 2}$, there exists a constant $\varpi = \varpi (\varepsilon, \theta) \in \big( 0, \frac{\varepsilon}{2} \big)$ and functions $a_{jk} \in \mathcal{C}^2 (\mathbb{R}^2)$ for each $j, k \in \{ x, y \}$ such that $a_{jk} = a_{kj}$ and the following two properties hold. 
	
	\begin{enumerate} 
		
		\item For each $j, k \in \{ x, y \}$ and $z \in \mathcal{B}_{\varpi} (s, t)$, we have that $a_{jk} (z) = \mathfrak{a}_{jk} (z)$. 
		
		\item For each $j, k \in \{ x, y \}$, we have that $\sup_{z \in \mathbb{R}^2} \big| a_{jk} (z) - A_{jk} \big| < \gamma$. 
		
	\end{enumerate}  

	Now let us for the moment suppose that $\varphi \in \mathcal{C}^{2, \alpha} (\overline{\mathcal{B}}_{1 / 8})$. Under this assumption, Theorem 15.13 of \cite{EDSO} guarantees the existence of a solution $F \in \mathcal{C}^{2, \alpha} (\overline{\mathcal{B}}_{1 / 8})$ to the equation
	\begin{flalign}
	\label{2ajkfequation}
	\displaystyle\sum_{j, k \in \{ x, y \}} a_{jk} \big( \nabla F(z) \big) \partial_j \partial_k F(z) = 0,
	\end{flalign}
	
	\noindent for each $z \in \mathcal{B}_{1 / 8}$, with boundary data $F |_{\partial \mathcal{B}_{1 / 8}} = \mathfrak{h} = \varphi |_{\partial \mathcal{B}_{1 / 8}}$. We will show that $F = \mathcal{H}$ by showing that $\big| \nabla F(z) - (s, t) \big| < \varpi$, for each $z \in \mathcal{B}_{1 / 8}$. 
	
	To implement this, let us first verify that $G$ also satisfies the equation \eqref{2ajkfequation} on $\mathcal{B}_{1 / 4}$. To that end, the second assumption on $g$ listed in \Cref{euv1v2estimategradient} yields a function $\Lambda: \overline{\mathcal{B}} \rightarrow \mathbb{R}$ such that $\sup_{z \in \partial{B}} \big| g(z) - \Lambda (z) \big| < \lambda^{2 \theta}$. For sufficiently small $\delta$, this implies that $(g, \Lambda |_{\partial \mathcal{B}})$ satisfies the conditions of $(\mathfrak{h}_1, \mathfrak{h}_2)$ in \Cref{perturbationboundary}. That result then yields a constant $C_1 = C_1 (\varepsilon, \theta) > 1$ such that 
	\begin{flalign} 
	\label{gradientgiestimate1}
	\sup_{z \in \mathcal{B}_{1 / 4}} \big| \nabla G (z) - (s, t) \big| \le C_1 \sup_{z \in \partial \mathcal{B}} \big| g (z) - \Lambda (z) \big| < C_1 \lambda^{2 \theta}.
	\end{flalign}
	
	Hence, $\nabla G(z) \in \mathcal{B}_{\varpi} (s, t) \subset \mathcal{T}_{\varepsilon / 2}$ for each $z \in \mathcal{B}_{1 / 4}$ and sufficiently small $\delta$. Therefore, from the first property satisfied by the $a_{jk}$, we deduce that $\mathfrak{a}_{jk} \big( \nabla G (z) \big) = a_{jk} \big( \nabla G(z) \big)$ for each $z \in \mathcal{B}_{1 / 4}$; moreover, by \Cref{haijequations}, $G$ satisfies the equation \eqref{aijh} on $\mathcal{B}_{1 / 4}$. Together, these two facts imply that $G$ also satisfies \eqref{2ajkfequation} on $\mathcal{B}_{1 / 4}$. 
	
	Now, by the third property listed in \Cref{euv1v2estimategradient}; \Cref{uestimategradient}, whose $M$ there equals $\lambda^{2 \theta - 1}$ here; and scaling by $\frac{1}{8}$ (as in \Cref{estimatehrho}), there exists a constant $C_1 = C_1 (\varepsilon, \theta) > 1$ such that 
	\begin{flalign} 
	\label{f1bestimate}
	[F]_{1, \alpha; \mathcal{B}_{1 / 8}} < C_1 \lambda^{2 \theta - 1}.  
	\end{flalign}
	
	\noindent Denoting $\nu = \lambda^{1 - \theta}$, we claim for sufficiently small $\delta$ that
	\begin{flalign}
	\label{stgradientf}
	\displaystyle\sup_{z \in \mathcal{B}_{1  / 8 - \nu}} \big| \nabla F(z) - (s, t) \big| \le 4 \lambda^{19 \theta / 20}. 
	\end{flalign}
	
	The proof of \eqref{stgradientf} will be similar to that of \Cref{fjderivativesestimate}, with the $\mathcal{C}^0$ bound on $F$ there replaced by the proximity of $F$ to the nearly linear function $G$ here. Specifically, assume to the contrary that there exists some $z_0 \in \mathcal{B}_{1 / 8 - \nu}$ such that $\big| \nabla F (z_0) - (s, t) \big| > 4 \lambda^{19 \theta / 20}$. Then, either $\big| \partial_x F(z_0) - s \big| > 2 \lambda^{19 \theta / 20}$ or $\big| \partial_y F (z_0) - t \big| > 2 \lambda^{19 \theta / 20}$. Let us assume the former, and also that $\partial_x F(z_0) > s + 2 \lambda^{19 \theta / 20}$. Then \eqref{f1bestimate} (and the fact that $\alpha = 1 - \frac{\theta}{20}$) yields 
	\begin{flalign*} 
	\displaystyle\inf_{z \in \mathcal{B}_{\nu} (z_0)} \big( \partial_x F(z) - s \big) > 2 \lambda^{19 \theta / 20} - C_1 \nu^{\alpha} \lambda^{2 \theta - 1} > \lambda^{19 \theta / 20},
	\end{flalign*} 
	
	\noindent for sufficiently small $\delta$. Denoting $w = (1, 0)$, it follows from integration that
	\begin{flalign}
	\label{fz0nuw} 
	F(z_0 + \nu w) - F(z_0) - s \nu > \lambda^{19 \theta / 20} \nu = \lambda^{1 - \theta / 20}.
	\end{flalign}
	
	However, the first property satisfied by $G$ listed in \Cref{euv1v2estimategradient} and the comparison principle \Cref{aijcomparison} for equations of the type \eqref{ajkfiequation1} (see also \Cref{h1h2gamma}) together imply that $\sup_{z \in \mathcal{B}_{1 / 8}} \big| F(z) - G(z) \big| \le \lambda$. This, with the bound \eqref{gradientgiestimate1}, yields for any $z_0 \in \mathcal{B}_{1 - \nu}$ and sufficiently small $\delta$ that
	\begin{flalign*}
	\displaystyle\sup_{z \in \mathcal{B}_{\nu} (z_0)} \big| F(z) - F(z_0) - (s, t) \cdot (z - z_0) \big| & \le \displaystyle\sup_{z \in \mathcal{B}_{\nu} (z_0)} \big| G (z) - G (z_0) - (s, t) \cdot (z - z_0) \big| + 2 \lambda & \\
	& <  C_1 \lambda^{2 \theta} \nu + 2 \lambda < 3 \lambda,
	\end{flalign*}
	
	\noindent which contradicts \eqref{fz0nuw}. This establishes \eqref{stgradientf}. 
	
	Next we bound $\sup_{z \in \mathcal{B}_{1 / 8}} \big| \nabla F(z) - (s, t) \big|$. To that end, observe from \eqref{f1bestimate} and \eqref{stgradientf} that 
	\begin{flalign}
	\label{gradientfstestimatelambda}
	\begin{aligned} 
	\big| \nabla F(z) - (s, t) \big| & \le \displaystyle\sup_{z_0 \in \mathcal{B}_{1 / 8 - \nu}} \big| \nabla F(z_0) - (s, t) \big| + \displaystyle\inf_{z_0 \in \mathcal{B}_{1 / 8 - \nu}} \big| \nabla F(z) - \nabla F (z_0) \big|\\
	&  \le 4 \lambda^{19 \theta / 20} + \nu^{\alpha} [F]_{1, \alpha; \mathcal{B}_{1 / 8}} \le 4 \lambda^{19 \theta / 20} + C_1 \lambda^{2 \theta - 1} \nu^{1 - \theta / 20} < \lambda^{15 \theta / 16},
	\end{aligned}
	\end{flalign}
	
	\noindent  for any $z \in \mathcal{B}_{1 / 8}$ and sufficiently small $\delta$. Hence, $\nabla F(z) \in \mathcal{B}_{\varpi} (s, t)$ for each $z \in \mathcal{B}_{1 / 8}$, and so $a_{jk} \big( \nabla F(z) \big) = \mathfrak{a}_{jk} \big( \nabla F(z) \big)$ by the first property satisfied by the $a_{jk}$. Therefore, $F \in \mathcal{C}^2 (\mathcal{B}_{1 / 8})$ solves the Euler-Lagrange equations for the function $\sigma$ (from \eqref{lsigma}), which yields $F = \mathcal{H}$ by \Cref{haijequations}. Thus, the proposition in the case $\varphi \in \mathcal{C}^{2, \alpha} (\overline{\mathcal{B}}_{1 / 8})$ follows from \eqref{gradientfstestimatelambda}.\footnote{Observe that the condition $\varphi \in \mathcal{C}^{2, \alpha} (\overline{\mathcal{B}}_{1 / 8})$ was only stipulated in order to ensure that $F \in \mathcal{C}^2 (\overline{\mathcal{B}}_{1 / 8})$. The bounds above did not depend on $\| \varphi \|_{\mathcal{C}^{2, \alpha} (\overline{\mathcal{B}}_{1 / 8})}$; they only depended on $\| \varphi \|_{\mathcal{C}^2 (\overline{\mathcal{B}}_{1 / 8})} \le \lambda^{2 \theta - 1}$.}
	
	If instead $\varphi \notin \mathcal{C}^{2, \alpha} (\overline{\mathcal{B}}_{1 / 8})$, then let $\varphi_1, \varphi_2, \ldots \in \mathcal{C}^{2, \alpha} (\overline{\mathcal{B}}_{1 / 8})$ denote a sequence of functions converging to $\varphi$ in $\mathcal{C}^2 (\overline{\mathcal{B}}_{1 / 8})$. For each $j \ge 1$, let $\mathcal{H}_j \in \Adm (\mathcal{B}_{1 / 8})$ denote the maximizer of $\mathcal{E}$ on $\mathcal{B}_{1 / 8}$ with boundary data $\varphi_j |_{\partial \mathcal{B}_{1 / 8}}$. Then the above reasoning implies that $\sup_{z \in \mathcal{B}_{1 / 8}} \big| \nabla \mathcal{H}_j (z) - (s, t) \big| < \lambda^{15 \theta / 16}$, for sufficiently large $j$.
	
	Now let us use \Cref{perturbationboundary} to show that $\lim_{j \rightarrow \infty} \nabla \mathcal{H}_j (z) = \nabla \mathcal{H} (z)$, for each $z \in \mathcal{B}_{1 / 8}$. To that end, fix $z_0 \in \mathcal{B}_{1 / 8}$ and let $\nu \in \big( 0, \frac{1}{8} \big)$ be such that $\mathcal{B}_{\nu} (z_0) \subset \mathcal{B}_{1 / 8}$. Define the functions $\mathcal{H}^{(\nu)}: \overline{\mathcal{B}} \rightarrow \mathbb{R}$ and $\mathcal{H}_j^{(\nu)}: \overline{\mathcal{B}} \rightarrow \mathbb{R}$ by setting 
	\begin{flalign*} 
	\mathcal{H}^{(\nu)} (z) = \nu^{-1} \big( \mathcal{H} (z_0 + \nu z) - \mathcal{H} (z_0) \big); \qquad \mathcal{H}_j^{(\nu)} (z) = \nu^{-1} \big( \mathcal{H}_j (z_0 + \nu z) - \mathcal{H}_j (z_0) \big),
	\end{flalign*} 
	
	\noindent for each $z \in \overline{\mathcal{B}}$ and $j \ge 1$. Then, $\mathcal{H}^{(\nu)}$ and each $\mathcal{H}_j^{(\nu)}$ are maximizers of $\mathcal{E}$ on $\mathcal{B}$; furthermore, $\nabla \mathcal{H} (z_0 + \nu z) = \nabla \mathcal{H}^{(\nu)} (z)$ and $\nabla \mathcal{H}_j (z_0 + \nu z) = \nabla \mathcal{H}_j^{(\nu)} (z)$ for each $z \in \mathcal{B}$ and $j \ge 1$. So, the bound $\sup_{z \in \mathcal{B}_{1 / 8}} \big| \nabla \mathcal{H}_j (z) - (s, t) \big| < \lambda^{15 \theta / 16}$ implies $\sup_{z \in \mathcal{B}} \big| \nabla \mathcal{H}_j^{(\nu)} (z) - (s, t) \big| < \lambda^{15 \theta / 16}$ for sufficiently large $j$. In particular, $\mathcal{H}_j^{(\nu)}$ is $\lambda^{15 \theta / 16}$ linear with slope $(s, t)$ on $\mathcal{B}$. 
	
	Since \Cref{h1h2gamma} implies that $\sup_{z \in \partial \mathcal{B}_{\nu} (z_0)} \big| \mathcal{H} (z) - \mathcal{H}_j (z) \big| \le \sup_{z \in \overline{\mathcal{B}}_{1 / 8}} \big| \varphi_j (z) - \varphi (z) \big| \le \nu \lambda^{15 \theta / 16}$ for sufficiently large $j$, it follows that $\mathcal{H}^{(\nu)}$ is $2 \lambda^{15 \theta / 16}$-linear with slope $(s, t)$ on $\mathcal{B}$. Therefore, for sufficiently small $\delta$ and large $j$, \Cref{perturbationboundary} yields the existence of a constant $C_2 = C_2 (\varepsilon) > 0$ such that 
	\begin{flalign*}
	\displaystyle\lim_{j \rightarrow \infty} \big| \nabla \mathcal{H} (z_0) - \nabla \mathcal{H}_j (z_0) \big| & = \displaystyle\lim_{j \rightarrow \infty} \big| \nabla \mathcal{H}^{(\nu)} (0, 0) - \nabla \mathcal{H}_j^{(\nu)} (0, 0) \big| \\
	&  \le C_2 \displaystyle\lim_{j \rightarrow \infty} \displaystyle\sup_{z \in \partial \mathcal{B}} \big| \mathcal{H}^{(\nu)} (z) - \mathcal{H}_j^{(\nu)} (z) \big| \\
	& = C_2 \nu^{-1} \displaystyle\lim_{j \rightarrow \infty} \displaystyle\sup_{z \in \partial \mathcal{B}} \big| \mathcal{H} (z_0 + \nu z) - \mathcal{H}_j (z_0 + \nu z) \big| \\
	&  \le C_2 \nu^{-1} \displaystyle\lim_{j \rightarrow \infty} \displaystyle\sup_{z \in \partial \mathcal{B}_{1 / 8}} \big| \varphi (z) - \varphi_j (z) \big| = 0,
	\end{flalign*}
	
	\noindent where in the fourth statement we applied \Cref{h1h2gamma} and in the fifth we used the fact that the $\varphi_j$ converge to $\varphi$ in $\mathcal{C}^2 (\overline{\mathcal{B}}_{1 / 8})$. Thus $\lim_{j \rightarrow \infty} \nabla \mathcal{H}_j (z) = \nabla \mathcal{H} (z)$, for each $z \in \mathcal{B}_{1 / 8}$, so the proposition follows from the fact that $\sup_{z \in \mathcal{B}_{1 / 8}} \big| \nabla \mathcal{H}_j (z) - (s, t) \big| < \lambda^{15 \theta / 16}$ for sufficiently large $j$.
\end{proof}

Now we can establish \Cref{euv1v2estimategradient}. 

\begin{proof}[Proof of \Cref{euv1v2estimategradient}]
		
	By \Cref{gradientfe}, the first bound in \eqref{derivativehestimates122} holds, so it suffices to establish the second. To that end, we will assume throughout this proof that $\mathcal{H} (0, 0) = 0$. 
	
	For any $z_0 \in \mathcal{B}_{1/8 - \mu}$, define the rescaled function $\mathcal{H}^{(\mu)}: \overline{\mathcal{B}} \rightarrow \mathbb{R}$ by 
	\begin{flalign*}
	\mathcal{H}^{(\mu)} (z) = \mathcal{H}^{(\mu; z_0)} (z) = \mu^{-1} \big( \mathcal{H} (z_0 + \mu z) - \mathcal{H} (z_0) \big), \quad \text{for each $z \in \overline{\mathcal{B}}$.}
	\end{flalign*}
	
	\noindent Then, for each $z \in \mathcal{B}$; $j, k \in \{x, y \}$; and $\alpha \in (0, 1)$, we have that
	\begin{flalign}
	\label{xrgradientderivatives} 
	\nabla \mathcal{H}^{(\mu)} (z) = \nabla \mathcal{H} (z_0 + \mu z); \quad \big[ \mathcal{H}^{(\mu)} \big]_{2, \alpha; \mathcal{B}}= \mu^{1 + \alpha} [\mathcal{H}]_{2, \alpha; \mathcal{B}_{\mu} (z_0)}; \quad \big[ \mathcal{H}^{(\mu)} \big]_{4, \alpha; \mathcal{B}_{1/2}} = \displaystyle\frac{\mu^3}{8} [\mathcal{H}]_{4, \alpha; \mathcal{B}_{\mu/2} (z_0)}.
	\end{flalign}
	
	Next, by the scale-invariance of the variational principle (recall \Cref{estimatehrho}), the function $\mathcal{H}^{(\mu)}$ is a maximizer of $\mathcal{E}$ on $\mathcal{B}$. Therefore, \Cref{derivativeshestimate}, \eqref{xrgradientderivatives}, and the first bound in \eqref{derivativehestimates122} together yield a constant $C_1 = C_1 (\varepsilon) > 1$ such that $\big[ \mathcal{H}^{(\mu)} \big]_{1, \theta / 4; \mathcal{B}_{1 / 2}} \le \big\| \mathcal{H}^{(\mu)} \big\|_{\mathcal{C}^2 (\overline{\mathcal{B}}_{1 / 2})} \le C_1$, since $\mathcal{H}^{(\mu)} (0, 0) = 0$.  
	
	To bound the $\mathcal{C}^{2, \theta / 4}$-norm of $\mathcal{H}^{(\mu)}$, we will use the interior Schauder estimate \Cref{aijuestimates}. This will be facilitated by first removing the ``approximate slope'' of $\mathcal{H}^{(\mu)}$, so let us define the functions 
	\begin{flalign*} 
	F (z) = F^{(\mu; z_0)} (z) = \mathcal{H}^{(\mu)} (z) - (s, t) \cdot z; \qquad a_{jk} (z) = \mathfrak{a}_{jk} \big( \nabla \mathcal{H}^{(\mu)} (z) \big), 
	\end{flalign*} 
	
	\noindent for each $j, k \in \{ x, y \}$ and $z \in \mathcal{B}$. Observe that the first bound in \eqref{derivativehestimates122} implies that $\nabla \mathcal{H} (z) \in \mathcal{T}_{\varepsilon / 2}$, for each $z \in \mathcal{B}_{1 / 8}$ and sufficiently small $\delta$. Thus $\nabla \mathcal{H}^{(\mu)} (z) \in \mathcal{T}_{\varepsilon / 2}$ for each $z \in \mathcal{B}$ and so, by \Cref{haijequations}, $F$ satisfies the linear equation \eqref{aijubiuf}, with the $b_j$ and $g$ there equal to $0$. 
	
	Now, the fact that $\big[ \mathcal{H}^{(\mu)} \big]_{1, \theta / 4; \mathcal{B}_{1 / 2}} \le C_1$ and the uniform smoothness of $\mathfrak{a}_{jk}$ on $\mathcal{T}_{\varepsilon / 2}$ (recall \Cref{concavesigmat}) together imply the existence of a constant $C_2 = C_2 (\varepsilon) > 0$ such that $\| a_{jk} \|_{0, \theta / 4; \mathcal{B}_{1 / 2}} \le C_2$. Thus, the interior Schauder estimate \Cref{aijuestimates} yields a constant $C_3 = C_3 (\varepsilon, \theta) > 1$ such that 
	\begin{flalign} 
	\label{hmuestimatederivatives2} 
	\big[ \mathcal{H}^{(\mu)} \big]_{2, \theta / 4; \mathcal{B}_{1 / 4}} = [F]_{2, \theta / 4; \mathcal{B}_{1 / 4}}  \le \| F \|_{\mathcal{C}^{2, \theta / 4} (\overline{\mathcal{B}}_{1 / 4})} \le C_3 \| F \big\|_{0; \mathcal{B}_{1 / 2}}.
	\end{flalign}
	
	\noindent To bound the right side of \eqref{hmuestimatederivatives2}, observe since $F(0, 0) = 0$ that 
	\begin{flalign}
	\label{ffgradienthmuh}
	\| F \|_{0; \mathcal{B}_{1 / 2}} \le \| \nabla F \|_{0; \mathcal{B}} = \big\| \nabla \mathcal{H}^{(\mu)} - (s, t) \big\|_{0; \mathcal{B}} = \big\| \nabla \mathcal{H} - (s, t) \big\|_{0; \mathcal{B}_{\mu} (z_0)} \le \lambda^{15 \theta / 16},
	\end{flalign} 
	
	\noindent where the third and fourth statements of \eqref{ffgradienthmuh} follow from the first equality in \eqref{xrgradientderivatives} and the first bound in \eqref{derivativehestimates122}, respectively. Inserting \eqref{ffgradienthmuh} into \eqref{hmuestimatederivatives2}, using the second identity in \eqref{xrgradientderivatives}, and recalling that $\mathcal{H}$ is $1$-Lipschitz and that $\mu = \lambda^{1 + \theta / 8}$, we deduce for sufficiently small $\delta$ that 
	\begin{flalign*}
	\| \mathcal{H} \|_{\mathcal{C}^{2, \theta / 4} (\overline{\mathcal{B}}_{1/8 - \mu})} & \le [\mathcal{H}]_{2, \theta / 4; \mathcal{B}_{1/8 - \mu}} + 2 \\
	& \le \mu^{-1 - \theta / 4} \displaystyle\sup_{z_0 \in \mathcal{B}_{1/8 - \mu}} \big[ \mathcal{H}^{(\mu; z_0)} \big]_{2, \theta / 4; \mathcal{B}_{1 / 4}} + 2 \le C_3 \mu^{-1 - \theta / 4} \lambda^{15 \theta / 16} + 2 \le \lambda^{\theta / 2 - 1}.
	\end{flalign*}
	
	\noindent This yields the second statement of \eqref{derivativehestimates122} and therefore establishes the proposition. 
\end{proof}

\section{Proof of \texorpdfstring{\Cref{numbersigma}}{}} 

\label{ProofstPn}

In this section we establish \Cref{numbersigma}, whose proof will closely follow Sections 6, 7, and 8 of \cite{VPDT}. As in that work, this will proceed by comparing torus tilings of a given (approximate) tile type proportion with suitably weighted torus tilings. Thus, we begin in \Cref{PnEstimate} by recalling and asymptotically analyzing the partition function for this weighted count of torus tilings. Then, in \Cref{EstimatePn2} we provide a variance estimate for weighted tilings that enables one to compare them with unweighted tilings of a given approximate tile type proportion. Throughout this section, we recall the notation of \Cref{EstimateTriangle}.

\subsection{Enumerating Weighted Tilings of a Torus}

\label{PnEstimate}  

In this section we establish \Cref{znzpnp} below, which approximates the number of lozenge tilings of $P_N$ that are weighted in a certain way. To state that result, we first require some notation. 

Fix $a, b, c \in \mathbb{R}_{> 0}$, and define the \emph{weight} of a tiling $\mathscr{M} \in \mathfrak{E} (P_N)$ to be $w (\mathscr{M}) = w_{a, b, c} (\mathscr{M}) = a^{\mathcal{N}_1 (\mathscr{M})} b^{\mathcal{N}_2 (\mathscr{M})} c^{\mathcal{N}_2 (\mathscr{M})}$. Further define the measure $\nu = \nu_{a, b, c} = \nu_{a, b, c; N} \in \mathfrak{P} (P_N)$ that assigns a tiling $\mathscr{M} \in \mathfrak{E} (P_N)$ probability $\nu (\mathscr{M}) = Z^{-1} w (\mathscr{M})$, where we have denoted the \emph{partition function} $Z = Z_N (a, b, c) = \sum_{\mathscr{M} \in \mathfrak{E} (P_N)} w(\mathscr{M})$ to ensure that $\sum_{\mathscr{M} \in \mathfrak{E} (P_N)} \nu (\mathscr{M}) = 1$. 

The following lemma is a modification of Lemma 7.3 of \cite{VPDT} with explicit error estimates. In what follows, we set
\begin{flalign*}
\mathfrak{Z} = \mathfrak{Z} (a, b, c) = \displaystyle\frac{1}{(2 \pi \textbf{i})^2} \displaystyle\int_{\gamma} \displaystyle\int_{\gamma} \log (a + bz + cw) w^{-1} z^{-1} dw dz,
\end{flalign*}

\noindent where $\log$ denotes the principal branch of the logarithm, and the contour $\gamma$ for $w$ and $z$ is a positively oriented circle of radius $1$ centered at $0$. 

The following lemma is a variant of Lemma 7.3 of \cite{VPDT} with an error estimate of order $N^{-1 / 4}$. 

\begin{lem}
	
	\label{znzpnp}
	
	There exists a constant $C > 1$ such that, for any real numbers $a, b, c \in \mathbb{R}_{> 0}$ and integer $N \in \mathbb{Z}_{> 1}$, we have that $\big| N^{-2} \log Z_N (a, b, c) - \mathfrak{Z} (a, b, c) \big| < C N^{-1 / 4}$. 
	
\end{lem}

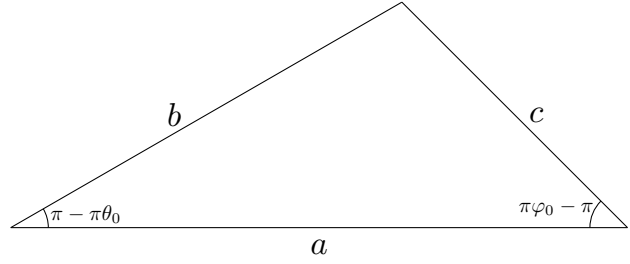
\begin{figure}

	\begin{center}

		\begin{tikzpicture}[
		>=stealth,
		auto,
		style={
			scale = 1
		}
		]

		\draw[-, black] (0, 0) -- (8.2, 0);
		\draw[-, black] (0, 0) -- (5.2, 3); 
		\draw[-, black] (5.2, 3) -- (8.2, 0);

		\draw[black] (.5, 0) arc (0:30:.5);
		\draw[black] (7.7, 0) arc (180:135:.5);
		
		\draw[] (4.1, 0) circle [radius = 0] node[below, scale = 1.25]{$a$};
		\draw[] (2.6, 1.5) circle [radius = 0] node[left = 5, scale = 1.25]{$b$};
		\draw[] (6.7, 1.5) circle [radius = 0] node[right = 1, scale = 1.25]{$c$};
		
		\draw[] (1, .4) circle [radius = 0] node[below, scale = .8]{$\pi - \pi \theta_0$}; 
		\draw[] (7.25, .47) circle [radius = 0] node[below, scale = .8]{$\pi \varphi_0 - \pi$};
		
		\end{tikzpicture}
		
	\end{center}

	\caption{\label{thetafigure} Depicted above are the angles $\theta_0$ and $\varphi_0$ from the proof of \Cref{znzpnp} (and \Cref{znzpnp2}).}

\end{figure}

\begin{proof}
	
	The proof will closely follow that of Lemma 7.3 of \cite{VPDT}. To that end, let us assume by scaling that $0 < c \le b \le a = 1$ and that $N$ is odd (the alternative case when $N$ is even is similar). Then, recall from equation (3) and Section 6 of \cite{OD} that 
	\begin{flalign}
	\label{zzij}
	Z_N (a, b, c) = \displaystyle\frac{Z_{00} + Z_{01} + Z_{10} - Z_{11}}{2}; \qquad \displaystyle\max_{i, j \in \{ 0, 1 \}} Z_{ij} \le Z_N (a, b, c) \le 2 \max_{i, j \in \{ 0, 1 \}} Z_{ij},
	\end{flalign}
	
	\noindent where $Z_{00} = Z_{00; N} (a, b, c)$, $Z_{01} = Z_{01; N} (a, b, c)$, $Z_{10} = Z_{10; N} (a, b, c)$, and $Z_{11} = Z_{11; N} (a, b, c)$ are defined (see Section 4.16 of \cite{OD}) by
	\begin{flalign}
	\label{zijdefinition} 
	\begin{aligned}
	& Z_{00} = \displaystyle\prod_{j = 0}^{N - 1} \displaystyle\prod_{k = 0}^{N - 1} \big( a + b  e^{2 \pi \textbf{i} j / N} + c e^{2 \pi \textbf{i} k /N} \big); \qquad Z_{10} = Z_{00; N} \big( a, b e^{\pi \textbf{i} / N}, c \big); \\
	& Z_{01} = Z_{00; N} \big( a, b, c e^{\pi \textbf{i} / N} \big); \qquad \qquad \qquad \qquad \quad Z_{11} = Z_{00; N} \big( a, b e^{\pi \textbf{i} / N}, c e^{\pi \textbf{i} / N} \big).
	\end{aligned} 
	\end{flalign}
	
	\noindent The second estimate in \eqref{zzij} and the definition \eqref{zijdefinition} of the $Z_{ij}$ together suggest that $N^{-2} \log Z \approx \mathfrak{Z} (a, b, c)$, except for the fact that terms in the products defining the $Z_{ij}$ might be singular. 
	
	To address this issue, define $\theta_0, \varphi_0 \in \big( \frac{1}{2}, \frac{3}{2} \big)$ as follows. If $a \ge b + c$, then set $\theta_0 = 1 = \varphi_0$. Otherwise, set $(\theta_0, \varphi_0)$ to be the unique pair in $\big( \frac{1}{2}, 1 \big) \times \big( 1, \frac{3}{2} \big)$ satisfying $a + b e^{\pi \textbf{i} \theta_0} + c e^{\pi \textbf{i} \varphi_0} = 0$; see \Cref{thetafigure}. Further set $\big( \widehat{\theta}_0, \widehat{\varphi}_0 \big) = (2 - \theta_0, 2 - \varphi_0) \in \big( 1, \frac{3}{2} \big) \times \big( \frac{1}{2}, 1 \big)$, which also satisfies $a + b e^{\pi \textbf{i} \widehat{\theta_0}} + c e^{\pi \textbf{i} \widehat{\varphi_0}} = 0$. The pairs $(\theta_0, \varphi_0)$ and $\big( \widehat{\theta}_0, \widehat{\varphi}_0 \big)$ correspond to the possibly singular terms in the products defining the $Z_{ij}$. 
	
	Let us first approximate the terms in these products that are not close to these singularities. To that end, one can quickly verify since $c \le b \le a = 1$ that
	\begin{flalign}
	\label{thetatheta0estimate}
	\big| a + b + e^{\pi \textbf{i} \theta} + c e^{\pi \textbf{i} \varphi} \big| > \delta^2, \qquad \text{if $\displaystyle\max \Big\{ |\theta - \theta_0|, \big| \theta - \widehat{\theta}_0 \big|\Big\} > 4 \delta$}.
	\end{flalign}
	
	\noindent Combining \eqref{thetatheta0estimate} for $\delta = N^{-2 / 3}$ with the bound 
	\begin{flalign*}
	\big| \log z - \log z' \big| \le |z - z'| \left( \displaystyle\frac{1}{|z|} + \displaystyle\frac{1}{|z'|} \right), \quad \text{for any $z, z' \in \mathbb{C}$},
	\end{flalign*}
	
	\noindent we obtain 
	\begin{flalign*}
	\Big| \log \big( a + b  e^{2 \pi \textbf{i} j / N} + c e^{2 \pi \textbf{i} k /N} \big) - \log (a + b e^{\pi \textbf{i} \theta} + c e^{\pi \textbf{i} \varphi}) \Big| \le \displaystyle\frac{4 \pi}{N^{1 / 3}},
	\end{flalign*} 
	
	\noindent if $\big| \frac{2j}{N} - \theta_0 \big|, \big| \frac{2j}{N} - \widehat{\theta}_0 \big| > 4 N^{-1 / 3}$; $| \theta - \theta_0|, \big| \theta - \widehat{\theta}_0 \big| > 4 N^{-1 / 3}$; and $\big| \frac{2j}{N} - \theta \big|, \big| \frac{2k}{N} - \varphi \big| \le \frac{2}{N}$ all hold. Integrating therefore yields 
	\begin{flalign}
	\label{abcjk}
	\left| 4 N^{-2} \log \big( a + b  e^{2 \pi \textbf{i} j / N} + c e^{2 \pi \textbf{i} k /N} \big) - \displaystyle\int_{\frac{2j}{N}}^{\frac{2j + 2}{N}} \displaystyle\int_{\frac{2k}{N}}^{\frac{2k + 2}{N}} \log \big( a + b e^{\pi \textbf{i} \theta } + c e^{\textbf{i} \pi \varphi} \big) d \theta d \varphi \right| < \displaystyle\frac{16 \pi}{N^{7 / 3}},
	\end{flalign}
	
	\noindent whenever $\big| \frac{2j}{N} - \theta_0 \big|, \big| \frac{2j}{N} - \widehat{\theta}_0 \big| > 4 N^{-1 / 3}$. An analogous statement holds if $\frac{2j}{N}$ is replaced by $\frac{2j + 1}{N}$ or $\frac{2k}{N}$ by $\frac{2k + 1}{N}$, which can be used to evaluate the other $Z_{ij}$ from \eqref{zijdefinition}. 
	
	Next we estimate the maximum possible values of the terms in the products defining the $Z_{ij}$ that are close to the singularities. To that end, observe that 
	\begin{flalign*}
	\displaystyle\max \Bigg\{ \displaystyle\min_{1 \le j \le N} \bigg\{ \Big| \frac{2j}{N} - \theta_0 \Big|, \Big| \frac{2j}{N} - \widehat{\theta}_0 \Big| \bigg\}, \displaystyle\min_{1 \le j \le N} \bigg\{ \Big| \frac{2 j + 1}{N} - \theta_0 \Big|, \Big| \frac{2j + 1}{N} - \widehat{\theta}_0 \Big| \bigg\} \Bigg\} \ge \frac{1}{2N},
	\end{flalign*} 
	
	\noindent and so it follows from \eqref{thetatheta0estimate} that 
	\begin{flalign}
	\label{jj1ntheta} 
	\Bigg| \displaystyle\max \bigg\{ \displaystyle\min_{1 \le j \le N}  \log \big| a + b e^{2 \pi \textbf{i} j / N} + c e^{2 \pi \textbf{i} k / N} \big|, \displaystyle\min_{1 \le j \le N} \log \big| a + b e^{\pi \textbf{i} (2j + 1)  / N} + c e^{2 \pi \textbf{i} k / N} \big| \bigg\} \Bigg| \le 2 \log (8N),
	\end{flalign} 
	
	\noindent for any integer $k \in [0, N]$. Additionally, integrating \eqref{thetatheta0estimate} yields 
	\begin{flalign}
	\label{thetatheta0}
	\begin{aligned}
	\displaystyle\max_{\omega_0 \in \{ \theta_0, \widehat{\theta_0} \}} \displaystyle\int_{|\theta - \omega_0| \le \frac{4}{N^{1 / 3}}} & \displaystyle\int_0^{2 \pi} \Big| \log \big( a + b e^{\pi \textbf{i} \theta } + c e^{\textbf{i} \varphi} \big) d \theta d \varphi \Big| \le \displaystyle\frac{16 \pi \log N}{N^{1 / 3}}.
	\end{aligned}
	\end{flalign} 

	\noindent Now we deduce the lemma from the second statement of \eqref{zzij}; \eqref{zijdefinition}; summing \eqref{abcjk} over all $j$ such that $\big| \frac{2j}{N} - \theta_0 \big|, \big| \frac{2j}{N} - \widehat{\theta}_0 \big| > 4 N^{-1 / 3}$ (or such that $\big| \frac{2j + 1}{N} - \widehat{\theta}_0 \big|, \big| \frac{2j + 1}{N} - \widehat{\theta}_0 \big| > 4 N^{-1 / 3}$, if the second element on the left side of \eqref{jj1ntheta} is larger); summing \eqref{jj1ntheta} over the remaining $j \in [1, N]$; and \eqref{thetatheta0}. 	
\end{proof}

\subsection{Variance Bounds and Proof of \Cref{numbersigma}}

\label{EstimatePn2}

In this section we establish \Cref{numbersigma}. To that end, we begin with the following lemma that estimates the expectations and variances of the $\mathcal{N}_i (\mathscr{M})$ enumerating the tile types of a weighted random lozenge tiling $\mathscr{M}$ as considered in \Cref{PnEstimate}. Variants of this result without the effective error bound of order $N^{-1 / 32}$ were established as Proposition 8.2 and Proposition 8.4 of \cite{VPDT}. Since the proof of the below lemma is similar to that of those results from \cite{VPDT} (with the analogous modifications as implemented in the derivation of \Cref{znzpnp}), we only outline it. 

In what follows, we define $\mathfrak{p}_a = \mathfrak{p}_a (a, b, c)$ and $\mathfrak{p}_b = \mathfrak{p}_b (a, b, c)$ by 
\begin{flalign}
\label{papb} 
& \mathfrak{p}_a = \displaystyle\frac{1}{(2 \pi \textbf{i})^2} \displaystyle\int_{\gamma} \displaystyle\int_{\gamma} \displaystyle\frac{a}{a + bz + cw} w^{-1} z^{-1} dz dw; \qquad \mathfrak{p}_b = \displaystyle\frac{1}{(2 \pi \textbf{i})^2} \displaystyle\int_{\gamma} \displaystyle\int_{\gamma} \displaystyle\frac{bz}{a + bz + cw} w^{-1} z^{-1} dz dw,
\end{flalign}

\noindent where the contour $\gamma$ for $w$ and $z$ is again a positively oriented circle of radius $1$ centered at $0$. 

\begin{lem}
	
	\label{znzpnp2} 
	
	There exists a constant $C > 1$ such that, for any $N \in \mathbb{Z}_{> 1}$ and $a, b, c \in \mathbb{R}_{> 0}$ such that $\max \{ a, b, c \} \le N^{1 / 32} \min \{ a, b, c \}$, the following two statements both (simultaneously) hold for at least one value of $K \in \{ N, N + 1 \}$. 
	
	\begin{enumerate} 
		
		\item Under the measure $\nu = \nu_{a, b, c; K}$, we have that
		\begin{flalign}
		\label{n1n2estimate}
		\Big| K^{-2} \mathbb{E} \big[ \mathcal{N}_1 (\mathscr{M}) \big]  -  \mathfrak{p}_a \Big| < C K^{-1 / 32}; \qquad \Big| K^{-2} \mathbb{E} \big[ \mathcal{N}_2 (\mathscr{M}) \big]  - \mathfrak{p}_b \Big| < C K^{-1 / 32}. 
		\end{flalign}
		
		\item Under the measure $\nu = \nu_{a, b, c; K}$, we have $K^{-4} \Var \big[ \mathcal{N}_i (\mathscr{M}) \big] < C K^{- 1 / 32}$ for each $i \in \{ 1, 2, 3 \}$. 
	\end{enumerate} 
\end{lem}

\begin{proof}[Proof (Outline)] 
	
	The proofs of the first and second statements of this lemma closely follow those of Proposition 8.2 and Proposition 8.4 of \cite{VPDT}, respectively. Therefore, we will only explain how to establish the first and omit the proof of the latter. Furthermore, since the verifications of the two bounds in \eqref{n1n2estimate} are entirely analogous, we will only discuss that of the former. Throughout this proof, we assume by scaling that $0 < c \le b \le a = 1$.
	
	Let us fix a constant $\mu \in (0, 1)$ (which we will later set to $\frac{1}{15}$) and first suppose that $b + c \ge a + 40 N^{- \mu}$. Recall the definitions of $(\theta_0, \varphi_0) \in \big( \frac{1}{2}, 1 \big) \times \big( 1, \frac{3}{2} \big)$ and $\big( \widehat{\theta}_0, \widehat{\varphi}_0 \big)  \in \big( 1, \frac{3}{2} \big) \times \big( \frac{1}{2}, 1 \big)$ from the proof of \Cref{znzpnp} (see \Cref{thetafigure}). The bounds $0 < c \le b \le a = 1$ and $b + c \ge 1 + 40 N^{- \mu}$ then imply that $b \ge \frac{1}{2}$, $c \ge 40 N^{- \mu}$, and $\theta_0 < 1 - 6 N^{- \mu}$, so one can quickly verify for sufficiently large $N$ and any $\theta, \varphi \in [0, 2 \pi]$ that
	\begin{flalign}
	\label{abcestimatelower} 
	\big| a + b e^{\pi \textbf{i} \theta} + c e^{\pi \textbf{i} \varphi} \big| \ge N^{-2 \mu} \min \Big\{ |\theta - \theta_0| + |\varphi - \varphi_0|, \big| \theta - \widehat{\theta}_0 \big| + \big| \varphi - \widehat{\varphi}_0 \big| \Big\}.
	\end{flalign}

	Now, we claim that $\min_{j \in \mathbb{Z}} \big| \theta_0 - \frac{j}{K} \big| > K^{-1 - \mu}$ holds for at least one $K \in \{ N, N + 1 \}$ when $N > 10^{1 / \mu}$. Indeed, assume to the contrary that this estimate is false for both $K \in \{ N, N + 1 \}$. Then, we must have that $\big| \frac{j}{N} - \theta_0 \big| < N^{- 1 - \mu}$ and $\big| \frac{j}{N + 1} - \theta_0 \big| < N^{- 1 - \mu}$ or that $\big| \frac{j}{N} - \theta_0 \big| < N^{- 1 - \mu}$ and $\big| \frac{j + 1}{N + 1} - \theta_0 \big| \le N^{-1 - \mu}$. In the former case, subtracting yields $\frac{j}{N} < 4 N^{-\mu}$, and so $\theta_0 < 5 N^{-\mu}$, which for $N > 10^{1 / \mu}$ contradicts the fact that $\theta_0 > \frac{1}{2}$. In the latter case, subtracting yields $\big| 1 - \frac{j}{N} \big| < 4 N^{-\mu}$, and so $1 - \theta_0 < 5 N^{-\mu}$, which contradicts the fact that $\theta_0 < 1 - 6 N^{-\mu}$. Thus, $\min_{j \in \mathbb{Z}} \big| \theta_0 - \frac{j}{K} \big| > K^{-1 - \mu}$ holds if $K = N$ or if $K = N + 1$; let us assume the former and also that $N$ is odd (the latter is so that the identities below match with those in the proof of \Cref{znzpnp}). 
	
	Next, recall the definitions of the $Z_{ij}$ from \eqref{zijdefinition} in the proof of \Cref{znzpnp}. Since $\mathbb{E} \big[ \mathcal{N}_1 (\mathscr{M}) \big] = \frac{a}{Z} \frac{\partial Z}{\partial a}$, it follows from the first identity in \eqref{zzij} that 
	\begin{flalign}
	\label{cijzij}
	\mathbb{E} \left[ \displaystyle\frac{\mathcal{N}_1 (\mathscr{M})}{N^2} \right] = \displaystyle\frac{1}{N^2} \left( c_{00} \displaystyle\frac{\partial \log Z_{00}}{\partial a} + c_{10} \displaystyle\frac{\partial \log Z_{10}}{\partial a} + c_{01} \displaystyle\frac{\partial \log Z_{01}}{\partial a} + c_{11} \displaystyle\frac{\partial \log Z_{11}}{\partial a} \right),
	\end{flalign} 
	
	\noindent where $c_{ij} = \frac{a Z_{ij}}{2Z}$ if $(i, j) \ne (1, 1)$ and $c_{11} = - \frac{a Z_{11}}{2Z}$. In particular, by \eqref{zzij} and the fact that each $Z_{ij}$ is nonnegative (due to \eqref{zijdefinition}), $c_{ij} \in [-1, 1]$ and $c_{00} + c_{01} + c_{10} + c_{11} = a$. So, it suffices to show that $\big| \frac{a}{N^2} \frac{\partial}{\partial a} \log Z_{ij} - \mathfrak{p}_a \big| < C N^{-1 / 32}$, for each $i, j \in \{ 0, 1 \}$ and some constant $C > 1$. Let us only verify this bound for $(i, j) = (0, 0)$, since the proofs in the other three cases are entirely analogous.
	
	To that end, we have from \eqref{zijdefinition} that
	\begin{flalign}
	\label{z00an}
	\displaystyle\frac{a}{N^2} \displaystyle\frac{\partial \log Z_{00}}{\partial a} = \displaystyle\frac{a}{N^2} \displaystyle\sum_{j = 1}^N \displaystyle\sum_{k = 1}^N \displaystyle\frac{1}{a + b e^{2 \pi \textbf{i} j / N} + c e^{2 \pi \textbf{i} k / N}}, 
	\end{flalign}
	
	\noindent which suggests that it should converge to $\mathfrak{p}_a$ as $N$ tends to $\infty$, except again that the denominators of certain summands can be singular. 
	
	Let us first address the summands that are not close to these singularities. To do this, let $\nu \in (0, 1)$ be a constant (which we will later set to $\frac{1}{5}$) and observe from \eqref{thetatheta0estimate} that if $\big| \frac{2j}{N} - \theta_0 \big| > 4 N^{- \nu}$ and $\big| \frac{2j}{N} - \widehat{\theta}_0 \big| > 4 N^{- \nu}$ then $\big| a + b e^{2 \pi \textbf{i} j / N} + c e^{2 \pi \textbf{i} k / N} \big| > N^{-2 \nu}$. So, it follows from \eqref{abcestimatelower} that, for sufficiently large $N$,
	\begin{flalign*}
	\Bigg| \displaystyle\frac{1}{a + b e^{2 \pi \textbf{i} j / N} + c e^{2 \pi \textbf{i} k / N}} - \displaystyle\frac{1}{a + b e^{\pi \textbf{i} \theta} + c e^{\pi \textbf{i} \varphi}} \Bigg| < \displaystyle\frac{4 \pi}{N^{1 - 4 \nu}},
	\end{flalign*}
	
	\noindent whenever $\big| \frac{2j}{N} - \theta \big|, \big| \frac{2k}{N} - \varphi \big|\le \frac{2}{N}$. Hence, if $\big| \frac{2j}{N} - \theta_0 \big| > 4 N^{- \nu}$ and $\big| \frac{2j}{N} - \widehat{\theta}_0\big| > 4 N^{- \nu}$, then 
	\begin{flalign}
	\label{estimateabcsum2}
	\Bigg| \displaystyle\frac{4}{N^2 \big(a + b e^{2 \pi \textbf{i} j / N} + c e^{2 \pi \textbf{i} k / N} \big)} - \displaystyle\int_{\frac{2j}{N}}^{\frac{2j + 2}{N}} \displaystyle\int_{\frac{2k}{N}}^{\frac{2k + 2}{N}} 	\displaystyle\frac{d \theta d \varphi}{a + b e^{\pi \textbf{i} \theta} + c e^{\pi \textbf{i} \varphi}} \Bigg| < \displaystyle\frac{16 \pi}{N^{3 - 4 \nu}}.
	\end{flalign}
	
	Next, let us estimate the $(j, k)$ terms for which $\frac{j}{N}$ is possibly close to either $\theta_0$ or $\widehat{\theta}_0$. To that end, we use \eqref{abcestimatelower} and the fact that $\max \big\{ \big| \frac{2 j}{N} - \theta_0 \big|, \big| \frac{2 j}{N} - \widehat{\theta_0} \big| \big\} > N^{-1 - \mu}$ to deduce that
	\begin{flalign}
	\label{jj01}
	\begin{aligned}
	\displaystyle\frac{1}{N^2} \Bigg| & \displaystyle\sum_{k = 1}^N \displaystyle\frac{1}{a  + b e^{2 \pi \textbf{i} j / N} + c e^{2 \pi \textbf{i} k /N}} \Bigg| \\
	& \le \displaystyle\frac{2}{N^2}  \displaystyle\sum_{k = 1}^N \Bigg( \displaystyle\frac{1}{N^{-1 - \mu} + N^{-2 \mu} \big|e^{2 \pi \textbf{i} k / N} - e^{\pi \textbf{i} \varphi_0} \big|} + \displaystyle\frac{1}{N^{- 1 - \mu} + N^{-2 \mu} \big| e^{2 \pi \textbf{i} k / N} -  e^{\pi \textbf{i} \widehat{\varphi}_0} \big|} \Bigg) \\
	&  \le \displaystyle\frac{2}{N^{1 - 2 \mu}} \displaystyle\sum_{k = 1}^N \Bigg( \displaystyle\frac{1}{N^{\mu} +  |2k - N \varphi_0 | } + \displaystyle\frac{1}{N^{\mu} + \big| 2k - N \widehat{\varphi}_0 \big| } \Bigg) \le \displaystyle\frac{16 \log N}{N^{1 - 2 \mu}}. 
	\end{aligned} 
	\end{flalign} 
	
	\noindent Additionally, one can quickly deduce from \eqref{abcestimatelower} that 
	\begin{flalign}
	\label{pajj0}
	\Bigg| \displaystyle\int_{|\theta - \theta_0| < 4 N^{- \nu}} \displaystyle\int_0^{2 \pi} \displaystyle\frac{d \theta d \varphi }{a + b e^{\pi \textbf{i} \theta} + c e^{\pi \textbf{i} \varphi}} \Bigg| + \Bigg| \displaystyle\int_{|\theta - \widehat{\theta_0}| < 4 N^{-\nu}} \displaystyle\int_0^{2 \pi} \displaystyle\frac{d \theta d \varphi }{a + b e^{\pi \textbf{i} \theta} + c e^{\pi \textbf{i} \varphi}} \Bigg| \le 256 N^{2 \mu - \nu},
	\end{flalign} 
	
	\noindent for instance by decomposing the integral over $(\theta, \varphi)$ into dyadic annuli centered at $(\theta_0, \varphi_0)$, $\big( \widehat{\theta_0}, \varphi_0 \big)$, $\big( \theta_0, \widehat{\varphi_0} \big)$, and $\big( \widehat{\theta_0}, \widehat{\varphi_0} \big)$.
	
	Then by \eqref{z00an}; summing \eqref{estimateabcsum2} over all $j$ such that $\big| \frac{2j}{N} - \theta_0 \big| > 4 N^{-\nu}$ and $\big| \frac{2j}{N} - \widehat{\theta}_0 \big| > 4 N^{-\nu}$; summing \eqref{jj01} over all remaining $j$; \eqref{pajj0}; and \eqref{papb}, we deduce that 
	\begin{flalign*} 
	\left| \displaystyle\frac{a}{N^2} \displaystyle\frac{\partial}{\partial a} \log Z_{00} - \mathfrak{p}_a \right| \le 256( N^{4 \nu - 1} + N^{2 \mu - \nu} \log N  + N^{2 \mu - \nu}) < N^{-1 / 16},
	\end{flalign*} 
	
	\noindent for sufficiently large $N$, upon setting $\nu = \frac{1}{5}$ and $\mu = \frac{1}{15}$. Similar reasoning provides analogous bounds on $\big| \frac{a}{N^2} \frac{\partial}{\partial a} \log Z_{ij} - \mathfrak{p}_a \big|$, for the remaining $(i, j) \in \{ 0, 1 \}$, and so the first statement of the lemma follows from \eqref{cijzij} if $b + c \ge a + 40 N^{-1 / 15}$. 
	
	Now suppose instead that $a \ge b + c + 40 N^{-1 / 15}$. Then set $a' = b + c + 40 N^{-1 / 15}$ and so, by the above, we have $K^{-2} \big| \mathbb{E} \big[ \mathcal{N}_1 (\mathscr{M}) \big] - \mathfrak{p}_{a'} \big| < K^{-1 / 16}$ under the measure $\nu' = \nu_{a', b, c; K}$ for at least one $K \in \{ N, N + 1 \}$. For this $K$, we have $\mathbb{E}_{\nu} \big[ \mathcal{N}_1 (\mathscr{M}) \big] \ge \mathbb{E}_{\nu'} \big[ \mathcal{N}_1 (\mathscr{M}) \big]$ since $\mathbb{E}_{\nu} \big[ \mathcal{N}_1 (\mathscr{M}) \big]$ is nondecreasing in $a$. Thus, $K^{-2} \mathbb{E}_{\nu} \big[ \mathcal{N}_1 (\mathscr{M}) \big] \ge \mathfrak{p}_{a'} - K^{-1 / 16}$, for sufficiently large $N$.
	
	Using the fact that $\mathfrak{p}_{a'} \ge 1 - C K^{-1 / 32}$ for some constant $C > 1$ (which can be deduced from the definition \eqref{papb} of $\mathfrak{p}_a$ and the facts that $a' = b + c + 40 N^{-1 / 15}$ and $c \ge N^{-1 / 32}$), it follows that $K^{-2} \mathbb{E}_{\nu} \big[ \mathcal{N}_1 (\mathscr{M}) \big] \ge 1 - 2C K^{-1 / 32}$. Thus, $\big| K^{-2} \mathbb{E}_{\nu} \big[ \mathcal{N}_1 (\mathscr{M}) \big] - \mathfrak{p}_a \big| \le 2 C K^{-1 / 32}$, since $1 - C K^{-1 / 32} \le \mathfrak{p}_{a'} \le \mathfrak{p}_a \le 1$. This yields the first part of the lemma; as mentioned previously, we omit the proof of the second since it closely follows that of Proposition 8.4 of \cite{VPDT}.
\end{proof}

Given \Cref{znzpnp} and \Cref{znzpnp2}, the proof of \Cref{numbersigma} closely follows that of Proposition 9.1 and Theorem 9.2 of \cite{VPDT} and so we will only outline it below. 

\begin{proof}[Proof of \Cref{numbersigma} (Outline)]
	
	Throughout this proof, for any $M \in \mathbb{Z}_{\ge 1}$ we abbreviate $\mathfrak{E}_M = \mathfrak{E} (s, t; M^{-1 / 70}; P_M)$. First observe that, if $(s, t) \notin \mathcal{T}_{7 N^{-1 / 32}}$, then Lemma 3.5 of \cite{VPDT} (with the $\varepsilon$ there equal to $7 N^{-1 / 32}$ here) implies that $N^{-2} \log |\mathfrak{E}_N| < C N^{-1 / 32} \log N$ for some constant $C > 1$. Thus, in this case, the proposition follows from the facts that that $\sigma$ is uniformly H\"{o}lder continuous of exponent $\frac{1}{2}$ on $\overline{\mathcal{T}}$ (recall \Cref{concavesigmat}) and that $\sigma (s, t) = 0$ for any $(s, t) \in \partial \mathcal{T}$. 
	
	Hence, we may assume that $(s, t) \in \mathcal{T}_{7 N^{-1 / 32}}$. Set $a, b, c \in \mathbb{R}_{> 0}$ so that $\max \{ a, b, c \} = 1$ and $(s, t) = (\mathfrak{p}_a, \mathfrak{p}_b)$; as indicated in Section 6.23 of \cite{OD} or Theorem 8.3 of \cite{VPDT}, such a choice is known to exist and can be quickly verified to satisfy $\min \{ a, b, c \} > N^{- 1 / 32}$. We will show for some $K \in \{ N, N + 1 \}$ that \eqref{steestimate} holds, with the $N$ there replaced by $K$ here. The verification that this implies the same bound for $K = N$ is omitted, as it is essentially given in the proof of Theorem 4.1 of \cite{VPDT}. 
	
	This will proceed by approximating $K^{-2} \log |\mathfrak{E}_K| \approx K^{-2} \log Z_K (a, b, c) - s \log a - t \log b - (1 - s - t) \log c$. To that end, the first and second parts of \Cref{znzpnp2}, together with a Markov estimate, imply for sufficiently large $N$ that $\mathbb{P} \big[ \mathscr{M} \notin \mathfrak{E}_K \big] \le K^{-1 / 400}$, if $\mathscr{M}$ is sampled under $\nu$; equivalently, 
	\begin{flalign*}
	( 1 - K^{-1 / 400})  Z_K (a, b, c)  \le \exp \left( \displaystyle\sum_{\mathscr{M} \in \mathfrak{E}_K} a^{\mathcal{N}_1 (\mathscr{M})} b^{\mathcal{N}_2 (\mathscr{M})} c^{\mathcal{N}_3 (\mathscr{M})} \right) \le Z_K (a, b, c).
	\end{flalign*}
	
	\noindent Since $\min \{ a, b, c \} > N^{-1 / 32}$, we have that
	\begin{flalign*}
	\Big| K^{-2} \log \big( a^{\mathcal{N}_1 (\mathscr{M})} b^{\mathcal{N}_2 (\mathscr{M})} c^{\mathcal{N}_3 (\mathscr{M})} \big) - \big( s \log a + t \log b + (1 - s - t) \log c \big) \Big| < K^{-1 / 70} \log K, 
	\end{flalign*} 
	
	\noindent for any $\mathscr{M} \in \mathfrak{E}_K$, and so  
	\begin{flalign*}
	& K^{-2} \log \big( (1 - K^{-1 / 400}) Z_K (a, b, c) \big) - \big( s \log a + t \log b + (1 - s - t) \log c\big) - K^{-1 / 70} \log K \\
	& \quad \le K^{-2} \log |\mathfrak{E}_K| \le K^{-2} \log Z_K (a, b, c) - \big( s \log a + t \log b + (1 - s - t) \log c\big) + K^{-1 / 70} \log K.
	\end{flalign*}
	
	\noindent Thus, \Cref{znzpnp} implies for sufficiently large $N$ that 
	\begin{flalign*}
	\bigg| K^{-2} \log |\mathfrak{E}_K| - \Big( \mathfrak{Z} (a, b, c) - ( s \log a + t \log b + (1 - s - t) \log c \big) \Big) \bigg| < K^{-1 / 75}.
	\end{flalign*}
	
	\noindent Hence, the proposition follows from the fact (shown in the proof of Proposition 9.1 of \cite{VPDT}) that $\sigma (s, t) = \mathfrak{Z} (a, b, c) - s \log a - t \log b + (s + t - 1) \log c$, if $a, b, c \in \mathbb{R}_{> 0}$ satisfy $(\mathfrak{p}_a, \mathfrak{p}_b) = (s, t)$. 
\end{proof}

\end{document}